\documentclass[11pt]{article}

\usepackage{fullpage}
\usepackage{amsmath, amsfonts, amssymb, amsthm, amsxtra}
\usepackage{mathrsfs, mathtools}
\usepackage{bbm}
\usepackage{bm}
\DeclareFontFamily{U}{mathx}{}
\DeclareFontShape{U}{mathx}{m}{n}{<-> mathx10}{}
\DeclareSymbolFont{mathx}{U}{mathx}{m}{n}
\DeclareMathAccent{\widehat}{0}{mathx}{"70}
\DeclareMathAccent{\widecheck}{0}{mathx}{"71}

\usepackage[shortlabels]{enumitem}
\usepackage{authblk}
\usepackage{natbib}
\usepackage[usenames, dvipsnames]{xcolor}
\usepackage[hidelinks]{hyperref}
\usepackage{comment}

\numberwithin{equation}{section}

\usepackage{booktabs}
\usepackage{multirow}
\usepackage{tikz}
\usepackage{tkz-graph}
\usepackage{subcaption}
\usepackage{wrapfig}
\usepackage{pgfplots}
\pgfplotsset{compat=newest}
\usepgfplotslibrary{statistics}
\usepgfplotslibrary{fillbetween}
\usetikzlibrary{positioning}
\usepackage{scalerel}
\usetikzlibrary{calc}
\usetikzlibrary{decorations.pathreplacing}
\tikzset{%
glow/.style={%
preaction={#1, draw, line join=round, line width=0.5pt, opacity=0.04,
preaction={#1, draw, line join=round, line width=1.0pt, opacity=0.04,
preaction={#1, draw, line join=round, line width=1.5pt, opacity=0.04,
preaction={#1, draw, line join=round, line width=2.0pt, opacity=0.04,
preaction={#1, draw, line join=round, line width=2.5pt, opacity=0.04,
preaction={#1, draw, line join=round, line width=3.0pt, opacity=0.04,
preaction={#1, draw, line join=round, line width=3.5pt, opacity=0.04,
preaction={#1, draw, line join=round, line width=4.0pt, opacity=0.04,
preaction={#1, draw, line join=round, line width=4.5pt, opacity=0.04,
preaction={#1, draw, line join=round, line width=5.0pt, opacity=0.04,
preaction={#1, draw, line join=round, line width=5.5pt, opacity=0.04,
preaction={#1, draw, line join=round, line width=6.0pt, opacity=0.04,
}}}}}}}}}}}}}}

\usepackage{cypriot}

\newtheorem{theorem}{Theorem}[section]
\newtheorem{definition}[theorem]{Definition}
\newtheorem{lemma}[theorem]{Lemma}
\newtheorem{corollary}[theorem]{Corollary}
\newtheorem{claim}[theorem]{Claim}
\newtheorem{remark}[theorem]{Remark}
\newtheorem{proposition}[theorem]{Proposition}
\newtheorem{conjecture}[theorem]{Conjecture}
\newtheorem{assumption}[theorem]{Assumption}

\newtheorem{problem}[theorem]{Problem}

\newcommand{\cC}{\mathcal{C}}

\newcommand{\cG}{\mathcal{G}}

\newcommand{\cM}{\mathcal{M}}
\newcommand{\cN}{\mathcal{N}}

\newcommand{\cP}{\mathcal{P}}
\newcommand{\cQ}{\mathcal{Q}}

\newcommand{\cT}{\mathcal{T}}

\newcommand{\EE}{\mathbb{E}}

\newcommand{\PP}{\mathbb{P}}

\newcommand{\RR}{\mathbb{R}}
\newcommand{\R}{\mathbbm{R}}

\makeatletter
\DeclareRobustCommand\Equiv{\mathrel{%
  \mathchoice
    {\Equiv@\textfont\displaystyle{.45}}
    {\Equiv@\textfont\textstyle{.45}}
    {\Equiv@\scriptfont\scriptstyle{.5}}
    {\Equiv@\scriptscriptfont\scriptscriptstyle{.55}}
}}
\newcommand{\Equiv@}[3]{%
  \rlap{\raisebox{#3\fontdimen5#12}{$\m@th#2 = $}}%
  \raisebox{-#3\fontdimen5#12}{$\m@th#2 = $}%
}
\makeatother

\newcommand{\KL}{\operatorname{\mathsf{KL}}}
\newcommand{\TV}{\operatorname{\mathsf{TV}}}

\newcommand{\ER}{Erd\H{o}s--R\'enyi}

\newcommand{\card}{\operatorname{\mathsf{card}}}

\newcommand{\aut}{\operatorname{\mathsf{aut}}}
\newcommand{\ud}{\mathrm{d}}
\newcommand{\Proj}{\operatorname{\mathrm{Proj}}}

\newcommand{\Tutte}{\operatorname{\mathsf{T}}}

\newcommand{\Cov}{\operatorname{\mathsf{Cov}}}
\newcommand{\Corr}{\operatorname{\mathsf{Corr}}}
\newcommand{\Var}{\operatorname{\mathsf{Var}}}

\newcommand{\notecm}[1]{{\sf \color{blue}[CM: #1]}}
\newcommand{\notetim}[1]{{\sf \color{red}[Tim: #1]}}

\newcommand{\bitm}{\begin{itemize}[leftmargin=*]}
\newcommand{\eitm}{\end{itemize}}

\newcommand{\benm}{\begin{enumerate}[leftmargin=*]}
\newcommand{\eenm}{\end{enumerate}}

\definecolor{forestgreen}{rgb}{0.13, 0.55, 0.13}

\def\abs#1{\left| #1 \right|}

\newcommand{\inparen}[1]{\left(#1\right)}             %\inparen{x+y}  is (x+y)
\newcommand{\inbraces}[1]{\left\{#1\right\}}           %\inbrace{x+y}  is {x+y}
\newcommand{\insquare}[1]{\left[#1\right]}             %\insquare{x+y}  is [x+y]
\makeatletter
\newcommand{\leqnomode}{\tagsleft@true\let\veqno\@@leqno}
\newcommand{\reqnomode}{\tagsleft@false\let\veqno\@@eqno}
\makeatother
\newcommand{\DrawKtwo}{
    \begin{tikzpicture}[scale=1]
        \SetGraphUnit{1}
        \GraphInit[vstyle=Hasse]
        \SetVertexSimple[MinSize=2pt]
        \Vertex{A}\EA(A){B}\Edge(A)(B)
    \end{tikzpicture}
}

\newcommand{\DrawPtwo}{
    \begin{tikzpicture}
        \SetGraphUnit{1}
        \GraphInit[vstyle=Hasse]
        \SetVertexSimple[MinSize=2pt]
        \Vertex[x = -0.5, y = 0.1]{A}
        \Vertex[x=0,y=0]{B}
        \Vertex[x = 0.5, y = 0.1]{C}
        \Edge(B)(C)
        \Edge(A)(B)
    \end{tikzpicture}
}

\newcommand{\DrawKtwoRep}{
    \begin{tikzpicture}
        \SetGraphUnit{1}
        \GraphInit[vstyle=Hasse]
        \SetVertexSimple[MinSize=2pt]
        \Vertex{A}\EA(A){B}
        \Edge[style={bend left = 10}](A)(B)
        \Edge[style={bend left = 10}](B)(A)
    \end{tikzpicture} 
}

\newcommand{\DrawKtwoRepRep}{
    \begin{tikzpicture}
        \SetGraphUnit{1}
        \GraphInit[vstyle=Hasse]
        \SetVertexSimple[MinSize=2pt]
        \Vertex{A}\EA(A){B}
        \Edge(A)(B)
        \Edge[style={bend left = 15}](A)(B)
        \Edge[style={bend left = 15}](B)(A)
    \end{tikzpicture}
}

\newcommand{\DrawKtwoRTripleRep}{
    \begin{tikzpicture}
        \SetGraphUnit{1}
        \GraphInit[vstyle=Hasse]
        \SetVertexSimple[MinSize=2pt]
        \Vertex{A}\EA(A){B}
        \Edge(A)(B)
        \Edge[style={bend left = 15}](A)(B)
        \Edge[style={bend left = 15}](B)(A)
        \Edge[style={bend left = 30}](A)(B)
    \end{tikzpicture}
}

\newcommand{\DrawPtwoRep}{
    \begin{tikzpicture}
        \SetGraphUnit{1}
        \GraphInit[vstyle=Hasse]
        \SetVertexSimple[MinSize=2pt]
        \Vertex[x = -0.5, y = 0.1]{A}
        \Vertex[x=0,y=0]{B}
        \Vertex[x = 0.5, y = 0.1]{C}
        \Edge(B)(C)
        \Edge[style={bend left = 15}](A)(B)
        \Edge[style={bend left = 15}](B)(A)
    \end{tikzpicture}
}

\newcommand{\DrawKthree}{
    \begin{tikzpicture}
        \SetGraphUnit{1}
        \GraphInit[vstyle=Hasse]
        \SetVertexSimple[MinSize=2pt]
        \Vertex[x = 0, y = 0]{A}
        \Vertex[x=0.5,y=0]{B}
        \Vertex[x = 0.25, y = 0.4]{C}
        \Edge(B)(C)
        \Edge(A)(B)
        \Edge(A)(C)
    \end{tikzpicture}
}

\newcommand{\DrawSthree}{
    \begin{tikzpicture}
        \SetGraphUnit{1}
        \GraphInit[vstyle=Hasse]
        \SetVertexSimple[MinSize=2pt]
        \Vertex[x = 0, y = 0]{A}
        \Vertex[x=0.5,y=0]{B}
        \Vertex[x = 1, y = 0.3]{C}
        \Vertex[x = 1, y = -0.3]{D}
        \Edge(B)(C)
        \Edge(A)(B)
        \Edge(B)(D)
    \end{tikzpicture}
}

\newcommand{\DrawPthree}{
    \begin{tikzpicture}
        \SetGraphUnit{1}
        \GraphInit[vstyle=Hasse]
        \SetVertexSimple[MinSize=2pt]
        \Vertex[x = 0, y = 0]{A}
        \Vertex[x=0.5,y=0]{B}
        \Vertex[x = 1, y = 0]{C}
        \Vertex[x = 1.5, y = 0]{D}
        \Edge(B)(C)
        \Edge(A)(B)
        \Edge(C)(D)
    \end{tikzpicture}
}

\newcommand{\DrawPtwoAdjRepRep}{
    \begin{tikzpicture}
        \SetGraphUnit{1}
        \GraphInit[vstyle=Hasse]
        \SetVertexSimple[MinSize=2pt]
        \Vertex[x = -0.5, y = 0.1]{A}
        \Vertex[x=0,y=0]{B}
        \Vertex[x = 0.5, y = 0.1]{C}
        \Edge[style={bend left = 15}](B)(C)
        \Edge[style={bend left = 15}](C)(B)
        \Edge[style={bend left = 15}](A)(B)
        \Edge[style={bend left = 15}](B)(A)
    \end{tikzpicture}
}

\newcommand{\DrawPtwoTripleRepRep}{
    \begin{tikzpicture}
        \SetGraphUnit{1}
        \GraphInit[vstyle=Hasse]
        \SetVertexSimple[MinSize=2pt]
        \Vertex[x = -0.5, y = 0.1]{A}
        \Vertex[x=0,y=0]{B}
        \Vertex[x = 0.5, y = 0.1]{C}
        \Edge[style={bend left = 0}](B)(C)
        \Edge[style={bend left = 0}](A)(B)
        \Edge[style={bend left = 30}](A)(B)
        \Edge[style={bend left = 30}](B)(A)
    \end{tikzpicture}   
}

\newcommand{\DrawSthreeRep}{
    \begin{tikzpicture}
        \SetGraphUnit{1}
        \GraphInit[vstyle=Hasse]
        \SetVertexSimple[MinSize=2pt]
        \Vertex[x = 0, y = 0]{A}
        \Vertex[x=0.5,y=0]{B}
        \Vertex[x = 1, y = 0.3]{C}
        \Vertex[x = 1, y = -0.3]{D}
        \Edge[style={bend left = 15}](A)(B)
        \Edge[style={bend left = 15}](B)(A)
        \Edge(B)(C)
        \Edge(B)(D)
    \end{tikzpicture}   
}

\newcommand{\DrawPthreeMiddleRep}{
    \begin{tikzpicture}
        \SetGraphUnit{1}
        \GraphInit[vstyle=Hasse]
        \SetVertexSimple[MinSize=2pt]
        \Vertex[x = 0, y = 0]{A}
        \Vertex[x=0.5,y=0]{B}
        \Vertex[x = 1, y = 0]{C}
        \Vertex[x = 1.5, y = 0]{D}
        \Edge[style={bend left = 15}](B)(C)
        \Edge[style={bend left = 15}](C)(B)
        \Edge(A)(B)
        \Edge(C)(D)
    \end{tikzpicture}   
}

\newcommand{\DrawPthreeLeavesRep}{
    \begin{tikzpicture}
        \SetGraphUnit{1}
        \GraphInit[vstyle=Hasse]
        \SetVertexSimple[MinSize=2pt]
        \Vertex[x = 0, y = 0]{A}
        \Vertex[x=0.5,y=0]{B}
        \Vertex[x = 1, y = 0]{C}
        \Vertex[x = 1.5, y = 0]{D}
        \Edge[style={bend left = 15}](A)(B)
        \Edge[style={bend left = 15}](B)(A)
        \Edge(B)(C)
        \Edge(C)(D)
    \end{tikzpicture}   
}

\newcommand{\DrawSfour}{
    \begin{tikzpicture}
        \SetGraphUnit{1}
        \GraphInit[vstyle=Hasse]
        \SetVertexSimple[MinSize=2pt]
        \Vertex[x = 0, y = 0.3]{A}
        \Vertex[x = 0, y = -0.3]{B}
        \Vertex[x=0.5,y=0]{C}
        \Vertex[x = 1, y = 0.3]{D}
        \Vertex[x = 1, y = -0.3]{E}
        \Edge(A)(C)
        \Edge(B)(C)
        \Edge(D)(C)
        \Edge(E)(C)
    \end{tikzpicture}   
}

\newcommand{\DrawSthreeWithTail}{
    \begin{tikzpicture}
        \SetGraphUnit{1}
        \GraphInit[vstyle=Hasse]
        \SetVertexSimple[MinSize=2pt]
        \Vertex[x = 0, y = 0]{A}
        \Vertex[x=0.5,y=0]{B}
        \Vertex[x=1,y=0]{C}
        \Vertex[x = 1.5, y = 0.3]{D}
        \Vertex[x = 1.5, y = -0.3]{E}
        \Edge(A)(B)
        \Edge(B)(C)
        \Edge(C)(D)
        \Edge(C)(E)
    \end{tikzpicture}  
}

\newcommand{\DrawPfour}{
    \begin{tikzpicture}
        \SetGraphUnit{1}
        \GraphInit[vstyle=Hasse]
        \SetVertexSimple[MinSize=2pt]
        \Vertex[x = 0, y = 0]{A}
        \Vertex[x=0.5,y=0]{B}
        \Vertex[x=1,y=0]{C}
        \Vertex[x = 1.5, y = 0]{D}
        \Vertex[x = 2, y = 0]{E}
        \Edge(A)(B)
        \Edge(B)(C)
        \Edge(C)(D)
        \Edge(D)(E)
    \end{tikzpicture}  
}

\newcommand{\simpleTrees}{\operatorname{\mathsf{simpleTrees}}}

\newcommand{\oneRepTrees}{\operatorname{\mathsf{oneRepTrees}}}
\newcommand{\TmRep}{\ensuremath{T_m^{\text{rep}}}}

\newcommand{\twoRepTrees}{\operatorname{\mathsf{twoRepTrees}}}
\newcommand{\tripleEdge}{\operatorname{\mathsf{tripleEdge}}}
\newcommand{\adjDD}{\operatorname{\mathsf{adjDD}}}
\newcommand{\sepDD}{\operatorname{\mathsf{sepDD}}}
\newcommand{\TmTwoRep}{\ensuremath{T_m^{=,=}}}
\newcommand{\TmTripleEdge}{\ensuremath{T_m^{\equiv}}}
\newcommand{\TmAdjDD}{\ensuremath{T_m^{==}}}
\newcommand{\TmSepDD}{\ensuremath{T_m^{= \cdots =}}}

\newcommand{\TmLabRed}{\ensuremath{\widetilde{T}_\text{red}}}
\newcommand{\TmLabBlue}{\ensuremath{\widetilde{T}_\text{blue}}}

\newcommand{\combined}{\operatorname{\mathsf{combined}}}
\newcommand{\combAdjDD}{\operatorname{\mathsf{combAdjDD}}}
\newcommand{\combSepDD}{\operatorname{\mathsf{combSepDD}}}

\newcommand{\signedKtwo}{\ensuremath{\widecheck{K_2}}}
\newcommand{\signedPtwo}{\ensuremath{\widecheck{P_2}}}

\title{Cluster expansion of the log-likelihood ratio: Optimal detection of planted matchings}

\makeatletter
\renewcommand\AB@affilsepx{, \@gobble} % \@gobble removes the next token (the footnote mark)
\makeatother

\author{Timothy L.~H.~Wee\thanks{Email: \url{timothy.wee@gatech.edu}} \ and Cheng Mao\thanks{Email: \url{cheng.mao@math.gatech.edu}}}
\affil{School of Mathematics, Georgia Institute of Technology}
\date{}
\begin{document}

\maketitle

\begin{abstract}
% To understand how hidden information can be extracted from statistical networks, planted models in random graphs have been the focus of intensive study in recent years. In this work, we consider the detection of a planted matching, i.e., an independent edge set, hidden in an Erd\H{o}s--R\'enyi random graph, which is formulated as a hypothesis testing problem. We identify the critical regime for this testing problem and prove that the log-likelihood ratio is asymptotically normal. Our main technical tool is the cluster expansion from statistical physics, which we use to analyze the log-likelihood ratio and show that it is dominated by the edge or wedge count. 
% \bigskip
To understand how hidden information can be extracted from statistical networks, planted models in random graphs have been the focus of intensive study in recent years. In this work, we consider the detection of a planted matching, i.e., an independent edge set, hidden in an Erd\H{o}s--R\'enyi random graph, which is formulated as a hypothesis testing problem. We identify the critical regime for this testing problem and prove that the log-likelihood ratio is asymptotically normal. Via analyses of computationally efficient edge or wedge count test statistics that attain the optimal limits of detection, our results also reveal the absence of a statistical-to-computational gap. Our main technical tool is the cluster expansion from statistical physics, which allows us to prove a precise, non-asymptotic characterization of the log-likelihood ratio. Our analyses rely on a careful reorganization and cancellation of terms that occur in the difference between monomer-dimer log partition functions on the complete and \ER~graphs. This combinatorial and statistical physics approach represents a significant departure from the more established methods such as orthogonal decompositions, and positions the cluster expansion as a viable technique in the study of log-likelihood ratios for planted models in general.
% This is accomplished by modeling the matching from the monomer-dimer model, and considering the combinatorial cancellations arising from the difference between two log partition functions of the monomer-dimer model. These techniques position cluster expansion as a viable alternative to established orthogonal decomposition methods for studying log-likelihood ratios. 
\end{abstract}

\tableofcontents

\section{Introduction}
Finding hidden information in networks is a central task in the study of statistical networks. In recent years, \emph{planted models} in random graphs have received considerable attention and resulted in a plethora of theoretical and algorithmic innovations. The most well-known of these is the planted clique problem \cite{jerrum1992large,alon1998finding}, which presents a celebrated statistical-to-computational gap whose full resolution remains elusive. 
% We discuss the planted clique from a cluster expansion perspective in Section \ref{sec:plantedClique}.
Other related examples include 
% Thresholds and algorithms for the detection problem have been extensively studied for a wide variety of planted subgraphs. These include 
planted dense subgraphs \cite{bhaskara2010detecting} or community detection \cite{arias2014community}, and planted partitions or stochastic block models \cite{abbe2018community}. 
Unlike the above models that possess low-rank structures, other planted combinatorial structures have also been studied more recently, such as planted Hamiltonian cycles \cite{bagaria2020hidden,ding2020consistent} or small-world networks \cite{mao2023detection}, planted trees \cite{massoulie2019planting,moharrami2025planted}, and planted $k$-factors \cite{gaudio2025finding,gaudio2025all}.

\subsection{Planted matching}
This work primarily focuses on the \emph{planted matching} model, which belongs to the latter class where the planted subgraph is characterized by local combinatorial constraints. 
More specifically, a matching refers to an \emph{independent edge set} consisting of edges that are not adjacent to each other, and it is planted in an otherwise random \ER\ graph. 
A weighted bipartite version of this model was considered by \cite{chertkov2010inference} to study tracking mobile objects such as particles in turbulent flows. 
The task was to recover the latent matching between two sets of spatial points, representing two consecutive snapshots of a random dynamical system of particles. 
This task corresponds to the \emph{recovery} problem for the planted matching model, that is, to estimate the hidden matching given the graph. 
Towards this end, there has been a line of research \cite{semerjian2020recovery,moharrami2021planted,ding2023planted} in recent years studying information-theoretic thresholds and algorithms for planted matching.

We instead consider the \emph{detection} problem for the planted matching model, formulated as hypothesis testing: given a graph $A$ on $n$ vertices, we test the null hypothesis that $A$ is a purely random \ER\ graph $G(n,q)$ against the alternative hypothesis that $A$ is an \ER\ graph $G(n,p)$ containing a hidden planted matching $M$. 
To ease the discussion, let us consider the case where the planted matching $M$ contains $\Theta(n)$ edges (note that the maximum size of a matching is $\lfloor n/2 \rfloor$), and where $p$ and $q$ are defined so that the two models have the same average edge densities. 

The detection problem turns out to be significantly different from the recovery problem in terms of the critical thresholds. 
In view of \cite[Remark 2]{ding2023planted}, the threshold for (almost exact and partial) recovery occurs at the order $p = \Theta(1/n)$. 
Moreover, it is a classical result \cite{erdHos1966existence,frieze2015introduction} that $q=(\log n)/n$ is the threshold above which a perfect matching exists (for $n$ even) with high probability in the null model $G(n,q)$. 
However, for the detection problem, we show that the critical threshold is $p = \Theta(1/\sqrt{n})$. 
In particular, if $(\log n)/n \ll p \ll 1/\sqrt{n}$, there exist many matchings of size $\Theta(n)$ in both the null and alternative models, but we are still able to test consistently whether one additional matching $M$ is planted or not. 
To the best of our knowledge, the testing threshold $p = \Theta(1/\sqrt{n})$ has not been identified in the literature on planted models, except when $n$ is even and $M$ is a perfect matching. In that setting, the threshold appears implicitly as an intermediate result in \cite{janson1994numbers}, which studies the number of perfect matchings in an \ER\ graph. 
Crucially, our main technique differs fundamentally from that in \cite{janson1994numbers}, and we discuss the connection in more detail in Section~\ref{sec:perfect-matching-results}.

Furthermore, in the critical regime $p = \Theta(1/\sqrt{n})$, we study the log-likelihood ratio $\log \frac{\ud \cP}{\ud \cQ}$ (with $\cQ$ denoting the null and $\cP$ denoting the alternative) and show that it is dominated by a simple statistic---the \emph{signed wedge\footnote{A wedge refers to a path of length two, denoted by $P_2$.} count} $\signedPtwo(A) := \sum_{j \in [n]} \sum_{\inbraces{i,k} \in \binom{[n] \setminus \{j\}}{2} } (A_{ij} - q) (A_{jk} - q)$. 
Our main result in this regime states that, for $A \sim \cQ$, the log-likelihood ratio satisfies 
\begin{equation}
\log \frac{\ud \cP}{\ud \cQ}(A) \approx - \frac{\sigma^2}{2} + \sigma \frac{\signedPtwo(A)}{\sqrt{\Var \signedPtwo(A)}} ,
\label{eq:log-likelihood-ratio-signed-wedge-count-informal}
\end{equation}
where $\sigma \approx \frac{1}{\sqrt{2n} q} \big(\frac{2 \abs{M}}{n}\big)^2$ and an $O_\PP \big( \frac{1}{\sqrt{n p}} \big)$ lower-order term is omitted for brevity. 
Since the likelihood ratio test is statistically optimal by the Neyman--Pearson lemma, the above approximation has several important consequences for our testing problem:
\begin{itemize}
\item 
The signed wedge count is a degree-two polynomial in $(A_{ij})$ and can be efficiently computed, so there is no statistical-to-computational gap for this detection problem. 

\item
The standardized statistic $\frac{\signedPtwo(A)}{\sqrt{\Var \signedPtwo(A)}}$ is asymptotically $\cN(0,1)$ by \cite{janson1994orthogonal}, from which it follows that the log-likelihood ratio is asymptotically $\cN(-\sigma^2/2, \sigma^2)$ for $A \sim \cQ$. 
The relation that the mean is $-1/2$ of the variance is the special condition that gives mutual \emph{contiguity} between $\cQ$ and $\cP$ in Le Cam's framework of local asymptotic normality \cite{cam1960locally,le2000asymptotics}. 
By Le Cam's third lemma \cite[Example~6.7]{van2000asymptotic}, we then see that the log-likelihood ratio is also asymptotically normal for $A \sim \cP$. 

\item
As a result of the asymptotic normality of the likelihood ratio, we can derive the precise asymptotic testing error, or, equivalently, the asymptotic total variation distance between $\cQ$ and $\cP$ with sharp constants in the critical regime.
\end{itemize}

\subsection{Cluster expansion}
To prove the approximation of the likelihood ratio \eqref{eq:log-likelihood-ratio-signed-wedge-count-informal}, we use the \emph{cluster expansion} technique from statistical physics. 
Briefly, the cluster expansion is a \emph{formal} series expansion of the logarithm of a partition function.
% Briefly, the cluster expansion is a formal procedure for taking the logarithm of a partition function. 
It is particularly useful when the partition function can be expressed as a sum over geometrical objects, abstractly called \emph{polymers}, whose interactions can be described in a pairwise manner. We refer to \cite[Chap.~5]{friedli2017statistical} and \cite{brydges1984short, faris2010combinatoricsAndCE} for general references on cluster expansions, and to \cite{gruber1971general, kotecky1986cluster} for the polymer formulation. 
While the cluster expansion has been applied to study statistical physics models on random graphs \cite{helmuth2023finite}, and to analyze certain signed subgraph counts \cite{bangachev2024detection}, we are not aware of any previous use of it to study the log-likelihood ratio for a planted model. 
We believe that applying the cluster expansion in statistical analysis is interesting in its own right and has the potential to open a new line of research.

More specifically, in our context, the cluster expansion of the log-likelihood ratio takes the form 
\begin{equation}
\log \frac{\ud \cP}{\ud \cQ}(A) = F(A) + \sum_{m \ge 1} \sum_{e_1,\dots,e_m} \phi(H(e_1,\dots,e_m)) \lambda^m \bigg( \frac{1}{p^m} \prod_{j=1}^{m} A_{e_j} - 1 \bigg) ,
\label{eq:log-likelihood-ratio-cluster-expansion-introduction}
\end{equation}
where (i) $F(A) := |A| \log \frac{p(1-q)}{q(1-p)} + \binom{n}{2} \log \frac{1-p}{1-q}$, which depends only on the number of edges $|A|$, (ii) the inner sum is over possibly repeated edges $e_1, \dots, e_m$ in $\binom{[n]}{2}$ that form a connected multigraph called a \emph{cluster}, (iii) $\phi(H(e_1,\dots,e_m))$ is known as the \emph{Ursell function}, which is related to cumulants, and (iv) $\lambda$ is a parameter determining the size of $M$. 
These definitions will be made precise in Section~\ref{sec:intro-cluster-expansion} where we formally introduce the cluster expansion. 
Note that each summand on the right-hand side of \eqref{eq:log-likelihood-ratio-cluster-expansion-introduction} includes the indicator $\prod_{j=1}^{m} A_{e_j}$ of the cluster $(e_1, \dots, e_m)$, so \eqref{eq:log-likelihood-ratio-cluster-expansion-introduction} can be understood as a weighted sum of subgraph counts if the sum is reorganized as follows: 
\begin{equation}
\log \frac{\ud \cP}{\ud \cQ}(A) = F(A) + \sum_{m \ge 1} \sum_{G : |G| = m} \phi(H(G)) \lambda^m \bigg( \frac{1}{p^m} \widetilde{G}(A) - \widetilde{G}(K_n) \bigg) ,
\label{eq:log-likelihood-ratio-cluster-expansion-subgraph-counts}
\end{equation}
where (i) the inner sum is over unlabeled multigraph $G$ with $m$ edges, (ii) the Ursell function can be written as $\phi(H(G))$ because it only depends on the shape of the cluster, not the labeling, and (iii) $\widetilde{G}(A) := \sum_{(e_1, \dots, e_m) \cong G} \prod_{j=1}^{m} A_{e_j}$ and $\widetilde{G}(K_n)$ is defined similarly for the complete graph $K_n$. 

The above expansion in terms of subgraph counts is reminiscent of the \emph{orthogonal decomposition} of functions on random graphs, first introduced in a series of works by Janson \cite{janson1994orthogonal,janson1994numbers} and more recently widely applied to study planted models \cite{hopkins2018statistical,kunisky2019notes,wein2025computational}. 
The comparison between the two expansions is of considerable interest, which we discuss in Section~\ref{sec:comparison-to-orthogonal-expansion}. 

To prove our main result \eqref{eq:log-likelihood-ratio-signed-wedge-count-informal} using the cluster expansion, it suffices to show that the sum in \eqref{eq:log-likelihood-ratio-cluster-expansion-subgraph-counts} is dominated by the signed wedge count $\signedPtwo(A)$. We remark that this is not simply done by showing that the dominating term in \eqref{eq:log-likelihood-ratio-cluster-expansion-subgraph-counts} corresponds to $m=2$ and $G$ being a wedge. Instead, the terms with $G$ being a \emph{tree} all contribute nontrivially to the log-likelihood ratio. However, these tree terms are all asymptotically perfectly correlated with the signed wedge count $\signedPtwo(A)$, thereby yielding the claimed result. The proof ideas are given in Section \ref{sec:nonMatchingEdges_intuition}.
% We compute the first few terms in the cluster expansion explicitly in Section~\ref{sec:nonMatchingEdges_intuition} for the simpler $p=q$ setting to 

\subsection{Related work}

\paragraph{Planted matchings in random graphs} As discussed above, the recovery problem for planted matchings has origins in statistical physics \cite{chertkov2010inference,semerjian2020recovery} with applications in tracking trajectories of particles. 
% The task is to recover the latent matching between two sets of spatial points, representing two consecutive snapshots of a random dynamical system of particles. 
The task can be interpreted as recovering a planted matching in a complete bipartite graph given its random weighted adjacency matrix, with planted and non-planted edges distinguished by having different distributions. 

In a sequence of recent papers \cite{moharrami2021planted,ding2023planted}, general information-theoretic thresholds were obtained in terms of the Bhattacharyya distance between the planted and non-planted edge distributions.
% Generally, these settings involve an underlying graph which contains a latent matching of interest. The edges in the matching are distinguished by having weights drawn from a distribution $\cP$, while the rest of the edges have weights drawn from another distribution $\cQ$. 
% Rigorous characterizations in the independent edges setting were subsequently given by \cite{ding2023planted}, which presented information-theoretic thresholds for the recovery of the planted matching in terms of the Bhattarchaya distance between $\cP$ and $\cQ$. 
More refined results were obtained for the case of exponentially distributed weights. In particular, the error curve for the fraction of correctly recovered planted edges for the maximum likelihood estimator (efficiently computable as a linear assignment problem) was shown 
% in the almost exact and partial recovery phase 
to be related to a system of ODEs arising as fixed point equations of a message-passing algorithm on a planted version of Aldous's \emph{Poisson-weighted infinite tree} \cite{Aldous2004}. 
% Furthermore, the threshold achieved by MLE is found to be information-theoretically optimal. 

Moreover, a variation of the problem with Gaussian weights was investigated in \cite{dai2023gaussian}, with applications to database alignment. 
Edge weights with dependencies, more closely aligned with the original formulation in \cite{chertkov2010inference}, were also studied in the context of geometric planted matchings by \cite{kunisky2022strong,dai2023gaussian,wang2022random}. 
Thresholds for 
% perfect and strong 
recovery, as well as error bounds, were obtained in terms of the ambient dimension of the particles. 
% and parameters related to their average displacement between time steps.
% The recovery problem for planted matchings in random graphs is briefly discussed in Remark BLA in \cite{mossel2025sharp} in the BLA.

% Our results complement these recovery problems by establishing the similar absence of a statistical-to-computational gap in the detection problem.

The detection problem has received far less attention than the recovery problem. As alluded to earlier, it was implicitly studied in \cite{janson1994numbers}, whose results and techniques bear an interesting comparison to ours. See Section~\ref{sec:perfect-matching-results} for more details.

\paragraph{Cluster expansion applications} 
The use of cluster expansions in statistical mechanics is vast and spans many decades. We mention only two recent instances of its applicability in the monomer-dimer model, which is the model we use for random matchings. Cluster expansion was used in a lattice version of this model to study correlation decay \cite{quitmann2024decay}, and also in a variant with short-range attractive interactions to study liquid-crystal properties \cite{alberici2016cluster}.

Outside its traditional sphere of influence, cluster expansion techniques have found great effect in combinatorics, algorithms, random graphs and various other fields. The influential work of \cite{scott2005repulsive} established striking connections between the zero-free region of the hard-core lattice gas partition function, convergence of the cluster expansion of its logarithm, Shearer's theorem, and the Lov{\'a}sz local lemma. The cluster expansion has also been applied to study sampling from the Potts model on expanders at low temperature \cite{jenssen2020algorithms}, structural properties and asymptotic enumeration of triangle-free graphs \cite{jenssen2025evolution}, precise phase coexistence characterizations in the random cluster model on random graphs \cite{helmuth2023finite}, independent sets in the hypercube \cite{jenssen2020independent,balogh2016applications}, and free energies in mean-field disordered systems \cite{dey2023mean, aizenman1987some}.

% The latter two works build upon Sapozhenko's graph container lemma to establish convergence of the cluster expansion. 
One of the goals of this paper is to bring these powerful cluster expansion techniques to the fore in statistics by demonstrating their effectiveness in a classical hypothesis testing framework.

Ideas from the cluster expansion are also used in \cite{bangachev2024detection} albeit in a very different manner---in their case, several steps from the formal derivation of the cluster expansion are used to give an expansion of certain expected signed subgraph counts under a random geometric graph model. Notably, this does not involve taking the logarithm of a grand canonical partition function or addressing the related questions of convergence.

\paragraph{Asymptotic distributions of log-likelihood ratios} 
The asymptotic distribution of the log-likelihood ratio is a central problem in hypothesis testing with a celebrated result due to Wilks \cite{wilks1938large}. Recent studies have focused on log-likelihood ratios in \emph{high-dimensional} versions of widely used statistical procedures, for instance, covariance testing \cite{bai2009corrections}, testing between Gaussians \cite{jiangYang2013central}, and logistic regression \cite{sur2019likelihood}. 
% It is a central problem in hypothesis testing to determine asymptotic distributions of log-likelihood ratios (LLR's), motivated in large part by likelihood ratio test statistics. A celebrated result in this area is Wilks's theorem \cite{wilks1938large}. More recent studies have focused on LLR's in \emph{high-dimensional} versions of classical statistical frameworks, for instance covariance testing \cite{bai2009corrections}, testing between Gaussians \cite{jiangYang2013central}, and logistic regression \cite{sur2019likelihood}. 

A line of work, more similar in spirit to this paper, studies log-likelihood ratios in signal detection in \emph{spiked} random matrix models \cite{onatski2013asymptotic,johnstoneOnatski2020testing,ElAlaoui20fundamental,banerjeeMa2022optimal,lu2023contextual}. In particular, \cite{banerjeeMa2022optimal} analyzes the asymptotic testing error attained by linear spectral statistics (positive result) and further establishes their optimality by computing the asymptotic  distribution of the log-likelihood ratio using a second moment method related to \cite{janson1995random} (negative result). This parallels the structure of this paper where our positive result follows from analyses of computationally tractable statistics.

Notable differences (aside from clearly different settings) are that (i) there is typically an absence of low-rank structure in many planted subgraph problems, including those considered in the present paper, and (ii) our techniques for analyzing the log-likelihood ratio are very different. For example, Gaussianity is used in \cite{banerjeeMa2022optimal} to decompose the log-likelihood ratio into \emph{bipartite signed cycle counts}, and it is also exploited in \cite{ElAlaoui20fundamental} through Gaussian interpolation techniques with connections to mean-field spin glasses. This paper instead leverages the connection between the log-likelihood ratio and abstract polymer models with pairwise interactions from statistical physics, which are amenable to cluster expansion techniques. 

% ADD WEI-KUO CHEN*, MADELINE HANDSCHY† AND GILAD LERMAN for more gaussian interpolation
% And Detection of Signal in the Spiked Rectangular Models Jung for resolvent method
% NVM these papers do not determine LLR

% One contrast is that we take the log in the preliminary step (at least formally), then collect the dominant terms. The approach in \cite{banerjeeMa2022optimal} \cite{janson1995random} typically requires one to `foresee' the dominant terms beforehand, and further utilizes intricate details about the joint distribution of these dominant terms in order to interpret the likelihood ratio as the generating function of a basis of Hermite polynomials so that it can be coerced into $\exp(\cdots)$ and then the log can be taken.

% In that regard (about partition functions), 
% \cite{abbe2022proof} showing log-normal limit in the binary symmetric perceptron. Mention also Aizenmann Lebowitz Ruelle that also used cluster expansion. Qiang Wu.

\paragraph{Other planted models}
The recent literature on planted models is extensive, and we focus here on the works most closely related to ours. 
For the detection of planted subgraphs, many specific models have been considered, and unifying frameworks have also been proposed by \cite{elimelech2025detecting,yu2025counting} to study either information-theoretic or computational thresholds. 
However, most existing results either suggest an all-or-nothing phenomenon for a planted model (such as the well-known $2 \log_2 n$ threshold for planted clique) or only determine the order at which the phase transition occurs. 
Notable exceptions include, for example, \cite{moitraWein2025precise,mossel2025weak}, which study the precise testing error at the critical threshold. 
For planted matchings, we can determine the testing error with sharp constants thanks to the asymptotic normality of the likelihood ratio, and, in particular, reveal a smooth phase transition in the critical regime. 
At a high level, this is in line with the ``infinite-order phase transition'' for the recovery of a planted matching \cite{ding2023planted}.
%but its nature remains elusive.

Hypothesis testing with a planted signal, although not necessarily involving graph structure, has also been studied, for example, in \cite{perkins2013forgetfulness,addarioberry2010onComb}. 
% \notetim{When discussing Will's balls bins paper, we might say that restricting to degree based statistics appears to work for the planted matching detection. But not expected to hold for general $k$-factors. E.g. For 2-factors, if triangle factors \cite{krivelevich1997triangle} then signed triangles dominates signed wedge. While in Hamiltonian cycle $P_2$ is dominates.}
The model in \cite{perkins2013forgetfulness} can be seen as a planted subgraph model where only the vertex degrees are observed (barring technical differences). 
It is shown that a degree-two polynomial of the degrees is the optimal statistic, which corresponds precisely to the signed wedge count statistic we use. 
However, the analysis of the likelihood ratio, which is our main contribution, is far more involved when a full graph is observed instead of only the degrees. 
In \cite{addarioberry2010onComb}, a planted vector model with Gaussian noise is studied and can be applied to obtain results for planted perfect matchings (see Section~4.3 of that paper), but the results are not directly comparable to ours.

Finally, there is a plethora of recent works using \emph{subgraph counts} or \emph{network motifs} as efficient statistics for detection of planted structures, many of which are based on the orthogonal decomposition \cite{janson1994orthogonal} and the low-degree polynomial framework \cite{hopkins2018statistical,wein2025computational}. 
Examples of such subgraphs include self-avoiding walks for community detection \cite{hopkins2017efficient}, stars as an optimal statistic among all constant-degree statistics \cite{yu2025counting}, balanced subgraphs for detecting a planted dense or general subgraph \cite{dhawan2025detection,elimelech2025detecting}, trees for detecting correlations between random graphs \cite{mao2024testing}, and triangles or four-cycles for detecting latent geometry in random graphs \cite{bubeck2016testing,bangachev2024detection}. 
The cluster expansion such as \eqref{eq:log-likelihood-ratio-cluster-expansion-introduction} for planted matchings also involves subgraph counts, so it may guide the design of low-degree statistics and algorithms in a way similar to the orthogonal decomposition---we discuss this point in Section~\ref{sec:intro-cluster-expansion}.

\subsection{Notation}

We use the standard big-O notation $O(\cdot)$, $o(\cdot)$, $\Theta(\cdot)$, $\dots$ for quantities depending on $n$ as $n \to \infty$. 
Let $\Phi$ denote the standard Gaussian cumulative distribution function (CDF). 
% The complementary error function is $\mathsf{erfc}(y) = \frac{2}{\sqrt{\pi}} \int_y^{\infty} \exp(-u^2) \ud u$.
% We use the convention that sums $\sum_{m =1}^{c \log n}$ are taken to mean $\sum_{m=1}^{\left\lceil c \log n \right\rceil}$. 
Let $K_n$ denote the complete graph on the vertex set $[n] := \{1, \dots, n\}$. For a graph $G$, we sometimes use the same notation $G$ for the graph itself, its edge set, and its adjacency matrix when there is no ambiguity. 
For an unlabeled, simple, template subgraph $G$, and for $A \sim G(n,q)$, define the subgraph count, the \emph{centered} subgraph count, and the \emph{signed} subgraph count respectively as follows:
\begin{align}\label{eq:ordinaryCenteredSignedSubgraphCount_definition}
    G(A) = \sum_{\substack{G' \subseteq K_n \\ G' \cong G } } \prod_{\inbraces{i,j}\in G'} A_{ij}, \quad \overline{G}(A) = G(A) - \EE G(A), \quad \text{and} \quad  \widecheck{G}(A) = \sum_{\substack{G' \subseteq K_n \\ G' \cong G } } \prod_{\inbraces{i,j}\in G'} \inparen{A_{ij} - q}.
\end{align}
We write $\aut(G)$ to denote the number of automorphisms of $G$. 
% Of particular importance are the signed subgraph counts related to the templates $G = K_2, P_2$, which we spell out:
% \begin{align}\label{eq:signedKtwo_signedPtwo_definition}
%     \signedKtwo(A) &= \sum_{\inbraces{i,j}} (A_{ij} - q) \qquad \text{and} \qquad 
%     \signedPtwo(A) = \sum_{j \in [n]} \sum_{\inbraces{i,k} \in \binom{[n] \backslash j}{2} } (A_{ij} - q) (A_{jk} - q),
% \end{align}
% Note that $\signedKtwo(A) = \overline{K}_2(A)$, although we typically use the former notation. 
Throughout the paper, we write $P_m$, $S_m$, and $T_m$ to refer respectively to an unlabeled path, star, and tree with $m$ edges. 
% The notation $\TmRep$ refers to an unlabeled multi-tree with $m+1$ edges and vertices (exactly one repeated edge). 

\section{Main results for detecting a planted matching}
\label{sec:main-results-matching}

\subsection{Problem formulation}

Let us start by defining the model for a random matching, known as the \emph{monomer-dimer model} in statistical physics. This has antecedents in lattice chemistry (see e.g.~\cite{kasteleyn1961statistics,fisher1961statistical}) but its modern mathematical formulation can be traced to \cite{heilmann1972theory}. The latter contains the seminal Heilmann-Lieb theorem on the location of the zeros of the monomer-dimer partition function. The partition function is also referred to as the \emph{matching polynomial} in algebraic graph theory \cite{farrell1979introduction,godsil1978matching}.

\begin{definition}[The monomer-dimer model for a random matching] \label{def:monomer-dimer}
    For a simple graph $G$, for dimer density $\lambda > 0$, the monomer-dimer Gibbs measure $\mu_\lambda = \mu_{\lambda,G}$ is a probability measure over matchings in $G$ given by
\begin{equation*}
    \mu_\lambda(M) = \frac{\lambda^{\abs{M}}}{Z_{G} (\lambda) }, \qquad \text{where} \qquad Z_{G} (\lambda) := \sum_{M \subset G} \lambda^{^{\abs{M}}},
\end{equation*}
where $|M|$ denotes the size of $M$, i.e., the number of edges in $M$, and the sum is over all possible (labeled) matchings $M$ in $G$.
\end{definition}

The model for a planted matching in a random graph is defined as follows.

\begin{definition}[The planted matching model] \label{def:planted-distribution}
For a positive integer $n$, $p \in (0,1)$, and $\lambda > 0$, the planted distribution $\cP_\lambda$ is the distribution of a random graph on $n$ vertices consisting of a matching $M \sim \mu_\lambda$ 
planted in an \ER~random graph $G(n,p)$, where $\mu_\lambda = \mu_{\lambda, K_n}$ is the monomer-dimer Gibbs measure on the complete graph $K_n$ given in Definition~\ref{def:monomer-dimer}. 
More precisely, let $A$ denote the adjacency matrix of a random graph from $\cP_\lambda$. Conditional on $M$, we have $A_{ij} = 1$ if $\{i,j\} \in M$ and $A_{ij} \sim \mathrm{Bernoulli}(p)$ independently if $\{i,j\} \notin M$. 
% We write $A \sim \cP_\lambda = M + G(n,p)$ for brevity.
\end{definition}

The detection of a planted matching is formulated as a hypothesis testing problem between two distributions $\cP_\lambda$ and $\cQ$. 

\begin{problem}[Detection of a planted matching] \label{prob:detection}
For a positive integer $n$, $p, q \in (0,1)$, and $\lambda > 0$, let $\cP_\lambda$ denote the planted model in Definition~\ref{def:planted-distribution}, and let $\cQ$ denote the \ER~random graph model $G(n,q)$. Given a random graph $A$, we test the null hypothesis $H_0 : A \sim \cQ$ against the alternative hypothesis $H_1 : A \sim \cP_\lambda$. 
% The null distribution $\cQ$ is a plain \ER~random graph. We refer to this problem as
% \begin{equation}\label{eq:detectionProblem_generalForm}
%   \tag{$\cP_\lambda \text{ v.s.~} \cQ$}
%   \begin{aligned}
%     \cP_\lambda &= M + G(n,p), \qquad \text{where } M \sim \mu_\lambda \\
%     \cQ &= G(n,q).
%   \end{aligned}
% \end{equation}
\end{problem}

Before proceeding to our main results for the detection of a planted matching, let us first build intuition for how the parameters scale in the planted matching model. 
Note the maximum size of a matching in $K_n$ is $\lfloor n/2 \rfloor$. It is easily seen that, as $\lambda \to \infty$ in Definition~\ref{def:monomer-dimer}, the Gibbs measure $\mu_\infty$ becomes the uniform distribution over perfect matchings. Less intuitively, as soon as $\lambda$ is of order $1/n$, the typical size of $M \sim \mu_\lambda$ is of order $n$. 
In this regime, the results from \cite{alberici2014mean} for the ``pure hard-core monomer-dimer model'' (in their terminology) establish the \emph{thermodynamic limits} for $n^{-1} \log Z_{K_n}(\lambda)$ and $2 \EE \abs{M}/n$ as $n \to \infty$. 
We map their results into our notation in Appendix~\ref{sec:thermodynamicLimits_MDmodel}. More precisely, we have the following result for $\EE |M|$ (see Theorem~\ref{thm:thermodynamicLimits_MD}).
% under Assumption \ref{asmpt:nonMatchingEdges},
% \begin{align}\label{eq:MDMean_convergence_to_c}
%     \lim_{n \rightarrow \infty} \frac{2\EE \abs{M}}{n} = c, \qquad \text{where} \qquad c = c(b) := 1 - \frac{1}{2} \inparen{  \sqrt{4 e^{2(1+b)} +  8e^{1+b} } - 2 e^{1+b}  } \in (0,1).
% \end{align}

\begin{lemma}
\label{lem:matching-limiting-size}
For $\zeta > 0$, suppose 
$$
\lambda = \lambda_n := \frac{1}{\zeta n }. 
$$
Then we have that 
\begin{equation}
\label{eq:MDMean_convergence_to_c}
\lim_{n \rightarrow \infty }\frac{2 \EE_{\mu_\lambda} \abs{M}}{n} = c \in (0,1) , \qquad \text{where} \qquad c = c(\zeta) := 1 - \frac{1}{2} \inparen{  \sqrt{\zeta^2 + 4\zeta } - \zeta } .
\end{equation}
% and that $\Var_{\mu_\lambda} |M| = O(n)$. 
\end{lemma}

% \noindent
Our main results will be most easily understood in the above limiting regime, although they have more general implications. 
Informally, the question we aim to answer is the following: For $n$ large, if we plant a matching of size $\Theta(n)$ in a random graph $G(n,p)$, what scaling of $p = p_n$ enables us to detect the presence of the hidden matching? 

% We consider two regimes of interest. 
% \begin{enumerate}
%     \item The $\lambda \sim \frac{1}{n}$ case. The monomer-dimer Gibbs measure is supported over matchings of variable size. The size of a matching is typically of order $c n/2$ for some $c \in (0,1)$ independent of $n$.
%     \item  The $\lambda \rightarrow \infty$ case. The monomer-dimer Gibbs measure is taken as the uniform distribution over perfect matchings in $K_n$. The size of a matching is fixed at $n/2$ ($n$ assumed to be divisible by two for simplicity).
% \end{enumerate}

\subsection{Equal ambient edge density and the edge count}

Let us start with the case $p=q$ in Problem~\ref{prob:detection}; that is, the planted model $\cP_\lambda$ has an \emph{ambient edge density} equal to that in the null model $\cQ$. In this simple case, the planted matching adds $\Theta(n)$ more edges in the model $\cP_\lambda$ compared to $\cQ$ as discussed above. Therefore, the edge count (i.e., the total number of edges in $A$) 
% $$
% |A| := \sum_{\{i,j\} \in \binom{[n]}{2}} A_{ij} 
% $$
is a natural test statistic that distinguishes the two hypotheses. Since the standard deviation of the edge count in $A \sim \cP_\lambda$ or $\cQ$ is $\Theta(n\sqrt{p(1-p)})$, it is easily seen that the edge count yields a consistent test if $p \to 0$, while the critical regime is when $p$ is a constant, which we now focus on.

\begin{assumption}
    \label{asmpt:nonMatchingEdges} 
 % has the following parameters.
 %    \begin{enumerate}[(i)]
 %        \item Let $b \geq \log 9$ and for each $n$, set the dimer density $\lambda = \lambda_n$ as
 %        \begin{align}
 %            \lambda = \frac{1}{2e^{1+b} n}.
 %        \end{align}
 %        \item The edge densities are $p = q$, where $p = O(1)$. For each $n$, it holds that $p > \frac{6 \log n}{n-1}$.
 %    \end{enumerate}
Consider Problem~\ref{prob:detection} with $p=q \in (0,1)$ being a constant. 
Suppose $\lambda = \frac{1}{\zeta n}$ for a constant $\zeta \ge 40$. Let $c$ be defined by \eqref{eq:MDMean_convergence_to_c}.
\end{assumption}

\noindent
The assumption $\zeta \ge 40$ is not optimized---the absolute constant can be made smaller. However, it cannot be completely lifted due to the convergence issue of the cluster expansion (see Theorem~\ref{thm:truncate_CE_clogn} and Section~\ref{sec:cluster-expansion-convergence}). 
This limits the size of the matching in view of \eqref{eq:MDMean_convergence_to_c}, and we discuss more about this in Section~\ref{sec:perfect-matching-results}.

% In both regimes of $\lambda$, we consider the settings of
% \begin{enumerate}
%     \item The case $p = q$, where planting a matching adds edges in the model $\cP_\lambda$, as well as
%     \item Matching edge densities $p \neq q$. Here $q$ chosen such that the expected edge density in both the planted and null distributions coincide, i.e.~$\EE_{A \sim \cP_\lambda} \abs{A} = \EE_{A \sim \cQ} \abs{A}$.
% \end{enumerate}

Consider the \emph{signed edge count} defined by 
\begin{equation}
    \signedKtwo(A) = \sum_{\inbraces{i,j} \in \binom{[n]}{2}} (A_{ij} - q) ,\label{eq:def-signedKtwo}
\end{equation}
which is simply the number of edges in $A$ centered to have mean zero. 
Define the threshold test $\varphi_n : \inbraces{0,1}^{\binom{[n]}{2}} \rightarrow \inbraces{0,1}$ by
\begin{equation} 
    \varphi_n(A) = \boldsymbol{1}\!\inbraces{ \frac{\signedKtwo(A)}{\sqrt{\binom{n}{2}p(1-p)}} \geq \frac{c}{2\sqrt{2}} \sqrt{\frac{1-p}{p}} }.
% \varphi_n(A) = \boldsymbol{1}\!\inbraces{ \signedKtwo(A) \geq c n (1-p)/4 } . 
\label{eq:edge-count-threshold-test}
\end{equation}
That is, $\varphi_n(A)$ returns $1$ (resp.~$0$) if the test result is that $A \sim \cP_\lambda$ (resp.~$A \sim \cQ$).
The next result is a simple consequence of the central limit theorem (CLT). 
See Section~\ref{sec:analysis-edge-wedge-count} for the proof.

\begin{theorem} \label{thm:threshold-edge-count}
Suppose Assumption \ref{asmpt:nonMatchingEdges} holds. As $n \to \infty$, the threshold test \eqref{eq:edge-count-threshold-test} satisfies
    % The sum of the type-I and type-II errors of the threshold test are asymptotically
    \begin{align*}
        \PP_{A \sim \cP_\lambda}\!\insquare{ \varphi_n(A) = 0  } + \PP_{A \sim \cQ}\!\insquare{ \varphi_n(A) = 1  } \longrightarrow 2 \Phi\!\inparen{  -\frac{c}{2\sqrt{2}} \sqrt{\frac{1-p}{p}}  } .
    \end{align*}
    % where $\Phi$ denotes the standard Gaussian CDF.
    % Consequently, the signed edge statistic $\signedKtwo(A)$ is an asymptotically optimal test statistic.
\end{theorem}

The above asymptotic error achieved by thresholding the edge count turns out to be statistically optimal. To prove a matching negative result, we study the likelihood ratio $\frac{\ud \cP_\lambda}{\ud \cQ}$ because it is known to be the optimal test statistic for simple hypothesis testing. The following result shows that, in fact, the log-likelihood ratio is dominated by the signed edge count.

\begin{theorem}\label{thm:nonMatching_LLdistributionUnderNull}
    Suppose Assumption \ref{asmpt:nonMatchingEdges} holds. 
    Let $\signedKtwo(A)$ be the signed edge count defined by \eqref{eq:def-signedKtwo}. 
    Then for $A \sim \cQ$ and for each $n$, the log-likelihood ratio satisfies 
    \begin{align}\label{eq:nonMatching_LLdistributionUnderNull}
        \log \frac{\ud \cP_\lambda}{\ud \cQ}(A) = - \frac{1-p}{p} \inparen{ \frac{\EE\abs{M}}{n}  }^2 + \sqrt{\frac{2(1-p)}{p}} \frac{\EE\abs{M}}{n} \frac{\signedKtwo(A)}{\sqrt{\Var \signedKtwo(A)}} + O_\PP\!\inparen{ \frac{1}{p \sqrt{n}}  }.
    \end{align}
\end{theorem}
\noindent 
Note that the main terms in \eqref{eq:nonMatching_LLdistributionUnderNull} are of constant order since $\EE |M| = \Theta(n)$, and that the remainder term vanishes in probability. We have opted to leave explicit the dependence on $p$ in the remainder term in \eqref{eq:nonMatching_LLdistributionUnderNull} even when $p = \Theta(1)$ in this regime because this will provide a useful comparison to the setting in Section~\ref{sec:equal-average-edge-density-results}.
% , where the scaling is $p \sqrt{n}  = O(1)$. 

The above theorem is proved in Section~\ref{sec:plantingVarSize_pEqualq} via a \emph{finite-sample} analysis. As a result, while the theorem is stated with asymptotic notation, the approximation \eqref{eq:nonMatching_LLdistributionUnderNull} is inherently non-asymptotic. Moreover, \eqref{eq:nonMatching_LLdistributionUnderNull} implies that the log-likelihood ratio is asymptotically normal and achieves the same asymptotic testing error as the signed edge count, which is therefore statistically optimal.

\begin{theorem}\label{thm:plantedMatchings_mean_is_minus_half_variance}
Suppose Assumption \ref{asmpt:nonMatchingEdges} holds. 
As $n \to \infty$, 
% the following occur.
% \begin{enumerate}[(i)]
    % \item 
the log-likelihood ratio 
    % \eqref{eq:nonMatchingLikelihoodRatio} 
    satisfies
    \begin{align*}\label{eq:plantedMatchings_mean_is_minus_half_variance}
        \log \frac{\ud \cP_\lambda}{\ud \cQ}(A) \overset{d}{\longrightarrow} \cN\!\inparen{ \pm   \frac{c^2}{4} \frac{1-p}{p}, \, \frac{c^2}{2} \frac{1-p}{p}  },
    \end{align*}
    where `$+$' holds for $A \sim \cP_\lambda$ and `$-$' holds for $A \sim \cQ$.
    % \item 
    % \label{thm:nonMatching_TV_mainResult} 
    Consequently, 
    \begin{align*}
        \inf_{\psi_n} \left( \PP_{A \sim \cP_\lambda}\!\insquare{ \psi_n(A) = 0  } + \PP_{A \sim \cQ}\!\insquare{ \psi_n(A) = 1  } \right) = 
     1 - \TV(\cP_\lambda, \cQ) \longrightarrow 2 \Phi\!\inparen{  -\frac{c}{2\sqrt{2}} \sqrt{\frac{1-p}{p}}  } ,
     % = 1 - \mathsf{erfc}\inparen{ \frac{c}{4}\sqrt{\frac{1-p}{p}}  }.
    \end{align*}
    where the infimum is taken over all tests $\psi_n : \inbraces{0,1}^{\binom{[n]}{2}} \rightarrow \inbraces{0,1}$.
% \end{enumerate}
\end{theorem}

% \begin{remark}In the $\lambda \sim \frac{1}{n}$ case $c < 1$. However, setting $c = 1$ (so $\EE |M| \approx n/2$, i.e.~the matching is perfect) in \eqref{eq:plantedMatchings_mean_is_minus_half_variance} achieves parity with the $\lambda \rightarrow \infty$ result Theorem \ref{thm:Pinfty_nonMatching_mainResult}. 
% \end{remark}

% In the proof of the above results, we show that the log-likelihood ratio is dominated by the signed edge count in this scenario, which explains why they achieve the same asymptotic testing error. 
We emphasize that, although the asymptotic behavior of the likelihood ratio is captured by a simple statistic, establishing this result is a sophisticated task. 
Moreover, as a consequence of the above theorems, there is no statistical-to-computational gap for this testing problem.

\subsection{Equal average edge density and the signed wedge count}
\label{sec:equal-average-edge-density-results}

We now consider the more challenging setting where the \emph{average edge density} in the planted model $\cP_\lambda$ is equal to that in the model $\cQ$, 
% Matching edge densities $p \neq q$. Here $q$ chosen such that the expected edge density in both the planted and null distributions coincide, 
i.e., $\EE_{\cQ} A_{ij} = \EE_{\cP_\lambda} A_{ij}$ which is equivalent to condition \eqref{eq:matchingEdgeDensities_q_definition}.
% The choice of $q$ in \eqref{eq:matchingEdgeDensities_q_definition} arises from the following computation. 
% This consideration leads to the condition 
% $$
% p + (1-p) \frac{\E |M|}{\binom{n}{2}} = q .
% $$
% The expected number of edges for a graph $A$ drawn from $\cP_\lambda$ is
% \begin{align*}
%     \EE_{A \sim \cP_\lambda}  K_2(A)  &= \EE\abs{M} + \binom{n}{2}p - \EE\abs{M}   =\binom{n}{2}\insquare{ p +  \frac{\EE\abs{M}}{\binom{n}{2}}(1-p) }.
% \end{align*}
In this case, the edge count is uninformative and thus does not trivialize the positive result. It turns out that another simple statistic, the \emph{signed wedge count} defined by 
\begin{equation}
    \signedPtwo(A) = \sum_{j \in [n]} \sum_{\inbraces{i,k} \in \binom{[n] \setminus \{j\}}{2} } (A_{ij} - q) (A_{jk} - q) , \label{eq:def-signedPtwo}
\end{equation}
is the optimal statistic. On the one hand, it is natural to consider counting wedges for two reasons: (i) a wedge is the next simplest network motif beyond an edge, and (ii) the planted model is expected to contain fewer wedges because the planted matching, by definition, contains no wedge. 
On the other hand, a planted matching is defined by the \emph{global constraint} that the edges in the matching are not adjacent to each other, so it is highly nontrivial why a simple network motif involving only two edges is optimal. 

What is perhaps surprising is the scaling of the edge density $p$ in $n$ in the critical regime. To see this critical scaling, we can compute $\EE_{\cQ} [\signedPtwo(A)] - \EE_{\cP_\lambda} [\signedPtwo(A)] = \Theta(n)$ and $\sqrt{\Var_{\cQ} (\signedPtwo(A))} \approx \sqrt{\Var_{\cP_\lambda} (\signedPtwo(A))} = \Theta(n^{3/2} p)$ (see Lemma~\ref{lemma:signed_P2_normality_P_Q} for a more precise statement), which suggests the scaling $p = \Theta(\frac{1}{\sqrt{n}})$. Consequently, in the regime $\frac{\log n}{n} \ll p \ll \frac{1}{\sqrt{n}}$, there are already plenty of matchings of size $\Theta(n)$ in a $G(n,p)$ random graph, but we can still consistently detect the presence of just one additional planted matching using the statistic $\signedPtwo(A)$.

The above considerations motivate the following assumption. 

\begin{assumption}
    \label{asmpt:MatchingEdgeDensity} 
Consider Problem~\ref{prob:detection} with 
% \begin{equation}\label{eq:p-root-n-scaling}
$p\sqrt{n} \rightarrow \theta$ 
% \end{equation}
as $n \to \infty$ for a constant $\theta > 0$
% , $p > \frac{6 \log n}{n-1}$, 
and 
\begin{equation}\label{eq:matchingEdgeDensities_q_definition}
q := p + \frac{\EE\abs{M}}{\binom{n}{2}}(1-p).
\end{equation}
Suppose $\lambda = \frac{1}{\zeta n}$ for a constant $\zeta \ge 60$. Let $c$ be defined by \eqref{eq:MDMean_convergence_to_c}.
% , or (ii) $\lambda = \infty$, i.e., the planted matching has size $\lfloor n/2 \rfloor$, in which case we let $c=1$. 
% 
% Problem~\ref{prob:detection} has the following parameters.
%     \begin{enumerate}[(i)]
%         \item Let $b \geq \log 11$ and for each $n$, set the dimer density $\lambda = \lambda_n$ as 
%         \begin{align}
%             \lambda := \frac{1}{2e^{1+b} n}.
%         \end{align}
%         \item The $G(n,p)$ component in $\cP_\lambda$ has edge density $p = p_n$ satisfying $p\sqrt{n} \rightarrow \theta \in \RR$. For each $n$, it holds that $p > \frac{6 \log n}{n-1}$.
%         \item The edge density $q = q_n$ under $\cQ$ is defined by 
%         \begin{align}\label{eq:matchingEdgeDensities_q_definition}
%             q := p + \frac{\EE\abs{M}}{\binom{n}{2}}(1-p).
%         \end{align}
%     \end{enumerate}
\end{assumption}

\noindent
Note that since $\EE |M| = \Theta(n)$, the conditions $p \sqrt{n} \to \theta$ and \eqref{eq:matchingEdgeDensities_q_definition} imply that $q - p = \Theta(\frac 1n)$ and $p \sim q \sim \frac{\theta}{\sqrt{n}}$. 

To formalize the result for testing with the signed wedge count, define the threshold test $\varphi'_n(A) : \inbraces{0,1}^{\binom{[n]}{2}} \rightarrow \inbraces{0,1}$ by
\begin{equation}
\varphi'_n(A) = \boldsymbol{1}\!\inbraces{ \frac{\signedPtwo(A)}{\sqrt{3\binom{n}{3}q^2(1-q^2)}} \leq -\frac{c^2}{2\sqrt{2}\theta} }.
    % \varphi'_n(A) = \boldsymbol{1}\!\inbraces{ \signedPtwo(A) \leq -\frac{c^2 n^{3/2}}{4 \theta} }.
    \label{eq:wedge-count-threshold-test}
\end{equation}
That is, $\varphi'_n(A)$ returns $1$ (resp.~$0$) if the test result is $A \sim \cP_\lambda$ (resp.~$A \sim \cQ$). 
This threshold test achieves the following asymptotic error, proved in Section~\ref{sec:analysis-edge-wedge-count}.

\begin{theorem}\label{thm:matchingEdgeDensities_compEfficientStat_mainResult}
Suppose Assumption \ref{asmpt:MatchingEdgeDensity} holds. As $n \to \infty$, the threshold test \eqref{eq:wedge-count-threshold-test} satisfies
    % The sum of the type-I and type-II errors of the midpoint test satisfy
    \begin{equation*}
        \PP_{A \sim \cP_\lambda}\!\insquare{ \varphi'_n(A) = 0  } + \PP_{A \sim \cQ}\!\insquare{ \varphi'_n(A) = 1  } \longrightarrow 2 \Phi\!\inparen{  -\frac{c^2}{2\sqrt{2} \theta }  }.
    \end{equation*}
    % Consequently, the signed edge statistic $\signedKtwo(A)$ is an asymptotically optimal test statistic.
\end{theorem}

% Define the midpoint test $\varphi_n : \inbraces{0,1}^{\binom{n}{2}} \rightarrow \inbraces{0,1}$ by
%     \begin{align*}
%         \varphi_n(A) = \boldsymbol{1}\!\inbraces{ \frac{\signedPtwo(A)}{\sqrt{3\binom{n}{3}q^2(1-q)^2}} \leq - \frac{1}{2\sqrt{2}\theta}  }.
%     \end{align*}
%     That is, $\varphi_n(A) = 1$ if the test returns $A \sim \cP$ and $\varphi_n(A) = 0$ if the test returns $A \sim \cQ$.
% \begin{theorem}\label{thm:Pinfty_pNotEqualq_compEfficientStat_mainResult}
%     The sum of the type-I and type-II errors of the midpoint test satisfy
%     \begin{align}
%         &\PP_{A \sim \cP}\!\insquare{ \varphi_n(A) = 0  } + \PP_{A \sim \cQ}\!\insquare{ \varphi_n(A) = 1  } \longrightarrow 2 \Phi\!\inparen{  -\frac{1}{2\sqrt{2} \theta }  }.
%         \label{eq:Pinfty_pNotEqualq_compEfficientStat_mainResult}
%     \end{align}
%     Consequently, $\signedPtwo(A)$ is an asymptotically optimal test statistic.
% \end{theorem}

Similar to the previous case, to prove the optimality of the $\signedPtwo$ statistic, we now show a matching negative result by considering the likelihood ratio $\frac{\ud \cP_\lambda}{\ud \cQ}$. 
The following result shows that the log-likelihood ratio is dominated by the signed wedge count asymptotically. 

\begin{theorem}\label{thm:MatchingEdgeDensity_LLdistributionUnderNull}
    Suppose Assumption \ref{asmpt:MatchingEdgeDensity} holds. 
    Let the signed wedge count $\signedPtwo(A)$ be defined by \eqref{eq:def-signedPtwo}.
    Then for $A \sim \cQ$ and for each $n$, the log-likelihood ratio satisfies 
    \begin{align}\label{eq:MatchingEdgeDensity_LLdistributionUnderNull}
        \log \frac{\ud \cP_\lambda}{\ud \cQ}(A) = - \frac{1}{4 n q^2 } \inparen{ \frac{2\EE\abs{M}}{n}  }^4 + \frac{1}{\sqrt{2n} q} \inparen{\frac{2\EE\abs{M}}{n} }^2 \frac{\signedPtwo(A)}{\sqrt{\Var \signedPtwo(A)}} + O_\PP\!\inparen{ \frac{1}{\sqrt{nq}}  } .
    \end{align}
\end{theorem}

\noindent
Note that the second term on the right-hand side of \eqref{eq:MatchingEdgeDensity_LLdistributionUnderNull} (i.e., the main random term) are of order $\frac{1}{q\sqrt{n}} \sim \frac{1}{p\sqrt{n}}$, which is the same as the remainder term in \eqref{eq:nonMatching_LLdistributionUnderNull}. Therefore, proving \eqref{eq:MatchingEdgeDensity_LLdistributionUnderNull} is a more challenging task because we need to carefully show that all the larger terms in the log-likelihood ratio cancel each other in the regime $p \sqrt{n} = \Theta(1)$. 

The above result is proved in Section~\ref{sec:plantingVarSize_pNotEqualq}. Similar to the previous case, the analysis is finite-sample and \eqref{eq:MatchingEdgeDensity_LLdistributionUnderNull} holds non-asymptotically.
Moreover, it readily implies the following.

\begin{theorem}\label{thm:matchingEdgeDensities_mean_is_minus_half_variance}
Suppose Assumption \ref{asmpt:MatchingEdgeDensity} holds. 
As $n \to \infty$, 
the log-likelihood ratio satisfies
\begin{align*}
% \label{eq:matchingEdgeDensities_mean_is_minus_half_variance}
        \log \frac{\ud \cP_\lambda}{\ud \cQ}(A) \overset{d}{\longrightarrow} \cN\!\inparen{ \pm   \frac{c^4}{4\theta^2}, \, \frac{c^4}{2\theta^2}  },
    \end{align*}
    where `$+$' holds for $A \sim \cP_\lambda$ and `$-$' holds for $A \sim \cQ$.
    % \item 
    % \label{thm:nonMatching_TV_mainResult} 
    Consequently, 
    \begin{align*}
        \inf_{\psi_n} \left( \PP_{A \sim \cP_\lambda}\!\insquare{ \psi_n(A) = 0  } + \PP_{A \sim \cQ}\!\insquare{ \psi_n(A) = 1  } \right) = 
     1 - \TV(\cP_\lambda, \cQ) \longrightarrow 2 \Phi\!\inparen{  -\frac{c^2}{2\sqrt{2} \theta}  } ,
     % = 1 - \mathsf{erfc}\inparen{ \frac{c}{4}\sqrt{\frac{1-p}{p}}  }.
    \end{align*}
    where the infimum is taken over all tests $\psi_n : \inbraces{0,1}^{\binom{[n]}{2}} \rightarrow \inbraces{0,1}$.
% \end{enumerate}
\end{theorem}

% \begin{remark}In the $\lambda \sim \frac{1}{n}$ case $c < 1$. However, setting $c = 1$ (so $\EE |M| \approx n/2$, i.e.~the matching is perfect) in \eqref{eq:plantedMatchings_mean_is_minus_half_variance} achieves parity with the $\lambda \rightarrow \infty$ result Theorem \ref{thm:Pinfty_nonMatching_mainResult}. 
% \end{remark}

The conclusion is also analogous to the previous case: the log-likelihood ratio is dominated by the signed wedge count, which is asymptotically normal in the regime $p \sqrt{n} \to \theta > 0$, and there is no statistical-to-computational gap for this testing problem.

\subsection{Planted perfect matching}
\label{sec:perfect-matching-results}

A limitation of the above results is the condition $\lambda = \frac{1}{\zeta n}$ for $\zeta$ larger than an absolute constant as in Assumptions~\ref{asmpt:nonMatchingEdges} and~\ref{asmpt:MatchingEdgeDensity}. By \eqref{eq:MDMean_convergence_to_c}, this means that the largest possible matching our results apply to has expected size $\EE |M| \sim c n/2$ for a certain constant $c \in (0,1)$. 
On the one hand, we believe our main results, Theorems~\ref{thm:nonMatching_LLdistributionUnderNull} and~\ref{thm:MatchingEdgeDensity_LLdistributionUnderNull}, can be extended to a regime where $\lambda = o(1/n)$ and $\EE |M| = o(n)$ with non-essential modifications of the proofs.
On the other hand, the convergence of the cluster expansion is a fundamental bottleneck that prohibits us from taking $\lambda$ to be sufficiently large so that $c$ is close to $1$, so we cannot cover the entire range of $\EE |M|$. 
This limitation is well-known in the cluster expansion literature and will be made clear by the proofs in Section~\ref{sec:cluster-expansion-convergence}. 
Nevertheless, we still expect our main theorems to hold for any $\lambda = \Omega(1/n)$ and $c \in (0,1)$, because the extreme case $\lambda = \infty$ and $c=1$ appeared implicitly in \cite{janson1994numbers} as intermediate results, which were proved using an entirely different approach. 

To be more precise, we now assume $n$ is even for simplicity. 
Let us consider the case $\lambda = \infty$ in Definition~\ref{def:planted-distribution} and Problem~\ref{prob:detection}. That is, we test the null model $\cQ$ against the alternative model $\cP_\infty$ where a uniformly random perfect matching (of size $n/2$) is planted in a $G(n,p)$ random graph. 
The goal is to show results analogous to Theorems~\ref{thm:threshold-edge-count}, \ref{thm:plantedMatchings_mean_is_minus_half_variance}, \ref{thm:matchingEdgeDensities_compEfficientStat_mainResult}, and~\ref{thm:matchingEdgeDensities_mean_is_minus_half_variance}. 
Our positive results about the edge count and the wedge count remain valid, and the negative results via the likelihood ratio follow from intermediate results in \cite{janson1994numbers}. 

\begin{theorem} \label{thm:perfect-matching-p=q}
Consider Problem~\ref{prob:detection} with $p=q \in (0,1)$ being a constant and $\lambda = \infty$. Let $c=1$. Then all the statements in Theorems~\ref{thm:threshold-edge-count} and \ref{thm:plantedMatchings_mean_is_minus_half_variance} hold.
\end{theorem}

\begin{theorem} \label{thm:perfect-matching-p-not-equal-q}
Consider Problem~\ref{prob:detection} with 
$p\sqrt{n} \rightarrow \theta > 0$ as $n \to \infty$, 
$q = p + \frac{\EE\abs{M}}{\binom{n}{2}}(1-p)$, and $\lambda = \infty$. Let $c=1$. Then all the statements in Theorems~\ref{thm:matchingEdgeDensities_compEfficientStat_mainResult} and~\ref{thm:matchingEdgeDensities_mean_is_minus_half_variance} hold.
\end{theorem}

\noindent
See Section~\ref{sec:lambda_infty} for the proofs of the above results.

Note that the asymptotic results in Theorems~\ref{thm:plantedMatchings_mean_is_minus_half_variance} and~\ref{thm:matchingEdgeDensities_mean_is_minus_half_variance} (and the above theorems) are weaker than the non-asymptotic results in Theorems~\ref{thm:nonMatching_LLdistributionUnderNull} and~\ref{thm:MatchingEdgeDensity_LLdistributionUnderNull}. 
It is not clear how to extract non-asymptotic results for the \emph{log-likelihood ratio} from \cite{janson1994numbers} because the paper's technique centers around the \emph{likelihood ratio} and proves that it is asymptotically \emph{log-normal}. 

More precisely, while studying the number of perfect matchings in an \ER\ graph, the paper \cite{janson1994numbers} analyzes $\frac{d\cP_\infty}{d\cQ}$ (which is never referred to as the likelihood ratio) and shows that its variance is dominated by the aggregate of the signed counts of $k$ disjoint wedges for $k \ge 1$. 
The proofs involve intricate combinatorics of perfect matchings, and are also crucially based on Janson's earlier book \cite{janson1994orthogonal} which develops fascinating theory about the orthogonal decomposition of functions on random graphs. 

Compared to Janson's approach, the cluster expansion has the advantage that it deals directly with the log-likelihood ratio for a fixed $n$ and yields finite-sample results about it. 
It remains an intriguing question how our approach can be extended beyond the bottleneck $\EE |M| \sim cn/2$ for a certain constant $c$. 
The above results for small and infinite $\lambda$ provide strong evidence that the formal cluster expansion, even when non-convergent in the $\lambda = \Omega(1/n)$ regime, still contains useful and ``correct'' information about the log-likelihood ratio. 
Making this observation rigorous is an interesting direction for future research.

\section{Cluster expansion for planted models}
\label{sec:intro-cluster-expansion}

We formally introduce the cluster expansion in this section. 
In addition to applying it to planted matchings, we also consider the planted clique model in Section~\ref{sec:plantedClique} to shed light on the potential use of the cluster expansion for other planted models. 
A comparison of the cluster expansion to the orthogonal decomposition is provided in Section~\ref{sec:comparison-to-orthogonal-expansion}.

Following \cite[Chapter 5]{friedli2017statistical}, we consider a \emph{polymer partition function} 
\begin{equation}
Z := \sum_{\Gamma' \subset \Gamma} \bigg( \prod_{\gamma \in \Gamma'} w(\gamma) \bigg) \bigg( \prod_{\{\gamma, \gamma'\} \subset \Gamma'} \delta(\gamma, \gamma') \bigg) ,
\label{eq:generic-polymer-partition-function}
\end{equation}
where $\Gamma$ is a finite set whose elements are called \emph{polymers}, $w(\gamma) \in \R$ is the weight of a polymer $\gamma$, and $\delta(\gamma, \gamma') \in \R$ is the pairwise interaction between polymers $\gamma$ and $\gamma'$, assumed to satisfy $\delta(\gamma, \gamma') = \delta(\gamma', \gamma)$, $\delta(\gamma, \gamma) = 0$, and $|\delta(\gamma, \gamma')| \le 1$ for all $\gamma, \gamma' \in \Gamma$. 
The cluster expansion refers to the formal series 
\begin{equation}
    \log Z \overset{\textrm{F}}{=} \sum_{m \geq 1} \sum_{\gamma_1,
    \dots,\gamma_m \in \Gamma} \phi(H(\gamma_1,\dots,\gamma_m)) \prod_{i = 1}^{m} w(\gamma_i) , \label{eq:generic-cluster-expansion}
\end{equation}
where $\overset{\textrm{F}}{=}$ means that the equality is formal (i.e., the convergence of the series has not been justified), and the coefficient $\phi(H(\gamma_1,\dots,\gamma_m))$, known as the \emph{Ursell function}, is defined as follows. 
\begin{definition}[Ursell function]
\label{def:Ursell_definition}
For any ordered tuple $(\gamma_1,\dots,\gamma_m)$ of possibly repeated polymers in $\Gamma$, define $H = H(\gamma_1,\dots,\gamma_m)$ to be the graph on the vertex set $\{\gamma_1,\dots,\gamma_m\}$\footnote{The vertex set $\{\gamma_1,\dots,\gamma_m\}$ is sometimes identified with $[m] = \{1, \dots, m\}$ when there is no ambiguity. If there are repeated polymers $\gamma_i = \gamma_j$, the latter notation emphasizes that they are distinct vertices in $H$.} 
% vertices, labeled say from $\inbraces{1,\dots,m}$, 
with edge $\inbraces{\gamma_i, \gamma_j}$ present if the weight $\delta(\gamma_i, \gamma_j) - 1$ is nonzero.
    The Ursell function $\phi$ of the graph $H$ is defined as follows. For $m = 1$, let $\phi(H) = 1$. For $m \geq 2$, let
$$
\phi(H) = \frac{1}{m!} \sum_{ \substack{S \subseteq H \\ \text{spann., conn.} }} \prod_{\{\gamma, \gamma'\} \in S} (\delta(\gamma, \gamma') - 1) ,
$$
where the sum is over spanning and connected subgraphs $S$ of $H$. 
\end{definition}

\subsection{Formal results for planted matching}
\label{sec:cluster-expansion-formal-planted matching}

% Next, we turn to the main problem studied in this work---detection of a planted matching in a random graph. 

To see why the cluster expansion can be used to study Problem~\ref{prob:detection}, we express the log-likelihood ratio using log-partition functions.

\begin{lemma} \label{lem:log-likelihood-ratio-partition-functions}
Let $|A|$ denote the number of edges in the graph $A$, and let $Z_G(\lambda)$ be given by Definition~\ref{def:monomer-dimer}. 
For Problem~\ref{prob:detection}, the log-likelihood ratio can be written as
\begin{equation}\label{eq:MatchingEdgeDensities_loglikelihood}
    \log \frac{\ud \cP_\lambda}{\ud \cQ}(A) = F(A) + \log Z_{A}(\lambda / p) - \log Z_{K_n}(\lambda) ,
\end{equation}
where 
\begin{equation}
F(A) := |A| \log \frac{p(1-q)}{q(1-p)} + \binom{n}{2} \log \frac{1-p}{1-q} \label{eq:def-F(A)} . 
\end{equation}
\end{lemma}

\begin{proof}
By the definitions of the models $\cP_\lambda$ and $\cQ$, we have
\begin{align*}
\frac{\ud \cP_\lambda}{\ud \cQ}(A) &= \frac{1}{Z_{K_n}(\lambda)} \sum_{M \subset K_n} \lambda^{|M|} \frac{\boldsymbol{1}\!\inbraces{M \subset A}}{q^{|M|}} \prod_{\inbraces{i,j} \notin M} \frac{p^{A_{ij}} (1-p)^{1 - A_{ij}} }{ q^{A_{ij}} (1-q)^{1 - A_{ij}} } \\
&= \inparen{ \frac{p(1-q)}{q(1-p)}  }^{|A|} \inparen{ \frac{1-p}{1-q} }^{\binom{n}{2}} \frac{Z_A\!\inparen{\lambda/p}}{Z_{K_n}(\lambda)} ,
\end{align*}
from which the result follows.
\end{proof}

% from which we deduce that
% \begin{equation*}\label{eq:MatchingEdgeDensities_likelihoodratio}
%     \frac{\ud \cP_\lambda}{\ud \cQ}(A) = . 
% \end{equation*}

As a result, to study the log-likelihood ratio for Problem~\ref{prob:detection}, we may analyze $\log Z_G(\lambda)$ using the cluster expansion. 
Comparing $Z_G(\lambda) = \sum_{M \subset G} \lambda^{|M|}$ to the generic polymer partition function \eqref{eq:generic-polymer-partition-function}, we note: (i) the polymers in this case are the \emph{edges} of $G$ which we denote by $e$, (ii) the weight of each polymer is $w(e) = \lambda$, and (iii) the pairwise interaction between two polymers is $\delta(e, e') = \boldsymbol{1}\!\inbraces{e \sim e'}$ where $e \sim e'$ means that the two edges are \emph{not} adjacent. 
This pairwise interaction is known as the \emph{hard-core repulsion} between edges. 
The notation $e \sim e'$, albeit unconventional in the context of graphs, means that $e$ is \emph{compatible} with $e'$, while $e \not\sim e'$ means
% Specifically, for each $n$, a polymer is the 2-tuple $e \in \binom{[n]}{2}$. The pairwise interaction for these polymers is a hardcore repulsion derived from the fact that two edges cannot both be in a matching if they share a vertex. 
% More precisely, we define the following 
the \emph{incompatibility} relation between polymers, i.e., the edges $e$ and $e'$ are adjacent. 

Next, following Definition~\ref{def:Ursell_definition}, we see that the graph $H = H(e_1, \dots, e_m)$ contains an edge $\{i,j\}$ with weight $-1$ if and only if $e_i \not\sim e_j$, i.e., $e_i$ and $e_j$ are adjacent in $G$. 
The graph $H$ is also known as the \emph{incompatibility graph} of $(e_1,\dots,e_m)$ and coincides with the line graph of the subgraph with edges $e_1,\dots,e_m$ in $G$ if there are no repeated polymers. 
The Ursell function is therefore
% as a sum over clusters. 
\begin{equation}\label{eq:Ursell_definition}
        \phi(H(e_1, \dots, e_m)) = \frac{1}{m!} \sum_{ \substack{S \subseteq H(e_1, \dots, e_m) \\ \text{spann., conn.} }} (-1)^{\abs{S}}.
    \end{equation}
A \emph{cluster} is an ordered tuple $(e_1,\dots,e_m)$ of possibly repeated polymers whose incompatibility graph is connected. 
% This constraint will be enforced by the following function. 
Observe that $\phi(H)$ is nonzero only when $(e_1,\dots,e_m)$ is a cluster, which is the namesake of the cluster expansion. 
\begin{comment}
Remark that $\phi(H(e_1,\dots,e_m))$ is commonly defined as $m!^{-1} \sum_{\beta} \prod_{\inbraces{i,j} \in \beta} - \boldsymbol{1}\!\inbraces{e_i \not\sim e_j}$, with the sum running over all spanning and connected subgraphs $\beta$ of $K_m$. The form \eqref{eq:Ursell_definition} is equivalent and more convenient for our purposes. 
\end{comment}

Furthermore, the cluster expansion \eqref{eq:generic-cluster-expansion}
% will yield a representation 
of the log-partition function becomes
% The cluster expansion for $Z_G(\lambda)$ is the formal series expansion
% \begin{align}\label{eq:generic_clusterExpansion_ZGlambda}
$$
    \log Z_G(\lambda) \overset{\textrm{F}}{=} \sum_{m \geq 1} \sum_{e_1,
    \dots,e_m \in G} \phi(H(e_1,\dots,e_m)) \lambda^m 
    $$
% \end{align}
which is a \emph{perturbative} expansion around $\lambda = 0$, where $G$ is identified with its own edge set.  
Here, and henceforth, we use the convention that the inner sum is over $e_1,\dots,e_m \in \binom{[n]}{2}$, i.e., over all ordered $m$-tuples of possibly repeated polymers in $K_n$. 
% The derivation of this formal identity from the polymer representation \eqref{eq:generic_ZGlambda} is standard---we refer for instance to \cite[Section 5.3]{friedli2017statistical}. 
% Applying \eqref{eq:generic_clusterExpansion_ZGlambda}, 
Specializing the above equation to $\log Z_{K_n}(\lambda)$ 
% , $\EE \abs{M}$, 
and $\log Z_A(\lambda/p)$, we obtain %\notecm{Missing $H(...)$ below?} Corrected, will also correct all the proofs. Before we didn't write H(...)
% 
% One may interpret the Ursell function as a function of the incompatibility graph (and so independent of the actual underlying polymers $e_1,\dots,e_m$):
% \begin{align}\label{eq:Ursell_definition_incompatibility_graph}
%     \phi(H) = \frac{1}{\abs{V(H)}!} \sum_{ \substack{\alpha \subseteq H(e_1,\dots,e_m)\\ \text{spanning, connected} }} (-1)^{\abs{\alpha}}
% \end{align} 
\begin{equation}\label{eq:logZKn_CE_formal}
    \log Z_{K_n}(\lambda) \overset{\textrm{F}}{=} \sum_{m \geq 1} \sum_{e_1,
        \dots,e_m } \phi(H(e_1,\dots,e_m)) \lambda^m,
\end{equation}
and 
% for every fixed adjacency matrix $A$,
\begin{equation}\label{eq:logZA_CE_formal}
    \log Z_A\!\inparen{\frac{\lambda}{p}} \overset{\textrm{F}}{=} \sum_{m \geq 1} \sum_{e_1,
    \dots,e_m } \phi(H(e_1,\dots,e_m)) \Big( \frac{\lambda}{p} \Big)^m \prod_{j = 1}^{m} A_{e_j} .
\end{equation}
% In constrast to \eqref{eq:generic_ZGlambda}, the cluster expansion sums are indeed infinite sums over $m$ because repeated polymers are possible. 
% The above two equations together with \eqref{eq:MatchingEdgeDensities_loglikelihood} provide \notecm{to be continued}
 % 
% Recall \eqref{eq:MatchingEdgeDensities_loglikelihood}, \eqref{eq:logZKn_CE_formal}, and \eqref{eq:logZA_CE_formal}, where we express the partition functions of the monomer-dimer model 
% that show up in \eqref{eq:nonMatchingLikelihoodRatio} 
% as perturbative expansions around $\lambda = 0$. 
% This is accomplished by the cluster expansion tool from statistical physics.

We now assuage concerns about convergence and the infinite nature of the above expansions. 
% In particular, we show next that for $\lambda \le \frac{c}{n}$ for a constant $c>0$, these cluster expansions converge absolutely.
In fact, these expansions can be truncated to $\Theta( \log n )$ terms with vanishing error. Consequently, for each fixed $n$, the cluster expansions we deal with are essentially finite sums over $m$.

\begin{theorem}\label{thm:truncate_CE_clogn}
    Suppose that $\lambda \le \frac{1}{30 n}$ and $\frac{9 \log n}{n} \le q \le 1.01 p$. 
    Then the following occur.
    \begin{enumerate}[(i)]
        \item The cluster expansion \eqref{eq:logZKn_CE_formal} for $\log Z_{K_n}(\lambda)$ converges absolutely. Moreover,
    \begin{align}
        \label{eq:logZ_Kn_truncate_clogn}
        \log Z_{K_n}(\lambda) = \sum_{m = 1}^{ 2 \log n} \sum_{e_1,\dots,e_m} \phi(H(e_1,\dots,e_m)) \lambda^m + \frac{1}{n} .
    \end{align}
        \item For $A \sim G(n,q)$, with probablity at least $1 - \frac{1}{n}$, the cluster expansion \eqref{eq:logZA_CE_formal} for $\log Z_{A}(\lambda/p)$ converges absolutely. Moreover,
    \begin{align}\label{eq:logZ_A_truncate_clogn}
        \log Z_{A}\!\inparen{\frac{\lambda}{p}} = \sum_{m = 1}^{ 2 \log n} \sum_{e_1,\dots,e_m} \phi(H(e_1,\dots,e_m)) \frac{\lambda^m}{p^m} \prod_{j = 1}^{m} A_{e_j} + \frac{1}{n}  .
    \end{align}
    \end{enumerate}
\end{theorem}

\noindent
See Section~\ref{sec:cluster-expansion-convergence} for the proof of the above result. 
The condition $\frac{9 \log n}{n} \le q \le 1.01 p$ is mild in view of the regimes we consider in Section~\ref{sec:main-results-matching}. 
On the other hand, the condition $\lambda \le \frac{1}{30 n}$ required for the convergence of the cluster expansion cannot be removed and is a limitation of the current theory as discussed in Section~\ref{sec:perfect-matching-results}. 
% 
% below, we will see that the conditions on the parameters $\lambda, p, q$ are satisfied in the regime of interest. Therefore, the cluster expansion provides us a rigorous way to analyze the log-likelihood ratio. 
% In particular, 
We remark that combining Lemma~\ref{lem:log-likelihood-ratio-partition-functions} with Theorem~\ref{thm:truncate_CE_clogn} yields \eqref{eq:log-likelihood-ratio-cluster-expansion-introduction} stated in the introduction.

\subsection{Heuristics for planted clique}
\label{sec:plantedClique}

While this work primarily considers planted matching detection, it is illuminating to apply the formal cluster expansion to the iconic problem of detecting a planted clique of size approximately $k$ in a random graph $G(n,1/2)$. Since the planted clique problem is well-studied in the literature, this informal discussion is \emph{not} meant to establish rigorous results---instead, the goal is to provide some heuristics about how the cluster expansion captures information in the log-likelihood ratio through a well-understood model.

For $\lambda > 0$ and a graph $G$ with vertex set $[n]$ and edge weights $G_{ij}$, consider the Gibbs measure $\nu_\lambda(V)$ over subsets $V \subset [n]$ defined by 
\begin{equation*}
    \nu_\lambda(V) = \frac{ \lambda^{|V|} \prod_{\{i,j\} \in E(V)} G_{ij} }{Q_{G} (\lambda)}  , \qquad \text{where} \qquad Q_{G} (\lambda) := \sum_{V \subset [n]} \lambda^{|V|} \prod_{\{i,j\} \in E(V)} G_{ij} ,
\end{equation*}
where $E(V)$ denotes the edge set of the complete graph on $V$. 
For $G = K_n$, we sample the vertex set of the planted clique from the Gibbs measure $\nu_\lambda(V) \propto \lambda^{|V|}$. 
If $\lambda = \frac{k}{n-k}$, this is equivalent to assuming that each vertex belongs to the planted clique independently with probability $k/n$ so that the expected size of the clique is $k$. 
Since all the interesting information-theoretic and computational thresholds for a planted clique of size $k$ in a random graph $G(n,1/2)$ occur at certain $k = o(n)$, it suffices to consider $\lambda \approx k/n$. 

Let $\cP$ denote the planted clique model: $A \sim \cP$ means that conditional on $V \sim \nu_\lambda$, we have $A_{ij} = 1$ if $i,j \in V$ and $A_{ij} \sim \mathrm{Bernoulli}(1/2)$ independently otherwise. Let $\cQ = G(n,1/2)$. 
Then the likelihood ratio satisfies
\begin{align*}
\frac{\ud\cP}{\ud\cQ}(A) 
&= \frac{1}{Q_{K_n}(\lambda)} \sum_{V \subset [n]} \lambda^{|V|} \frac{\prod_{\{i,j\} \in E(V)} A_{ij} \prod_{\{i,j\} \notin E(V)} (1/2)^{A_{ij}} (1-1/2)^{1-A_{ij}}}{\prod_{\{i,j\} \subset [n]} (1/2)^{A_{ij}} (1-1/2)^{1-A_{ij}}} \\
&= \frac{\sum_{V \subset [n]} \lambda^{|V|} \prod_{\{i,j\} \in E(V)} (2 A_{ij})}{\sum_{V' \subset [n]} \lambda^{|V'|}} 
= \frac{Q_{2A}(\lambda)}{Q_{K_n}(\lambda)} ,
\end{align*}
where $2A$ denotes the graph $A$ with edge weights $2 A_{ij}$. 
As a result, we have
\begin{equation}
\log \frac{\ud\cP}{\ud\cQ}(A) = \log Q_{2A}(\lambda) - \log Q_{K_n}(\lambda) ,
\label{eq:log-likelihood-planted-clique-partition-functions}
\end{equation}
and the cluster expansion can be applied to study the two log-partition functions above.

Note that $Q_G(\lambda)$ is in the form of \eqref{eq:generic-polymer-partition-function} where (i) the polymers are vertices, (ii) the weight of each polymer is $w(i) = \lambda$, and (iii) the pairwise interaction between two polymers is $\delta(i,j) = G_{ij}$. Therefore, the Ursell function in Definition~\ref{def:Ursell_definition} is given by 
$$
\phi(H(i_1,\dots,i_m)) = \frac{1}{m!} \sum_{ \substack{S \subseteq H(i_1,\dots,i_m) \\ \text{spann., conn.} }} \prod_{\{i, j\} \in S} (G_{i j} - 1) .
$$
For $G = K_n$, we have $G_{ij} - 1 = 0$ if $i \ne j$ and $G_{ii} - 1 = -1$, so the Ursell function $\phi(H(i_1,\dots,i_m))$ is zero unless $i_1 = \cdots = i_m$. 
Moreover, $\phi(H(i,\dots,i))$ is the same for $G = K_n$ and $G = 2A$. As a result, by \eqref{eq:log-likelihood-planted-clique-partition-functions} and \eqref{eq:generic-cluster-expansion}, we obtain
\begin{equation}
\log \frac{\ud\cP}{\ud\cQ}(A) \overset{\textrm{F}}{=} \sum_{m \geq 2} \sum_{\substack{i_1,
    \dots,i_m \in [n] \\ \text{not all equal} }} \frac{1}{m!} \sum_{ \substack{S \subseteq H(i_1,\dots,i_m) \\ \text{spann., conn.} }} \prod_{\{i,j\} \in S} (2 A_{ij} - 1) \lambda^m .
\label{eq:cluster-expansion-planted-clique}
\end{equation}

The issue with the formal series \eqref{eq:cluster-expansion-planted-clique}, which is essentially equivalent to the cluster expansion of the partition function for the hard-core model, is that its convergence requires $\lambda = O(1/n)$ \cite{scott2005repulsive}. 
This means that the planted clique has a constant size and is therefore too restrictive. 
Nevertheless, it turns out that a truncated version of \eqref{eq:cluster-expansion-planted-clique} captures sufficiently interesting information for planted clique detection. 

% In the sequel, we use $\sim$ to denote an approximation that is not rigorously justified. 
% To provide some heuristics about the above cluster expansion, 

To be more precise, let us consider the partial sum over \emph{distinct} $i_1, \dots, i_m \in [n]$ in \eqref{eq:cluster-expansion-planted-clique}:
$$
\bigg[ \log \frac{\ud\cP}{\ud\cQ}(A) \bigg]_{\mathrm{part}}
:= \sum_{m = 2}^n \frac{\lambda^m}{m!} \sum_{\substack{i_1,
    \dots,i_m \in [n] \\ \text{distinct} }}  \sum_{ \substack{S \subseteq H(i_1,\dots,i_m) \\ \text{spann., conn.} }} \prod_{\{i,j\} \in S} (2 A_{ij} - 1) ,
$$
where convergence is no longer an issue because the sum is finite once $i_1, \dots, i_m$ are required to be distinct. 
% Moreover, for $m$ small, one gains a factor of a power of $1/n$. Thus 
We then deduce that 
\begin{equation}
\bigg[ \log \frac{\ud\cP}{\ud\cQ}(A) \bigg]_{\mathrm{part}}
= \sum_{m=2}^n \lambda^{m} \sum_{\substack{\alpha \subset K_n \  \text{conn.} \\ |V(\alpha)| = m}} \prod_{\{i,j\} \in \alpha} (2 A_{ij} - 1) ,
\label{eq:truncated-log-likelihood-planted-clique}
\end{equation}
where $\alpha$ is a connected subgraph of $K_n$ (coming from labeling the vertices of $S$ by $i_1, \dots, i_m$ in the previous display) and $V(\alpha)$ denotes the vertex set of $\alpha$. 

Furthermore, since the Kullback--Leibler (KL) divergence is defined by $\KL(\cP, \cQ) = \EE_{A \sim \cP} \log \frac{\ud\cP}{\ud\cQ}(A)$, we can analogously introduce 
$$
\left[ \KL(\cP, \cQ) \right]_{\mathrm{part}}
:= \EE_{A \sim \cP} \bigg[ \log \frac{\ud\cP}{\ud\cQ}(A) \bigg]_{\mathrm{part}}
= \sum_{m=2}^n \lambda^{m} \sum_{\substack{\alpha \subset K_n \  \text{conn.} \\ |V(\alpha)| = m}} \EE_{V \sim \nu_\lambda} \bigg[ \prod_{\{i,j\} \in \alpha} \EE_{A \sim \cP} [ 2 A_{ij} - 1 \mid V] \bigg] .
$$
Since $2 A_{ij} - 1 = 1$ if $i,j \in V$ and otherwise it has mean zero conditional on $V$, so the outer expectation above is equal to $\PP_{V \sim \nu_\lambda} [ V(\alpha) \subset V]$. 
Note that this probability is $(k/n)^m$ if we set $\lambda = \frac{k}{n-k} \approx \frac{k}{n}$ by our earlier discussion. 
As a result, 
\begin{equation}
\left[ \KL(\cP, \cQ) \right]_{\mathrm{part}}
= \sum_{m=2}^n (k/n)^{m} \sum_{\substack{\alpha \subset K_n \  \text{conn.} \\ |V(\alpha)| = m}} (k/n)^m 
= \sum_{\substack{\alpha \subset K_n \  \text{conn.} \\ |\alpha| \ge 1}} (k/n)^{2 |V(\alpha)|} .
\label{eq:truncated-KL-planted-clique}
\end{equation}

The expansion 
% \eqref{eq:truncated-log-likelihood-planted-clique} and 
\eqref{eq:truncated-KL-planted-clique} is reminiscent of the (rigorous) 
% orthogonal expansion of the likelihood ratio\footnote{These expansions can be easily derived from the general theory \cite{janson1994orthogonal,hopkins2018statistical}---see the tutorial \cite{mao2025tutorial}.}
% $$
% \log \frac{\ud\cP}{\ud\cQ} = \sum_{\alpha \subset K_n} (k/n)^{|V(\alpha)|} \prod_{\{i,j\} \in \alpha} (2 A_{ij} - 1)
% $$
% and 
expansion of the $\chi^2$-divergence\footnote{This identity can be easily derived using the general theory \cite{janson1994orthogonal,hopkins2018statistical}. See the tutorial \cite{mao2025tutorial}, especially Equation (5) with $D = \binom{n}{2}$ and (6) which is an equality for the planted model where each vertex belongs to the clique independently with probability $k/n$.}
\begin{equation}
\chi^2(\cP,\cQ) = \sum_{\alpha \subset K_n : |\alpha| \ge 1} (k/n)^{2 |V(\alpha)|} ,
\label{eq:chi-sq-expansion-planted-clique}
\end{equation}
where the only difference is that the subgraph $\alpha$ is required to be connected in 
% \eqref{eq:truncated-log-likelihood-planted-clique} and 
\eqref{eq:truncated-KL-planted-clique}. 
Moreover, from the expansion \eqref{eq:chi-sq-expansion-planted-clique}, one can obtain both the information-theoretic threshold $k \sim 2 \log_2 n$ and the computational threshold $k \asymp \sqrt{n}$ in the low-degree polynomial framework (see Theorem~2.5 in the tutorial \cite{mao2025tutorial}). 
Since the connectedness of $\alpha$ is not essential for obtaining these thresholds from \eqref{eq:chi-sq-expansion-planted-clique}, they can be extracted from the expansion \eqref{eq:truncated-KL-planted-clique} too. 
It is intriguing that the truncated cluster expansion contains sufficient information to recover both thresholds for planted clique detection, even though the formal series is not expected to converge.

\subsection{Comparison to the orthogonal decomposition}
\label{sec:comparison-to-orthogonal-expansion}

For testing the null model $\cQ = G(n,q)$ against any alternative random graph model $\cP$, the orthogonal decomposition of the likelihood ratio (see \cite{janson1994orthogonal,hopkins2018statistical,kunisky2019notes}) takes the form
$$
\frac{\ud\cP}{\ud\cQ}(A) = \sum_{\alpha \subset K_n} \EE_{\cP} [\phi_\alpha] \cdot \phi_\alpha(A), \quad \text{ where } \phi_\alpha(A) := \prod_{\{i,j\} \in \alpha} \frac{A_{ij} - q}{\sqrt{q(1-q)}} .
$$
We compare this to the cluster expansion:
\begin{itemize}
\item
Most notably, the orthogonal decomposition is for the likelihood, while the cluster expansion is for the log-likelihood. As a result, we can directly obtain non-asymptotic approximations of the log-likelihood ratio which subsequently yields its asymptotic distribution.

\item
The orthogonal decomposition is a rigorous finite sum. 
On the other hand, the cluster expansion is a formal series \eqref{eq:generic-cluster-expansion} whose convergence needs to be proved. 

\item
The orthogonal decomposition is the same for any planted model $\cP$. 
The cluster expansion, however, is a technique rather than a unique expansion, because for different planted models we may expand the log-likelihood ratios in very different ways such as \eqref{eq:log-likelihood-ratio-cluster-expansion-introduction} versus \eqref{eq:cluster-expansion-planted-clique}. 

\item
Both expansions involve (signed) subgraph counts. 
In line with the above comparison, the orthogonal decomposition is always in terms of signed subgraph counts (note the definition of $\phi_\alpha$ above), but the cluster expansion may involve subgraph counts as in \eqref{eq:log-likelihood-ratio-cluster-expansion-subgraph-counts} or the signed version as in \eqref{eq:cluster-expansion-planted-clique}.

\item
By restricting the sums to small template subgraphs, both expansions may be used to inform computational thresholds and low-degree polynomial algorithms. 
This aspect of the cluster expansion is not formally developed in this work due to the lack of a statistical-to-computational gap for planted matching detection.
Nonetheless, the resemblance between \eqref{eq:truncated-KL-planted-clique} and \eqref{eq:chi-sq-expansion-planted-clique} suggests that cluster expansion techniques can potentially be used from the perspective of low-degree polynomials.
% 
% \item
% Because the orthogonality is in an $L_2$ sense with respect to the inner product $\langle f, g \rangle := \EE_{A \sim \cQ}[f(A) g(A)]$, the orthogonal expansion is natural for analyzing $\chi^2(P,Q) = \|\frac{\ud\cP}{\ud\cQ}-1\|^2$. On the other hand, the cluster expansion is potentially related to the KL divergence as we have seen above.
\end{itemize}
In view of the broad applications of the orthogonal decomposition in statistical problems, we believe the link between the cluster expansion and planted models established by this work opens an interesting direction for future research.

\section{First few terms of the log-likelihood ratio}
\label{sec:nonMatchingEdges_intuition}

% The reader who wishes only to see the general results may skip this informal section.

To understand the proof strategy for our main results, it is helpful to explicitly compute the first few terms in the cluster expansion of the log-likelihood ratio in the simple $p=q$ case. 
This provides intuition about the asymptotic normality of the log-likelihood ratio and also outlines the proof of Theorem~\ref{thm:nonMatching_LLdistributionUnderNull}. The strategy for proving Theorem~\ref{thm:MatchingEdgeDensity_LLdistributionUnderNull} is analogous.

In light of the absolute convergence in Theorem \ref{thm:truncate_CE_clogn}, we can reorganize the sum over polymers into sums over template subgraphs (which include multigraphs) as in \eqref{eq:log-likelihood-ratio-cluster-expansion-subgraph-counts}. The main message of this section is that the dominating terms in the cluster expansion correspond to template subgraphs that are \emph{simple trees} and \emph{trees with one repeated edge}. In particular, they give rise to the zero-mean fluctuation part and the deterministic mean part respectively in \eqref{eq:nonMatching_LLdistributionUnderNull}: 
\begin{align*}
    \text{simple trees} &\overset{d}{\approx} \cN\inparen{ 0, \,  \frac{2(1-p)}{p} \inparen{\frac{\EE\abs{M}}{n}}^2  }, \quad 
    \text{one repeated edge trees} \approx - \frac{1-p}{p} \inparen{ \frac{\EE\abs{M}}{n}  }^2 .
\end{align*}
In addition, the fact that the limiting Gaussian has mean exactly $-1/2$ of the variance (contiguity condition) will already be apparent from the first few terms. 

More precisely, by Lemma~\ref{lem:log-likelihood-ratio-partition-functions} (note that $F(A) = 1$ for $p=q$) together with Theorem~\ref{thm:truncate_CE_clogn}, with high probability over $A \sim \cQ$, we have
\begin{align}
    \log \frac{\ud \cP_\lambda}{\ud \cQ}(A) &= \log Z_{A}\!\inparen{ \frac{\lambda}{p}  }  - \log Z_{K_n}(\lambda) \notag \\
    &\approx \sum_{m =1}^{2 \log n} \sum_{e_1,\dots,e_m} \phi(H(e_1,\dots,e_m)) \lambda^m \insquare{ \prod_{j=1}^{m} \frac{A_{e_j}}{p} - 1  } \nonumber \\
    &= \sum_{m =1}^{2 \log n} \sum_{G : |G| = m} \sum_{(e_1,\dots,e_m) \cong G} \phi(H(e_1,\dots,e_m)) \lambda^m \insquare{ \prod_{j=1}^{m} \frac{A_{e_j}}{p} - 1  } ,
    % &= \text{($m = 1$ terms)} + \text{($m = 2$ terms)} + \text{($m = 3$ terms)} + \cdots . 
    \label{eq:nonMatchingLL_intuition_expansion}
\end{align}
where $G$ denotes a template subgraph with $m$ edges.  
To compute the innermost sum corresponding to each $G$, the counts and Ursell functions of clusters up to size 4 are given in Table \ref{table:clusterExpansion_matchingPolynomial}. 
Recall the notation in \eqref{eq:ordinaryCenteredSignedSubgraphCount_definition}: for a template $G$, we use $G(A)$ to denote the number of copies of $G$ in $A$.

\begin{table}[ht]
\centering
\renewcommand{\arraystretch}{2}
\setlength{\tabcolsep}{6pt}
\begin{tabular}{@{}c|cccc@{}}
\toprule
\multicolumn{1}{l}{}   & Cluster & $\psi$ & Ordering & Ursell \\ \midrule
$m=1$                  & \DrawKtwo & 1 & 1 & 1 \\ \midrule
\multirow{2}{*}{$m=2$} & \DrawKtwoRep & 1 & 1 & $-\frac{1}{2}$\\
                       &\DrawPtwo& 1 & 2! & $-\frac{1}{2}$ \\ \midrule
\multirow{5}{*}{$m=3$} &\DrawKtwoRepRep& 1 & 1 &$\frac{1}{3}$ \\
                       &\DrawPtwoRep& 2 & $\frac{3!}{2!}$ &$\frac{1}{3}$ \\
                       &\DrawKthree& 1 & 3! & $\frac{1}{3}$ \\
                       &\DrawSthree& 1 & 3! & $\frac{1}{3}$\\
                       &\DrawPthree& 1  & 3! & $\frac{1}{6}$\\ \bottomrule
\end{tabular}
\qquad
\begin{tabular}{@{}c|cccc@{}}
\toprule
\multicolumn{1}{l}{}   & Cluster & $\psi$ & Ordering & Ursell \\ \midrule
\multirow{9}{*}{$m=4$} &\DrawKtwoRTripleRep& 1 & 1 & $-\frac{1}{4}$ \\
                        &\DrawPtwoAdjRepRep& 1 & $\frac{4!}{2!2!}$ & $-\frac{1}{4}$\\
                        &\DrawPtwoTripleRepRep& 2 & $\frac{4!}{3!}$& $-\frac{1}{4}$\\
                        % &\DrawKthreeRep& 3 &$\frac{4!}{2!}$ &$-\frac{1}{4}$\\
                        &\DrawSthreeRep& 3 &$\frac{4!}{2!}$&$-\frac{1}{4}$\\
                        &\DrawPthreeMiddleRep&1&$\frac{4!}{2!}$&$-\frac{1}{6}$\\
                        &\DrawPthreeLeavesRep&2&$\frac{4!}{2!}$&$-\frac{1}{12}$\\
                        % &\DrawCfour&1&4!&$-\frac{1}{8}$\\
                        % &\DrawKThreeWithTail &1&4!&$-\frac{1}{6}$\\
                        &\DrawSfour&1&4!&$-\frac{1}{4}$\\
                        &\DrawSthreeWithTail &1& 4!&$-\frac{1}{12}$\\
                        &\DrawPfour&1&4!&$-\frac{1}{24}$\\ \bottomrule 
\end{tabular}
\caption{Clusters of size four (selected) and below with corresponding quantities appearing in each summand in \eqref{eq:nonMatchingLL_intuition_expansion}. For each cluster template $G$, let $G_0$ be the simple graph obtained from $G$ by removing any repeated edges. The quantity $\psi(G)$ is the number of ways to place any repeated edges in a labeled version of $G_0$ so that the resulting graph is $G$. 
% up to isomorphism. 
If $G$ is simple set $\psi(G) = 1$. 
The factor ``ordering'' is present because $e_1, \dots, e_m$ are ordered. 
The contribution corresponding to the template $G$ in \eqref{eq:nonMatchingLL_intuition_expansion} is then $\lambda^m \insquare{ G_0(A)/p^m - G_0(K_n)  } \psi(G) \cdot \inbraces{\text{ordering} } \cdot \inbraces{\text{Ursell}}$. 
% \notetim{Can be moved to Appendix} 
}
\label{table:clusterExpansion_matchingPolynomial}
\end{table}

\begin{table}[ht]
\renewcommand{\arraystretch}{2.5}
\setlength{\tabcolsep}{10pt}
\centering
\begin{tabular}{@{}ccc@{}}
\toprule
\multicolumn{1}{c|}{$m=1$}   & \DrawKtwo & $\displaystyle \lambda \insquare{\frac{K_2(A)}{p} - \binom{n}{2}}$ \\ \midrule
\multicolumn{1}{c|}{\multirow{2}{*}{$m=2$}} & \DrawKtwoRep  & $\displaystyle - \frac{\lambda^2}{2} \insquare{ \frac{K_2(A)}{p^2} - \binom{n}{2}   }$     \\
\multicolumn{1}{c|}{}                      & \DrawPtwo    &   $\displaystyle -\lambda^2 \insquare{  \frac{P_2(A)}{p^2} - 3 \binom{n}{3} }$ \\
\bottomrule   
\end{tabular}
\qquad
\begin{tabular}{@{}c|ll@{}}
\toprule
\multirow{5}{*}{$m=3$} & \DrawKtwoRepRep    & $\displaystyle \frac{\lambda^3}{3} \insquare{ \frac{K_2(A)}{p^3} - \binom{n}{2}   }$ \\
                       &   \DrawPtwoRep    &  $\displaystyle 2 \lambda^3 \insquare{ \frac{P_2(A)}{p^3} - 3\binom{n}{3}   }$ \\
                       &  \DrawKthree    &   $\displaystyle 2 \lambda^3 \insquare{ \frac{K_3(A)}{p^3} - \binom{n}{3}   }$     \\
                       &  \DrawSthree    &    $\displaystyle 2 \lambda^3 \insquare{ \frac{S_3(A)}{p^3} - 4\binom{n}{4}   }$     \\
                       &  \DrawPthree      &    $\displaystyle \lambda^3 \insquare{ \frac{P_3(A)}{p^3} - \frac{4!}{2}\binom{n}{4}   }$  \\
\bottomrule  
\end{tabular}
\caption{Subgraph templates and contributions for first few terms in \eqref{eq:nonMatchingLL_intuition_expansion}, where $G_0(A)$ denotes the number of copies of $G_0$ in the graph $A$. We use $K_m$ to denote the complete graph on $m$ vertices, $P_m$ to denote the path of length $m$, and $S_m$ to denote the star with $m$ edges.}
\label{table:nonMatchingLL_intuition_expansion}
\end{table}

With the calculations in Table~\ref{table:clusterExpansion_matchingPolynomial}, we then obtain the contributions corresponding to the first few templates $G$ in Table \ref{table:nonMatchingLL_intuition_expansion}. 
Let $G_0$ be the simple graph obtained from $G$ by removing any repeated edges. 
We make the following observations, bearing in mind $\lambda = \Theta( \frac{1}{n} )$.
\begin{enumerate}
    \item The first time a template subgraph $G_0 = G$ appears (no repeated edges), the corresponding term cancels in expectation and hence produces a zero-mean fluctuation term (e.g.~\DrawKtwo has zero mean and variance $O(1)$).
    \item The second time a base template subgraph $G_0$ appears (exactly one repeated edge in $G$), it contributes essentially a constant order deterministic term (e.g.~$\DrawKtwoRep$ has mean $O(1)$ and variance $O(n^{-2})$).
    \item The third and subsequent times a base template subgraph $G_0$ appears (two or more repeated edges in $G$), it is of smaller order (e.g.~$\DrawKtwoRepRep$ is $O(n^{-1})$). 
    \item Cyclic subgraphs are of smaller order (e.g.~\DrawKthree has zero-mean and variance $O(n^{-2})$ so it is $O_\PP(n^{-1})$). Thus we expect tree subgraphs to dominate. 
\end{enumerate}
The above observations extend also to $m \geq 4$. Importantly, terms corresponding to template graphs that are cyclic or have at least two repeated edges, as well as the small fluctuations coming from graphs with one repeated edge, will be shown to be small in aggregate---they do not conspire to produce non-negligible $O(1)$ terms in the limit. We take this for granted momentarily and carry forward the computation for only the first few terms corresponding to trees with at most one repeated edge.

By the classical CLT and a variance computation, we obtain (recall the notation for the signed edge count $\signedKtwo$ in \eqref{eq:ordinaryCenteredSignedSubgraphCount_definition})
\begin{align*}
    \DrawKtwo = \frac{\lambda}{p} \signedKtwo(A) \overset{d}{\approx} \cN\inparen{ 0, \, \frac{\lambda^2 n^2}{2} \frac{1-p}{p}   } \qquad \text{and} \qquad \DrawKtwoRep \approx \EE\insquare{\DrawKtwoRep} \approx - \frac{\lambda^2 n^2}{4} \frac{1-p}{p}.
\end{align*}
In other words, the edge term and the double edge term combine into a Gaussian with mean equal to $-1/2$ of the variance.

We consider next the wedge term (recall the notation for the centered wedge count $\overline{P}_2$ in \eqref{eq:ordinaryCenteredSignedSubgraphCount_definition}). Note that 
\begin{align*}
    \Corr\insquare{  \overline{P}_2(A), \signedKtwo  } = \frac{\Cov\insquare{  \overline{P}_2(A), \signedKtwo  } }{\sqrt{\Var \overline{P}_2(A)} \sqrt{\Var \signedKtwo} } \approx \frac{\binom{n}{3} 6 p^2 (1 - p)}{\sqrt{ \binom{n}{4} 2 \cdot 4! \cdot  p^3(1-p)}\sqrt{ \binom{n}{2} p (1-p) }} \longrightarrow 1.
\end{align*}
Therefore $\overline{P}_2(A)$ is asymptotically a linear function of $\signedKtwo(A)$ in an $L_2$ sense (in fact this is true for all centered subgraph counts). We have 
\begin{align*}
    \DrawPtwo = - \frac{\lambda^2}{p^2} \overline{P}_2(A) \approx - \frac{\lambda^2}{p^2} \frac{\Cov\insquare{  \overline{P}_2(A), \signedKtwo(A) }}{\Var \signedKtwo(A)} \signedKtwo(A) \approx - \frac{2\lambda^2 n}{p} \signedKtwo(A).
\end{align*}
In particular, the randomness in $\overline{P}_2(A)$ is approximately the same as in the signed edge count $\signedKtwo$. 

Repeat this procedure for simple trees with $m=3$ edges---$S_3$ and $P_3$, in each case projecting the centered subgraph count in the direction of $\signedKtwo$. Let
\begin{align*}
    \sigma = \lambda n \sqrt{\frac{1-p}{2p}} \qquad \text{and} \qquad Z \sim \cN(0,1),
\end{align*}
noting that $\sigma = O(1)$. The contributions from the first few terms are summarized as follows. The zero-mean fluctuation contributions from the $m=1$, $2$, $3$ terms are:
\begin{align*}
    & m = 1 \qquad \begin{tikzpicture}[scale=1]
        \SetGraphUnit{1}
        \GraphInit[vstyle=Hasse]
        \SetVertexSimple[MinSize=2pt]
        \Vertex{A}\EA(A){B}\Edge(A)(B)
    \end{tikzpicture} :\qquad \qquad\qquad\qquad\qquad\quad   \sigma Z \\
    & m = 2 \qquad \begin{tikzpicture}
        \SetGraphUnit{1}
        \GraphInit[vstyle=Hasse]
        \SetVertexSimple[MinSize=2pt]
        \Vertex[x = -0.5, y = 0.1]{A}
        \Vertex[x=0,y=0]{B}
        \Vertex[x = 0.5, y = 0.1]{C}
        \Edge(B)(C)
        \Edge(A)(B)
    \end{tikzpicture} : \qquad\qquad\qquad\qquad\qquad\quad  -2\lambda n \sigma Z \\
    & m = 3 \qquad \begin{tikzpicture}
        \SetGraphUnit{1}
        \GraphInit[vstyle=Hasse]
        \SetVertexSimple[MinSize=2pt]
        \Vertex[x = 0, y = 0]{A}
        \Vertex[x=0.5,y=0]{B}
        \Vertex[x = 1, y = 0.3]{C}
        \Vertex[x = 1, y = -0.3]{D}
        %\SOEA[Lpos=10, unit=0.5](A){C}
        %\NOEA[unit=0.5](C){B}
        \Edge(B)(C)
        \Edge(A)(B)
        \Edge(B)(D)
    \end{tikzpicture}, \quad \begin{tikzpicture}
        \SetGraphUnit{1}
        \GraphInit[vstyle=Hasse]
        \SetVertexSimple[MinSize=2pt]
        \Vertex[x = 0, y = 0]{A}
        \Vertex[x=0.5,y=0]{B}
        \Vertex[x = 1, y = 0]{C}
        \Vertex[x = 1.5, y = 0]{D}
        \Edge(B)(C)
        \Edge(A)(B)
        \Edge(C)(D)
    \end{tikzpicture}  : \qquad\qquad\quad 5 \lambda^2 n^2 \sigma Z.
\end{align*}
They together contribute a variance $(\sigma - 2 \lambda n \sigma + 5 \lambda^2 n^2 \sigma)^2 = \sigma^2 (1 - 4 \lambda n + 14 \lambda^2 n^2 + O(\lambda^3 n^3))$. 
The mean (deterministic) contribution from the $m=2$, $3$, $4$ terms are: 
\begin{align*}
    & m = 2 \qquad \begin{tikzpicture}
        \SetGraphUnit{1}
        \GraphInit[vstyle=Hasse]
        \SetVertexSimple[MinSize=2pt]
        \Vertex{A}\EA(A){B}
        \Edge[style={bend left = 10}](A)(B)
        \Edge[style={bend left = 10}](B)(A)
    \end{tikzpicture} : \qquad\qquad\qquad\qquad\qquad\qquad\qquad -\frac{1}{2}\sigma^2  \\
    & m = 3 \qquad \begin{tikzpicture}
        \SetGraphUnit{1}
        \GraphInit[vstyle=Hasse]
        \SetVertexSimple[MinSize=2pt]
        \Vertex[x = -0.5, y = 0.1]{A}
        \Vertex[x=0,y=0]{B}
        \Vertex[x = 0.5, y = 0.1]{C}
        \Edge(B)(C)
        \Edge[style={bend left = 15}](A)(B)
        \Edge[style={bend left = 15}](B)(A)
    \end{tikzpicture} : \qquad\qquad\qquad\qquad\qquad\qquad\qquad  2\lambda n \sigma^2  \\
    & m = 4  \qquad    % S_3 one repeated edge
    \begin{tikzpicture}
        \SetGraphUnit{1}
        \GraphInit[vstyle=Hasse]
        \SetVertexSimple[MinSize=2pt]
        \Vertex[x = 0, y = 0]{A}
        \Vertex[x=0.5,y=0]{B}
        \Vertex[x = 1, y = 0.3]{C}
        \Vertex[x = 1, y = -0.3]{D}
        \Edge[style={bend left = 15}](A)(B)
        \Edge[style={bend left = 15}](B)(A)
        \Edge(B)(C)
        \Edge(B)(D)
    \end{tikzpicture}, \quad     % P_3 one repeated edge in middle
    \begin{tikzpicture}
        \SetGraphUnit{1}
        \GraphInit[vstyle=Hasse]
        \SetVertexSimple[MinSize=2pt]
        \Vertex[x = 0, y = 0]{A}
        \Vertex[x=0.5,y=0]{B}
        \Vertex[x = 1, y = 0]{C}
        \Vertex[x = 1.5, y = 0]{D}
        \Edge[style={bend left = 15}](B)(C)
        \Edge[style={bend left = 15}](C)(B)
        \Edge(A)(B)
        \Edge(C)(D)
    \end{tikzpicture}, \quad 
    % P_3 one repeated edge on leaves
    \begin{tikzpicture}
        \SetGraphUnit{1}
        \GraphInit[vstyle=Hasse]
        \SetVertexSimple[MinSize=2pt]
        \Vertex[x = 0, y = 0]{A}
        \Vertex[x=0.5,y=0]{B}
        \Vertex[x = 1, y = 0]{C}
        \Vertex[x = 1.5, y = 0]{D}
        \Edge[style={bend left = 15}](A)(B)
        \Edge[style={bend left = 15}](B)(A)
        \Edge(B)(C)
        \Edge(C)(D)
    \end{tikzpicture}: \qquad  - 7 \lambda^2 n^2 \sigma^2.
\end{align*}
Altogether, they combine to give a Gaussian random variable
\begin{align}\label{eq:first_1,2,3,4_clusters_mean_and_fluctuation}
    &\cN\inparen{ \inparen{-\frac{1}{2} + 2\lambda n - 7 \lambda^2 n^2 + \cdots } \sigma^2, \, \inparen{1 - 4\lambda n + 14 \lambda^2 n^2 + \cdots} \sigma^2  }.
\end{align}
The pattern that the mean equals $-1/2$ of the variance continues to hold. On the other hand, similar computations reveal that the series for $\EE\abs{M}$ in \eqref{eq:E|M|_truncate_clogn} is also dominated by the simple tree templates (made precise in Proposition \ref{prop:GibbsMatching_mean_var_from_CE}). Using Table \ref{table:clusterExpansion_matchingPolynomial}, the dominant first few terms of $\EE\abs{M}$ are seen to be
\begin{align*}
    \EE \abs{M} \sim \frac{n^2 \lambda}{2} - n^3 \lambda^2 + \frac{5}{2} n^4 \lambda^3 + \cdots.
\end{align*}
Rewriting the series in \eqref{eq:first_1,2,3,4_clusters_mean_and_fluctuation} as a square, we find that
\begin{align}\label{eq:nonMatching_LLfirstFewTerms_GaussianForm}
    \log \frac{\ud \cP_\lambda}{\ud \cQ}(A)
    &\overset{d}{\approx} \cN\!\inparen{  - \frac{1}{4}\inparen{ \underbrace{  \lambda n - 2\lambda^2 n^2 + 5 \lambda^3 n^3 + \cdots }_{= 2 \EE\abs{M} /n } }^2 \frac{1-p}{p},  \, \frac{1}{2}\inparen{ \underbrace{ \lambda n - 2\lambda^2 n^2 + 5 \lambda^3 n^3 + \cdots }_{= 2 \EE\abs{M} /n } }^2 \frac{1-p}{p}  },
\end{align}
which explains why we expect Theorem~\ref{thm:nonMatching_LLdistributionUnderNull} to hold!

\section{Concluding remarks and future directions}

This paper studies a hypothesis testing problem of distinguishing between two models $\cP_{\lambda}$ and $\cQ$. The planted model $\cP_{\lambda}$ consists of a matching $M$ drawn from the monomer-dimer model on $K_n$ with dimer density $\lambda$ superimposed with an \ER~$G(n,p)$. The null model $\cQ$ is a plain \ER~$G(n,q)$. In the critical regime, we provide a precise, finite-sample characterization of the log-likelihood ratio $\log \frac{\ud \cP_\lambda}{\ud \cQ}(A)$ for $A \sim \cQ$. This is accomplished for both the cases (i) $p = q$, i.e.~equal ambient edge density, and (ii) $p \neq q$ with $p$, $q$ chosen such that $\cP_\lambda$ and $\cQ$ have equal average edge density. This allows us to elucidate the fundamental limits of detection. Together with the computationally efficient edge or wedge count test statistics which attain the optimal total variation rate, our results confirm the absence of a statistical-to-computational gap. 

Additionally, one of the goals of this paper is to demonstrate the value of the cluster expansion as a tool in mathematical statistics. The techniques presented here for studying log-likelihood ratios are very different from more established methods such as orthogonal decompositions of the likelihood ratio. To list just one striking difference: in cluster expansions the log is taken at the very \emph{first} step, whereas if orthogonal decomposition techniques are employed to study log-likelihood ratios, the log is typically taken at the very \emph{last} step \cite{janson1994numbers}.

Although the cluster expansion can provide remarkably precise results---as demonstrated here in a statistical setting and elsewhere through its vast successes in other fields---there remains a limitation regarding convergence. Outside the disk of convergence, statements can only remain formal. Nevertheless, we offer some encouraging observations. As shown by the similar asymptotic log-likelihood distributions for both $\lambda = \Theta(1/n)$ and $\lambda = \infty$ in this manuscript, cluster expansions may still provide useful and ``correct" information outside the disk of convergence. One plausible explanation, at least where the monomer-dimer model is concerned, comes from the Heilmann-Lieb theorem \cite{heilmann1972theory} providing analyticity of the log partition function for all real $\lambda > 0$, yielding an absence of phase transitions in the Lee-Yang sense (see e.g.~\cite[Section 3.7]{friedli2017statistical}) across all such $\lambda$. In other words, the technical issue of convergence may turn out to have no bearing on certain qualitative aspects of the system. We refer to \cite{quitmann2024decay} who extended the exponential decay of correlations in the monomer-dimer model on lattice graphs, obtained by cluster expansion at small densities, across the entire range of physical parameter values. Establishing analogous extensions in the planted matching detection problem is an interesting problem for future research.

Along these lines, the planted clique heuristics discussed in Section \ref{sec:plantedClique} culminated in the formal $\KL$ approximation \eqref{eq:truncated-KL-planted-clique} which at least exposes the familiar information-theoretic and computational thresholds for detection. Given the considerable interest in the planted clique model as a canonical example for studying statistical-to-computational gaps, we consider it an exciting direction to extract rigorous insights that build upon these preliminary heuristics. We point to \cite{mousset2020probability} for an example of cluster expansion-type techniques being used to give asymptotics of probabilities of subgraph containment in random graphs or arithmetic progressions in random subset of integers, even when operating in regimes where the full expansion may be non-convergent.

Finally, it is natural to consider applications of the techniques in this paper to other planted subgraph problems. For instance, $k$-factors consisting of vertex-disjoint components (e.g.~triangle factors \cite{krivelevich1997triangle}) are suitable candidates as they also display hardcore repulsive interactions. It is also of interest to reach towards hypergraph settings \cite{adomaityte2022planted}. On a different note, one may consider other ambient random graph ensembles besides the standard \ER. Inhomogeneous \ER~graphs for instance, may exhibit non-Gaussian asymptotic subgraph count distributions \cite{bhattacharya2023fluctuations}. We remark that the asymptotic jointly Gaussian distribution of signed subgraph counts features heavily in the orthogonal decomposition techniques in \cite{janson1994numbers,janson1994orthogonal}, whence Wick's formula and Hermite polynomial identities are critical in establishing log-normality of the likelihood ratio. The cluster expansion may therefore be advantageous in this case since, for instance, Gaussianity plays no role whatsoever in the proofs of Theorems \ref{thm:nonMatching_LLdistributionUnderNull} and \ref{thm:MatchingEdgeDensity_LLdistributionUnderNull}.

% \begin{itemize}
%     \item We may not be able to show that the cluster expansion of the difference of log-partition functions converges for larger values of $\lambda$. But if we just worked formally, we will achieve Janson's result if we can pick $\lambda \gg 1/n$. 

%     Possibly reference \cite{mousset2020probability} for CE being useful outside of radius of convergence.

%     The other approach is to analytically extend the results for small $\lambda$ to larger $\lambda$. See Theorem 3.2 of Alexandra Quitmann---Decay of correlations in the monomer-dimer model.
%     %\item More generally, we want to study sharp thresholds for planted graphs that are built from unions of disjoint components. We first focus on planting matchings.

%     \item Planted matching/$k$-factors on hypergraphs. \cite{adomaityte2022planted}.
% \end{itemize}

\section*{Acknowledgements}

We are very grateful to Will Perkins for invaluable discussions about the monomer-dimer model and cluster expansions, in particular pointing us to the use of the Penrose tree-graph bound. CM was supported in part by NSF grant DMS-2338062.

\bibliography{testing}
\bibliographystyle{alpha}

\appendix

\section{Thermodynamic limits of the monomer-dimer model} 
% \label{sec:monomer-dimer-model}

% \subsection{Thermodynamic limits of the monomer-dimer model}
\label{sec:thermodynamicLimits_MDmodel}
% \notetim{This section onwards is needed. DO NOT DELETE.}

We now state a result by \cite{alberici2014mean} that immediately implies Lemma~\ref{lem:matching-limiting-size}. 

\begin{theorem}(\cite[Proposition 2 and Remark 9]{alberici2014mean})
    \label{thm:thermodynamicLimits_MD}
    Let $b > 0$ and suppose 
    \begin{align*}
        \lambda = \lambda_n := \frac{1}{2e^{1 + b} n }. 
    \end{align*}
    Define $h := \frac{1+b}{2} + \log \sqrt{2}$ and 
    \begin{align}
        g(h) := \frac{1}{2} \inparen{  \sqrt{e^{4h} + 4e^{2h} } - e^{2h}  }.
    \end{align}
    Then
    \begin{enumerate}[(i)]
        \item the thermodynamic limit of the free energy of the monomer-dimer model exists and
        \begin{align}
            \label{eq:thermodynamicLimits_MD_FE}
        \lim_{n \rightarrow \infty }\frac{1}{n} \log Z_{K_{n}}(\lambda) = - \frac{1 - g(h)}{2} - \log g(h).
        \end{align} 
        \item Additionally, the expected matching size of the monomer-dimer model converges as
        \begin{align}
            \label{eq:thermodynamicLimits_MD_mean}
            \lim_{n \rightarrow \infty }\frac{2 \EE_{M \sim \mu_\lambda} \abs{M}}{n} = 1 - g(h).
        \end{align}
    \end{enumerate}
\end{theorem}
\begin{remark}
    The mapping between our notation and that of \cite{alberici2014mean} is as follows. The partition function of the monomer-dimer model on $K_n$ considered in \cite{alberici2014mean} is
\begin{align*}
    Z^{\text{MD}}_n(h, w) := \sum_{M \text{ matching} } w^{\abs{M}} e^{h(N - 2\abs{M})}. 
\end{align*}
That is, $e^h$ and $w$ are the monomer and dimer activity parameters respectively. Setting $h := - \frac{1}{2} \log (\lambda n)$, it is easily checked that
\begin{align*}
    \frac{1}{n} \log Z_{K_n}(\lambda) = -h + \frac{1}{n} \log Z_{n}^{\text{MD}}\inparen{  h, \frac{1}{n}  }.
\end{align*}
Then \cite[Proposition 2 and Remark 9]{alberici2014mean} gives
\begin{align*}
    \lim_{n \rightarrow \infty} \frac{1}{n} \log Z_{n}^{\text{MD}}\inparen{  h, \frac{1}{n}  } = h - \frac{1 - g(h)}{2} - \log g(h).
\end{align*}
This leads to \eqref{eq:thermodynamicLimits_MD_FE}. Next, the expected monomer density defined in \cite[Remark 2]{alberici2014mean} as 
\begin{align*}
    m_n^{\text{MD}} := \frac{\partial}{\partial h_n} \frac{1}{n} \log Z_{n}^{\text{MD}}\inparen{  h, \frac{1}{n}  }
\end{align*}
satisfies the relation
\begin{align*}
    m_n^{\text{MD}} = 1 - \frac{2\EE_{M \sim \mu_\lambda} \abs{M}}{n}.
\end{align*}
Then \cite[Remark 9]{alberici2014mean} gives $m_n^{\text{MD}} \longrightarrow g(h)$ and this leads to \eqref{eq:thermodynamicLimits_MD_mean}.
\end{remark}

\section{Analysis of the edge count and the wedge count}
\label{sec:analysis-edge-wedge-count}

In this section, we analyze the signed edge count and the signed wedge count, thereby establishing the positive results for detection, Theorems~\ref{thm:threshold-edge-count} and~\ref{thm:matchingEdgeDensities_compEfficientStat_mainResult}.

\subsection{Proof of Theorem~\ref{thm:threshold-edge-count}}
% \paragraph{Computationally efficient test statistic.}
% Remark that for $K_2$, the signed and centered versions coincide: $\signedKtwo(A) \equiv \overline{K}_2(A)$, where $\overline{K}_2(A) = K_2(A) - \EE_{A \sim \cQ} K_2(A)$. However, it is conceptually clearer to use the signed notation. 

The following lemma gives the mean and variance of the signed edge count and establishes its asymptotic normality under the planted and null distributions.
Theorem~\ref{thm:threshold-edge-count} then follows immediately.

% \begin{lemma}\label{lemma:nonMatching_signedKtwo_plantedAndNull_meanVar} 
% Suppose Assumption \ref{asmpt:nonMatchingEdges} holds. Then we have 
%     \begin{alignat*}{2}
%         &\text{(i)}\,\, \EE_{\cQ} \signedKtwo(A) = 0, &\qquad\qquad \text{(ii)}\,\, &\Var_{\cQ} \signedKtwo(A) = \binom{n}{2} p(1-p), \\
%         & \text{(iii)}\,\, \EE_{\cP_\lambda} \signedKtwo(A) = \EE\abs{M}(1-p), &\qquad\qquad  \text{(iv)}\,\, &\Var_{\cP_\lambda} \signedKtwo(A) = \binom{n}{2} p(1-p) + O(n).
%     \end{alignat*}
%     % \begin{enumerate}[(i)]
%     %     \item $\EE_{\cQ} \signedKtwo(A) = 0$
%     %     \item $\Var_{\cQ} \signedKtwo(A) = \binom{n}{2} p(1-p)$
%     %     \item $\EE_{\cP_\lambda} \signedKtwo(A) = \EE\abs{M}(1-p)$
%     %     \item $\Var_{\cP_\lambda} \signedKtwo(A) = \binom{n}{2} p(1-p) + O(n)$.
%     % \end{enumerate}
% \end{lemma}

% \begin{proof}
% They follow by direct computation.
% \end{proof}

% The result follows immediately from Lemma \ref{lemma:signed_K2_normality_P_Q} and Theorem \ref{thm:plantedMatchings_mean_is_minus_half_variance} \ref{thm:nonMatching_TV_mainResult}.

% We then prove the asymptotic normality of $\signedKtwo$ under the two models.

\begin{lemma}\label{lemma:signed_K2_normality_P_Q}
    % For $c$ defined in \eqref{eq:MDMean_convergence_to_c}, we have

Suppose Assumption \ref{asmpt:nonMatchingEdges} holds. Let $\signedKtwo$ be defined by \eqref{eq:def-signedKtwo}. Then we have 
    \begin{alignat*}{2}
        &\text{(i)}\,\, \EE_{\cQ} \signedKtwo(A) = 0, &\qquad\qquad \text{(ii)}\,\, &\Var_{\cQ} \signedKtwo(A) = \binom{n}{2} p(1-p), \\
        & \text{(iii)}\,\, \EE_{\cP_\lambda} \signedKtwo(A) = \EE\abs{M}(1-p), &\qquad\qquad  \text{(iv)}\,\, &\Var_{\cP_\lambda} \signedKtwo(A) = \binom{n}{2} p(1-p) + O(n^{3/2}).
    \end{alignat*}
Moreover, 
    \begin{equation*}
        \frac{\signedKtwo(A)}{\sqrt{\binom{n}{2}p(1-p)}} \overset{d}{\longrightarrow}_{\cQ} \cN\!\inparen{ 0, 1  } \quad \text{and} \quad \frac{\signedKtwo(A)}{\sqrt{\binom{n}{2}p(1-p)}} \overset{d}{\longrightarrow}_{\cP_\lambda} \cN\!\inparen{  \sqrt{\frac{1-p}{p}} \frac{c}{\sqrt{2}}, \, 1  } .
    \end{equation*}
    % For $A \sim \cQ$,
    % \begin{align}
    %     \frac{\signedKtwo(A)}{\sqrt{\binom{n}{2}p(1-p)}} \overset{d}{\longrightarrow} \cN\!\inparen{ 0, 1  }
    % \end{align}
    % For $A \sim \cP_\lambda$,
    % \begin{align}
    %     \frac{\signedKtwo(A)}{\sqrt{\binom{n}{2}p(1-p)}} \overset{d}{\longrightarrow} \cN\!\inparen{  \sqrt{\frac{1-p}{p}} \frac{c}{\sqrt{2}}, \, 1  }
    % \end{align}
\end{lemma}

\begin{proof}
It is straightforward to compute the mean and variance of $\signedKtwo$ under $\cQ$, and its asymptotic normality is immediate by the classical CLT.
       
    Let $A \sim \cP_\lambda$. Let $M := \inbraces{M_{ij}}$ be indicator random variables with $M_{ij} = 1$ if edge $\inbraces{i,j}$ is in the planted matching, and $M_{ij} = 0$ otherwise. Note that $A \sim \cP_\lambda$ can be regarded as the union between $\tilde{A} \sim G(n,p)$ and $M$, with $\tilde{A}$ independent of $M$. We may write
    \begin{align*}
        \signedKtwo(A) &\overset{d}{=} \sum_{\inbraces{i,j} \in \binom{[n]}{2} } (M_{ij} (1-\widetilde{A}_{ij}) + \widetilde{A}_{ij} - p) = U + V
    \end{align*}
    where 
    \begin{align*}
        U &:= \sum_{\inbraces{i,j} \in \binom{[n]}{2} } M_{ij}(1- \tilde{A}_{ij}), \qquad\text{and}\qquad V := \sum_{\inbraces{i,j} \in \binom{[n]}{2} } (\tilde{A}_{ij} - p).
    \end{align*}
    Note that $\EE U = \EE\abs{M}(1-p)$ and $\EE V = 0$, so $\EE \signedKtwo(A) = \EE\abs{M}(1-p)$. By the law of total variance, 
    \begin{align*}
        \Var U = \Var \insquare{ \EE\insquare{ U \, | \, M  }  } + \EE\insquare{\Var \insquare{  U \, | \, M  }}.
    \end{align*}
    We have $\Var \insquare{ \EE\insquare{ U \, | \, M  }  } = (1-p)^2 \Var \abs{M} = O(n)$ by Proposition~\ref{prop:Var|M|_Cn_bound},
    % if $\lambda = \frac{1}{\zeta n}$ and $\Var \abs{M} = 0$ if $\lambda = \infty$, 
    and also $\EE\insquare{\Var \insquare{  U \, | \, M  }} = p(1-p) \EE \abs{M} = O(n)$ by Lemma~\ref{lem:matching-limiting-size}. Thus $\Var U = O(n)$. 
    In addition, $\Var V = \binom{n}{2} p(1-p)$.
    We conclude that 
    $$
    \Var \signedKtwo = \Var U + \Var V + O\left( \sqrt{\Var U \cdot \Var V} \right) = 
    \binom{n}{2} p(1-p) + O(n^{3/2}) .
    $$
    Moreover, by Lemma~\ref{lem:matching-limiting-size}, 
    $$
        \frac{U}{\sqrt{\binom{n}{2}p(1-p)}} = \frac{ \EE\abs{M}(1-p) + O_{\PP}(\sqrt{n})}{\sqrt{\binom{n}{2}p(1-p)}} = \sqrt{\frac{1-p}{p}} \frac{c}{\sqrt{2}} + o_{\PP}(1).
    $$
    On the other hand, $V / \sqrt{\binom{n}{2}p(1-p)} \overset{d}{\longrightarrow} \cN(0,1)$ by the classical CLT. This completes the proof.
\end{proof}

\subsection{Proof of Theorem~\ref{thm:matchingEdgeDensities_compEfficientStat_mainResult}}

% \paragraph{Computationally efficient test statistic.} 
% The mean and variance of $\signedPtwo$ under the planted and null distributions are computed directly and stated below. In what follows, we suppose that Assumption \ref{asmpt:MatchingEdgeDensity} holds, and $c$ is defined in \eqref{eq:MDMean_convergence_to_c}.

The following lemma gives the mean and variance of the signed wedge count and establishes its asymptotic normality under the planted and null distributions.
Theorem~\ref{thm:matchingEdgeDensities_compEfficientStat_mainResult} then follows immediately.

% \begin{lemma}\label{lemma:matchingEdgeDensities_signedKPwo_plantedAndNull_meanVar} We have 
%     \begin{alignat*}{2}
%         &\text{(i)}\,\, \EE_{\cQ} \signedPtwo(A) = 0, &\qquad\qquad \text{(ii)}\,\, &\Var_{\cQ} \signedPtwo(A) = 3\binom{n}{3} q^2(1-q)^2, \\
%         & \text{(iii)}\,\, \EE_{\cP_\lambda} \signedPtwo(A) \sim -\frac{2\inparen{\EE\abs{M}}^2}{n}, &\qquad\qquad  \text{(iv)}\,\, &\Var_{\cP_\lambda} \signedPtwo(A) = 3\binom{n}{3} q^2(1-q)^2 + o(n^2).
%     \end{alignat*}
%     % \begin{enumerate}[(i)]
%     %     \item $\EE_{\cQ} \signedKtwo(A) = 0$
%     %     \item $\Var_{\cQ} \signedKtwo(A) = \binom{n}{2} p(1-p)$
%     %     \item $\EE_{\cP_\lambda} \signedKtwo(A) = \EE\abs{M}(1-p)$
%     %     \item $\Var_{\cP_\lambda} \signedKtwo(A) = \binom{n}{2} p(1-p) + O(n)$.
%     % \end{enumerate}
% \end{lemma}

% Theorem \ref{thm:matchingEdgeDensities_compEfficientStat_mainResult} result follows immediately from Lemma \ref{lemma:signed_P2_normality_P_Q} and Theorem \ref{thm:matchingEdgeDensities_mean_is_minus_half_variance} \ref{thm:matchingEdgeDensities_TV_mainResult}.

\begin{lemma}\label{lemma:signed_P2_normality_P_Q}
Suppose 
Assumption~\ref{asmpt:MatchingEdgeDensity} holds. Let $\signedPtwo$ be defined by \eqref{eq:def-signedPtwo}. Write $\sim$ to mean equality to leading order terms. 
    % Let $c$ be defined in \eqref{eq:MDMean_convergence_to_c}. 
    Then we have
\begin{alignat*}{2}
        &\text{(i)}\,\, \EE_{\cQ} \signedPtwo(A) = 0, &\qquad\qquad \text{(ii)}\,\, &\Var_{\cQ} \signedPtwo(A) = 3\binom{n}{3} q^2(1-q)^2, \\
        & \text{(iii)}\,\, \EE_{\cP_\lambda} \signedPtwo(A) \sim -\frac{2\inparen{\EE\abs{M}}^2}{n}, &\qquad\qquad  \text{(iv)}\,\, &\Var_{\cP_\lambda} \signedPtwo(A) = 3\binom{n}{3} q^2(1-q)^2 + o(n^2).
    \end{alignat*}
    Moreover, 
    \begin{equation*}
        \frac{\signedPtwo(A)}{\sqrt{3\binom{n}{3}q^2(1-q^2)}} \overset{d}{\longrightarrow}_{\cQ} \cN\!\inparen{ 0, 1  } \quad \text{and} \quad \frac{\signedPtwo(A)}{\sqrt{3\binom{n}{3}q^2(1-q^2)}} \overset{d}{\longrightarrow}_{\cP_\lambda} \cN\!\inparen{  -\frac{c^2}{\sqrt{2} \theta}, \, 1  } .
    \end{equation*}
\end{lemma}
\begin{proof}[Proof of Lemma \ref{lemma:signed_P2_normality_P_Q}]
It is straightforward to compute the mean and variance of $\signedPtwo$ under $\cQ$, and its asymptotic normality follows immediately from \cite[Theorem 1]{janson1994orthogonal}.

    We next consider $A \sim \cP_\lambda$. 
    % Associate any matching $M$ drawn from $\mu_\lambda$ with the indicators $\inbraces{M_{ij}}$ with $M_{ij} = 1$ if $\inbraces{i,j}$ is in the matching and zero otherwise. 
    Using the same notation as in the proof of Lemma~\ref{lemma:signed_K2_normality_P_Q}, 
    we have $A_{ij} = M_{ij} (1-\widetilde{A}_{ij}) + \widetilde{A}_{ij}$, where $\widetilde{A} \sim G(n,p)$ is independent of $M$. In what follows, we write $\sum_{i\text{--}j\text{--}k}$ to mean $\sum_{j \in [n]} \sum_{\inbraces{i,k} \in \binom{[n]\setminus j}{2}}$. Decompose $\signedPtwo(A)$ as
    \begin{align*}
        \signedPtwo(A) &= \sum_{i\text{--}j\text{--}k} \inparen{ M_{ij} + (1 - M_{ij})\widetilde{A}_{ij} - p + p - q  } \inparen{ M_{jk} + (1 - M_{jk})\widetilde{A}_{jk} - p + p - q  } \\
        &= \text{I} + \text{II} + \text{III} + \text{IV} + \text{V} + \text{VI},
    \end{align*}
    where, using symmetry, 
    \begin{alignat*}{2}
        &\text{I} = \sum_{i\text{--}j\text{--}k} M_{ij} (1 - \widetilde{A}_{ij}) M_{jk} (1 - \widetilde{A}_{jk}), 
        &\qquad & 
        \text{II} = 2 \sum_{i\text{--}j\text{--}k} M_{ij} (1-\widetilde{A}_{ij})(p-q) ,\\
        &\text{III} = 2\sum_{i\text{--}j\text{--}k} M_{ij} (1-\widetilde{A}_{ij}) (\widetilde{A}_{jk} - p) ,
        &\qquad  
        &\text{IV} = 2\sum_{i\text{--}j\text{--}k} (\widetilde{A}_{ij} - p) (p - q),\\
        &\text{V} = \sum_{i\text{--}j\text{--}k} (p - q)^2 ,
        &\qquad  
        &\text{VI} = \sum_{i\text{--}j\text{--}k} (\widetilde{A}_{ij} - p) (\widetilde{A}_{jk} - p).
    \end{alignat*}
    Term $\text{I}$ is identically zero, since $\inbraces{i,j}$ and $\inbraces{j,k}$ cannot simultaneously be in a matching. The expected value of $\text{II}$ is, using $p - q \sim - 2\EE\abs{M}/n^2$ (see Assumption~\ref{asmpt:MatchingEdgeDensity}), 
    \begin{align*}
        \EE\,\text{II} &= \EE\insquare{ 2(p-q)(n-2)\sum_{\inbraces{i,j}} M_{ij}(1-\widetilde{A}_{ij})  } \\
        &\sim - \frac{4\EE\abs{M}}{n} \EE\insquare{ \EE\insquare{   \sum_{\inbraces{i,j}} M_{ij} (1-\widetilde{A}_{ij})  \,\bigg\rvert\, M  }  } = - \frac{4\inparen{\EE\abs{M}}^2}{n}.
    \end{align*}
    In addition, 
    \begin{align*}
        \Var\insquare{ \sum_{\inbraces{i,j}} M_{ij} (1-\widetilde{A}_{ij})  } &= \EE \Var\insquare{  \sum_{\inbraces{i,j}} M_{ij} (1-\widetilde{A}_{ij}) \,\bigg\rvert\, M  } + \Var \, \EE\insquare{ \sum_{\inbraces{i,j}} M_{ij} (1-\widetilde{A}_{ij})  \,\bigg\rvert\, M  } \\
        &= \EE \insquare{ \abs{M} p(1-p)  } + \Var\insquare{ \abs{M}(1-p)   } 
        = O(n) 
    \end{align*}
    by Lemma~\ref{lem:matching-limiting-size} and Proposition~\ref{prop:Var|M|_Cn_bound}.  
    It follows that
    \begin{equation*}
        \Var \text{II} = O \left( \Var\insquare{ \sum_{\inbraces{i,j}} M_{ij} (1-\widetilde{A}_{ij})  } \right) = O(n) . 
    \end{equation*}
    % \begin{align*}
    %     \Var \text{II} &= O \left( \Var\insquare{ \sum_{\inbraces{i,j}} M_{ij} (1-\widetilde{A}_{ij})  } \right) \\
    %     &= C\inbraces{\EE \Var\insquare{  \sum_{\inbraces{i,j}} M_{ij} (1-\widetilde{A}_{ij}) \,\bigg\rvert\, M  } + \Var \, \EE\insquare{ \sum_{\inbraces{i,j}} M_{ij} (1-\widetilde{A}_{ij})  \,\bigg\rvert\, M  } } \\
    %     &= \EE \insquare{ \abs{M} p(1-p)  } + \Var\insquare{ \abs{M}(1-p)   } \leq Cn, 
    % \end{align*}
    % using the variance bound Proposition \ref{prop:Var|M|_Cn_bound}. 
    One can similarly show that 
    \begin{align*}
        \EE \text{III}  = 0, \,\Var \text{III} = O(n^2 q), \quad \quad \EE \text{IV}  = 0, \,\Var \text{IV} = O(n^2 q) \quad\text{and}\quad \text{V}  \sim \frac{2(\EE\abs{M})^2}{n}.
    \end{align*}
    % \begin{alignat*}{2}
    %     &\EE \text{III}  = 0, &\qquad& \Var \text{III} = O(n^2 q), \\
    %     &\EE \text{IV}  = 0, &\qquad& \Var \text{IV} = O(n^2 q), \\
    %     &\text{V}  \sim 2n(\EE\abs{M})^2.
    % \end{alignat*}
    % Note that $\sqrt{\Var \signedPtwo(A)} = O(n^{3/2} q) = O(n)$ by Lemma \ref{lemma:matchingEdgeDensities_signedKPwo_plantedAndNull_meanVar}.
Term VI has mean zero and variance $3\binom{n}{3} q^2(1-q)^2$. 
The mean and variance of $\signedPtwo$ then follow by combining terms I--VI. 
% Furthermore, note that we have $\sqrt{3\binom{n}{3}q^2(1-q^2)} = \Theta(n)$. 
Furthermore, scaling terms I--V by $1/\sqrt{3\binom{n}{3}q^2(1-q^2)} = \Theta(1/n)$, only terms $\text{II}$ and $\text{V}$ contribute a deterministic $\Theta(1)$ term:
    \begin{align*}
        \frac{\text{I} + \text{II} + \text{III} + \text{IV} + \text{V}}{\sqrt{3\binom{n}{3}q^2(1-q^2)}} = - \frac{1}{\sqrt{2}\sqrt{n}q } \inparen{ \frac{2\EE\abs{M}}{n}  }^2 + o_\PP(1) \overset{d}{\longrightarrow} - \frac{c^2}{\sqrt{2} \theta} .
    \end{align*}
    The proof is complete by \cite[Theorem 1]{janson1994orthogonal} giving 
    \begin{align*}
        \frac{\text{VI}}{\sqrt{3\binom{n}{3}q^2(1-q^2)}} &\overset{d}{\longrightarrow} \cN(0,1).\qedhere
    \end{align*}
\end{proof}

\section{Proofs for the cluster expansion}

\subsection{Cluster expansion convergence}
\label{sec:cluster-expansion-convergence}

We first prove the convergence of the cluster expansion, Theorem \ref{thm:truncate_CE_clogn}.
The main tool is the celebrated Penrose tree-graph bound. 

\begin{lemma}(Penrose tree-graph bound \cite[Equation 7]{penrose1967convergence}). 
    \label{lemma:Penrose_tree_bound}
    Let $H$ be a graph, identified with its own edge set, and let $\inbraces{w_e}_{e \in H}$ be complex edge weights. Suppose that $\abs{1 + w_e} \leq 1$ for all $e$. Then
    \begin{align*}
        \bigg| \sum_{\substack{C \subseteq H  \\ \text{conn., spann.}}}   \prod_{e \in C} w_e \bigg| \leq \sum_{\substack{T \subseteq H, \text{ tree}  \\ \text{conn., spann.}}} \prod_{e \in T} \abs{w_e} ,
    \end{align*}
    where on the right-hand side, the sum is over connected spanning trees $T$ in $H$.
\end{lemma}

\begin{proof}[Proof of Theorem \ref{thm:truncate_CE_clogn}]
    (i) To establish the absolute convergence and \eqref{eq:logZ_Kn_truncate_clogn}, it suffices to show that
    \begin{align*}
        \sum_{m > 2 \log n} \sum_{e_1,\dots,e_m} \abs{ \phi(H(e_1,\dots,e_m)) } \lambda^m  \leq \frac{1}{n},
    \end{align*}
    % where $C(b)$ is a constant depending only on $b$. 
    Fix $m$. Let $\cT_{m-1}^{\text{lab}}$ be the set of labeled trees on vertex set $[m]$ and let $\cT(H)^{\text{lab}}$ be the set of labeled spanning trees of a graph $H$. 
    As discussed in Section~\ref{sec:cluster-expansion-formal-planted matching}, the incompatibility graph of cluster $(e_1,\dots,e_m)$, denoted by $H = H(e_1,\dots,e_m)$, contains an edge $\{i,j\}$ with weight $-1$ if $e_i \not\sim e_j$, i.e., $e_i$ and $e_j$ are adjacent. 
    % In what follows, we use the abbreviated notation $H$ whenever clear from the context. 
    % We use $[m]$ for the vertex set of $H = H(e_1,\dots,e_m)$ without loss of generality. 
    By the Penrose tree-graph bound, Lemma \ref{lemma:Penrose_tree_bound}, applied to \eqref{eq:Ursell_definition}, we have
    \begin{align*}
        \sum_{e_1,\dots,e_m} \abs{ \phi(H(e_1,\dots,e_m)) } &= \frac{1}{m!} \sum_{e_1,\dots,e_m} \bigg| \sum_{ \substack{ S \subseteq H \\ \text{conn., spann.} } } \prod_{\inbraces{i,j} \in S } -\boldsymbol{1}\!\inbraces{e_i \not\sim e_j}  \bigg| \\
        &\leq \frac{1}{m!} \sum_{e_1,\dots,e_m} \sum_{t \in \cT_{m-1}^{\text{lab}}} \boldsymbol{1}\!\inbraces{ t \in \cT(H)^{\text{lab}}  } \\
        &= \frac{1}{m!}  \sum_{t \in \cT_{m-1}^{\text{lab}}} \sum_{e_1,\dots,e_m} \boldsymbol{1}\!\inbraces{ t \in \cT(H)^{\text{lab}}  }.
    \end{align*}

    Fix $t \in \cT_{m-1}^{\text{lab}}$. We next describe an iterative process to construct clusters $(e_1,\dots,e_m)$ such that the incompatibility graph $H$ contains $t$ as a spanning tree.

    \underline{Step 1:} Pick a polymer $\tilde{e}$ to assign to vertex $i_1 = 1$ of $t$. There are $\binom{n}{2}$ ways to do this.

    \underline{Step 2:} Iteratively, suppose vertices $i_1 = 1,i_2,\dots,i_j$ have been assigned to polymers $e_{i_1} = \tilde{e}, e_{i_2}, \dots, e_{i_j}$. There must exist $i_{j+1} \in [m]\setminus \inbraces{i_1,\dots,i_j}$ adjacent to one of $\inbraces{i_1,\dots,i_j}$ in $t$. Without loss of generality suppose $\inbraces{i_j, i_{j+1}} \in t$. Then there are at most $2(n-2) + 1 = 2n - 3 =: \Delta$ choices for $e_{i_{j+1}}$, corresponding to all possible adjacent edges to $e_{i_j}$, as well as itself.

    By Cayley's theorem $\abs{\cT_{m-1}^{\text{lab}}} = m^{m-2}$. Note also $m^m / m! \leq e^m$. It follows that 
    \begin{align}\label{eq:logZ_Kn_truncate_clogn_fixed_m_bound}
        \sum_{e_1,\dots,e_m} \abs{ \phi(H(e_1,\dots,e_m)) } \leq \frac{m^{m-2}}{m!} \binom{n}{2} \Delta^{m-1} \leq \frac{n}{2m^2 } (e \Delta)^m.
    \end{align}
    Multiplying by $\lambda^m$ and summing over $m \geq 2 \log n$, we obtain
    \begin{align}
        \label{eq:logZ_Kn_truncate_clogn_fixed_m_bound_sum_over_m}
        \sum_{m \geq 2 \log n} \sum_{e_1,\dots,e_m} \abs{\phi(H(e_1,\dots,e_m))} \lambda^m \leq \frac{n}{2} \sum_{m \geq 2 \log n} (e\lambda \Delta)^m 
        \leq n (e \lambda \Delta)^{2 \log n} 
        \leq \frac{1}{n},
    \end{align}
    if $e \lambda \Delta \leq \frac{1}{e}$ which holds by assumption. This establishes \eqref{eq:logZ_Kn_truncate_clogn}.

\smallskip
\noindent
(ii) 
    The argument for \eqref{eq:logZ_A_truncate_clogn} is similar to above. The difference is that the underlying graph is random. We will show that with probability at least $1 - \frac{1}{n}$, 
    \begin{align*}
        \sum_{m > 2 \log n} \sum_{e_1,\dots,e_m} \abs{ \phi(H(e_1,\dots,e_m)) } \frac{\lambda^m}{p^m} \prod_{j = 1}^{k} A_{e_j}  \leq \frac{1}{n}.
    \end{align*}
    Fix $m$. By a similar application of the Penrose tree-graph bound Lemma \ref{lemma:Penrose_tree_bound}, we obtain 
    \begin{align}
        \label{eq:logZ_A_truncate_clogn_apply_penrose}
        \sum_{e_1,\dots,e_m} \abs{ \phi(H(e_1,\dots,e_m)) } \frac{\lambda^m}{p^m} \prod_{j = 1}^{m} A_{e_j} &\leq \frac{1}{m!} \frac{\lambda^m}{p^m}  \sum_{t \in \cT_{m-1}^{\text{lab}}} \sum_{e_1,\dots,e_m \in A} \boldsymbol{1}\!\inbraces{ t \in \cT(H)^{\text{lab}}  }.
    \end{align}
    Fix $t \in \cT_{m-1}^{\text{lab}}$. We describe a similar iterative process to construct clusters $(e_1,\dots,e_m)$ where $e_i$'s are in $A$, and such that the incompatibility graph $H$ contains $t$ as a spanning tree.

    \underline{Step 1:} Pick a polymer $\tilde{e} \in A$ to assign to vertex $i_1 = 1$ of $t$. There are $|A|$ ways to do this.

    \underline{Step 2:} Iteratively, suppose vertices $i_1 = 1,i_2,\dots,i_j$ have been assigned to polymers $e_{i_1} = \tilde{e}, e_{i_2}, \dots, e_{i_j}$. There must exist $i_{j+1} \in [m]\setminus \inbraces{i_1,\dots,i_j}$ adjacent to one of $\inbraces{i_1,\dots,i_j}$ in $t$. Without loss of generality suppose $\inbraces{i_j, i_{j+1}} \in t$. Then there are at most $2(\Delta(A) - 1) + 1$ choices for $e_{i_{j+1}}$, corresponding to all possible distinct adjacent edges to $e_{i_{j}}$ in $A$, as well as $e_{i_j}$ itself, where $\Delta(A)$ denotes the max degree in $A$.

For $A \sim G(n,q)$, the Chernoff bound together with a union bound implies that $\Delta(A) \le 2nq$ and $|A| \le n^2 q$ 
% (so $ 2\Delta(A) -1 \leq 3 \Delta p$ for $n \geq 3$) 
with probability at least $1 - \frac{1}{n}$ if $q \ge \frac{9 \log n}{n}$. 
% We also have for any $\epsilon > 0$ that $|A| \leq \frac{n^2 p}{2}(1 + \epsilon)$ with probability at least $1 - O\inparen{\frac{1}{\epsilon^2 n^2}}$. Fix $\epsilon = 1$. Then with probability at least $1 - \frac{1}{n}$, 
Conditional on this event, we arrive after similar simplifications at 
    \begin{equation*}
        \sum_{e_1,\dots,e_m} \abs{ \phi(H(e_1,\dots,e_m)) } \frac{\lambda^m}{p^m} \prod_{j = 1}^{m} A_{e_j}
        \le \frac{m^{m-2}}{m!} \frac{\lambda^m}{p^m} |A| (2 \Delta(A) - 1)^{m-1} 
        % \le \frac{e^m}{4 m^2} \frac{\lambda^m}{p^m} n (4nq)^{m} 
         \le \frac{n}{4 m^2} \Big(\frac{4e \lambda nq}{p}\Big)^{m} .
        % \leq \frac{n}{3 m^2} (3e\Delta)^m.
    \end{equation*}
    Summing over $m \geq 2 \log n$ and using the condition $q \le 1.01p$, we have 
\begin{align}\label{eq:logZA_truncate_clogn_sum_over_m}
        \sum_{m > 2 \log n} \sum_{e_1,\dots,e_m} \abs{ \phi(H(e_1,\dots,e_m)) } \frac{\lambda^m}{p^m} \prod_{j = 1}^{k} A_{e_j} \leq \frac{n}{4} \sum_{m > 2 \log n} \inparen{4.04 e \lambda n}^m \leq \frac{1}{n}
    \end{align}
    if $4.04 e \lambda n \leq \frac{1}{e}$ which holds by assumption. This establishes \eqref{eq:logZ_A_truncate_clogn} and finishes the proof.
\end{proof}

\begin{remark}
    The condition 
    % $b > 1 + \log 3$ 
    $\lambda \le \frac{1}{30 n}$ assumed in Theorem \ref{thm:truncate_CE_clogn} 
    could be improved if we are willing to sum up to $m = C \log n$ for some bigger constant $C > 2$, or if we are willing to have a smaller rate of decay of the tail sum (say $n^{-\delta}$, for some $0 < \delta < 1$). However, since $\lambda = O(1/n)$ is necessary in view of the above proof, we choose not to optimize the constant.
\end{remark}

In addition, using the cluster expansion of the log-partition function together with the identities $\EE_{M \sim \mu_\lambda} \abs{M} = \lambda \inparen{  \log Z_{K_n}(\lambda)  }'$ and $\Var_{M \sim \mu_\lambda} \abs{M} = \lambda \inparen{\EE_{M \sim \mu_\lambda}\abs{M}}'$, we have the cluster expansions
\begin{align}
    \EE_{M \sim \mu_\lambda} \abs{M} &\overset{\textrm{F}}{=} \sum_{m \geq 1} \sum_{e_1,\dots,e_m } \phi(H(e_1,\dots,e_m)) m \lambda^m ,\label{eq:MD_mean_CE_formal}\\
    \Var_{M \sim \mu_\lambda} \abs{M} &\overset{\textrm{F}}{=} \sum_{m \geq 1} \sum_{e_1,\dots,e_m } \phi(H(e_1,\dots,e_m)) m^2 \lambda^m. \label{eq:MD_Var_CE_formal}
\end{align}
The $\EE_{M \sim \mu_\lambda}\abs{M}$ series satisfies similar desirable properties as those of $\log Z_{K_n}(\lambda)$.
\begin{proposition}\label{prop:truncate_mean_clogn}
    Suppose $\lambda \le \frac{1}{30 n}$. Then the cluster expansion \eqref{eq:MD_mean_CE_formal} for $\EE_{M \sim \mu_\lambda} \abs{M}$ converges absolutely. Moreover,
    \begin{equation}\label{eq:E|M|_truncate_clogn}
        \EE_{M \sim \mu_\lambda} \abs{M} = \sum_{m = 1}^{ 2 \log n} \sum_{e_1,\dots,e_m} \phi(H(e_1,\dots,e_m)) m \lambda^m + \frac{1}{n}  .
    \end{equation}
\end{proposition}

From \eqref{eq:MDMean_convergence_to_c} we deduce that $\EE \abs{M} = O(n)$. The following result indicates that the variance is on the same order, implying a concentration around the mean of the matching size for the monomer-dimer model in the $\lambda = \Theta( \frac{1}{n} )$ regime. 
\begin{proposition}\label{prop:Var|M|_Cn_bound}
    Suppose $\lambda \le \frac{1}{30 n}$. Then we have $\Var_{M \sim \mu_\lambda}(\abs{M}) = O(n) $.
\end{proposition}

\begin{proof}[Proof of Proposition \ref{prop:truncate_mean_clogn}]
    The absolute convergence and truncation for $\EE_{M \sim \mu_\lambda} \abs{M}$ follows by a straightforward modification of that of \eqref{eq:logZ_Kn_truncate_clogn}. Indeed, in \eqref{eq:logZ_Kn_truncate_clogn_fixed_m_bound} there was an extra factor of $1/m^2$. Therefore the additional $m$ factor in the cluster expansion for $\EE_{M \sim \mu_\lambda} \abs{M}$ does not present any additional difficulty.
\end{proof}

\begin{proof}[Proof of Proposition \ref{prop:Var|M|_Cn_bound}]
    The result follows by a straightforward modification of the proof of absolute convergence and truncation of \eqref{eq:logZ_Kn_truncate_clogn}. In \eqref{eq:logZ_Kn_truncate_clogn_fixed_m_bound}, the extra factor of $1/m^2$ handles the additional factor of $m^2$ appearing in \eqref{eq:MD_Var_CE_formal}. Next, in \eqref{eq:logZ_Kn_truncate_clogn_fixed_m_bound_sum_over_m}, we sum over $m \geq 1$ instead of $m \geq \log n$, leading to the desired bound.
\end{proof}

\subsection{Tree terms in the cluster expansion}

Our analysis of the cluster expansions relies on the crucial observation that the dominating terms correspond to clusters that are \emph{trees}. 
We first make this precise for the expansion \eqref{eq:E|M|_truncate_clogn} of the mean size of a matching from the monomer-dimer model---it admits the following useful approximation as a sum over trees up to size $O(\log n)$. 
% In the following, we state a refined result about the mean size of a matching, whose 
% The proof can be found in Section~\ref{sec:nonMatchingEdges_E|M|_in_trees}. 
% Recall that $T_m$ generically denotes an unlabeled simple connected tree with $m$ edges.

% Let $\mathscr{T}_m$ denote the set of unlabeled connected trees on $m$ edges. 

\begin{proposition}\label{prop:GibbsMatching_mean_var_from_CE}
    Suppose $\lambda \le \frac{1}{30 n}$. Then for each $n$,
    \begin{align*}
        \EE_{M \sim \mu_\lambda} \abs{M} &= \sum_{m = 1}^{2 \log n} \sum_{T_m}  m! \phi(H(T_m)) m \lambda^m \frac{(n)_{m+1}}{\aut (T_m)} + O(1),
        % \\
        % &\qquad + \text{(terms with cycles but no repeated edges)} + O(1).
    \end{align*}
    where the sum is over unlabeled connected trees $T_m$ with $m$ edges, $(n)_{m+1}$ denotes the falling factorial, $\aut(T_m)$ denotes the number of automorphisms of $T_m$, and $H(T_m)$ denotes the line graph of $T_m$; and for $m = 1$, $\phi(H(K_2)) = 1$, and for $m \geq 2$, $\phi(H(T_m)) = \sum_{S \subseteq H(T_m)} (-1)^{\abs{S}}/m!$, where the sum ranges over all connected and spanning subgraphs of $H(T_m)$. 
\end{proposition}

Proposition \ref{prop:MDMean_cycles_O(1)} is a consequence of the following result. 

\begin{proposition}
% \label{prop:MDMean_repEdge_O(1)}
\label{prop:MDMean_cycles_O(1)}
Suppose $\lambda \le \frac{1}{30 n}$. Then
$$
        \sum_{m \geq 2} \sum_{\substack{e_1,\dots,e_m \text{ contains} \\ 
       \text{a repeated edge or a cycle} }} m \lambda^m |\phi(H(e_1,\dots,e_m))| = O(1) .
$$
\end{proposition}

\begin{proof}[Proof of Proposition \ref{prop:GibbsMatching_mean_var_from_CE}]
By \eqref{eq:E|M|_truncate_clogn} and Proposition~\ref{prop:MDMean_cycles_O(1)}, we have 
$$
\EE_{M \sim \mu_\lambda} \abs{M} = \sum_{m = 1}^{ 2 \log n} \sum_{\substack{e_1,\dots,e_m \\ \text{simple tree}}} \phi(H(e_1,\dots,e_m)) m \lambda^m + O(1)  
$$
which is equivalent to the statement in Proposition \ref{prop:GibbsMatching_mean_var_from_CE}, because there are $\frac{(n)_{m+1}}{\aut (T_m)} m!$ ways to assign the edges of an unlabeled tree $T_m$ to $e_1, \dots, e_m$. 
\end{proof}

\begin{proof}[Proof of Proposition \ref{prop:MDMean_cycles_O(1)}]
    Fix integers $m \geq 0$ and $r \geq 2$. 
    Consider clusters $\inbraces{e_1,\dots, e_{m + r}}$ that contain a cycle $C_r$ or (if $r=2$) a repeated edge which we denote by $C_2$. 
    % let $G(e_1,\dots,e_{m+r})$ denote the subgraph induced by the distinct edges (polymers) in $\inbraces{e_1,\dots, e_{m + r}}$. Let $H(e_1,\dots, e_{m + r})$ denote its incompatibility graph. Write $G$ for  $G(e_1,\dots, e_{m + r})$ and $H$ for $H(e_1,\dots, e_{m + r})$ when clear from the context. 
    Let us use $[m+r]$ for the vertex set of $H = H(e_1,\dots,e_{m+r})$. 
    Let $\cT_{m+r}^{\text{lab}}$ denote the set of all labeled trees with vertex set $[m+r]$. Let $\cT(H)^{\text{lab}}$ denote the set of all labeled spanning trees of a graph $H$.

    Similar to the proof of Theorem~\ref{thm:truncate_CE_clogn}, the Penrose tree-graph bound Lemma \ref{lemma:Penrose_tree_bound} implies that
    \begin{align*}
        \sum_{ \substack{e_1, \dots, e_{m + r} \\ \text{contains } C_r } } \abs{\phi(H(e_1,\dots,e_{m+r}))}
        &= \frac{1}{(m+r)!} \sum_{ \substack{e_1, \dots, e_{m + r} \\ \text{contains } C_r } } \bigg| \sum_{ \substack{ S \subseteq H \\ \text{conn., spann.} } } \prod_{\inbraces{i,j} \in S } -\boldsymbol{1}\!\inbraces{e_i \not\sim e_j}  \bigg| \\
        &\leq \frac{1}{(m+r)!} \sum_{ \substack{e_1, \dots, e_{m + r} \\ \text{contains } C_r } } \sum_{t \in \cT_{m+r}^{\text{lab}}} \boldsymbol{1}\!\inbraces{ t \in \cT(H)^{\text{lab}} } \\
        &= \frac{1}{(m+r)!} \sum_{t \in \cT_{m+r}^{\text{lab}}} \sum_{ \substack{e_1, \dots, e_{m + r} \\ \text{contains } C_r } }  \boldsymbol{1}\!\inbraces{ t \in \cT(H)^{\text{lab}} }.
    \end{align*}
    Fix $t \in \cT_{m+r}^{\text{lab}}$. We describe an iterative process to construct clusters $\inbraces{e_1,\dots,e_{m+r}}$ such that $t \in \cT(H)^{\text{lab}}$ and the cluster contains at least one $C_r$. 
    % This is always possible if we allow repeated polymers.
    
    \underline{Step 1:} Fix $V' \subseteq V(t) = [m+r]$ with $\abs{V'} = r$. The set $V'$ will be the index set such that $\{e_i : i \in V'\}$ forms $C_r$. 
    % coordinates in the cluster $\inbraces{e_1, \dots, e_{m+r}}$ which contain a single $r$-cycle. 
    There are $\binom{m + r}{r}$ ways to choose $V'$. 
    
    \underline{Step 2:} Choose $r$ distinct polymers to make up a single $C_r$: there are $\binom{n}{r} \frac{r!}{2r}$ ways to do this if $r \ge 3$ and $\binom{n}{2}$ ways if $r = 2$. 
    
    \underline{Step 3:} Pick a polymer $\tilde{e}$ from the above chosen $r$ polymers to assign to an arbitrary vertex $i_1 \in V'$. (We may take $i_1$ to be the smallest index in $V'$.) There are $r$ choices for $\tilde{e}$. 
    % out of the chosen $r$-cycle polymers. 

    \underline{Step 4:} Iteratively, suppose $i_1, \dots, i_j$ have been assigned to polymers $e_{i_1} = \tilde{e}, e_{i_2},\dots, e_{i_j}$. There must exist $i_{j+1} \in [m+r] \setminus \inbraces{i_1,\dots,i_{j}}$ such that $i_{j+1}$ is adjacent to one of $\inbraces{i_1,\dots,i_j}$ in $t$. Without loss of generality suppose $\inbraces{i_j, i_{j+1}} \in t$. Now
    \begin{itemize}
        \item if $i_{j+1} \in V'$, then we attempt to assign a polymer in the chosen $C_r$ to $i_{j+1}$. There are at most two choices for $e_{i_{j+1}}$, which has to be compatible with the assignment of $e_{i_j}$ to $i_j$. 
        If there are no compatible choices for $e_{i_{j+1}}$, we terminate the iteration and output an incomplete assignment (which does not contribute to the sum).
        
        \item if $i_{j+1} \notin V'$, then there are at most $2(n-2) + 1 = 2n - 3 := \Delta$ choices for $e_{i_{j+1}}$, corresponding to all possible distinct incident edges to $e_{i_j}$, as well as itself. 
    \end{itemize}
    For a chosen $C_r$, the subset of completed cluster assignments that had utilized all chosen polymers in $C_r$ contain all the desired ordered clusters $\inbraces{e_1,\dots,e_{m+r}}$ satisfying $t \in \cT(H)^{\text{lab}}$ and $e_1,\dots,e_{m+r}$ containing that chosen $C_r$.

    % If we instead chose V' and then immediately assigned the r cycle vertices, thus incurring a r! factor, it would fail. I think this is because it massively overcounts the number of possible assignments. Especially if V' vertices are adjacent in T, the cycle assignments are very restricted. This sequential approach works better. 

    In this way, we have 
    \begin{align*}
        \sum_{ \substack{e_1, \dots, e_{m + r} \\ \text{contains } C_r } }  \boldsymbol{1}\!\inbraces{ t \in \cT(H)^{\text{lab}} } &\leq \binom{m+r}{r} \binom{n}{r} \frac{r!}{r} r \Delta^m 2^{r-1}  \leq \frac{1}{2}\binom{m+r}{r} n^r \Delta^m 2^r
    \end{align*}
    By Cayley's theorem $\abs{\cT_{m+r}^{\text{lab}}} = (m+r)^{m+r - 2}$. It follows that
    \begin{align}
        \sum_{ \substack{e_1,\dots, e_{m + r} \\ \text{contains } C_r } } \abs{\phi(H(e_1,\dots,e_{m+r})) } \lambda^{m+r}(m+r) 
        &\leq \frac{1}{2(m+r)} \binom{m+r}{r} \frac{(m+r)^{m+r}}{(m+r)!} n^r \Delta^m \lambda^{m+r} 2^r \nonumber \\
        &\leq \frac{1}{2(m+r)} \binom{m+r}{r} (e \lambda \Delta)^{m+r} 2^r.
        \label{eq:MDMean_cycles_O(1)_afterCayley}
    \end{align}
    Summing over $m$ and $r$ gives, since $\lambda \le \frac{1}{30 n}$,
    \begin{align*}
        &\sum_{m \geq 0, r \geq 2} \sum_{ \substack{e_1,\dots, e_{m + r} \\ \text{contains } C_r } } \abs{\phi(H(e_1,\dots,e_{m+r})) } \lambda^{m+r}(m+r) \leq \sum_{m \geq 0, r \geq 2} \binom{m+r}{r} (e \lambda \Delta)^{m+r} 2^r \\
        &\quad= \sum_{\ell \geq 2} (e\lambda \Delta)^\ell \sum_{r = 2}^{\ell} \binom{\ell}{r} 2^r 
        \leq \sum_{\ell \geq 2} (3e\lambda \Delta)^\ell 
        = O(1) \qedhere
    \end{align*}
    % where we used the hypothesis $\abs{3e\lambda \Delta} < 1$, and where $C > 0$ is some constant depending only on $b$.
\end{proof}

\subsection{Combinatorial identities for the Ursell function}
\label{sec:Ursell_CombinatorialIdentities}

Before proceeding to analyze the mean part of the log-likelihood ratio, we prove some combinatorial identities about the Ursell functions. 
In what follows, $G(\text{spann., conn.})$ denotes the set of spanning connected subgraphs of a connected graph $G$. 

% \begin{tikzpicture}[scale=1.1]

% % styles (small, filled, no borders)
% \tikzstyle{redv}=[circle, fill=red, inner sep=2pt]
% \tikzstyle{bluev}=[circle, fill=blue, inner sep=2pt]

% % red vertices
% \node[redv] (r1) at (0,1) {};
% \node[redv] (r2) at (1,2) {};
% \node[redv] (r3) at (1,0) {};
% \node[redv] (vs) at (2,1) {}; % v_*

% % blue vertices
% \node[bluev] (vss) at (3,1) {}; % v_{**}
% \node[bluev] (b1) at (4,2) {};
% \node[bluev] (b2) at (4,1) {};
% \node[bluev] (b3) at (4,0) {};

% % edges inside the red induced subgraph H[V_red]
% \draw (r1) -- (r2) -- (vs) -- (r3) -- (r1);

% % edges inside the blue induced subgraph H[V_blue]
% \draw (vss) -- (b1) -- (b2) -- (vss);
% \draw (vss) -- (b3);

% % cross-color edges
% \draw (vs) -- (vss);   % edge v_* -- v_{**}
% \draw (r3) -- (vss);   % extra connection so v_* v_{**} is not a cut edge

% % labels for v_* and v_{**} (optional)
% \node at (2,1.4) {$v_*$};
% \node at (3,1.4) {$v_{**}$};

% \end{tikzpicture}

\begin{lemma}
    \label{lemma:Ursell_trees_adjVertices_identity}
    Let $(V(H), H)$ be a connected graph. Let $v_*$ and $v_{**}$ be two adjacent vertices in $H$. Define the following subset of bi-colorings of $V(H)$:
    % Let $H = H(T)$ be the line graph of (possibly multi) tree $T$. Let $v_*$ and $v_{**}$ be two adjacent vertices in $H$ such that $v_*$ and $v_{**}$ are either both corresponding to the same repeated edge in $T$, or both do not correspond to any repeated edge. Define the following subset of bi-colorings of $V(H)$:
    \begin{align}
        \cC(H; v_{*}, v_{**}) := \inbraces{  (V_{\text{red}}, V_{\text{blue}}) : \, 
    \begin{aligned}
        & V_{\text{red}} \cup V_{\text{blue}} = V(H) \text{ disjoint}, \, V_{\text{red}} \ni v_*, \, V_{\text{blue}} \ni v_{**},\\
        & H[V_{\text{red}}] \text{ and } H[V_{\text{blue}}] \text{ are each connected subgraphs } 
    \end{aligned}
    }. \label{eq:def-bi-coloring}
    \end{align}
    (See Figure \ref{fig:nonMatching_TredTblueJoin_and_IncompatibilityGraph} (Right) for an example of such a bi-coloring in $\cC(H; v_{*}, v_{**})$.) 
    Then 
    \begin{align}\label{eq:Ursell_trees_adjVertices_identity}
        \sum_{ \substack{ S \subseteq H(\text{spann., conn.})  } } (-1)^{\abs{S}} = \sum_{ \substack{ (V_{\text{red}}, V_{\text{blue}}) \\ \in \cC(H; v_{*}, v_{**}) } } \sum_{\substack{  S_{\text{red}} \subseteq H[V_{\text{red}}](\text{spann., conn.})  \\ S_{\text{blue}} \subseteq H[V_{\text{blue}}](\text{spann., conn.})   } } (-1)^{\abs{S_{\text{red}}} + \abs{S_{\text{blue}}} + 1 }.
    \end{align}
\end{lemma}

\begin{proof}[Proof of Lemma \ref{lemma:Ursell_trees_adjVertices_identity}]
    Denote $e_* := \inbraces{v_*, v_{**}}$. We claim that 
    \begin{align}\label{eq:Ursell_trees_adjVertices_identity_eStarCutEdge}
        \sum_{ \substack{ S \subseteq H(\text{spann., conn.})  } } (-1)^{\abs{S}} = \sum_{ \substack{ S \subseteq H(\text{spann., conn.}) \\ S \ni e_*, \text{ $e_*$ is cut-edge} } } (-1)^{\abs{S}}.
    \end{align}
    To see this, partition $H(\text{spann., conn.})$ into two sets: $H(e_*)$ and $H(\text{no } e_*)$ which consists of spanning and connected subgraphs that respectively contain and do not contain $e_*$. Further partition $H(e_*)$ into two sets $H(e_*, \text{cut})$ and $H(e_*, \text{not cut})$ which contain the subgraphs where $e_*$ is respectively a cut-edge and not a cut-edge. 
    Any $S \in H(e_*, \text{not cut})$ can be uniquely paired with an $S \setminus \{e_*\}$ that lives in $H(\text{no } e_*)$. In other words, there is a bijection between $H(e_*, \text{not cut})$ and $H(\text{no } e_*)$ obtained by including and not including $e_*$. The summands corresponding to these pairs in the LHS of \eqref{eq:Ursell_trees_adjVertices_identity_eStarCutEdge} cancel since they differ by exactly one edge. Therefore, it remains only to sum over $H(e_*, \text{cut})$. This establishes \eqref{eq:Ursell_trees_adjVertices_identity_eStarCutEdge}. 

    The set $H(e_*, \text{cut})$ can be generated by the following procedure
    \begin{enumerate}
        \item 
        Color the vertices of $H$ red and blue and call the resulting colored vertex sets $V_{\text{red}}$ and $V_{\text{blue}}$ respectively, such that $V_{\text{red}} \ni v_*$, $V_{\text{blue}} \ni v_{**}$, and the induced subgraphs $H[V_{\text{red}}]$ and $H[V_{\text{blue}}]$ are each connected. 
        \item Join any spanning and connected $ S_{\text{red}} \subseteq H[V_{\text{red}}]$ with a spanning and connected $S_{\text{blue}} \subseteq H[V_{\text{blue}}]$ with the edge $\inbraces{v_*, v_{**}}$ to form a spanning and connected subgraph of $H$. 
        \item The uncolored collection of all such joinings over all choices of $(V_{\text{red}}, V_{\text{blue}}) $ forms the desired set $H(e_*, \text{cut})$.
    \end{enumerate}
    The size of any such subgraph generated by the above procedure is $\abs{S_{\text{red}}} + \abs{S_{\text{blue}}} + 1$. This proves the equality in \eqref{eq:Ursell_trees_adjVertices_identity}.
\end{proof}

\begin{figure}[ht]\centering
%Tred
\begin{tikzpicture}
%\draw[step=1, lightgray] (0,0) grid (6,6);
\node (1) at (1,2) {$1$};
\node (8) at (1,4) {$8$};
\node (3) at (2,3) {$3$};
\node (7) at (4,3) {$7$};
\node (4) at (5,3) {$4$};

\draw[thick, red] (3) -- node[label={[label distance=-0.5em]90:$v_{*}$}] {} (7);
\draw[thick, red] (4) -- (7);
\draw[thick, red] (1) -- (3);
\draw[thick, red] (3) -- (8);
\end{tikzpicture}
\qquad\qquad % <----------------- SPACE BETWEEN PICTURES
%Tblue
\begin{tikzpicture}
%\draw[step=1, lightgray] (0,0) grid (8,8);
\node (9) at (1,2) {$9$};
\node (10) at (2,3) {$10$}; 
\node (5) at (1,4) {$5$}; 
\node (6) at (0,3) {$6$};
\node (7) at (3,4) {$7$};
\node (3) at (4.8,3) {$3$};
\node (11) at (6,4) {$11$};
\node (2) at (7,3) {$2$}; 

\draw[thick, blue] (9) -- (10);
\draw[thick, blue] (5) -- (10);
\draw[thick, blue] (5) -- (6);
\draw[thick, blue] (7) -- (10);
\draw[thick, blue] (3) -- (11);
\draw[thick, blue] (2) -- (11);

\draw[thick, blue] (3) -- node[label={[label distance=-0.5em]267:$v_{**}$}] {} (7);
\end{tikzpicture}
\caption{Example of tuple \eqref{eq:nonMatching_GenericTuple}. Left: $\widetilde{T}_{\text{red}}(v_*)$ with $\abs{\widetilde{T}_{\text{red}}(v_*)} = \ell = 4$. Right: $\widetilde{T}_{\text{blue}}(v_{**})$ with $\abs{\widetilde{T}_{\text{blue}}(v_{**})} = m + 1 -\ell = 7$. Their join by superimposing $v_*$ and $v_{**}$ gives the one-repeated-edge tree with $m + 1 = 11$ edges and $11$ vertices as in Figure \ref{fig:nonMatching_TredTblueJoin_and_IncompatibilityGraph} (Left).}
\label{fig:nonMatching_Tred_Tblue_separate}
\end{figure}
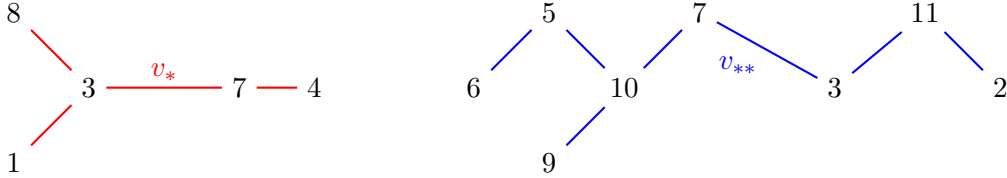

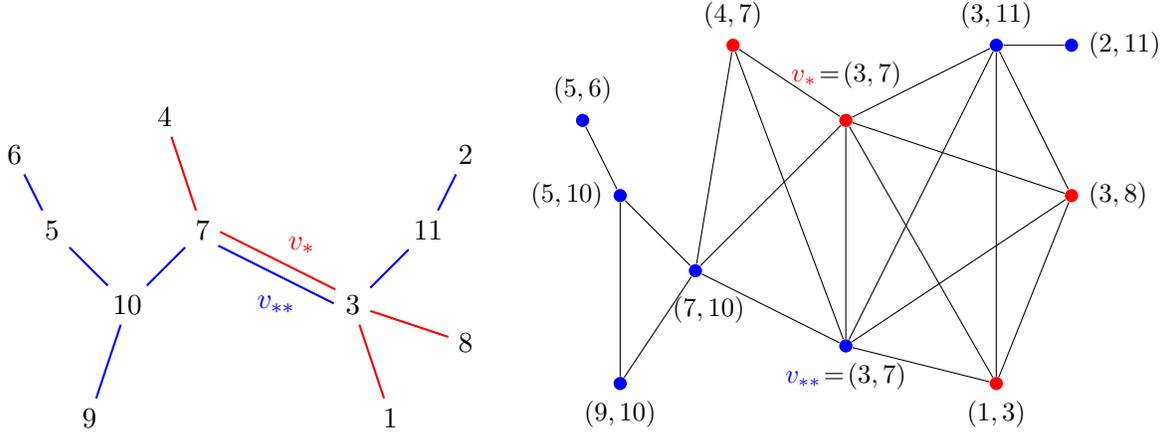
\begin{figure}[ht]\centering
\begin{tikzpicture}
%\draw[step=1, lightgray] (0,0) grid (8,8);
\node (9) at (1.5,1.5) {$9$};
\node (10) at (2,3) {$10$}; 
\node (5) at (1,4) {$5$}; 
\node (6) at (0.5,5) {$6$};
\node (7) at (3,4) {$7$};
\node (4) at (2.5,5.5) {$4$};
\node (3) at (5,3) {$3$};
\node (1) at (5.5,1.5) {$1$};
\node (8) at (6.5,2.5) {$8$};
\node (11) at (6,4) {$11$};
\node (2) at (6.5,5) {$2$}; 

\draw[thick, blue] (9) -- (10);
\draw[thick, blue] (5) -- (10);
\draw[thick, blue] (5) -- (6);
\draw[thick, blue] (7) -- (10);
\draw[thick, blue] (3) -- (11);
\draw[thick, blue] (2) -- (11);

\draw[thick, transform canvas={yshift=0.24em}, red] (3) -- node[label={[label distance=-0.5em]87:$v_{*}$}] {} (7);

\draw[thick, transform canvas={xshift=-0.06em, yshift=-0.24em}, blue] (3) -- node[label=below:$v_{**}$] {} (7);

\draw[thick, red] (4) -- (7);
\draw[thick, red] (1) -- (3);
\draw[thick, red] (3) -- (8);
\end{tikzpicture}
\quad 
%%%%%%%% H(Tred,Tblue) %%%%%%%%
\begin{tikzpicture}[every node/.style={font=\small}, redCirc/.style={circle,fill=red, minimum size=5pt, inner sep=0pt}, blueCirc/.style={circle,fill=blue, minimum size=5pt, inner sep=0pt}]

%\draw[step=1, lightgray] (0,0) grid (8,8);

% left side nodes
\node[blueCirc,label=below:{$(9,10)$}] (9-10) at (1,1.5) {};
\node[blueCirc, label={[xshift=5pt,yshift=-3pt]below:$(7,10)$}] (7-10) at (2,3) {};
\node[blueCirc, label=left:{$(5,10)$}] (5-10) at (1,4) {};
\node[blueCirc, label=above:{$(5,6)$}] (5-6) at (0.5,5) {};
\node[redCirc, label=above:{$(4,7)$}] (4-7) at (2.5,6) {};

% v*, v** 
\node[redCirc,label={[label distance=0.51em]88:$\textcolor{red}{v_{*}}\!=\!(3,7)$}] (vStar) at (4,5) {};
\node[blueCirc,label=below:{$\textcolor{blue}{v_{**}}\! = \!(3,7)$}] (vStarStar) at (4,2) {};

% right side nodes
\node[blueCirc,label=above:{$(3,11)$}] (3-11) at (6,6) {};
\node[blueCirc,label=right:{$(2,11)$}] (2-11) at (7,6) {};
\node[redCirc,label=right:{$(3,8)$}] (3-8) at (7,4) {};
\node[redCirc,label=below:{$(1,3)$}] (1-3) at (6,1.5) {};

% edges
\draw (9-10) -- (7-10);
\draw (5-10) -- (7-10);
\draw (5-6) -- (5-10);
\draw (5-10) -- (9-10);
\draw (4-7) -- (7-10);

\draw (4-7) -- (vStar);
\draw (4-7) -- (vStarStar);
\draw (7-10) -- (vStar);
\draw (7-10) -- (vStarStar);

\draw (vStar) -- (vStarStar);

\draw (3-11) -- (vStar);
\draw (3-11) -- (vStarStar);
\draw (3-8) -- (vStar);
\draw (3-8) -- (vStarStar);
\draw (1-3) -- (vStar);
\draw (1-3) -- (vStarStar);

\draw (1-3) -- (3-11);
\draw (1-3) -- (3-8);
\draw (3-8) -- (3-11);

\draw (2-11) -- (3-11);
\end{tikzpicture}
\caption{(Left) The repeated edge tree represented by $\inparen{\widetilde{T}_{\text{red}}(v_*), \widetilde{T}_{\text{blue}}(v_{**})}$ from Figure \ref{fig:nonMatching_Tred_Tblue_separate}. (Right) The incompatibility graph $H$ with the corresponding coloring.}
\label{fig:nonMatching_TredTblueJoin_and_IncompatibilityGraph}
\end{figure}

\begin{lemma} \label{lem:ursell-function-tmrep-tm-identity}
Let $T_m$ denote a generic unlabeled simple tree on $m+1$ vertices, and let $\TmRep$ denote a generic unlabeled tree with one repeated edge on $m+1$ vertices. 
Then we have
\begin{align}\label{eq:convolutionOfUrsell_overTrees}
    &\sum_{\TmRep} \frac{\widetilde{\phi}(H(\TmRep)) }{2 \aut(\TmRep)} = - \sum_{\ell = 1}^{m} \sum_{ \inparen{T_\ell, T_{m+1-\ell}}  }\frac{\ell   \widetilde{\phi}(H(T_\ell))  }{\aut(T_\ell) }  \frac{(m+1 - \ell) \widetilde{\phi}(H(T_{m + 1 - \ell}))}{\aut(T_{m + 1 - \ell})} ,
\end{align}
where $\widetilde{\phi}(H) := m! \cdot \phi(H)$ denotes the unnormalized Ursell function, and $\aut(\cdot)$ denotes the number of authomorphisms. 
\end{lemma}

\begin{proof}
% for every $1 \leq m \leq 2\log n - 1$, that
We will rewrite the LHS and RHS of \eqref{eq:convolutionOfUrsell_overTrees} over ``labeled'' and ``colored and labeled'' trees respectively, and show that \eqref{eq:convolutionOfUrsell_overTrees} is equivalent to (with notation to be explained below)
\begin{align}
    \sum_{\widetilde{\TmRep} } \widetilde{\phi}\!\inparen{H\!\inparen{ \widetilde{\TmRep}  }} = - \sum_{ (\widetilde{T}_{\text{red}}(v_*), \widetilde{T}_{\text{blue}}(v_{**}))  } \widetilde{\phi}(H(  \widetilde{T}_{\text{red}}(v_{*})) ) \widetilde{\phi}(H( \widetilde{T}_{\text{blue}}(v_{**}) )).
    \label{eq:rewrite-as-labeled-colored}
\end{align}

More precisely, on the LHS of \eqref{eq:convolutionOfUrsell_overTrees} we rewrite the sum over $\widetilde{\TmRep}$'s which are vertex-labeled trees with labels in $[m+1]$ and with $m+1$ edges. The number of such $\widetilde{\TmRep}$'s that can be generated from a single unlabeled $\TmRep$ is
% \begin{align}\label{eq:nonMatching_meanPart_LHSCombFactor}
    $\frac{(m+1)!}{\aut(\TmRep)}$.
% \end{align}
Therefore, the LHS of \eqref{eq:rewrite-as-labeled-colored} is $2 \cdot (m+1)!$ times the LHS of \eqref{eq:convolutionOfUrsell_overTrees}. 

On the other hand, rewrite the RHS of \eqref{eq:convolutionOfUrsell_overTrees} as a sum over tuples generically denoted by
\begin{align}
    \inparen{ \widetilde{T}_{\text{red}}(v_*), \,  \widetilde{T}_{\text{blue}}(v_{**}) }
    \label{eq:nonMatching_GenericTuple}
\end{align}
satisfying the following: 
\begin{itemize}
    \item $\widetilde{T}_{\text{red}}(v_*)$ and $\widetilde{T}_{\text{blue}}(v_{**})$ are vertex-labeled simple trees.
    
    \item $\widetilde{T}_{\text{red}}(v_*)$ and $\widetilde{T}_{\text{blue}}(v_{**})$ each have a distinguished edge\footnote{Note that $v_*$ is an edge of $\widetilde{T}_{\text{red}}(v_*)$ but corresponds to a vertex of $H(\widetilde{T}_{\text{red}}(v_*))$, and hence the notation.} $v_{*}$ and $v_{**}$ respectively. The label set of the two vertices incident to $v_{*}$ must coincide with that for $v_{**}$.
The label set of all the other vertices of $\widetilde{T}_{\text{red}}(v_*)$ has an empty intersection with the label set of all the other vertices of $\widetilde{T}_{\text{blue}}(v_{**})$. The vertices are labeled using $[m+1]$.
    
    \item Joining $\widetilde{T}_{\text{red}}(v_*)$ and $\widetilde{T}_{\text{blue}}(v_{**})$ by superimposing the vertices with the same labels (so that $v_*$ and $v_{**}$ form the double edge) gives a multi-tree with $m+1$ edges and $m+1$ vertices.
\end{itemize}
We refer to Figure \ref{fig:nonMatching_Tred_Tblue_separate} for an example of such a tuple \eqref{eq:nonMatching_GenericTuple}, and to Figure \ref{fig:nonMatching_TredTblueJoin_and_IncompatibilityGraph} (Left) for the corresponding joined tree.

The number of tuples \eqref{eq:nonMatching_GenericTuple} that can be generated from a single unlabeled pair $(T_\ell, T_{m+1-\ell})$ for a fixed $1 \leq \ell \leq m$ is
% \begin{align}\label{eq:nonMatching_meanPart_RHSCombFactor}
$$
    \binom{m+1}{\ell + 1} \frac{(\ell + 1)!}{\aut(T_\ell)} \ell (m+1- \ell) \cdot 2 \cdot \frac{(m-\ell)!}{\aut(T_{m+1-\ell})} 
    = \frac{2 \cdot (m+1)! \ell (m+1-\ell)}{\aut(T_\ell) \aut(T_{m+1-\ell})}.
    $$
% \end{align}
The Ursell functions are independent of the coloring or labeling of the graphs they are applied to. 
% Rewriting \eqref{eq:convolutionOfUrsell_overTrees} as described and scaling the LHS and RHS by the combinatorial factors \eqref{eq:nonMatching_meanPart_LHSCombFactor} and \eqref{eq:nonMatching_meanPart_RHSCombFactor} respectively, we obtain
Therefore, we see that the RHS of \eqref{eq:rewrite-as-labeled-colored} is $2 \cdot (m+1)!$ times the RHS of \eqref{eq:convolutionOfUrsell_overTrees}. 

Consequently, it suffices to show for a fixed $\widetilde{\TmRep}$ that
\begin{align}\label{eq:convolutionOfUrsell_fixed_T_m^rep}
    \widetilde{\phi}\!\inparen{ H\!\inparen{ \widetilde{\TmRep}  }} = - \sum_{ ( \widetilde{T}_{\text{red}}(v_*), \widetilde{T}_{\text{blue}}(v_{**}) ) \cong \widetilde{\TmRep}  } \widetilde{\phi}(H( \widetilde{T}_{\text{red}}(v_{*}))) \widetilde{\phi}(H( \widetilde{T}_{\text{blue}}(v_{**}))),
\end{align}
where we write $ ( \widetilde{T}_{\text{red}}(v_*), \widetilde{T}_{\text{blue}}(v_{**}) ) \cong \widetilde{\TmRep}$ to mean that an uncolored version of the join of $( \widetilde{T}_{\text{red}}(v_*), \widetilde{T}_{\text{blue}}(v_{**}) ) $ is isomorphic to $\widetilde{\TmRep}$. In what follows, we fix $H := H\!\inparen{ \widetilde{\TmRep}  }$ the incompatibility graph of $\widetilde{\TmRep}$. 
For convenience, we also denote the two vertices in $H$ corresponding to the repeated edges by $v_{*}$ and $v_{**}$. 
Define the subset $\cC(H; v_*, v_{**})$ of bi-colorings of $V(H)$ as in \eqref{eq:def-bi-coloring}. 
% \begin{align}
%     \cC(H) := \inbraces{  (V_{\text{red}}, V_{\text{blue}}) : \, 
%     \begin{aligned}
%         & V_{\text{red}} \cup V_{\text{blue}} = V(H) \text{ disjoint}, \, V_{\text{red}} \ni v_*, \, V_{\text{blue}} \ni v_{**},\\
%         & H[V_{\text{red}}] \text{ and } H[V_{\text{blue}}] \text{ are each connected subgraphs } 
%     \end{aligned}
%     }.
% \end{align}
There is a bijection between $\cC(H; v_*, v_{**})$ and the set $\inbraces{  ( \widetilde{T}_{\text{red}}(v_*), \widetilde{T}_{\text{blue}}(v_{**}) ) \cong \widetilde{\TmRep}   }$. We refer to Figure \ref{fig:nonMatching_TredTblueJoin_and_IncompatibilityGraph} (Right) for an example of such a bi-coloring of $V(H)$ that corresponds to a splitting of $\widetilde{T}_{m}^{\text{rep}}$ along the repeated edges into a red and a blue tree.
% Recalling that $H(\text{spann., conn.})$ denotes the set of spanning and connected subgraphs of $H$ and 
Recalling the Ursell function defined by \eqref{eq:Ursell_definition}, 
% (Definition \ref{def:Ursell_definition}), 
we see that \eqref{eq:convolutionOfUrsell_fixed_T_m^rep} 
% \begin{align}\label{eq:convolutionOfUrsell_fixed_T_m^rep_lineGraphForm}
%     \sum_{ \substack{ S \subseteq H(\text{spann., conn.})  } } (-1)^{\abs{S}} = \sum_{ (V_{\text{red}}, V_{\text{blue}}) \in \cC(H) } \sum_{\substack{  S_{\text{red}} \subseteq H[V_{\text{red}}](\text{spann., conn.})  \\ S_{\text{blue}} \subseteq H[V_{\text{blue}}](\text{spann., conn.})   } } (-1)^{\abs{S_{\text{red}}} + \abs{S_{\text{blue}}} + 1 }.
% \end{align}
is exactly the equality shown in Lemma \ref{lemma:Ursell_trees_adjVertices_identity}.
\end{proof}

%%%%%%% COMBINING
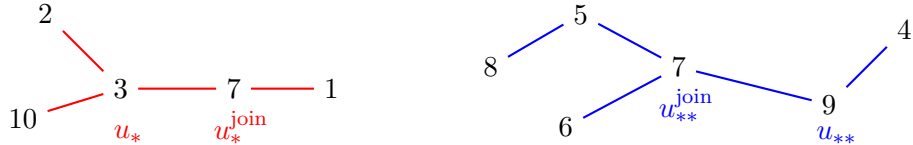
\begin{figure}[ht]\centering
%Tred(u_*, u_*^join)
\begin{tikzpicture}
%\draw[step=1, lightgray] (0,0) grid (8,8);
\node (10) at (0.7,2.6) {$10$};
\node (2) at (1,4) {$2$};
\node (3) at (2,3) {$3$};
\node [below=of 3, xshift=0.1cm, yshift=0.9cm] {$\textcolor{red}{u_*}$};
\node (7) at (3.5,3) {$7$};
\node [below=of 7, xshift=0.1cm, yshift=1.1cm] {$\textcolor{red}{u_{*}^{\text{join}}}$};
\node (1) at (4.8,3) {$1$};

\draw[thick, red] (3) -- (7);
\draw[thick, red] (1) -- (7);
\draw[thick, red] (2) -- (3);
\draw[thick, red] (3) -- (10);
\end{tikzpicture}
\qquad\qquad % <----------------- SPACE BETWEEN PICTURES
%Tblue(u_**, u_**^join)
\begin{tikzpicture}
%\draw[step=1, lightgray] (0,0) grid (8,8);
\node (6) at (1.5,3.2) {$6$}; 
\node (5) at (1.7,4.7) {$5$}; 
\node (8) at (0.5,4) {$8$};
\node (7) at (3,4) {$7$};
\node [below=of 7, xshift=0.1cm, yshift=1.1cm] {$\textcolor{blue}{u_{**}^{\text{join}}}$};
\node (9) at (5,3.5) {$9$};
\node [below=of 9, xshift=0.1cm, yshift=1.1cm] {$\textcolor{blue}{u_{**}}$};
\node (4) at (6,4.5) {$4$};

\draw[thick, blue] (6) -- (7);
\draw[thick, blue] (5) -- (7);
\draw[thick, blue] (5) -- (8);
\draw[thick, blue] (7) -- (9);
\draw[thick, blue] (4) -- (9);
\end{tikzpicture}
\caption{Example of tuple \eqref{eq:eqAverageED_FlucPart_P2Decor_TredTblueTuple}. Left: $\widetilde{T}_{\text{red}}(u_*, u_*^{\text{join}})$ with size $\ell = 4$. Right: $\widetilde{T}_{\text{blue}}(u_{**}, u_{**}^{\text{join}})$ with size $m + 1 -\ell = 5$. Their join by superimposing on the `join' vertices gives the $P_2$ decorated tree with $m + 1 = 9$ edges and $10$ vertices as in Figure \ref{fig:equalAverageED_TredTblueJoin_and_IncompatibilityGraph} (Left).}
\label{fig:equalAverageED_flucPart_TredTblueSeparate}
\end{figure}

%%%%%%%% COMBINING P2 Decor Pictures %%%%%%%%
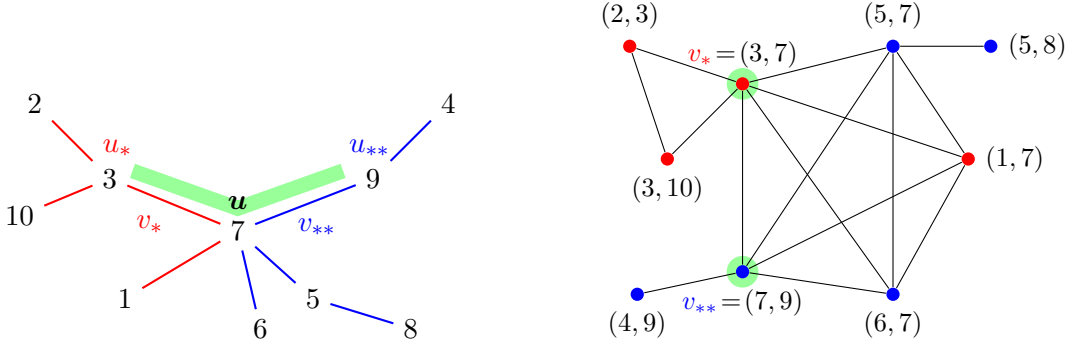
\begin{figure}[ht]\centering
%%%%%%%% (Tred,Tblue) with P2 decor %%%%%%%%
\begin{tikzpicture}
%\draw[step=1, lightgray] (0,0) grid (8,8);

% Draw the P2 decoration first so that other things are above it
\draw[green, opacity=0.4, line width=0.2cm] (2.3,4.1) -- (3.7,3.6) -- (5.1,4.1);

\node (10) at (0.8,3.5) {$10$};
\node (2) at (1,5) {$2$};
\node (3) at (2,4) {$3$};
\node [above=of 3, xshift=0.1cm, yshift=-1.1cm] {$\textcolor{red}{u_*}$};
\node (7) at (3.7,3.3) {$7$};
\node [above=of 7, xshift=0cm, yshift=-1.1cm] {$\bm{u}$};
\node (9) at (5.5,4) {$9$};
\node [above=of 9, xshift=-0.05cm, yshift=-1.1cm] {$\textcolor{blue}{u_{**}}$};
\node (4) at (6.5,5) {$4$};
\node (1) at (2.2,2.4) {$1$};
\node (5) at (4.7,2.4) {$5$};
\node (8) at (6,2) {$8$};
\node (6) at (4,2) {$6$};

\draw[thick, red] (2) -- (3);
\draw[thick, red] (3) -- (10);
\draw[thick, red] (1) -- (7);

\draw[thick, red] (3) -- node[label={[label distance=-0.35em]267:$v_{*}$}] {} (7);
\draw[thick, blue] (7) -- node[label=below:$v_{**}$, xshift=0.15cm, yshift=0.1cm] {} (9);

\draw[thick, blue] (4) -- (9);
\draw[thick, blue] (5) -- (7);
\draw[thick, blue] (5) -- (8);
\draw[thick, blue] (6) -- (7);

\end{tikzpicture}
\qquad\qquad % <----------------- SPACE BETWEEN PICTURES
%%%%%%%% H(Tred,Tblue) P2 decor %%%%%%%%
\begin{tikzpicture}[every node/.style={font=\small}, redCirc/.style={circle,fill=red, minimum size=5pt, inner sep=0pt}, blueCirc/.style={circle,fill=blue, minimum size=5pt, inner sep=0pt}]

%\draw[step=1, lightgray] (0,0) grid (8,8);

% Draw the P2 decoration first so that other things are above it
\node[circle, fill=green, opacity=0.4, minimum size=12pt,inner sep=0pt] at (4,5) {};
\node[circle, fill=green, opacity=0.4, minimum size=12pt,inner sep=0pt] at (4,2.5) {};

% left side nodes
\node[blueCirc,label=below:{$(4,9)$}] (4-9) at (2.6,2.2) {};
\node[redCirc, label=below:{$(3,10)$}] (3-10) at (3,4) {};
\node[redCirc, label=above:{$(2,3)$}] (2-3) at (2.5,5.5) {};

% v*, v** 
\node[redCirc,label={[label distance=-0.05em]88:$\textcolor{red}{v_{*}}\!=\!(3,7)$}] (vStar) at (4,5) {};
\node[blueCirc,label=below:{$\textcolor{blue}{v_{**}}\! = \!(7,9)$}] (vStarStar) at (4,2.5) {};

% right side nodes
\node[blueCirc,label=above:{$(5,7)$}] (5-7) at (6,5.5) {};
\node[blueCirc,label=right:{$(5,8)$}] (5-8) at (7.3,5.5) {};
\node[redCirc,label=right:{$(1,7)$}] (1-7) at (7,4) {};
\node[blueCirc,label=below:{$(6,7)$}] (6-7) at (6,2.2) {};

% edges
%% Left side nodes
\draw (2-3) -- (3-10);

\draw (4-9) -- (vStarStar);

\draw (2-3) -- (vStar);
\draw (3-10) -- (vStar);

%% Center
\draw (vStar) -- (vStarStar);

%% Right side nodes
\draw (5-7) -- (vStar);
\draw (1-7) -- (vStar);
\draw (6-7) -- (vStar);
\draw (5-7) -- (vStarStar);
\draw (1-7) -- (vStarStar);
\draw (6-7) -- (vStarStar);

\draw (5-7) -- (1-7);
\draw (5-7) -- (6-7);
\draw (1-7) -- (6-7);

\draw (5-7) -- (5-8);

\end{tikzpicture}
\caption{(Left) The joined tree represented by $\inparen{\widetilde{T}_{\text{red}}(u_*, u_*^{\text{join}}), \widetilde{T}_{\text{blue}}(u_{**}, u_{**}^{\text{join}})}$ from Figure \ref{fig:equalAverageED_flucPart_TredTblueSeparate}. The $P_2$ decoration is highlighted in green. (Right) The incompatibility graph $H$ with the corresponding coloring. Vertices highlighted in green correspond to the $P_2$ decoration.}
\label{fig:equalAverageED_TredTblueJoin_and_IncompatibilityGraph}
\end{figure}

\begin{lemma} \label{lem:ursell-function-tm-identity-2}
Let $T_m$ denote a generic unlabeled simple tree on $m+1$ vertices. 
Then we have
\begin{align}
    \label{eq:MatchingEdgeDensity_fluctuation_fix_m}
    \sum_{T_m} \widetilde{\phi}(H(T_m)) \frac{\gamma(T_m)}{\aut(T_m)} = -2 \sum_{\ell = 1}^{m-1} \sum_{(T_\ell, T_{m-\ell})} \frac{ \ell (m-\ell) \widetilde{\phi}(H(T_\ell)) \widetilde{\phi}(H(T_{m-\ell})) }{\aut(T_\ell) \aut(T_{m - \ell})} ,
\end{align}
% where $\widetilde{\phi}(H) := m! \cdot \phi(H)$ denotes the ``unnormalized'' Ursell function, and $\aut(\cdot)$ denotes the number of authomorphisms. 
where $\gamma(\cdot)$ is defined in \eqref{eq:gamma_Tm_definition}. 
\end{lemma}

\begin{proof}
% Fix $m$. 
Similar to the proof of Lemma~\ref{lem:ursell-function-tmrep-tm-identity}, we will label and color the trees in \eqref{eq:MatchingEdgeDensity_fluctuation_fix_m} and show that it is equivalent to (with notation to be explained)
\begin{align}
    \label{eq:MatchingEdgeDensity_fluctuation_fix_m_rewrite}
    \sum_{\widetilde{T_m}(u_{*}, u_{**})} \widetilde{\phi}(H(\widetilde{T_m})) = -\frac{1}{2}  \sum_{\inparen{ \widetilde{T_{\text{red}}}(u_{*}^{\text{join}}, u_{*}), \, \widetilde{T_{\text{blue}}}(u_{**}^{\text{join}}, u_{**})  }} \widetilde{\phi}(H(\widetilde{T_{\text{red}}})) \widetilde{\phi}(H(\widetilde{T_{\text{blue}}})),
\end{align}
where we have suppressed any mention of distinguished vertices in the Ursell functions where clear from context. 

Rewrite the LHS of \eqref{eq:MatchingEdgeDensity_fluctuation_fix_m} in terms of ``labeled trees decorated with a $P_2$''. Formally, rewrite the LHS as a sum over generic elements $\widetilde{T_m}(u_{*}, u_{**})$ such that 
\begin{itemize}
    \item the vertices are labeled with $[m+1]$, and
    \item the (unique) path between the distinguished vertices $u_{*}$ and $u_{**}$ forms a $P_2$.
\end{itemize}
Note that two identically labeled identical trees $\widetilde{T_m}(u_{*}, u_{**})$ and $\widetilde{T_m}(u'_{*}, u'_{**})$ are considered different if $\inbraces{u_{*}, u_{**}}$ and $\inbraces{u'_{*}, u'_{**}}$ are different pairs (i.e., the $P_2$ decoration is different for them). The number of elements $\widetilde{T_m}(u_{*}, u_{**})$ that can be generated from an unlabeled and undecorated $T_m$ is
% \begin{align}
     % \label{eq:MatchingEdgeDensity_fluctuation_fix_m_LHS_CombFactor}
    $\frac{(m+1)!}{\aut(T_m)} \gamma(T_m).$
% \end{align}
Therefore, the LHS of \eqref{eq:MatchingEdgeDensity_fluctuation_fix_m_rewrite} is $(m+1)!$ times the LHS of \eqref{eq:MatchingEdgeDensity_fluctuation_fix_m}.

On the other hand, rewrite the RHS of \eqref{eq:MatchingEdgeDensity_fluctuation_fix_m} in terms of ``colored, labeled trees whose join is decorated with a $P_2$''. Formally, the RHS will be written as a sum over generic elements 
\begin{align}\label{eq:eqAverageED_FlucPart_P2Decor_TredTblueTuple}
    \inparen{ \widetilde{T_{\text{red}}}(u_{*}^{\text{join}}, u_{*}), \, \widetilde{T_{\text{blue}}}(u_{**}^{\text{join}}, u_{**})  }
\end{align}
satisfying the following: 
\begin{itemize}
    \item $\widetilde{T_{\text{red}}}(u_{*}^{\text{join}}, u_{*})$ and $\widetilde{T_{\text{blue}}}(u_{**}^{\text{join}}, u_{**})$ are vertex-labeled trees with distinguished vertices as indicated in parenthesis, and have vertices and edges colored red and blue respectively.
    \item The vertex label for $u_{*}^{\text{join}}$ coincides with that of $u_{**}^{\text{join}}$.
    \item Joining the two trees by superimposing $u_{*}^{\text{join}}$ and $u_{**}^{\text{join}}$ gives a tree that is labeled by $[m+1]$.
    \item Denote the joining vertex by $u := u_{*}^{\text{join}} = u_{**}^{\text{join}}$. In the joined tree, $u_{*} \text{---} u \text{---} u_{**}$ forms a $P_2$.
\end{itemize}
An example of such a tuple \eqref{eq:eqAverageED_FlucPart_P2Decor_TredTblueTuple} and its corresponding join is given in Figure \ref{fig:equalAverageED_flucPart_TredTblueSeparate} and Figure \ref{fig:equalAverageED_TredTblueJoin_and_IncompatibilityGraph} (Left).

By construction, the joined tree has vertex sets $V_{\text{red}}$ and $V_{\text{blue}}$ colored red and blue respectively. The induced subgraphs $H[V_{\text{red}}]$ and $H[V_{\text{blue}}]$ are connected sub-trees. There is only one vertex $u$ that is colored both red and blue. The number of such generic elements that can be generated from a pair $(T_\ell, T_{m-\ell})$ is 
\begin{align}
    \label{eq:MatchingEdgeDensity_fluctuation_fix_m_RHS_CombFactor}
    \binom{m+1}{\ell + 1} \frac{(\ell + 1)!}{\aut(T_\ell)} \ell \cdot 2 \cdot  (m-\ell)  \cdot 2 \cdot \frac{(m-\ell)!}{\aut(T_{m-\ell})} 
    = \frac{4 (m + 1)! \ell (m-\ell)}{\aut(T_\ell) \aut(T_{m-\ell})} .
\end{align}
Rewriting \eqref{eq:MatchingEdgeDensity_fluctuation_fix_m} as described, we see that the the RHS of \eqref{eq:MatchingEdgeDensity_fluctuation_fix_m_rewrite} is $(m+1)!$ times the RHS of \eqref{eq:MatchingEdgeDensity_fluctuation_fix_m}, 
% and scaling the LHS and RHS by the combinatorial factors \eqref{eq:MatchingEdgeDensity_fluctuation_fix_m_LHS_CombFactor} and \eqref{eq:MatchingEdgeDensity_fluctuation_fix_m_RHS_CombFactor} respectively, we obtain \eqref{eq:MatchingEdgeDensity_fluctuation_fix_m_rewrite} 
where $\ell = \abs{ \widetilde{T_{\text{red}}}  }$ and so $m - \ell = \abs{ \widetilde{T_{\text{blue}}}  }$.

Therefore, it suffices to show that for a fixed $\widetilde{T_m}(u_{*}, u_{**})$, 
\begin{align}
    \label{eq:MatchingEdgeDensity_fluctuation_fix_m_fix_lab_dec_Tm}
    \widetilde{\phi}(H(\widetilde{T_m})) = - \sum_{\inparen{ \widetilde{T_{\text{red}}}(u_{*}^{\text{join}}, u_{*}), \, \widetilde{T_{\text{blue}}}(u_{**}^{\text{join}}, u_{**})  } \cong  \widetilde{T_m}(u_{*}, u_{**})} \widetilde{\phi}(H(\widetilde{T_{\text{red}}})) \widetilde{\phi}(H(\widetilde{T_{\text{blue}}})),
\end{align}
where the sum constraint means that an uncolored version of the join of $\widetilde{T_{\text{red}}}(u_{*}^{\text{join}}, u_{*})$ and $\widetilde{T_{\text{blue}}}(u_{**}^{\text{join}}, u_{**})  $ is isomorphic to $\widetilde{T_m}(u_{*}, u_{**})$. In particular, the $P_2$ decoration of the joined tree must also coincide with that of $\widetilde{T_m}(u_{*}, u_{**})$. Note that every such valid pair $\inparen{ \widetilde{T_{\text{red}}}(u_{*}^{\text{join}}, u_{*}), \, \widetilde{T_{\text{blue}}}(u_{**}^{\text{join}}, u_{**})  }$ has a corresponding pair $\inparen{ \widetilde{T_{\text{blue}}}(u_{*}^{\text{join}}, u_{*}), \, \widetilde{T_{\text{red}}}(u_{**}^{\text{join}}, u_{**})  }$ with the colors switched. By symmetry, these give the same contribution. We fix without loss of generality that $u_*$ is always colored red and $u_{**}$ is always colored blue\footnote{Note that this is different from the proof of Lemma~\ref{lem:ursell-function-tmrep-tm-identity} where the red edge in the double edge is denoted by $v_*$ and the blue one denoted by $v_{**}$. In that proof if we swapped the colors, there would be double counting. However, here $u_*$ and $u_{**}$ are different vertices in the tree, so swapping colors indeed contributes a factor $2$, canceling the factor $1/2$ in \eqref{eq:MatchingEdgeDensity_fluctuation_fix_m_rewrite}.}, thus absorbing the factor of $\frac{1}{2}$ in \eqref{eq:MatchingEdgeDensity_fluctuation_fix_m_rewrite}.

Suppose the $P_2$ decoration in $\widetilde{T_m}(u_{*}, u_{**})$ is $u_{*} \text{---} u \text{---} u_{**}$. In what follows, fix $H = H(\widetilde{T_m}(u_{*}, u_{**}))$ the incompatibility graph (line graph) of $\widetilde{T_m}(u_{*}, u_{**})$. We distinguish the two vertices in $H$, calling them $v_{*}$ and $v_{**}$, corresponding to the two edges $\inbraces{u_{*}, u}$ and $\inbraces{u_{**}, u}$ in the $P_2$ decoration in $\widetilde{T_m}(u_{*}, u_{**})$. 
Define the subset $\cC(H; v_*, v_{**})$ of bi-colorings of $V(H)$ as in \eqref{eq:def-bi-coloring}. 
% Define the following subset $\cC(H)$ of bi-colorings of $V(H)$:
% \begin{align}
%     \cC(H) := \inbraces{  (V_{\text{red}}, V_{\text{blue}}) : \, 
%     \begin{aligned}
%         & V_{\text{red}} \cup V_{\text{blue}} = V(H) \text{ disjoint}, \, V_{\text{red}} \ni v_*, \, V_{\text{blue}} \ni v_{**},\\
%         & H[V_{\text{red}}] \text{ and } H[V_{\text{blue}}] \text{ are each connected subgraphs. } 
%     \end{aligned}
%     }.
% \end{align}
There is a bijection between $\cC(H; v_*, v_{**})$ and the set $\inbraces{  \inparen{ \widetilde{T_{\text{red}}}(u_{*}^{\text{join}}, u_{*}), \, \widetilde{T_{\text{blue}}}(u_{**}^{\text{join}}, u_{**})   }  \cong  \widetilde{T_m}(u_{*}, u_{**}) }$. Figure \ref{fig:equalAverageED_TredTblueJoin_and_IncompatibilityGraph} (Right) gives an example of such a bi-coloring of $V(H)$ that corresponds to a splitting of $\widetilde{T_m}(u_{*}, u_{**})$ on vertex $u$ into a red and a blue tree given in (Left).
% Let $H(\text{spann., conn.})$ denote the set of spanning and connected subgraphs of $H$. 
Recalling the definition of the Ursell function \eqref{eq:Ursell_definition}, we see that \eqref{eq:MatchingEdgeDensity_fluctuation_fix_m_fix_lab_dec_Tm} reduces to \eqref{eq:Ursell_trees_adjVertices_identity}. 
% \begin{align}\label{eq:MatchingEdgeDensity_fluctuation_fix_m_fix_lab_dec_Tm_lineGraphForm}
%     \sum_{ \substack{ S \subseteq H(\text{spann., conn.})  } } (-1)^{\abs{S}} = \sum_{ (V_{\text{red}}, V_{\text{blue}}) \in \cC(H) } \sum_{\substack{  S_{\text{red}} \subseteq H[V_{\text{red}}](\text{spann., conn.})  \\ S_{\text{blue}} \subseteq H[V_{\text{blue}}](\text{spann., conn.})   } } (-1)^{\abs{S_{\text{red}}} + \abs{S_{\text{blue}}} + 1 },
% \end{align}
The proof is complete by Lemma \ref{lemma:Ursell_trees_adjVertices_identity}.
\end{proof}

\begin{lemma}
Let $T_m$ denote an unlabeled simple tree, let $T_m^{\text{rep}}$ denote an unlabeled tree with one twice repeated edge, and let $\TmTripleEdge$ denote an unlabeled tree with one edge repeated three times. 
Then 
\begin{align}
        \frac{1}{3!} \sum_{\TmTripleEdge} \frac{\tilde \phi(H(\TmTripleEdge)) }{\aut(\TmTripleEdge)} = -\frac{2}{3} \sum_{\ell = 1}^{m} \sum_{\inparen{T_\ell^{\text{rep}}, T_{m+1-\ell}} } (m+1-\ell) \frac{\widetilde{\phi}(H(T_\ell^{\text{rep}})) \widetilde{\phi}(H(T_{m+1-\ell})) }{ \aut(T_\ell^{\text{rep}}) \aut(T_{m+1-\ell}) } .
        \label{eq:tripleEdge_cancels_logprefactorO1_fix_m}
    \end{align}
\end{lemma}

\begin{proof}
We first show that
\eqref{eq:tripleEdge_cancels_logprefactorO1_fix_m} is equivalent to  
    \begin{align}
    \sum_{\widetilde{T}_m^{\equiv}} \widetilde{\phi}(H(\widetilde{T}_m^{\equiv})) = -2 \sum_{\inparen{ \widetilde{T}_{\text{red}}^{\text{rep}}, \widetilde{T}_{\text{blue}}  } } \widetilde{\phi}(H(\widetilde{T}_{\text{red}}^{\text{rep}}))  \widetilde{\phi}(H(\widetilde{T}_{\text{blue}})) \label{eq:T_m-triple-equiv-form}
    \end{align}
    with notation to be explained. 
Note that we can rewrite the LHS of \eqref{eq:tripleEdge_cancels_logprefactorO1_fix_m} as a sum over trees with exactly one triple repeated edge, with vertices labeled from $[m+1]$. We generically denote such trees by $\widetilde{\TmTripleEdge}$. The number of ways to label an unlabeled $\TmTripleEdge$ is $(m+1)! / \aut(\TmTripleEdge)$. 
Therefore, the LHS of \eqref{eq:T_m-triple-equiv-form} is $6 \cdot (m+1)!$ times the LHS of \eqref{eq:tripleEdge_cancels_logprefactorO1_fix_m}.

    On the other hand, we will rewrite the RHS of \eqref{eq:tripleEdge_cancels_logprefactorO1_fix_m} in terms of edge-colored and labeled trees. These are generically denoted by the tuple
    \begin{align}\label{eq:equalAverageED_TripleEdge_Tuple}
        \inparen{ \widetilde{T}_{\text{red}}^{\text{rep}}, \widetilde{T}_{\text{blue}}  },
    \end{align}
    satisfying the following: 
    \begin{itemize}
        \item $\widetilde{T}_{\text{red}}^{\text{rep}}$ is a labeled tree with exactly one repeated edge. $\widetilde{T}_{\text{blue}}$ is a labeled simple tree. Both have vertex labels in $[m+1]$.
        \item There is an edge in $\widetilde{T}_{\text{blue}}$ labeled the same as the repeated edge in $\widetilde{T}_{\text{red}}^{\text{rep}}$.
        \item Superimposing on the same labeled edge (matching the corresponding vertices by label) gives a triple edge tree labeled in $[m+1]$.
    \end{itemize}
    Note that by construction, the triple edge of the joined tree will have two edges colored red, and one colored blue. We refer to Figure \ref{fig:equalAverageED_TripleEdge} (Left) for an example of the joined tree. The number of such tuples that can be generated from an unlabeled, uncolored pair $\inparen{T_\ell^{\text{rep}}, T_{m+1-\ell}}$ is 
    \begin{align*}
        \binom{m+1}{\ell + 1} \frac{(\ell + 1)!}{\aut(\TmRep)} (m+1-\ell) \cdot 2 \cdot \frac{(m-\ell)!}{\aut(T_{m+1-\ell})}
        =  \frac{2 \cdot (m+1)!}{\aut(\TmRep) \aut(T_{m+1-\ell})} .
    \end{align*}
    For any $\widetilde{T}_{\text{red}}^{\text{rep}}$, define $\ell = \abs{\widetilde{T}_{\text{red}}^{\text{rep}}} - 1$. Thus any corresponding $\widetilde{T}_{\text{blue}}$ satisfies $|\widetilde{T}_{\text{blue}}| = m+1-\ell$. 
    % Rewriting the RHS of \eqref{eq:tripleEdge_cancels_logprefactorO1_fix_m} as stated, and scaling by the appropriate combinatorial factors, 
    We see that the RHS of \eqref{eq:T_m-triple-equiv-form} is $6 \cdot (m+1)!$ times the RHS of \eqref{eq:tripleEdge_cancels_logprefactorO1_fix_m}.

    Therefore, to prove \eqref{eq:tripleEdge_cancels_logprefactorO1_fix_m}, it suffices to show for fixed $\widetilde{T}_m^{\equiv}$ that
    \begin{align}
        \widetilde{\phi}(H(\widetilde{T}_m^{\equiv})) = -2 \sum_{\inparen{ \widetilde{T}_{\text{red}}^{\text{rep}}, \widetilde{T}_{\text{blue}}  } \cong  \widetilde{T}_m^{\equiv} }  \widetilde{\phi}(H(\widetilde{T}_{\text{red}}^{\text{rep}}))  \widetilde{\phi}(H(\widetilde{T}_{\text{blue}})),
        \label{eq:tripleEdge_cancels_logprefactorO1_fix_m_fix_T_m3rep}
    \end{align}
    where the sum constraint means that the \emph{uncolored} joined tree of $\inparen{ \widetilde{T}_{\text{red}}^{\text{rep}}, \widetilde{T}_{\text{blue}}  }$ is isomorphic to $\widetilde{T}_m^{\equiv}$. Fix now the incompatibility (i.e.~line) graph $H = H(\widetilde{T}_m^{\equiv})$. Let the three vertices in $H$ corresponding to the triple edge be $v_*$, $v_{**}$, and $v_3$. Define the following subset of bi-colorings of $V(H)$: 
    \begin{align*}
        \cC(H; v_3 \text{ red} ) := \inbraces{  (V_{\text{red}}, V_{\text{blue}}) : \, 
    \begin{aligned}
        & V_{\text{red}} \cup V_{\text{blue}} = V(H) \text{ disjoint}, \, V_{\text{red}} \ni v_*, v_3, \, V_{\text{blue}} \ni v_{**},\\
        & H[V_{\text{red}}] \text{ and } H[V_{\text{blue}}] \text{ are each connected subgraphs} 
    \end{aligned}
    }. 
    \end{align*}
    Define the set $\cC(H; v_3 \text{ blue} )$ analogously, with $v_{3}$ always in $V_{\text{blue}}$ instead. Note that with $\cC(H; v_*, v_{**})$ defined in \eqref{eq:def-bi-coloring}, by symmetry
    \begin{align*}
        \cC(H; v_*, v_{**}) = \cC(H; v_3 \text{ red} ) \cup \cC(H; v_3 \text{ blue} ).
    \end{align*}
    There is a bijection between $\cC(H; v_3 \text{ red} )$ and $\inbraces{\inparen{ \widetilde{T}_{\text{red}}^{\text{rep}}, \widetilde{T}_{\text{blue}}  } \cong  \widetilde{T}_m^{\equiv}}$. Figure \ref{fig:equalAverageED_TripleEdge} (Right) gives an example of such a bi-coloring. Recalling the definition of the Ursell function, \eqref{eq:tripleEdge_cancels_logprefactorO1_fix_m_fix_T_m3rep} reduces to 
    \begin{align*}
        \sum_{S \subseteq H(\text{conn.,~spann.})} \inparen{-1}^{\abs{S}} &= 2 \sum_{ (V_{\text{red}}, V_{\text{blue}}) \in \cC(H; v_3 \text{ red}) } \sum_{\substack{  S_{\text{red}} \subseteq H[V_{\text{red}}](\text{spann., conn.})  \\ S_{\text{blue}} \subseteq H[V_{\text{blue}}](\text{spann., conn.})   } } (-1)^{\abs{S_{\text{red}}} + \abs{S_{\text{blue}}} + 1 } \\
        &= \sum_{ (V_{\text{red}}, V_{\text{blue}}) \in \cC(H; v_*, v_{**}) } \sum_{\substack{  S_{\text{red}} \subseteq H[V_{\text{red}}](\text{spann., conn.})  \\ S_{\text{blue}} \subseteq H[V_{\text{blue}}](\text{spann., conn.})   } } (-1)^{\abs{S_{\text{red}}} + \abs{S_{\text{blue}}} + 1 }.
    \end{align*}
    The proof is complete by Lemma \ref{lemma:Ursell_trees_adjVertices_identity}.
\end{proof}

\begin{figure}[htp]\centering
%%%%%%%% (TRepRed,Tblue) with v3 %%%%%%%%
\begin{tikzpicture}
\tikzset{%
glow/.style={%
preaction={#1, draw, line join=round, line width=0.5pt, opacity=0.04,
preaction={#1, draw, line join=round, line width=1.0pt, opacity=0.04,
preaction={#1, draw, line join=round, line width=1.5pt, opacity=0.04,
preaction={#1, draw, line join=round, line width=2.0pt, opacity=0.04,
preaction={#1, draw, line join=round, line width=2.5pt, opacity=0.04,
preaction={#1, draw, line join=round, line width=3.0pt, opacity=0.04,
preaction={#1, draw, line join=round, line width=3.5pt, opacity=0.04,
preaction={#1, draw, line join=round, line width=4.0pt, opacity=0.04,
preaction={#1, draw, line join=round, line width=4.5pt, opacity=0.04,
preaction={#1, draw, line join=round, line width=5.0pt, opacity=0.04,
preaction={#1, draw, line join=round, line width=5.5pt, opacity=0.04,
preaction={#1, draw, line join=round, line width=6.0pt, opacity=0.04,
}}}}}}}}}}}}}}
%\draw[step=1, lightgray] (0,0) grid (8,8);

\draw[yellow, opacity=0.4, line width=0.15cm] (2,3) -- (4,4);

\node (2) at (2,3) {$2$};
\node (8) at (4,4) {$8$};
\node (3) at (1,4) {$3$};
\node (1) at (0.5,5.2) {$1$};
\node (5) at (1,2) {$5$};
\node (7) at (4.5,5.5) {$7$};
\node (4) at (4.5,3) {$4$};
\node (6) at (5.7,2.2) {$6$};
\node (9) at (4,2) {$9$};

\draw[thick, red] (2) -- (3);
\draw[thick, red] (2) -- (5);
\draw[thick, red] (1) -- (3);
\draw[thick, red] (7) -- (8);

\draw[thick, blue] (4) -- (8);
\draw[thick, blue] (4) -- (6);
\draw[thick, blue] (4) -- (9);

\draw[thick, transform canvas={xshift=-0.15em, yshift=0.55em}, red] (2) -- node[label={[label distance=-0.3em]90:$v_{*}$}] {} (8);
%\draw[thick, red, glow = green] (2) -- (8);
\draw[thick, red] (2) -- (8);
\draw[thick, transform canvas={xshift=0.15em, yshift=-0.55em}, blue] (2) -- node[label=below:$v_{**}$] {} (8);
\end{tikzpicture}
\qquad\qquad % <----------------- SPACE BETWEEN PICTURES
%%%%%%%% H(TRepRed,Tblue) with v3 %%%%%%%%
\begin{tikzpicture}[every node/.style={font=\small}, redCirc/.style={circle,fill=red, minimum size=5pt, inner sep=0pt}, blueCirc/.style={circle,fill=blue, minimum size=5pt, inner sep=0pt}]
\tikzset{%
glow/.style={%
preaction={#1, draw, line join=round, line width=0.5pt, opacity=0.04,
preaction={#1, draw, line join=round, line width=1.0pt, opacity=0.04,
preaction={#1, draw, line join=round, line width=1.5pt, opacity=0.04,
preaction={#1, draw, line join=round, line width=2.0pt, opacity=0.04,
preaction={#1, draw, line join=round, line width=2.5pt, opacity=0.04,
preaction={#1, draw, line join=round, line width=3.0pt, opacity=0.04,
preaction={#1, draw, line join=round, line width=3.5pt, opacity=0.04,
preaction={#1, draw, line join=round, line width=4.0pt, opacity=0.04,
preaction={#1, draw, line join=round, line width=4.5pt, opacity=0.04,
preaction={#1, draw, line join=round, line width=5.0pt, opacity=0.04,
preaction={#1, draw, line join=round, line width=5.5pt, opacity=0.04,
preaction={#1, draw, line join=round, line width=6.0pt, opacity=0.04,
preaction={#1, draw, line join=round, line width=6.5pt, opacity=0.04,
preaction={#1, draw, line join=round, line width=7.0pt, opacity=0.04,
preaction={#1, draw, line join=round, line width=7.5pt, opacity=0.04,
preaction={#1, draw, line join=round, line width=8.0pt, opacity=0.04,
preaction={#1, draw, line join=round, line width=8.5pt, opacity=0.04,
preaction={#1, draw, line join=round, line width=9.0pt, opacity=0.04,
preaction={#1, draw, line join=round, line width=9.5pt, opacity=0.04,
}}}}}}}}}}}}}}}}}}}}}

\node[circle, fill=yellow, opacity=0.6, minimum size=12pt,inner sep=0pt] at (2,3) {};

%\draw[step=1, lightgray] (0,0) grid (8,8);
%\node[redCirc,label=left:{$v_3$}, glow=green] (vThree) at (2,3) {};
\node[redCirc,label=left:{$v_3$}] (vThree) at (2,3) {};
\node[redCirc,label=below:{$v_*$}] (vStar) at (4,4) {};
\node[blueCirc,label=right:{$v_{**}$}] (vStarStar) at (6,3) {};

\node[redCirc, label=above:{$(2,3)$}] (2-3) at (2.5,5) {};
\node[redCirc, label=above:{$(2,5)$}] (2-5) at (5,5) {};
\node[redCirc, label=left:{$(1,3)$}] (1-3) at (1.5,4.5) {};
\node[redCirc, label=below:{$(7,8)$}] (7-8) at (3,1) {};
\node[blueCirc, label={[xshift=-4pt]below:{$(4,8)$}}] (4-8) at (5,1.4) {};
\node[blueCirc, label=right:{$(4,6)$}] (4-6) at (6.5,2) {};
\node[blueCirc, label=right:{$(4,9)$}] (4-9) at (6,0.5) {};

\draw (vThree) -- (vStar);
\draw (vThree) -- (vStarStar);
\draw (vStar) -- (vStarStar);

\draw (2-3) -- (vThree);
\draw (2-3) -- (vStar);
\draw (2-3) to [bend left=16] (vStarStar);
\draw (2-5) to [bend right=16] (vThree);
\draw (2-5) -- (vStar);
\draw (2-5) -- (vStarStar);
\draw (1-3) -- (2-3);
\draw (2-3) -- (2-5);

\draw (7-8) -- (vThree);
\draw (7-8) to [bend left=20] (vStar);
\draw (7-8) -- (vStarStar);
\draw (4-8) -- (vThree);
\draw (4-8) to [bend right=20] (vStar);
\draw (4-8) -- (vStarStar);
\draw (4-8) -- (7-8);
\draw (4-8) -- (4-6);
\draw (4-8) -- (4-9);
\draw (4-6) -- (4-9);
\end{tikzpicture}
\caption{(Left) A joined tree represented by the tuple \eqref{eq:equalAverageED_TripleEdge_Tuple}. Edge $v_3$ is indicated in green. (Right) The corresponding incompatibility graph $H$ where vertices $v_*$, $v_{**}$, and $v_3$ correspond to the repeated edge $(2,8)$. The bi-coloring depicted is in $\cC(H; v_3 \text{ red} )$.}
\label{fig:equalAverageED_TripleEdge}
\end{figure}
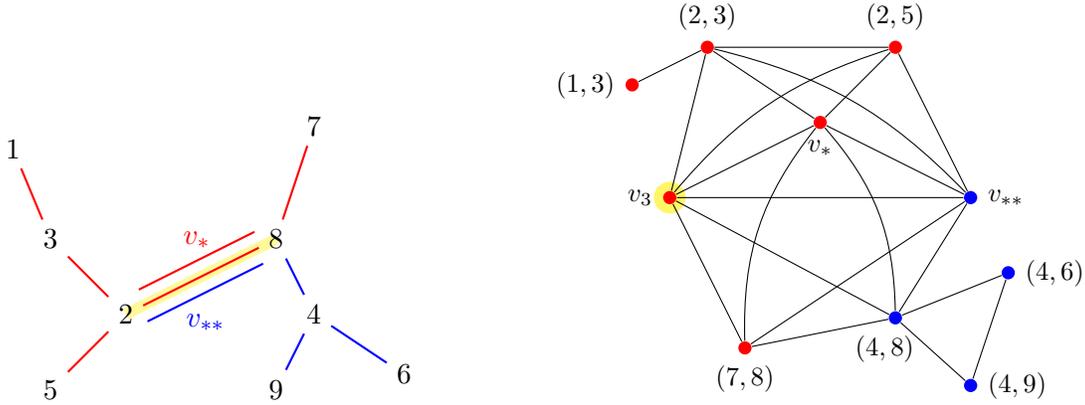

\section{Analysis of the log-likelihood ratio: equal ambient edge density}
\label{sec:plantingVarSize_pEqualq}

This section focuses on analyzing the likelihood ratio for Problem~\ref{prob:detection} in the setting of Assumption~\ref{asmpt:nonMatchingEdges}. 
Let us first show that Theorem~\ref{thm:plantedMatchings_mean_is_minus_half_variance} is an immediate consequence of Theorem~\ref{thm:nonMatching_LLdistributionUnderNull}.

\begin{proof}[Proof of Theorem~\ref{thm:plantedMatchings_mean_is_minus_half_variance}]
The asymptotic normality of the log-likelihood ratio for $A \sim \cQ$ follows immediately from Theorem~\ref{thm:nonMatching_LLdistributionUnderNull} combined with Lemma~\ref{lemma:signed_K2_normality_P_Q}. 
The corresponding statement for $A \sim \cP_\lambda$ is deduced from Le Cam's third lemma by considering the limiting joint distribution of $\inparen{\log \frac{\ud \cP_\lambda}{\ud \cQ}, \log \frac{\ud \cP_\lambda}{\ud \cQ}}$ under $\cQ$ as in \cite[Example~6.7]{van2000asymptotic}. 
The second statement in Theorem~\ref{thm:plantedMatchings_mean_is_minus_half_variance} follows from the Neyman--Pearson lemma together with the optimal error for testing between two Gaussian hypotheses, achieved by thresholding the log-likelihood at zero.
% we use that 
% \begin{align*}
%     \TV(\cP_\infty,\cQ) = 1 - \PP_{A \sim \cP_\infty}\!\insquare{ \frac{\ud \cP_\infty}{\ud \cQ} (A) \leq 1  } - \PP_{A \sim \cQ}\!\insquare{ \frac{\ud \cP_\infty}{\ud \cQ} (A) \geq 1  }.
% \end{align*}
% The result follows since $\PP_{A \sim \cP_\infty}\!\insquare{ \log \frac{\ud \cP_\infty}{\ud \cQ} (A) \leq 0  }$ and $\PP_{A \sim \cQ}\!\insquare{ \log \frac{\ud \cP_\infty}{\ud \cQ} (A) \geq 0  }$ both converge to the same limit $\Phi\!\inparen{ - \sqrt{\frac{1-p}{8p}}  } = \frac{1}{2}\mathsf{erfc}\inparen{\sqrt{\frac{1-p}{16p}}}$. 
% \notetim{Depending on the ordering the paper let's only state the Le Cam and TV part formally once when it first appears.}
\end{proof}

\noindent
The rest of this section is devoted to proving Theorem~\ref{thm:nonMatching_LLdistributionUnderNull}. 

\subsection{Approximation of the log-likelihood ratio}

We collect several definitions and results which will prove Theorem \ref{thm:nonMatching_LLdistributionUnderNull}. With Lemma~\ref{lem:log-likelihood-ratio-partition-functions} together with the cluster expansion absolute convergence and truncation provided by Theorem \ref{thm:truncate_CE_clogn}, let us further decompose the log-likelihood as 
\begin{align}\label{eq:decompose_loglikelihood}
    \log \frac{\ud \cP_\lambda}{\ud \cQ}(A) &= \sum_{m = 1}^{2 \log n} \sum_{e_1,\dots,e_m} \phi(H(e_1,\dots,e_m)) \lambda^m \insquare{ \prod_{j=1}^{m} \frac{A_{e_j}}{p} - 1  } + O_\PP\!\inparen{ \frac{1}{n}  }  \nonumber\\
    &= \simpleTrees + \oneRepTrees + \mathsf{remainder}_{\leq 2 \log n} + O_\PP\!\inparen{ \frac{1}{n}  },
\end{align}
where 
\begin{align*}
    \simpleTrees &:= \sum_{m = 1}^{2 \log n} \sum_{\substack{e_1,\dots,e_m \\ \text{simple tree}}} \phi(H(e_1,\dots,e_m)) \lambda^m \insquare{ \prod_{j=1}^{m} \frac{A_{e_j}}{p} - 1  } \\
    \oneRepTrees &:= \sum_{m = 2}^{2 \log n} \sum_{\substack{e_1,\dots,e_m \\ \text{tree with one rep.~edge}}} \phi(H(e_1,\dots,e_m)) \lambda^m \insquare{ \prod_{j=1}^{m} \frac{A_{e_j}}{p} - 1  } \\
    \mathsf{remainder}_{\leq 2 \log n} &:= \sum_{m = 1}^{2 \log n} \sum_{e_1,\dots,e_m} \phi(H(e_1,\dots,e_m))\lambda^m \insquare{ \prod_{j=1}^{m} \frac{A_{e_j}}{p} - 1  } - \simpleTrees - \oneRepTrees.
\end{align*}

In what follows, Assumption \ref{asmpt:nonMatchingEdges} is in force. In Section \ref{sec:nonMatchingFluctuationPart} we will show that $\simpleTrees$ gives the zero-mean fluctuation part of \eqref{eq:nonMatching_LLdistributionUnderNull}. 
\begin{proposition}\label{prop:nonMatching_simpleTrees_mainResult}
Let $A \sim \cQ$. Then
    \begin{align}\label{eq:nonMatching_simpleTrees_mainResult}
         \simpleTrees = \sqrt{\frac{2(1-p)}{p}} \frac{\EE\abs{M}}{n} \frac{\signedKtwo(A)}{\sqrt{\Var \signedKtwo(A)}} + O_\PP\!\inparen{ \frac{1}{p \sqrt{n}}  }.
    \end{align}
\end{proposition} 
\noindent
In Section \ref{sec:nonMatchingMeanPart} we will show that $\oneRepTrees$ gives the deterministic mean part of \eqref{eq:nonMatching_LLdistributionUnderNull}. 
\begin{proposition}\label{prop:nonMatching_oneRepTrees_mainResult}
Let $A \sim \cQ$. Then
    \begin{align}\label{eq:nonMatching_oneRepTrees_mainResult}
        \oneRepTrees = - \frac{1-p}{p} \inparen{\frac{\EE\abs{M}}{n}}^2 + O_\PP\!\inparen{\frac{(\log n)^2}{n p^{3/2}}}.
    \end{align}
\end{proposition} 
\noindent
Finally, in Section \ref{sec:nonMatchingRemainderLessThanTwoLogn} we will show that $ \mathsf{remainder}_{\leq 2 \log n}$ is small.
\begin{proposition}\label{prop:nonMatching_remainderLessThanTwoLogn_mainResult}
    Let $A \sim \cQ$. Then
    \begin{align}\label{eq:nonMatching_remainderLessThanTwoLogn_mainResult}
        \mathsf{remainder}_{\leq 2 \log n} = O_\PP\!\inparen{ \frac{1}{p^{2} n }  }.
    \end{align}
\end{proposition}
\begin{proof}[Proof of Theorem \ref{thm:nonMatching_LLdistributionUnderNull}]
    Immediate from \eqref{eq:decompose_loglikelihood} and the above three propositions.
\end{proof} 

% It will be seen for $A \sim \cQ$ that
% \begin{itemize}
%     \item $\mathsf{remainder} = \cN\!\inparen{0, \frac{c^2}{2}\frac{1-p}{p}} + o_{\PP}(1)$,  
%     \item $\oneRepTrees = \EE\insquare{\oneRepTrees} + o_{\PP}(1) = -\frac{c^2}{4}\frac{1-p}{p} + o_{\PP}(1)$, and
%     \item $\mathsf{remainder} = o_{\PP}(1)$.
% \end{itemize} 

\subsection{Fluctuation part}\label{sec:nonMatchingFluctuationPart}
In this section we establish Proposition \ref{prop:nonMatching_simpleTrees_mainResult}. Let $T_m$ denote a generic unlabeled connected simple tree with $m$ edges. 
% For finite $n$, 
Recall the notation \eqref{eq:ordinaryCenteredSignedSubgraphCount_definition} that $\overline{T}_m(A) = T_m(A) - \EE_{A \sim \cQ} T_m(A)$. Note that $\EE_{A \sim \cQ} T_m(A) = T_m(K_n) p^m$. Observe that $\simpleTrees$ as defined in \eqref{eq:decompose_loglikelihood} can be written as
\begin{align*}
    \simpleTrees = \sum_{m = 1}^{2 \log n} \sum_{T_m} \phi(H(T_m)) m! \frac{\lambda^m}{p^m} \overline{T}_m(A).
\end{align*}
For each $T_m$, define
\begin{equation}
    \alpha(T_m) := \frac{\Cov\insquare{\overline{T}_m(A), \signedKtwo(A) }}{ \Var \signedKtwo(A)}, \qquad \text{and} \qquad r(T_m, A) := \overline{T}_m(A) - \alpha(T_m) \signedKtwo(A).
    \label{eq:def-alpha-r-Tm-A}
\end{equation}
Decompose $\simpleTrees$ as 
\begin{align}
    \simpleTrees = \Proj_{\signedKtwo}\!\inparen{  \simpleTrees  } + \Proj^{\perp}_{\signedKtwo}\!\inparen{  \simpleTrees  },
\end{align}
where 
\begin{align*}
    \Proj_{\signedKtwo}\!\inparen{  \simpleTrees  } &:= \sum_{m = 1}^{2 \log n} \sum_{T_m} \phi(H(T_m)) m! \frac{\lambda^m}{p^m} \alpha(T_m) \signedKtwo(A), \\
    \Proj^\perp_{\signedKtwo}\!\inparen{  \simpleTrees  } &:= \sum_{m = 1}^{2 \log n} \sum_{T_m} \phi(H(T_m)) m! \frac{\lambda^m}{p^m} r(T_m, A).
\end{align*}
Note that both these projections have zero mean. The proof of Proposition \ref{prop:nonMatching_simpleTrees_mainResult} will be immediate from Lemmas \ref{lemma:nonMatching_ProjK2simpleTrees_distribution} and \ref{lemma:VarProjK_2_Perp_simpleTrees} below. In particular, these lemmas make precise the zero-mean fluctuation statement implied in the heuristic \eqref{eq:nonMatching_LLfirstFewTerms_GaussianForm}.
\begin{lemma}\label{lemma:nonMatching_ProjK2simpleTrees_distribution}
    Let $A \sim \cQ$. Then 
    \begin{align}
        \Proj_{\signedKtwo}\!\inparen{  \simpleTrees  } = \sqrt{ \frac{2(1-p)}{p}  } \frac{\EE\abs{M}}{n} \frac{\signedKtwo(A)}{\sqrt{\Var \signedKtwo(A)}} + O_\PP\!\inparen{\frac{1}{n}}.
    \end{align}
\end{lemma}

\begin{proof}[Proof of Lemma \ref{lemma:nonMatching_ProjK2simpleTrees_distribution}]
    Compute for each $m$ and $T_m$,
    \begin{align}\label{eq:nonMatching_coeff_signedKtwo}
        \alpha(T_m) = \left. \frac{(n)_{m+1}}{\aut(T_m)} m p^m (1-p) \middle/   \binom{n}{2}p(1-p) \right. .
    \end{align}
    Plugging in $\alpha(T_m)$ into $\Proj_{\signedKtwo}\!\inparen{  \simpleTrees  }$, and comparing the resulting expression to the $\EE\abs{M}$ series from Proposition \ref{prop:GibbsMatching_mean_var_from_CE}, we find
    \begin{align*}
        \Proj_{\signedKtwo}\!\inparen{  \simpleTrees  } &= \inparen{ \sum_{m = 1}^{2 \log n} \sum_{T_m} \phi(H(T_m)) m! m \lambda^m \frac{(n)_{m+1}}{\aut(T_m)}   } \frac{\signedKtwo(A)}{\binom{n}{2} p} \\
        &= \sqrt{ \frac{2(1-p)}{p}  } \inparen{  \frac{\EE_{M \sim \mu_\lambda}\abs{M} + O(1) }{\sqrt{n(n-1)}} } \frac{\signedKtwo(A)}{\sqrt{\Var \signedKtwo(A)}}.
        \qedhere
    \end{align*}
\end{proof}

\begin{lemma}
    \label{lemma:VarProjK_2_Perp_simpleTrees}
    Let $A \sim \cQ$. Then 
    % for some universal constant $C > 0$
    \begin{align}
        \Var \Proj^\perp_{\signedKtwo}\!\inparen{  \simpleTrees  } = O\left( \frac{1}{p^2 n} \right) .
    \end{align}
\end{lemma}
We require the following estimate. Recall that $G(A)$ denotes the number of copies of $G$ in $A$. 
% Let $\mathscr{C}_{N,m}$ denote the set of connected unlabeled simple graphs on $N$ vertices and $m$ vertices.
Define
\begin{equation}\label{eq:gamma_Tm_definition}
    \gamma(G) := \sum_{v \in V(G)} \binom{\text{degree}(v)}{2}.
\end{equation}
The quantity $\gamma(G)$ can be interpreted as the number of ways to superimpose a wedge $P_2$ on $G$. 
% We next establish that the component in $\overline{\simpleTrees}$ orthogonal to the span of $\signedKtwo$ and $\signedPtwo$ is vanishing. For this purpose a more refined variance estimate than Claim \ref{claim:Var_G_N_m} is required. 

\begin{claim}
    \label{claim:Var_G_N_m}
    Suppose that $m = o(\sqrt{nq})$ and $A \sim G(n,q)$. Let $G_{N,m}$ be any connected unlabeled simple graph with $N$ vertices and $m$ edges. Then 
    \begin{align}\label{eq:Var_G_N_m_General}
        \Var G_{N,m}(A) = \frac{2m^{2} q^{2m-1} (1-q) (n)_N (n)_{N-2} }{\aut(G_{N,m})^2} \left( 1 + o(1) + O \Big( \frac{m^4}{n q^2} \Big) \right)
        %O\!\inparen{\frac{n^{2N - 3} m^{4} q^{2m-2} }{\aut(G_{N,m})^2}} ,
    \end{align}
    where $(n)_N$ denotes the falling factorial.
    % and $\aut(\cdot)$ denotes the number of automorphisms.
    Moreover, for any $T_m$, with $\gamma(T_m)$ defined in \eqref{eq:gamma_Tm_definition}, 
    % \begin{align}\label{eq:Var_G_N_m_Trees}
    %     \Var T_m(A) \leq \frac{2 m^2 (1-q)  q^{2m - 1} (n)_{m+1} (n)_{m-1} } {\aut(T_m)^2} + O\!\inparen{  \frac{n^{2m- 1} q^{2m - 2} m^4 }{\aut(T_m)^2}   }.
    % \end{align}
    \begin{align}\label{eq:Var_Tm_LeadAndSublead}
        \Var T_m(A) &= \frac{2 m^2 (1-q)  q^{2m - 1} (n)_{m+1} (n)_{m-1} } {\aut(T_m)^2} \nonumber \\
        &\qquad + \frac{2 \gamma(T_m)^2 (1-q^2)  q^{2m - 2} (n)_{m+1} (n)_{m-2} } {\aut(T_m)^2} 
        + O\!\inparen{  \frac{ m^6 n^{2m- 2} q^{2m - 3} }{\aut(T_m)^2}   } \\
        &= \frac{2 m^2 (1-q)  q^{2m - 1} (n)_{m+1} (n)_{m-1} } {\aut(T_m)^2} \nonumber 
        + O\!\inparen{  \frac{ m^4 n^{2m- 1} q^{2m - 2} }{\aut(T_m)^2}   } .
    \end{align} 
\end{claim}

\begin{proof}[Proof of Claim \ref{claim:Var_G_N_m}]
    Write 
    \begin{equation*}
    \Var G_{N,m}(A) 
    = \sum_{G, G' \cong G_{N,m}} \Cov\insquare{\prod_{\{i,j\} \in G} A_{ij}, \prod_{\{i,j\} \in G'} A_{ij}} 
    = \sum_{\ell=1}^m \sum_{\substack{G, G' \cong G_{N,m} : \\ |G \cap G'| = \ell}} q^{2m-\ell} (1 - q^\ell) ,
    \label{eq:generic-variance-decomposition}
    \end{equation*}
    where the sums range over pairs of labeled copies of $G_{N,m}$ in $K_n$. 
    \begin{itemize}
    \item
    The leading order term corresponds to the pairs $(G,G')$ with exactly one overlapping edge, i.e., $\ell = 1$, with contribution 
    \begin{equation}%\label{eq:Var_G_N_m_leadingOrder}
        \frac{(n)_N}{\aut(G_{N,m})} 2 m^2 \frac{(n-N)_{N-2}}{\aut (G_{N,m})} q^{2m-1} (1-q)
        = \Theta \left( \frac{n^{2N-2} m^2 q^{2m-1}}{\aut(G_{N,m})^2} \right) ,
        \label{eq:variance-G-contribution-term-1}
    \end{equation}
    because there are $\frac{(n)_N}{\aut(G_{N,m})}$ ways to label the vertices of $G$, there are $2m^2$ ways that $G$ and $G'$ overlap at one edge, and there are $\frac{(n-N)_{N-2}}{\aut (G_{N,m})}$ ways to label the remaining vertices of $G'$. 

\item
The next largest contribution is from the pairs $(G,G')$ with an overlapping wedge (two overlapping adjacent edges), which are part of the $\ell=2$ inner sum. The contribution of such terms is 
    \begin{equation}
        \frac{(n)_N}{\aut(G_{N,m})} \cdot 2 \cdot \gamma(G_{N,m})^2 \frac{(n-N)_{N-3}}{\aut(G_{N,m})} q^{2m-2} (1-q^2) = \Theta \left( \frac{n^{2N-3} \gamma(G_{N,m})^2 q^{2m-2}}{\aut(G_{N,m})^2} \right) ,
        \label{eq:variance-G-contribution-term-2}
    \end{equation}
    where the $2 \gamma(T_m)^2$ factor arises as the number of ways to superimpose the pair along two adjacent edges. This leads to the subleading order term in \eqref{eq:Var_Tm_LeadAndSublead}. 

    \item
    The other terms with $\ell = 2$ correspond to pairs $(G,G')$ containing two non-adjacent overlapping edges. They contribute at most
    \begin{equation}
        \frac{(n)_N}{\aut(G_{N,m})} \cdot 8 \binom{m}{2}^2 \frac{(n-N)_{N-4}}{\aut(G_{N,m})} q^{2m-2} (1-q^2) = O \left( \frac{n^{2N-4} m^4 q^{2m-2}}{\aut(G_{N,m})^2} \right) ,
        \label{eq:variance-G-contribution-term-3}
    \end{equation}
    where the counting is done similarly.

    \item
    The terms corresponding to pairs $(G,G')$ with $\ell = 3,\dots, m$ overlapping edges are similarly upper bounded by
    % \begin{align}
    %     &\sum_{\ell = 3}^{m} O\!\inparen{  \binom{n}{2N - \ell - 1} \frac{(2N - \ell - 1)_{N}}{\aut(G_{N,m})} \ell! \binom{m}{\ell}^2 \frac{(N-\ell - 1)!}{\aut(G_{N,m})} q^{2m - \ell} (1-q^\ell)   } \nonumber \\
    %     &\qquad = \frac{n^{2N - 1} q^{2m}}{\aut(G_{N,m})^2} O\!\inparen{ \sum_{\ell = 2}^{m} \inparen{\frac{m^2}{nq}}^\ell  } = O\!\inparen{ \frac{n^{2N - 3} m^4 q^{2m - 2}}{\aut(G_{N,m})^2}}. 
    %     \label{eq:Var_G_N_m_subLeadingOrder}
    % \end{align}
     \begin{align}
        &\sum_{\ell = 3}^{m} O\!\inparen{ \frac{(n)_N}{\aut(G_{N,m})} \ell! \cdot 2^\ell \binom{m}{\ell}^2 \frac{(n)_{N-\ell}}{\aut(G_{N,m})} q^{2m - \ell} (1-q^\ell)   } \nonumber \\
        &\qquad = O\!\inparen{ \frac{n^{2N} q^{2m}}{\aut(G_{N,m})^2} \sum_{\ell = 3}^{m} \inparen{\frac{m^2}{nq}}^\ell  } = O\!\inparen{ \frac{n^{2N - 3} m^6 q^{2m - 3}}{\aut(G_{N,m})^2}} 
        \label{eq:variance-G-contribution-term-4}%\label{eq:Var_G_N_m_subLeadingOrder}
    \end{align}
    provided that $\frac{m^2}{nq} \le \frac 12$, where in particular the factor $(n)_{N-\ell}$ is due to that $G'$ has at most $N-\ell$ vertices that do not overlap with $G$. 
This last counting is tight if $G \cap G'$ is a cycle of length $\ell$; however, if $G_{N,m} = T_m$, a tree does not contain any cycle, so $G'$ has at most $N-\ell-1$ vertices that do not overlap with $G$. Therefore, for $G_{N,m} = T_m$ and $N = m+1$, the above bound can be improved to 
\begin{equation}
O\!\inparen{ \frac{n^{2m - 2} m^6 q^{2m - 3}}{\aut(T_m)^2}} .
        \label{eq:variance-G-contribution-term-5}
\end{equation}
    \end{itemize}
    % This establishes the first statement. The second statement follows by setting $N = m+1$. 
Note that $m - 1 \le \gamma(G_{N,m}) \le m^2$. 
To finish the proof for a general $G_{N,m}$, it remains to combine \eqref{eq:variance-G-contribution-term-1}, \eqref{eq:variance-G-contribution-term-2}, \eqref{eq:variance-G-contribution-term-3}, and \eqref{eq:variance-G-contribution-term-4}; for a tree $G_{N,m} = T_m$ and $N = m+1$, it suffices to combine \eqref{eq:variance-G-contribution-term-1}, \eqref{eq:variance-G-contribution-term-2}, \eqref{eq:variance-G-contribution-term-3}, and \eqref{eq:variance-G-contribution-term-5}.
\end{proof}

\begin{proof}[Proof of Lemma \ref{lemma:VarProjK_2_Perp_simpleTrees}]
    By the triangle inequality, 
    \begin{align}\label{eq:nonMatching_VarProjK_2_Perp_simpleTrees_triangleInequality}
        \sqrt{\Var \Proj^\perp_{\signedKtwo}\!\inparen{  \simpleTrees  } } \leq \sum_{m = 1}^{2 \log n} \sum_{T_m} \abs{ \phi(H(T_m)) } m! \frac{\lambda^m}{p^m} \sqrt{\Var r(T_m, A) }.
    \end{align}
    We also compute
    \begin{align*}
        \EE\insquare{ \overline{T}_m(A) \signedKtwo(A) } = \frac{(n)_{m+1}}{\aut(T_m)} m p^m (1-p).
    \end{align*}
    In what follows, $C > 0$ denotes a constant independent of $n$ that may differ from line to line. Using Claim \ref{claim:Var_G_N_m} for trees together with \eqref{eq:def-alpha-r-Tm-A}, we have for fixed $1 \leq m \leq 2 \log n$ and $T_m$ that
    \begin{align*}
        &\Var r(T_m, A) = \Var\insquare{  \overline{T}_m(A) } - \frac{\EE\insquare{ \overline{T}_m(A) \signedKtwo(A) }^2}{\Var \signedKtwo(A)}\\
        &\leq 2 m^2 p^{2m - 1} (1-p) \frac{(n)_{m+1}}{\aut(T_m)^2} \insquare{ (n)_{m-1} - \frac{(n)_{m+1}}{n(n-1)}  } + C  \frac{n^{2m- 1} p^{2m - 2} m^4 }{\aut(T_m)^2} \leq C \frac{n^{2m - 1} p^{2m - 2} m^4 }{\aut(T_m)^2}
    \end{align*}
    where we have used the fact that
    \begin{align}\label{eq:nonMatching_Var_rTMA_LeadCancellation_nIdentity}
        (n)_{m+1}\insquare{(n)_{m-1} - \frac{(n)_{m+1}}{n(n-1)} } = (n)_{m+1} (n-2)_{m-3} \insquare{ -2n + 2mn - m^2 + m }  \leq 3 m n^{2m - 1}.
    \end{align}
    Therefore, from \eqref{eq:nonMatching_VarProjK_2_Perp_simpleTrees_triangleInequality},
    \begin{align}
        &\sqrt{\Var \Proj^\perp_{\signedKtwo}\!\inparen{  \simpleTrees  } } \leq C\sum_{m = 1}^{2 \log n} \sum_{T_m} \abs{ \phi(H(T_m)) } m! \frac{\lambda^m}{p^m} \sqrt{   \frac{n^{2m- 1} p^{2m - 2} m^4 }{\aut(T_m)^2}   }  \nonumber \\
        &= \frac{C}{p} \sum_{m = 1}^{2 \log n} \sum_{T_m } \abs{ \phi(H(T_m)) } m! \lambda^m m^2 \frac{n^{m - \frac{1}{2}}}{\aut(T_m)}  \nonumber =  \frac{C}{p} \sum_{m = 1}^{2 \log n}  \sum_{\substack{e_1,\dots,e_m \\ \text{simple tree}}}  \abs{ \phi(H(e_1,\dots,e_m)) } \lambda^m m^2 \frac{n^{m - \frac{1}{2}}}{(n)_{m+1}}  \nonumber \\
        &\leq \frac{C}{p n^{3/2}} \exp\inparen{ \frac{4 (\log n)^2}{n} + O\!\inparen{ \frac{(\log n)^3}{n^2}  } } \underbrace{\sum_{m = 1}^{2 \log n}  \sum_{\substack{e_1,\dots,e_m \\ \text{simple tree}}}  \abs{ \phi(H(e_1,\dots,e_m)) } \lambda^m m^2}_{ \leq Cn} \label{eq:nonMatching_VarProjK_2_Perp_simpleTrees_finalBound}
    \end{align}
    where we have used that $n^{m+1} = (n)_{m+1} \exp\inparen{ \frac{m(m+1)}{2n} + O\!\inparen{ \frac{m^3}{n^2}   }  }$, and where the sum in the last step is bounded by $C n$ by a straightforward modification of the proof of \eqref{eq:logZ_Kn_truncate_clogn} as follows. Indeed, in \eqref{eq:logZ_Kn_truncate_clogn_fixed_m_bound}, we have an extra $1/m^2$ factor which handles the additional $m^2$ factor in \eqref{eq:nonMatching_VarProjK_2_Perp_simpleTrees_finalBound}. Finally, in \eqref{eq:logZ_Kn_truncate_clogn_fixed_m_bound_sum_over_m} we sum over $m \geq 1$ instead of $m \geq 2 \log n$. This finishes the proof.
\end{proof}

\subsection{Mean part}\label{sec:nonMatchingMeanPart}

In this section, we establish Proposition \ref{prop:nonMatching_oneRepTrees_mainResult}. 
% Observe that $\oneRepTrees$ does not have a zero mean. Write 
% \begin{equation*}
%     \oneRepTrees = \EE \insquare{\oneRepTrees } + \overline{\oneRepTrees},
% \end{equation*}
% where $\overline{\oneRepTrees} := \oneRepTrees - \EE\insquare{ \oneRepTrees }$. 
The following lemmas show that $\EE \oneRepTrees$ carries the deterministic part of the asymptotic log-likelihood distribution in \eqref{eq:nonMatching_LLdistributionUnderNull}. In particular, these results make precise the deterministic statement implied in the heuristic \eqref{eq:nonMatching_LLfirstFewTerms_GaussianForm}.
\begin{lemma}
    \label{lemma:E_oneRepTrees_mainResult}
    Let $A \sim \cQ$. Then
   \begin{align}
        \EE\insquare{\oneRepTrees} = - \frac{1-p}{p} \inparen{\frac{\EE\abs{M}}{n}}^2 + O\!\inparen{\frac{(\log n)^2}{n p}}.
        \label{eq:EoneRepTrees_is_meanSquare}
    \end{align}
\end{lemma}
% \begin{lemma}
%     \label{lemma:E_oneRepTrees_limit}
%     Let $c$ be defined as in Conjecture \ref{thm:plantedMatchings_mean_is_minus_half_variance}, then
%     \begin{align}
%         \label{eq:E_oneRepTrees_limit}
%         \lim_{n \rightarrow \infty} \EE_{A \sim \cQ} \insquare{\oneRepTrees }  = - \frac{c^2}{4}\frac{1-p}{p}.
%     \end{align}
% \end{lemma}

% Since $\EE [\overline{\oneRepTrees}] = 0$, the following lemma establishes that $\overline{\oneRepTrees} = O_{\PP}(p^{-3/2}n^{-2})$.
\begin{lemma}
    \label{lemma:Var_oneRepTrees_small}
    Let $A \sim \cQ$. Then
    \begin{align}
        \Var (\oneRepTrees) = O\left( \frac{1}{n^2 p^3} \right) .
    \end{align}
\end{lemma}

The proof of Proposition \ref{prop:nonMatching_oneRepTrees_mainResult} is immediate from the above two lemmas. 
In the sequel, we first prove Lemma \ref{lemma:Var_oneRepTrees_small} and then Lemma~\ref{lemma:E_oneRepTrees_mainResult}.

Let $\TmRep$ generically denote an unlabeled connected multi-tree with $m+1$ vertices and $m+1$ edges (so exactly one repeated edge). 
Define the (non-injective) map $s$ that maps a $\TmRep$ to the corresponding simple graph $T_m = s(\TmRep)$ by removing the repeated edge.  
Let $\psi(\TmRep)$ be the number of ways to place the repeated edge in a labeled version of $T_m$ so that the resulting graph is $\TmRep$. We suppress any notational reference to the map $s$ when clear from the context.
% Let $\mathscr{T}_m^{\text{rep}}$ be the set of unlabeled multi-trees with $m+1$ vertices and $m+1$ edges (so exactly one repeated edge). For $\TmRep \in \mathscr{T}_m^{\text{rep}}$ let $\psi(\TmRep)$ be the ``number of ways to place the repeated edge in a labeled version of $\TmRep$ up to isomorphism''. \notetim{Want to get rid of this $\psi$.... but without expectation don't know a better way to write it.} Define the (non-injective) map $\TmRep \mapsto T_m \in \mathscr{T}_m$ which removes the repeated edge; we suppress any notational reference to this map when clear from the context.
\begin{proof}[Proof of Lemma \ref{lemma:Var_oneRepTrees_small}]
    Recall \eqref{eq:decompose_loglikelihood}. 
    % and $\overline{T}_m(A) := T_m(A) - \EE_{A \sim \cQ}\insquare{  T_m(A)  }$. 
    The number of copies of $T^{\text{rep}}_m$ in $A$ is $T_m(A) \psi(\TmRep)$, and there are $(m+1)!/2$ ways to associate the edges of $T^{\text{rep}}_m$ with $e_1, \dots, e_{m+1}$. Therefore, we have the identity
    \begin{equation}
        \oneRepTrees = \sum_{m=1}^{2 \log n - 1} \sum_{T^{\text{rep}}_m} \frac{(m+1)!}{2} \phi\inparen{ H( \TmRep ) } \psi(\TmRep) \lambda^{m+1} \left( \frac{T_m(A)}{p^{m+1}} - 1 \right).
        \label{eq:one-rep-tree-identity}
    \end{equation}
    By Claim \ref{claim:Var_G_N_m} and 
    % mean value theorem 
    $\sqrt{1 + z} = 1 + O(z)$,
    \begin{align*}
        \sqrt{ \Var T_m(A) } \leq \frac{\sqrt{2(1-p)} m n^m p^{m - 1/2} }{\aut(T_m)} \inparen{1 + O\!\inparen{ \frac{m^2}{np} } }.
    \end{align*}
    By the triangle inequality and the above estimate,
    \begin{align*}
        \sqrt{\Var (\oneRepTrees)} &\leq \sum_{m=1}^{2 \log n - 1} \sum_{T^{\text{rep}}_m} \frac{(m+1)!}{2} \abs{ \phi\inparen{ H( \TmRep ) } } \psi(\TmRep) \inparen{\frac{\lambda}{p}}^{m+1} \sqrt{ \Var T_m(A)} \\
        &\leq \frac{C}{p^{3/2}} \sum_{m=1}^{2 \log n - 1} \sum_{T^{\text{rep}}_m} m \frac{(m+1)!}{2} \abs{ \phi\inparen{ H( \TmRep ) } } \frac{\psi(\TmRep)}{\aut(T_m)} \lambda^{m+1} n^{m} \\
        % &= \frac{C}{p^{3/2}} \sum_{m=1}^{2 \log n - 1} \sum_{T^{\text{rep}}_m} m \frac{(m+1)!}{2} \abs{ \phi\inparen{ H( \TmRep ) } } \frac{T_m(K_n) \psi(\TmRep)}{(n)_{m+1}} \lambda^{m+1} n^{m} \\
        &= \frac{C}{p^{3/2}} \sum_{m=1}^{2 \log n - 1} \sum_{\substack{e_1,\dots,e_{m+1} \\ \text{one rep.~edge tree}}} m \abs{ \phi\inparen{ H( e_1,\dots,e_{m+1} ) } } \frac{\lambda^{m+1} n^{m}}{(n)_{m+1}} ,
        \end{align*}
since $\aut(T_m) = \frac{(n)_{m+1}}{T_m(K_n)}$, the number of copies of $T^{\text{rep}}_m$ in $K_n$ is $T_m(K_n) \psi(\TmRep)$, and there are $(m+1)!/2$ ways to associate the edges of $T^{\text{rep}}_m$ with $e_1, \dots, e_{m+1}$.
Furthermore, changing $m$ to $m-1$ and using $n^{m} = (n)_{m} \exp\inparen{ \frac{m(m-1)}{2n} + O\!\inparen{ \frac{m^3}{n^2}   }  }$, we obtain
    %     \begin{align*}
    %     \sqrt{\Var \overline{\oneRepTrees}}  
    %     &\leq \frac{C}{p^{3/2}} \sum_{m=2}^{2 \log n} \sum_{\substack{e_1,\dots,e_m \\ \text{one rep.~edge tree}}} m \lambda^m \abs{ \phi(H(e_1,\dots,e_m)) } \frac 1n \\
    %     &\le \frac{C}{n p^{3/2} } 
    %     % \exp\inparen{ \frac{2 (\log n)^2}{n} + O\!\inparen{ \frac{(\log n)^3}{n^2}  } } 
    %     \underbrace{ \sum_{m=2}^{2 \log n} \sum_{\substack{e_1,\dots,e_m \\ \text{one rep.~edge tree}}} m \lambda^m \abs{ \phi(H(e_1,\dots,e_m)) } }_{\leq C},
    % \end{align*}
     \begin{equation*}
        \sqrt{\Var (\oneRepTrees)}  
        \leq \frac{C}{n p^{3/2}} \sum_{m=2}^{2 \log n} \sum_{\substack{e_1,\dots,e_m \\ \text{one rep.~edge tree}}} m \lambda^m \abs{ \phi(H(e_1,\dots,e_m)) } 
        = O \left( \frac{1}{n p^{3/2} } \right) 
    \end{equation*}
    by Proposition \ref{prop:MDMean_cycles_O(1)}. 
\end{proof}

The main challenge in the proof of Lemma \ref{lemma:E_oneRepTrees_mainResult} is to show that the series $\EE \oneRepTrees$, as given by taking the expectation of $\oneRepTrees$ as defined in \eqref{eq:decompose_loglikelihood}, is related to the square of another series \eqref{prop:GibbsMatching_mean_var_from_CE} for $\EE\abs{M}$. 
A key component of the following proof is the combinatorial identity established in Lemma~\ref{lem:ursell-function-tmrep-tm-identity}.

\begin{proof}[Proof of Lemma \ref{lemma:E_oneRepTrees_mainResult}]
    By \eqref{eq:one-rep-tree-identity}, we have
    \begin{align*}
        \EE \insquare{  \oneRepTrees  } = \frac{1-p}{p} \sum_{m = 1}^{2 \log n - 1} \sum_{\TmRep} \frac{(m+1)!}{2} \phi\inparen{H(\TmRep)} \lambda^{m+1} \frac{(n)_{m+1}}{\aut(\TmRep)} .
    \end{align*}
    For $m \leq 2 \log n -1$, we have the approximation
    \begin{align}
        \label{eq:fallingfac_to_power_approx}
        (n)_{m+1} = n^{m+1} \exp \inparen{  - \frac{m(m+1)}{2n} + O\!\inparen{\frac{m^3}{n^2}}   } = n^{m+1} \insquare{1 + O\!\inparen{\frac{(\log n)^2}{n}}}.
    \end{align}
    Using this, we have 
    \begin{equation}\label{eq:nonMatching_EoneRepTree_LHS_W_Expansion}
        \EE \insquare{  \oneRepTrees  } = W + O\!\inparen{\frac{(\log n)^2}{n}}W 
    \end{equation}
    where 
    \begin{equation*}
        W := \frac{1-p}{p} \sum_{m = 1}^{2 \log n - 1} (n \lambda)^{m+1} \sum_{\TmRep} \frac{(m+1)!}{2 \aut (\TmRep)} \phi\inparen{H(\TmRep)}.
    \end{equation*}

    On the other hand, from Proposition \ref{prop:GibbsMatching_mean_var_from_CE}, and with the approximation \eqref{eq:fallingfac_to_power_approx}, we have
    \begin{align}
        - \frac{1-p}{p} \inparen{\frac{\EE\abs{M}}{n}}^2 &= -\frac{1-p}{p n^2 }\inparen{  \sum_{m=1}^{2 \log n} \sum_{T_m} m! m \phi(H(T_m)) \lambda^m \frac{(n)_{m+1}}{\aut(T_m)}   }^2 + O\inparen{\frac{1}{n}} \nonumber \\
        &= X + O\!\inparen{\frac{(\log n)^2}{n}} X +  O\inparen{\frac{1}{n}} 
        \label{eq:nonMatching_EEoneRepTree_RHS_expand_in_X}
    \end{align}
    where in the first line we used that $\EE\abs{M}/n = O(1)$ as deduced from \eqref{eq:MDMean_convergence_to_c}, and where
    \begin{align*}
        X &:= - \frac{1-p}{p} \inparen{  \sum_{m=1}^{2 \log n} \sum_{T_m} m! m \phi(H(T_m)) \lambda^m \frac{n^m }{\aut(T_m)}   }^2.
    \end{align*}
For a graph $H$ on $m$ vertices, we denote the unnormalized Ursell function by
$$
\widetilde{\phi}(H) := m! \cdot \phi(H).
$$
    Expand the square in $X$ and write 
    \begin{equation*}
        X = X_{\leq 2 \log n} + X_{> 2 \log n} 
    \end{equation*}
    where 
    \begin{align*}
        X_{\leq 2 \log n} &:= - \frac{1-p}{p} \sum_{m = 1}^{2 \log n - 1} (n\lambda)^{m+1}  \sum_{\ell = 1}^{m} \sum_{(T_\ell, T_{m+1 -\ell})} \frac{\ell \widetilde{\phi}(H(T_\ell))}{\aut (T_\ell)}  \frac{(m+1-\ell) \widetilde{\phi}(H(T_{m+1-\ell}))}{\aut(T_{m+1-\ell})},\\
        X_{> 2 \log n} &:= - \frac{1-p}{p} \sum_{m = 2 \log n}^{4 \log n - 1} (n\lambda)^{m+1} \sum_{\ell = m+1 - 2\log n}^{2 \log n} \sum_{(T_\ell, T_{m+1 -\ell})} \frac{  \ell \widetilde{\phi}(H(T_\ell))}{\aut (T_\ell)}  \frac{(m+1-\ell) \widetilde{\phi}(H(T_{m+1-\ell}))}{\aut(T_{m+1-\ell})}.
    \end{align*} 
    In words, $X_{\leq 2 \log n}$ and $X_{> 2 \log n}$ sum over the pairs of simple trees $(T_\ell, T_{m+1 - \ell})$ which have respectively $\leq 2 \log n$ and $> 2 \log n$ total number of edges.

    We state the following claims.
    \begin{alignat*}{2}
        &\text{(i)}\,\, W = O\!\inparen{\frac{1}{p}}, &\qquad\qquad \text{(ii)}\,\, &X = O\!\inparen{\frac{1}{p}}, \\
        & \text{(iii)}\,\, X_{> 2 \log n} = O\!\inparen{ \frac{\log n}{n^2 p}  }, &\qquad\qquad  \text{(iv)}\,\, &W = X_{\leq 2 \log n}.
    \end{alignat*}
% 
    % \begin{enumerate}[(i)]
    %     \item $O\!\inparen{\frac{(\log n)^2}{n}} W = O\!\inparen{\frac{(\log n)^2}{n}}$.
    %     \item $O\!\inparen{\frac{(\log n)^2}{n}} X = O\!\inparen{\frac{(\log n)^2}{n}}$.
    %     \item $X_{> 2 \log n} = O\!\inparen{ \frac{\log n}{n^2}  }$.
    %     \item $W = X_{\leq 2 \log n}$. 
    % \end{enumerate}
    Using the above claims in \eqref{eq:nonMatching_EoneRepTree_LHS_W_Expansion} and \eqref{eq:nonMatching_EEoneRepTree_RHS_expand_in_X} yields the desired \eqref{eq:EoneRepTrees_is_meanSquare}. It remains to prove the claims.

\smallskip
    \noindent
    \underline{\textbf{Proof of Claim (i):}} 
From the definition of $W$, we deduce that
        \begin{equation*}
        |W| \le \frac{C}{p} \sum_{m = 1}^{2 \log n - 1} \sum_{\substack{e_1,\dots,e_{m+1} \\ \text{one rep.~edge tree}}} \lambda^{m+1}  |\phi\inparen{H(e_1,\dots,e_{m+1})}| ,
    \end{equation*}
    so the claim follows from Proposition \ref{prop:MDMean_cycles_O(1)}.

\smallskip
    \noindent
    \underline{\textbf{Proof of Claim (ii):}} The claim holds because the expression within the square in $X$ can be shown to be bounded in absolute value by 
    $$
    C \sum_{m=1}^{2 \log n} \sum_{\substack{e_1,\dots,e_m \\ \text{simple tree}}} m \lambda^m |\phi(H(e_1,\dots,e_m))| \le C n
    $$
    by a straightforward modification of the proof of Theorem \ref{thm:truncate_CE_clogn}. In particular, in \eqref{eq:logZ_Kn_truncate_clogn_fixed_m_bound_sum_over_m}, we sum over $m \geq 1$ instead of $m \geq 2 \log n$.

\smallskip
    \noindent
    \underline{\textbf{Proof of Claim (iii):}} Rewrite $X_{> 2 \log n}$ as
    \begin{align*}
        X_{> 2 \log n} = -\frac{1-p}{pn^2} \sum_{\ell= 1}^{2 \log n} \sum_{\ell' = 2 \log n + 1 - \ell}^{2 \log n} \sum_{(T_\ell, T_{\ell'})} \ell! \ell \phi(H(T_{\ell})) \lambda^{\ell} \frac{n^{\ell+1}}{\aut (T_{\ell})} \ell'! \ell' \phi(H(T_{\ell'})) \lambda^{\ell'} \frac{n^{\ell'+1}}{\aut (T_{\ell'})}.
    \end{align*}
    By the triangle inequality, and using \eqref{eq:fallingfac_to_power_approx}, we have 
    \begin{align*}
        \abs{X_{> 2 \log n}} \leq \frac{C}{n^2 p} \sum_{\ell= 1}^{2 \log n}  \sum_{T_\ell}& \ell! \ell \abs{ \phi(H(T_{\ell})) } \lambda^{\ell} \frac{(n)_{\ell+1}}{\aut (T_{\ell})} \underbrace{ \sum_{\ell' = 2 \log n + 1 - \ell}^{2 \log n} \ell'! \ell' \abs{ \phi(H(T_{\ell'})) }\lambda^{\ell'} \frac{(n)_{\ell'+1}}{\aut (T_{\ell'})}}_{Y(\ell)} ,
    \end{align*}
    where, rewriting in terms of polymers, 
    \begin{align*}
        Y(\ell) = \sum_{\ell' = 2 \log n + 1 - \ell}^{2 \log n} \sum_{\substack{e_1,\dots,e_{\ell'} \\ \text{simple tree}} } \ell' \abs{\phi(H(e_1,\dots,e_{\ell'}))} \lambda^{\ell'}.
    \end{align*}
    Let $\Delta := 2n - 3$. By similar arguments as in the proof of Theorem \ref{thm:truncate_CE_clogn} using the Penrose tree-graph bound, we have analogously to \eqref{eq:logZ_Kn_truncate_clogn_fixed_m_bound_sum_over_m}, for fixed $\ell'$,
    \begin{align}
        \label{eq:meanExpression_simpleTreeBound}
        \sum_{\substack{e_1,\dots,e_{\ell'} \\ \text{simple tree}} } \ell' \abs{\phi(H(e_1,\dots,e_{\ell'}))} \lambda^{\ell'} \leq \frac{n}{2} (e\lambda \Delta)^{\ell'}.
    \end{align}
    By Assumption \ref{asmpt:nonMatchingEdges}, $e\lambda \Delta \leq \frac{1}{e}$. It follows that
    \begin{align*}
        Y(\ell) \leq \frac{n}{2}\sum_{\ell' = 2 \log n + 1 - \ell}^{2 \log n} (e \lambda \Delta)^{\ell'} \leq n (e\lambda \Delta)^{2 \log n + 1 - \ell} \leq \frac{1}{n}  (e\lambda \Delta)^{- \ell}.
    \end{align*}
    Using this upper bound for $Y(\ell)$ together with \eqref{eq:meanExpression_simpleTreeBound} for the sum in $\ell$, we have 
    \begin{align*}
        &\sum_{\ell= 1}^{2 \log n}  \sum_{T_\ell} \ell! \ell \abs{ \phi(H(T_{\ell})) } \lambda^{\ell} \frac{(n)_{\ell+1}}{\aut (T_{\ell})} Y(\ell) 
        % &\abs{X_{> 2 \log n}}
        \leq \frac{1}{n} \sum_{\ell= 1}^{2 \log n} (e\lambda \Delta)^{- \ell}  \sum_{T_\ell} \ell! \ell \abs{ \phi(H(T_{\ell})) } \lambda^{\ell} \frac{(n)_{\ell+1}}{\aut (T_{\ell})} \\
        &\quad\leq \frac{1}{n} \sum_{\ell= 1}^{2 \log n} (e\lambda \Delta)^{- \ell}  \sum_{\substack{e_1,\dots,e_{\ell} \\ \text{simple tree}} } \ell \abs{\phi(H(e_1,\dots,e_{\ell}))} \lambda^{\ell} 
        \leq \frac{1}{n} \sum_{\ell= 1}^{2 \log n} (e\lambda \Delta)^{- \ell}  \frac{n}{2} (e \lambda \Delta)^\ell \leq \log n.
    \end{align*}
    This leads to the claimed bound on $\abs{X_{> 2 \log n}}$.
    % \begin{align*}
    %     \abs{X_{> 2 \log n}} &\leq C(b,p) \exp\inparen{ O\!\inparen{ \frac{(\log n)^2}{n} } } \frac{ \log n }{n^2 p} .
    % \end{align*}
    % Previously tried to squeeze all in one equation. Uses less vertical space but very ugly I think because of the exp(...)
    % \begin{align*}
    %     \abs{X_{> 2 \log n}} &\leq \frac{C(b,p)}{n^3} \exp\inparen{ O\!\inparen{ \frac{(\log n)^2}{n} } } \sum_{\ell= 1}^{2 \log n} (e\lambda \Delta)^{- \ell}  \sum_{T_\ell} \ell! \ell \abs{ \phi(H(T_{\ell})) } \lambda^{\ell} \frac{(n)_{\ell+1}}{\aut (T_{\ell})} \\
    %     &= \frac{C(b,p)}{n^3} \exp\inparen{ O\!\inparen{ \frac{(\log n)^2}{n} } } \sum_{\ell= 1}^{2 \log n} (e\lambda \Delta)^{- \ell}  \sum_{\substack{e_1,\dots,e_{\ell} \\ \text{simple tree}} } \ell \abs{\phi(H(e_1,\dots,e_{\ell}))} \lambda^{\ell} \\
    %     &\leq \frac{C(b,p)}{n^3} \exp\inparen{ O\!\inparen{ \frac{(\log n)^2}{n} } } \sum_{\ell= 1}^{2 \log n} (e\lambda \Delta)^{- \ell}  \frac{n}{2} (e \lambda \Delta)^\ell \\
    %     &= C(b,p) \frac{ \log n }{n^2} \exp\inparen{ O\!\inparen{ \frac{(\log n)^2}{n} } }.
    % \end{align*}

\smallskip
\noindent
\underline{\textbf{Proof of Claim (iv):}} Comparing the expressions for $X_{\leq 2 \log n}$ and $W$, we see that it is equivalent to show, for every $1 \leq m \leq 2\log n - 1$, that \eqref{eq:convolutionOfUrsell_overTrees} holds. Hence, the proof is complete.
\end{proof}

\subsection{Dropping cycles and \texorpdfstring{$\geq 2$}{>= 2} repeated edge subgraphs}\label{sec:nonMatchingRemainderLessThanTwoLogn}
% \begin{proposition} 
%     \label{prop:remainder_oPP(1)} Suppose $b > \log 9$. Then with $\mathsf{remainder}$ as defined in \eqref{eq:decompose_loglikelihood}, $\mathsf{remainder} = o_{\PP}(1)$.
% \end{proposition}
In this section we establish Proposition \ref{prop:nonMatching_remainderLessThanTwoLogn_mainResult} which  will follow from a series of lemmas that bound the sub-sums of $\mathsf{remainder}$ defined by:
\begin{align}\label{eq:nonMatching_remainderLessThanTwoLogn_decomposition}
    \mathsf{remainder}_{\leq 2 \log n} = \mathsf{simpleCyclic} + \mathsf{oneRepCyclic} + \mathsf{moreThanTwoRep},
\end{align}
where
\begin{align*}
    \mathsf{simpleCyclic} &:= \sum_{m = 3}^{2 \log n} \sum_{ \substack{e_1,\dots,e_m \\ \text{$e_i$'s distinct} \\ \text{contains cycle}  } } \phi(H(e_1,\dots,e_m)) \lambda^m \insquare{ \prod_{j=1}^{m} \frac{A_{e_j}}{p} - 1  }, \\
    \mathsf{oneRepCyclic} &:= \sum_{m = 4}^{2 \log n} \sum_{ \substack{e_1,\dots,e_m \\ \text{only one rep.~edge,} \\ \text{contains cycle}  } } \phi(H(e_1,\dots,e_m)) \lambda^m \insquare{ \prod_{j=1}^{m} \frac{A_{e_j}}{p} - 1  }, \\
    \mathsf{moreThanTwoRep} &:= \sum_{m = 3}^{2 \log n} \sum_{\substack{e_1,\dots,e_m \\ \text{at least two rep.~edges}} } \phi(H(e_1,\dots,e_m)) \lambda^m \insquare{ \prod_{j=1}^{m} \frac{A_{e_j}}{p} - 1  }.
\end{align*}
In what follows, assume that Assumption \ref{asmpt:nonMatchingEdges} is in force.

\begin{lemma}\label{lemma:nonMatching_simpleCyclic_small}
    It holds that $\mathsf{simpleCyclic} = O_{\PP}\!\inparen{ \frac{1}{n \sqrt{p}}  }$.
\end{lemma}

\begin{proof}
    As in Claim \ref{claim:Var_G_N_m}, let $G_{N,m}$ generically denote any connected unlabeled simple graph on $N$ vertices and $m$ vertices that contains a cycle. Define $\overline{G}_{N,m}(A) := G_{N,m}(A) - \EE_{A \sim \cQ}  G_{N,m}(A)$. We have the identity
    \begin{equation*}
        \mathsf{simpleCyclic} = \sum_{N = 3}^{2 \log n} \sum_{m = N}^{ \binom{N}{2} \wedge 2 \log n } \sum_{G_{N,m}} m! \phi(H(G_{N,m})) \frac{\lambda^m}{p^m} \overline{G}_{N,m}(A).
    \end{equation*}
    Clearly $\EE \insquare{ \mathsf{simpleCyclic} } = 0$. It suffices to show $\Var \insquare{ \mathsf{simpleCyclic} }$ is vanishing---indeed we will show $\Var \insquare{ \mathsf{simpleCyclic} } = O( \frac{1}{p n^2} )$.

    By Claim \ref{claim:Var_G_N_m} and 
    % mean value theorem 
    $\sqrt{1 + z} = 1 + O(z)$, for some constant $C > 0$ that may differ from line to line, 
    \begin{equation*}
        \sqrt{ \Var \overline{G}_{N,m}(A) } \leq C \frac{ m n^{N-1} p^{m - 1/2} }{\aut(G_{N,m})} \inparen{1 + \frac{m^4}{np^2} } .
    \end{equation*}
    By the triangle inequality and the above estimate, 
    \begin{align*}
        \sqrt{ \Var \insquare{ \mathsf{simpleCyclic} } } &\leq \sum_{N = 3}^{2 \log n} \sum_{m = N}^{ \binom{N}{2} \wedge 2 \log n } \sum_{G_{N,m} } m! \abs{ \phi(H(G_{N,m})) } \frac{\lambda^m}{p^m} \sqrt{\Var \overline{G}_{N,m}(A)} \\
        &\leq C \sqrt{\frac{1}{p}}  \sum_{N = 3}^{2 \log n} \sum_{m = N}^{ \binom{N}{2} \wedge 2 \log n } \sum_{G_{N,m}} m \cdot m! \abs{ \phi(H(G_{N,m})) } \lambda^m \frac{n^{N-1}}{\aut(G_{N,m})} \\
        &= C \sqrt{\frac{1}{p}}  \sum_{N = 3}^{2 \log n} \sum_{m = N}^{ \binom{N}{2} \wedge 2 \log n } \sum_{\substack{e_1,\dots,e_m \text{ cyclic} \\ \text{$N$ vertices, $e_i$'s distinct}}} m \abs{ \phi(H(e_1,\dots,e_m)) } \lambda^m \frac{n^{N-1}}{ \binom{n}{N} N! } \\
        &\leq \frac{C}{n} \sqrt{\frac{1}{p}}  \underbrace{\sum_{N = 3}^{2 \log n}  \sum_{m = N}^{ \binom{N}{2} \wedge 2 \log n } \sum_{\substack{e_1,\dots,e_m \text{ cyclic}  \\ \text{$N$ vertices, $e_i$'s distinct}}} m \abs{ \phi(H(e_1,\dots,e_m)) } \lambda^m}_{\leq C}
    \end{align*}
    where we used $n^N = (n)_{N} \exp\inparen{  \frac{N^2}{2n} + O\!\inparen{ \frac{N^3}{n^2} }  }$, and where the sum is bounded by Proposition~\ref{prop:MDMean_cycles_O(1)}.
\end{proof}

\begin{lemma}\label{lemma:nonMatching_oneRepCyclic_small}
    It holds that $\mathsf{oneRepCyclic} = O_{\PP}\!\inparen{ \frac{1}{n p}}$.
\end{lemma}

\begin{proof}
    Write 
    \begin{align}\label{eq:nonMatching_oneRepCyclic_decomposeIntoRandomDeterministic}
        \mathsf{oneRepCyclic} = \mathsf{oneRepCyclicRandom}  - \mathsf{oneRepCyclicDeterministic},
    \end{align}
    where 
    \begin{align*}
        \mathsf{oneRepCyclicRandom} &= \sum_{m = 4}^{2 \log n} \sum_{ \substack{e_1,\dots,e_m \\ \text{only one rep.~edge,} \\ \text{contains cycle}  } } \phi(H(e_1,\dots,e_m))  \frac{\lambda^m}{p^m} \prod_{j=1}^{m} A_{e_j}  \\
        \mathsf{oneRepCyclicDeterministic} &= \sum_{m = 4}^{2 \log n} \sum_{ \substack{e_1,\dots,e_m \\ \text{only one rep.~edge,} \\ \text{contains cycle}  } } \phi(H(e_1,\dots,e_m))  \lambda^m  
    \end{align*}
    We have
    \begin{align*}
        \abs{ \mathsf{oneRepCyclicRandom} } \leq \sum_{m = 1}^{2 \log n - 3} \sum_{r = 3}^{2 \log n - m} \sum_{\substack{e_1,\dots,e_{m+r} \\ G \supseteq C_r \\ \text{some } e_{j_1} = e_{j_2} }  } \abs{ \phi(H(e_1,\dots,e_{m+r})) } \frac{\lambda^{m+r}}{p^{m+r}} \prod_{j=1}^{m+r} A_{e_j}.
    \end{align*}
    Fix $m$ and $r$. By a similar application of the Penrose tree-graph bound as in equation \eqref{eq:logZ_A_truncate_clogn_apply_penrose}, we have 
    \begin{align*}
         &\sum_{\substack{e_1,\dots,e_{m+r} \\ G \supseteq C_r \\ \text{some } e_i = e_j }  } \abs{ \phi(H(e_1,\dots,e_{m+r})) } \prod_{j=1}^{m+r} A_{e_j} \leq \frac{1}{(m+r)!} \sum_{t \in \cT_{m+r - 1}^{\text{lab}}} \sum_{\substack{e_1,\dots,e_{m+r} \in A \\ G \supseteq C_r \\ \text{some } e_i = e_j }  } \boldsymbol{1}\!\inbraces{  t \in \cT(H)^{\text{lab}}  }. 
    \end{align*}
    Fix $t \in \cT_{m+r-1}^{\text{lab}}$. We next describe an iterative process to construct clusters $\inbraces{e_1,\dots,e_{m+r}}$ with $e_i$'s in $A$, and such that its incompatibility graph $H$ contains $t$ as a spanning tree, and where $G$ contains at least one $r$-cycle, and has at least one repeated polymer.
    % This is always possible if we allow repeated polymers.
    
    \underline{Step 1:} Fix $V' \subseteq V(t) = [m+r]$ with $\abs{V'} = r$. The set $V'$ will be the coordinates in the cluster $\inbraces{e_1, \dots, e_{m+r}}$ which contain a single $r$-cycle. There are at most $\binom{m + r}{r}$ ways to do this. 
    
    \underline{Step 2:} Choose the $r$ distinct polymers in $A$ to make up a single $r$-cycle: there are at most $C_r(A)$ ways to do this, where $C_r(A)$ is the number of labeled $r$-cycles in $A$.

    \underline{Step 3:} Pick an edge $\inbraces{i_*,j_*}$ in $t$ that will correspond to a link between a pair of repeated polymers. Not all edges in $t$ can be chosen, for instance any edge between vertices in $t$ that are chosen to represent the distinct cycle polymers is excluded. Nevertheless, there are at most $m+r - 1$ ways to do this.
    
    \underline{Step 4:} Pick a cycle polymer $\tilde{e}$ to assign to an arbitrary vertex $i_1 \in V'$. (We may take $i_1$ to be the smallest index in $V'$.) There are $r$ choices for $\tilde{e}$ out of the chosen $r$-cycle polymers. 

    Iteratively, suppose coordinates $i_1, \dots, i_j$ have been assigned to polymers $e_{i_1} = \tilde{e}, e_{i_2},\dots, e_{i_j}$. There must exist $i_{j+1} \in [m+r] \setminus \inbraces{i_1,\dots,i_{j}}$ such that $i_{j+1}$ is adjacent to one of $\inbraces{i_1,\dots,i_j}$ in $t$. Without loss of generality suppose $\inbraces{i_j, i_{j+1}} \in t$. 
    \begin{itemize}
        \item If $i_{j+1} \in V'$, then we attempt to assign a cycle polymer to $i_{j+1}$. There are at most two choices for $e_{i_{j+1}}$, which has to be compatible with the assignment of $e_{i_j}$ to $i_j$. If there are no compatible choices for a cycle polymer for $e_{i_{j+1}}$, we terminate the iteration and output an incomplete assignment.
        \item If $i_{j+1} \notin V'$ and $\inbraces{i_j, i_{j+1} } \neq \inbraces{i_*, j_*}$, then we can assign all possible distinct incident edges to $e_{i_j}$ that are in $A$, as well as $e_{i_j}$ itself. There are at most $2(\Delta(A) - 1) + 1 $ such choices for $e_{i_{j+1}}$.
        \item If $i_{j+1} \notin V'$ and $\inbraces{i_j, i_{j+1} } = \inbraces{i_*, j_*}$, then we assign $e_{i_{j}}$ to $e_{i_{j+1}}$.
    \end{itemize}
    For a chosen $r$-cycle, the subset of completed cluster assignments that had utilized all chosen $r$-cycle polymers contain all the desired ordered clusters $\inbraces{e_1,\dots,e_{m+r}}$ satisfying $t \in \cT(H\inparen{e_1,\dots,e_{m+r}})$ and $G\inparen{e_1,\dots,e_{m+r}}$ containing that chosen $r$-cycle. By Cayley's theorem $\abs{\cT_{m+r - 1}^{\text{lab}} } = (m+r)^{m+r - 2}$.

    Note that $\EE\insquare{C_r(A)} = (n)_r p^r / {2r}$. Claim \ref{claim:Var_G_N_m} applied with $G_{r,r}$ provides an upper bound on $\Var\insquare{C_r(A)}$. This leads to 
    \begin{align}
        \PP\insquare{ C_r(A) > 2 \EE\insquare{ C_r(A)}  } \leq \frac{\Var\insquare{C_r(A)}}{\inparen{\EE C_r(A)}^2} \leq \frac{C (\log n)^2}{np}.
    \end{align}
    Therefore with probability at least $1 - O(\frac{1}{n})$, we have $C_r(A) \leq (n)_r p^r/r$. Furthermore, with probability at least $1 - \frac{1}{n}$, $\Delta(A) < 2.02 np$. 
    
    Combining the above, with probability at least $1 - O\!\inparen{\frac{1}{n}}$,
    \begin{align}
         \sum_{\substack{e_1,\dots,e_{m+r} \\ G \supseteq C_r \\ \text{some } e_i = e_j }  } \abs{ \phi(H(e_1,\dots,e_{m+r})) } \prod_{j=1}^{m+r} A_{e_j} &\leq \frac{(m+r)^{m+r-2}}{(m+r)!} \binom{m+r}{r} \frac{(n)_r p^r}{r} (m+r - 1) r (4.04np )^{m-1} 2^{r-1} \nonumber \\
         &\leq \frac{C}{ n p} \frac{1}{m+r} \binom{m+r}{r} (4.04e n p)^{m+r} \label{eq:nonMatching_oneRepCyclicRandom_afterPenroseCayley_beforeFinal}
    \end{align}
    Multiplying by $\frac{\lambda^{m+r}}{p^{m+r}}$ and summing over $m$ and $r$, we have with probability at least $1 - O\!\inparen{\frac{1}{n}}$,
    \begin{align}
        &\abs{ \mathsf{oneRepCyclicRandom} } \leq \frac{C}{n p} \sum_{m=1}^{2 \log n - 3} \sum_{r=3}^{2 \log n - m}  (4.04e\lambda n)^{m+r} \binom{m+r}{r}  \label{eq:nonMatching_oneRepCyclicRandom_afterPenroseCayley} \\
        &\quad= \frac{C}{n p} \sum_{\ell = 4}^{2 \log n}   (4.04e\lambda n)^{\ell} \underbrace{ \sum_{r = 3}^{\ell - 1} \binom{\ell}{r}}_{\leq 2^\ell} 
        \leq \frac{C}{n p} \sum_{\ell \geq 4}   (8.08e\lambda n )^{\ell}  
        \leq \frac{C}{np}, \label{eq:nonMatching_oneRepCyclicRandom_afterPenroseCayley_final}
    \end{align}
    where the final inequality used hypothesis $\abs{8.08e\lambda n} < 1$. This shows $\mathsf{oneRepCyclicRandom} = o_{\PP}(1)$. 

    An almost identical argument will show that $\mathsf{oneRepCyclicDeterministic} = O\!\inparen{\frac{1}{n}}$. We only have to replace every instance of the random $\Delta(A)$ and $C_r(A)$ above with the deterministic $\Delta(K_n) = n-1$ and $C_r(K_n) = (n)_{r}/2r$ respectively.
\end{proof}

The next result shows that the terms in the log-likelihood ratio with more than two repeated edges are small in aggregate.

\begin{lemma}\label{lemma:nonMatching_moreThanTwoRep}
    It holds that $\mathsf{moreThanTwoRep} = O_{\PP}\!\inparen{ \frac{1}{n p^2}}$.
\end{lemma}

\begin{proof}
    Write 
    \begin{align*}
        \mathsf{moreThanTwoRep} = \mathsf{moreThanTwoRepRandom}  - \mathsf{moreThanTwoRepDeterministic},
    \end{align*}
    where 
    \begin{align*}
        \mathsf{moreThanTwoRepRandom} &= \sum_{m = 3}^{2 \log n} \sum_{\substack{e_1,\dots,e_m \\ \text{at least two rep.~edges}} } \phi(H(e_1,\dots,e_m))  \prod_{j=1}^{m} A_{e_j} \frac{\lambda}{p}  \\
        \mathsf{moreThanTwoRepDeterministic} &= \sum_{m = 3}^{2 \log n} \sum_{\substack{e_1,\dots,e_m \\ \text{at least two rep.~edges}} } \phi(H(e_1,\dots,e_m))  \lambda^m  
    \end{align*}
    We first show that with probability at least $1 - O\inparen{\frac{1}{n}}$, $\abs{\mathsf{moreThanTwoRepRandom}} \leq \frac{C}{p^2 n }$. Thus $\mathsf{moreThanTwoRepRandom} = o_{\PP}(1)$. To begin,
    \begin{align}
        \abs{\mathsf{moreThanTwoRepRandom} } &\leq \sum_{m \geq 3} \sum_{\substack{e_1,\dots,e_m \\ \text{at least two rep.~edges}} } \abs{ \phi(H(e_1,\dots,e_m)) }\prod_{j = 1}^{m} A_{e_j} \frac{\lambda}{p} \nonumber \\
        &\leq \sum_{m \geq 3}  \sum_{\substack{e_1,\dots,e_m \\ e_{i_1} = e_{i_2} = e_{i_3}  } } \inparen{\cdots} + \sum_{m \geq 4}  \sum_{\substack{e_1,\dots,e_m \\ e_{i_1} = e_{i_2}, e_{j_1} = e_{j_2}  } } \inparen{\cdots},
        \label{eq:logZ_A_atleastTwoRepEdges_CE_SPLIT}
    \end{align}
    where on the RHS of \eqref{eq:logZ_A_atleastTwoRepEdges_CE_SPLIT}, the constraint in the first sum means there exists distinct indices $\inbraces{i_1,i_2,i_3} \subseteq [m]$ such that $e_{i_1} = e_{i_2} = e_{i_3}$, and the constraint in the second sum means there exists distinct indices $\inbraces{i_1,i_2,j_1, j_2} \subseteq [m]$ such that $e_{i_1} = e_{i_2}$ and $e_{j_1} = e_{j_2}$ (it is possible that $e_{i_1} = e_{i_2} = e_{j_1} = e_{j_2}$).

    Let $\cT_{m-1}^{\text{lab}}$ be the set of labeled trees on vertex set $[m]$ and let $\cT(H)^{\text{lab}}$ be the set of labeled spanning trees of a graph $H$. In what follows, we denote by $H = H(e_1,\dots,e_m)$ the incompatibility graph of cluster $(e_1,\dots,e_m)$, using the abbreviated notation whenever clear from the context. Applying the Penrose tree-graph bound Lemma \ref{lemma:Penrose_tree_bound}, 
    \begin{align}
        &\sum_{m \geq 3}  \sum_{\substack{e_1,\dots,e_m \\ e_{i_1} = e_{i_2} = e_{i_3}  } } \abs{ \phi(H(e_1,\dots,e_m)) }\prod_{j = 1}^{m} A_{e_j} \frac{\lambda}{p} \nonumber\\
        &\qquad = \sum_{m \geq 3} \frac{1}{m!} \sum_{\substack{e_1,\dots,e_m \\ e_{i_1} = e_{i_2} = e_{i_3}  } } \abs{  \sum_{ \substack{ S \subseteq H \\ \text{conn., spann.} } } \prod_{\inbraces{i,j} \in S } -\boldsymbol{1}\!\inbraces{e_i \not\sim e_j}  } \prod_{j = 1}^{m} A_{e_j} \frac{\lambda}{p} \nonumber \\
        &\qquad \leq \sum_{m \geq 3} \frac{1}{m!} \sum_{\substack{e_1,\dots,e_m \\ e_{i_1} = e_{i_2} = e_{i_3}  } } \sum_{t \in \cT_{m-1}^{\text{lab}}} \boldsymbol{1}\!\inbraces{t \in \cT(H)^{\text{lab}}} \prod_{j=1}^{m} A_{e_j} \frac{\lambda}{p} \nonumber\\
        &\qquad = \sum_{m \geq 3} \frac{1}{m!} \sum_{t \in \cT_{m-1}^{\text{lab}}} \sum_{\substack{e_1,\dots,e_m \\ e_{i_1} = e_{i_2} = e_{i_3}  } }  \boldsymbol{1}\!\inbraces{t \in \cT(H)^{\text{lab}}} \prod_{j=1}^{m} A_{e_j} \frac{\lambda}{p}. \nonumber
    \end{align} 

    Fix $m$ and $t \in \cT_{m-1}^{\text{lab}}$. We next describe an iterative process to construct clusters $(e_1,\dots,e_m)$ where $e_{i} \in A$, and some $e_{i_1} = e_{i_2} = e_{i_3}$, and whose incompatibility graph $H$ contains $t$ as a spanning tree. 
    
    \underline{Step 1:} Fix $V' \subseteq V(t) = [m]$ with $\abs{V'} = 3$. The set $V'$ will be the coordinates in the cluster $(e_1,\dots,e_m)$ that contain a repeated edge. There are at most $\binom{m}{3}$ ways to do this.

    \underline{Step 2:} Pick the repeated edge $\tilde{e}$. There are $K_2(A)$ ways to do this. Assign the repeated edge $\tilde{e}$ to the vertices in $V'$.

    \underline{Step 3:} Iteratively, suppose coordinates $i_1,\dots,i_j$ have been assigned to polymers $e_{i_1}$, $e_{i_2}, \dots, e_{i_{j}}$. There must exist $i_{j+1} \in [m]\setminus \inbraces{i_1,\dots,i_j}$ adjacent to one of $\inbraces{i_1,\dots,i_j}$ in $t$. Without loss of generality suppose $\inbraces{i_j, i_{j+1}} \in t$. Then there are at most $2(\Delta(A) - 1) + 1$ choices for $e_{i_{j+1}}$, corresponding to all possible distinct incident edges to $e_{i_{j}}$ in $A$, as well as $e_{i_j}$ itself. (Here $\Delta(A)$ denotes the max degree in $A$).

    In this way, we obtain
    \begin{align}
        \sum_{m \geq 3}  \sum_{\substack{e_1,\dots,e_m \\ e_{i_1} = e_{i_2} = e_{i_3}  } } \abs{ \phi(H(e_1,\dots,e_m)) }\prod_{j = 1}^{m} A_{e_j} \frac{\lambda}{p} \leq \sum_{m \geq 3} \frac{1}{m!} \frac{\lambda^m}{p^m} \sum_{t \in \cT_{m-1}^{\text{lab}}} K_2(A) \binom{m}{3} \inparen{2\Delta(A) - 1}^{m - 3} .
        \label{eq:nonMatching_moreThanTwoRepRandom_ApplyPenroseTreeBound}
    \end{align}
    
    Since $\frac{9 \log n}{n} \leq 1.01 p$, we have for $A \sim G(n,p)$ that $\Delta(A) < 2.02np$ and $\abs{A} \leq 1.01 n^2 p$ with probability at least $1 - \frac{1}{n}$. By Cayley's theorem $\abs{\cT_{m-1}^{\text{lab}}} = m^{m-2}$. Note that $m^m / m! \leq e^m$. Hence with probability at least $1 - \frac{1}{n}$, we arrive after simplifications at
    \begin{align*}
        \sum_{m \geq 3}  \sum_{\substack{e_1,\dots,e_m \\ e_{i_1} = e_{i_2} = e_{i_3}  } } \abs{ \phi(H(e_1,\dots,e_m)) }\prod_{j = 1}^{m} A_{e_j} \frac{\lambda}{p} \leq \frac{C}{n p^2} \sum_{m \geq 3} m (4.04 e n \lambda)^m.
    \end{align*}
    By hypothesis $\abs{ 4.04 e n \lambda} < 1$ and the desired bound for the first sum in \eqref{eq:logZ_A_atleastTwoRepEdges_CE_SPLIT} follows.

    The argument for the second sum in \eqref{eq:logZ_A_atleastTwoRepEdges_CE_SPLIT} is largely similar, with differences only in the iterative process for construction of clusters. By similar application of the Penrose tree-graph bound Lemma \ref{lemma:Penrose_tree_bound}, we have
    \begin{align*}
        &\sum_{m \geq 4}  \sum_{\substack{e_1,\dots,e_m \\ e_{i_1} = e_{i_2}, e_{j_1} = e_{j_2}  } } \abs{ \phi(H(e_1,\dots,e_m)) }\prod_{j = 1}^{m} A_{e_j} \frac{\lambda}{p} \\
        &\qquad \leq \sum_{m \geq 4} \frac{1}{m!} \sum_{t \in \cT_{m-1}^{\text{lab}}} \sum_{\substack{e_1,\dots,e_m \\ e_{i_1} = e_{i_2}, e_{j_1} = e_{j_2}  } }  \boldsymbol{1}\!\inbraces{t \in \cT(H)^{\text{lab}}} \prod_{j=1}^{m} A_{e_j} \frac{\lambda}{p}.
    \end{align*}

    Fix $m$ and $t \in \cT_{m-1}^{\text{lab}}$. We next give an iterative process to construct clusters $(e_1,\dots,e_m)$ where $e_{i} \in A$, and some $e_{i_1} = e_{i_2}$ and $e_{j_1} = e_{j_2}$, and whose incompatibility graph $H$ contains $t$ as a spanning tree. 
    
    \underline{Step 1:} Distinguish two edges $\inbraces{i_{*,1}, i_{*,2}}$ and $\inbraces{j_{*,1},j_{*,2}}$ in $t$. These will correspond to the links between repeated polymers. There are at most $\binom{m-1}{2}$ ways to do this.

    \underline{Step 2:} Pick an arbitrary edge $\tilde{e}$ in $A$ to assign to vertex $i_1 := 1$ in $t$. There are $K_2(A)$ ways to do this.

    \underline{Step 3:} Iteratively, suppose coordinates $i_1 = 1,i_2,\dots,i_j$ have been assigned to polymers $e_{i_1} = \tilde{e}$, $e_{i_2}, \dots, e_{i_{j}}$. There must exist $i_{j+1} \in [m]\setminus \inbraces{i_1,\dots,i_j}$ adjacent to one of $\inbraces{i_1,\dots,i_j}$ in $t$. Without loss of generality suppose $\inbraces{i_j, i_{j+1}} \in t$. 
    \begin{itemize}
        \item If $\inbraces{i_j, i_{j+1}}$ is either of the distinguished edges $\inbraces{i_{*,1}, i_{*,2}}$ or $\inbraces{j_{*,1},j_{*,2}}$, assign polymer $e_{j}$ to $i_{j+1}$. That is, set $e_{j+1} = e_{j}$. 
        \item If $\inbraces{i_j, i_{j+1}}$ is not a distinguished edge, there are at most $2(\Delta(A) - 1) + 1$ choices for $e_{i_{j+1}}$, corresponding to all possible distinct incident edges to $e_{i_{j}}$ in $A$, as well as $e_{i_j}$ itself. (Here $\Delta(A)$ denotes the max degree in $A$).
    \end{itemize}
    Then similarly as before, with probability at least $1 - \frac{1}{n}$,  
    \begin{align*}
        \sum_{m \geq 4}  \sum_{\substack{e_1,\dots,e_m \\ e_{i_1} = e_{i_2}, e_{j_1} = e_{j_2}  } } \abs{ \phi(H(e_1,\dots,e_m)) }\prod_{j = 1}^{m} A_{e_j} \frac{\lambda}{p} &\leq \sum_{m \geq 4} \frac{m^{m-2}}{m!} \frac{\lambda^m}{p^m} K_2(A) \binom{m-1}{2} \inparen{  2 \Delta(A) - 1  }^{m - 3} \\
        &\leq \frac{C}{n p^2} \sum_{m \geq 4} (4.04 e n \lambda)^{m}.
    \end{align*}
    By hypothesis $\abs{ 4.04 e n \lambda} < 1$. This gives the desired bound on the second sum in \eqref{eq:logZ_A_atleastTwoRepEdges_CE_SPLIT} and finishes the bound for $\mathsf{moreThanTwoRepRandom}$.

    An almost identical argument will show that $\mathsf{moreThanTwoRepDeterministic} = O\!\inparen{\frac{1}{n}}$. We only have to replace every instance of the random $\Delta(A)$ and $\abs{A}$ above with the deterministic $\Delta(K_n) = n-1$ and $\abs{K_n} = \binom{n}{2}$ respectively.
\end{proof}

\section{Analysis of the log-likelihood ratio: equal average edge density}
\label{sec:plantingVarSize_pNotEqualq} 

%\subsection{Preliminaries}
In this section, we study the likelihood ratio for Problem~\ref{prob:detection} in the setting of Assumption~\ref{asmpt:MatchingEdgeDensity}. 
First, we note that Theorem~\ref{thm:matchingEdgeDensities_mean_is_minus_half_variance} follows from Theorem~\ref{thm:MatchingEdgeDensity_LLdistributionUnderNull} in the same way as Theorem~\ref{thm:plantedMatchings_mean_is_minus_half_variance} follows from Theorem~\ref{thm:nonMatching_LLdistributionUnderNull}. 
Therefore, it suffices to prove Theorem~\ref{thm:MatchingEdgeDensity_LLdistributionUnderNull}.

% a more involved version of Problem~\ref{prob:detection} as seen in Section \ref{sec:plantingVarSize_pEqualq}. The planted matching is similarly drawn from the monomer-dimer model for $\lambda \sim \frac{1}{n}$. However, the average edge densities between the planted and null model now coincide so that $p \neq q$. Formally, the setting is given as follows.

% The likelihood ratio in this case also features a ratio of monomer-dimer partition functions as in \eqref{eq:nonMatchingLikelihoodRatio}. 
% However, an additional factor $\exp F(A)$ appears:
% since
% $$
% \frac{\ud \cP_\lambda}{\ud \cQ}(A) = \EE_{M \sim \mu_\lambda} \prod_{\inbraces{i,j} \in M} \frac{1}{q^{A_{ij}}} \prod_{\inbraces{i,j} \notin M} \frac{p^{A_{ij}} (1-p)^{1 - A_{ij}} }{ q^{A_{ij}} (1-q)^{1 - A_{ij}} } \boldsymbol{1}\!\inbraces{M \subset A}, 
% $$
% we deduce that
% \begin{align}\label{eq:MatchingEdgeDensities_likelihoodratio}
%     \frac{\ud \cP_\lambda}{\ud \cQ}(A) = e^{F(A)} \cdot \frac{Z_A\!\inparen{\lambda/p}}{Z_{K_n}(\lambda)}, \quad \text{where} \quad e^{F(A)} = \inparen{ \frac{p(1-q)}{q(1-p)}  }^{K_2(A)} \inparen{ \frac{1-p}{1-q} }^{\binom{n}{2}}.
% \end{align}

Recall \eqref{eq:MatchingEdgeDensities_loglikelihood}. 
Intuitively, the effect of $F(A)$ is to cancel the dependence on the signed edge count. It will be seen in the sequel that $F(A)$ cancels overly large ($\gg 1$) and specific $\Theta(1)$ deterministic terms in the log-likelihood ratio coming from the ratio of partition functions. 
These cancellations lead to a pleasing conclusion. In the $p=q$ case in Section \ref{asmpt:nonMatchingEdges} the log-likelihood ratio has fluctuations and deterministic part carried by $\signedKtwo$ and one-repeated edge trees respectively---the latter arising from superimposing pairs of simple trees each having a marked \emph{edge}. Here, the fluctuation part is replaced by $\signedPtwo$, and the deterministic part by two-repeated edge trees arising from superimposing pairs of simple trees each having a marked \emph{wedge}.

Let us first establish a few preliminary results. 
Define
\begin{align*}
    c_n := \frac{2\EE\abs{M}}{n-1}. 
\end{align*}
% The $n-1$ is just for convenience, coming from $\binom{n}{2}$. 
Note that $c_n = O(1)$, with $c_n \rightarrow c \in (0,1)$ as given by \eqref{eq:MDMean_convergence_to_c}. 
It is easy to obtain the following.

\begin{claim}\label{claim:identities_pq} 
Suppose Assumption~\ref{asmpt:MatchingEdgeDensity} holds. 
We have
% (Identities for $p,q$.)
    \begin{align*}
        \frac{p(1-q)}{q(1-p)} &= 1 -  \frac{c_n}{nq}, \qquad \frac{q}{p} =  1 + \frac{c_n}{n}\frac{1-p}{p}, \qquad \text{and} \qquad
        \frac{1-q}{1-p} = 1 - \frac{c_n}{n}.
    \end{align*}
    For any $r \geq 3$, 
    \begin{align}
        \label{eq:identities_pq_(q/p)^r}
        \inparen{\frac{q}{p}}^r = 1 + \frac{c_n}{n}\frac{1-p}{p}r + \frac{c_n^2}{2n^2}\inparen{\frac{1-p}{p}}^2 r (r-1) + O\!\inparen{ \frac{c_n^3}{n^2 p}r^3 }.  
    \end{align}
\end{claim}
\begin{lemma} \label{lemma:logprefactor_simplification}
Suppose Assumption~\ref{asmpt:MatchingEdgeDensity} holds. 
    The factor $F(A)$ defined in \eqref{eq:def-F(A)} has the decomposition
    \begin{equation}\label{eq:logprefactor_simplification}
        F(A) = F_1(A) + F_2 + F_3 + O_\PP\!\inparen{ \frac{1}{\sqrt{nq}}  },
    \end{equation}
    where 
    \begin{equation*}
        F_1(A) = - \frac{c_n}{n q} \signedKtwo(A), \quad F_2 = - \frac{c_n^2}{4} \frac{1-q}{q}, \quad\text{and}\quad F_3 = - \frac{c_n^3}{6n} \frac{1 - q^2}{q^2} .
    \end{equation*}
    % \begin{align}
    %     \log(\mathsf{prefactor}) = \textcolor{red}{-\frac{c_n}{nq} \signedKtwo(A)}  \textcolor{forestgreen}{ - \frac{c_n^2}{4} \frac{1-q}{q} } \textcolor{blue}{- \frac{c_n^3}{6n} \frac{1 - q^2}{q^2}} + o_\PP(1).
    %     \label{eq:logprefactor_simplification}
    % \end{align}
\end{lemma}
\begin{proof}[Proof of Lemma \ref{lemma:logprefactor_simplification}]
    Using Claim \ref{claim:identities_pq} and the Taylor expansion of $\log(1 + x)$ at $x = 0$, we have
    % $\log(1 + x) =  x - \frac{x^2}{2} + \frac{x^3}{3} + O(x^4)$ for $x \rightarrow 0$,
    \begin{align*}
        F(A) &= \signedKtwo(A) \log \inparen{ 1 - \frac{c_n}{nq}  } + \binom{n}{2}q \log \inparen{ 1 - \frac{c_n}{nq}  } - \binom{n}{2} \log \inparen{ 1 - \frac{c_n}{n}  } \\
        &= \signedKtwo(A) \insquare{ - \frac{c_n}{nq} + O\!\inparen{\frac{1}{n^2 q^2}}  } + \binom{n}{2}q \insquare{  - \frac{c_n}{nq} - \frac{c_n^2}{2n^2 q^2} - \frac{c_n^3}{3 n^3 q^3} + O\!\inparen{ \frac{c_n^4}{n^4 q^4}  } } \\
        &\qquad -\binom{n}{2} \insquare{  -\frac{c_n}{n} - \frac{c_n^2}{2n^2} + O\!\inparen{ \frac{c_n^3}{n^3}  }  } \\
        &= - \frac{c_n}{n q} \signedKtwo(A) + O\!\inparen{  \frac{1}{n q^{3/2}}  } \frac{\signedKtwo(A)}{\sqrt{\Var \signedKtwo(A)}} - \frac{c_n^2}{4} \frac{1-q}{q} - \frac{c_n^3}{6n} \frac{1 - q^2}{q^2} + O\!\inparen{\frac{1}{nq}}.\qedhere
    \end{align*}
\end{proof}

\subsection{Approximation of the log-likelihood ratio}

We collect several definitions and results which will prove Theorem \ref{thm:MatchingEdgeDensity_LLdistributionUnderNull}. In light of 
Lemma~\ref{lem:log-likelihood-ratio-partition-functions}, Theorem~\ref{thm:truncate_CE_clogn}, and Lemma~\ref{lemma:logprefactor_simplification}, we may decompose the log-likelihood as 
\begin{align}\label{eq:matchingEdgeDensities_decompose_loglikelihood}
    &\log \frac{\ud \cP_\lambda}{\ud \cQ}(A) = F(A) + \sum_{m = 1}^{2 \log n} \sum_{e_1,\dots,e_m} \phi(H(e_1,\dots,e_m))\lambda^m \insquare{ \prod_{j=1}^{m} \frac{A_{e_j}}{p} - 1  } + O_\PP\!\inparen{ \frac{1}{n}  }  \nonumber\\
    &= F_1(A) + F_2 + F_3 + \simpleTrees + \oneRepTrees + \twoRepTrees + \mathsf{rem}_{\leq 2 \log n} + O_\PP\!\inparen{ \frac{1}{\sqrt{nq}}  },
\end{align}
where $\simpleTrees$ and $\oneRepTrees$ are defined as in \eqref{eq:decompose_loglikelihood}, and 
\begin{align*}
    \twoRepTrees &:= \sum_{m = 2}^{2 \log n} \sum_{\substack{e_1,\dots,e_m \\ \text{tree with two rep.~edge}}} \phi(H(e_1,\dots,e_m)) \lambda^m \insquare{ \prod_{j=1}^{m} \frac{A_{e_j}}{p} - 1  } , \\
    \mathsf{rem}_{\leq 2 \log n} &:= \sum_{m = 1}^{2 \log n} \sum_{e_1,\dots,e_m} \phi(H(e_1,\dots,e_m))\lambda^m \insquare{ \prod_{j=1}^{m} \frac{A_{e_j}}{p} - 1  } \\
    & \qquad - \simpleTrees - \oneRepTrees - \twoRepTrees.
\end{align*}
Observe that the random variable $\simpleTrees$ does not have a zero mean due to the mismatch between $p$ and $q$ (in contrast to Section \ref{sec:plantingVarSize_pEqualq}). Further decompose $\simpleTrees$ as 
\begin{align*}
    \simpleTrees &= \sum_{m =1}^{2 \log n} \sum_{T_m} m! \phi(H(T_m)) \insquare{  \frac{\lambda^m}{p^m} T_m(A) - \lambda^m T_m(K_n)  } = \overline{\simpleTrees} + \EE\insquare{ \simpleTrees  },
\end{align*}
where, with $\overline{T}_m(A) := T_m(A) - \EE_{A \sim \cQ} T_m(A)$,
\begin{align*}
    \overline{\simpleTrees} &:= \sum_{m =1}^{2 \log n} \sum_{T_m } m! \phi(H(T_m)) \frac{\lambda^m}{p^m} \overline{T}_m(A) , \\
    \EE\insquare{ \simpleTrees  } &:= \sum_{m =1}^{2 \log n} \sum_{T_m } m! \phi(H(T_m)) T_m(K_n) \lambda^m \insquare{\frac{q^m}{p^m} - 1}.
\end{align*}

Let us summarize at a high level the origin of the fluctuation and deterministic parts of \eqref{eq:MatchingEdgeDensity_LLdistributionUnderNull} from \eqref{eq:matchingEdgeDensities_decompose_loglikelihood}. The random parts $F_1(A)$ and $\overline{\simpleTrees}$ will combine to give the zero-mean fluctuation. On the other hand, the $\oneRepTrees$ and $\twoRepTrees$ concentrate around their means and so are essentially deterministic. These will combine with $\EE\simpleTrees$ and the deterministic $F_2$ and $F_3$ to give the mean part of \eqref{eq:MatchingEdgeDensity_LLdistributionUnderNull}. Finally, the remainder term $\mathsf{rem}_{\leq 2 \log n}$ will be small. In the following three propositions, Assumption \ref{asmpt:MatchingEdgeDensity} is in force. 
\begin{proposition}\label{prop:Matching_centeredSimpleTrees_and_Fone_mainResult}
    Let $A \sim \cQ$. Then 
    \begin{align}
        F_1(A) + \overline{\simpleTrees} = \frac{1}{\sqrt{2n} q} \inparen{\frac{2\EE\abs{M}}{n} }^2 \frac{\signedPtwo(A)}{\sqrt{\Var \signedPtwo(A)}} + O_\PP\!\inparen{ \frac{1}{\sqrt{nq}}  } . %+ O_\PP\!\inparen{ \frac{(\log n)^2}{n}  } .
    \end{align}
\end{proposition}
\begin{proposition}\label{prop:Matching_meanPart_mainResult}
    Let $A \sim \cQ$. Then 
    \begin{align}
        F_2 + F_3 + \EE\insquare{\simpleTrees} + \oneRepTrees + \twoRepTrees  = - \frac{4}{n q^2 } \inparen{ \frac{\EE\abs{M}}{n}  }^4 + O_\PP\!\inparen{ \frac{1}{\sqrt{nq}}  }.
    \end{align}
\end{proposition}
\begin{proposition}\label{prop:Matching_remLessThanTwoLogn_mainResult}
    Let $A \sim \cQ$. Then 
    \begin{align}
        \mathsf{rem}_{\leq 2 \log n}  = O_\PP\!\inparen{ \frac{1}{nq}  } .
    \end{align}
\end{proposition}
\begin{proof}[Proof of Theorem \ref{thm:MatchingEdgeDensity_LLdistributionUnderNull}]
Immediate from \eqref{eq:matchingEdgeDensities_decompose_loglikelihood} and the above three propositions.
\end{proof}

\subsection{Fluctuation part}

In this section, we establish Proposition \ref{prop:Matching_centeredSimpleTrees_and_Fone_mainResult}. 
Recall the notation \eqref{eq:ordinaryCenteredSignedSubgraphCount_definition}. 
Decompose $\overline{\simpleTrees}$ into three components:
\begin{align}\label{eq:centeredSimpleTrees_decompose_into_three}
    \overline{\simpleTrees} = \Proj_{\signedKtwo}\inparen{ \overline{\simpleTrees}   } + \Proj_{\signedPtwo}\inparen{ \overline{\simpleTrees}   } + \Proj^{\perp}_{\signedKtwo, \signedPtwo}\inparen{ \overline{\simpleTrees}   },
\end{align}
where 
\begin{align*}
    \Proj_{\signedKtwo}\inparen{ \overline{\simpleTrees}   } &= \inparen{  \sum_{m =1}^{2 \log n} \sum_{T_m} m! \phi(H(T_m)) \frac{\lambda^m}{p^m} \alpha(T_m) }\signedKtwo(A), \\
    \Proj_{\signedPtwo}\inparen{ \overline{\simpleTrees}   } &= \inparen{  \sum_{m =1}^{2 \log n} \sum_{T_m} m! \phi(H(T_m)) \frac{\lambda^m}{p^m} \beta(T_m) }\signedPtwo(A), \\
    \Proj^{\perp}_{\signedKtwo, \signedPtwo}\inparen{ \overline{\simpleTrees}   } &=   \sum_{m =1}^{2 \log n} \sum_{T_m} m! \phi(H(T_m)) \frac{\lambda^m}{p^m} r^\perp_{\signedKtwo, \signedPtwo}\!\inparen{T_m, A},
\end{align*}
where
\begin{equation*}
    \alpha(T_m) := \frac{\EE\insquare{\overline{T}_m(A) \cdot \signedKtwo(A)} }{\Var \signedKtwo(A)}, %=  \frac{(n)_{m+1}}{\aut(T_m)} m q^{m-1} \binom{n}{2}^{-1}. 
    \qquad \qquad \beta(T_m) := \frac{\EE\insquare{\overline{T}_m(A) \cdot \signedPtwo(A)} }{\Var \signedPtwo(A)} ,
\end{equation*}
and
\begin{equation}\label{eq:centeredSimpleTrees_def_rPperp_K2P2}
    r^\perp_{\signedKtwo, \signedPtwo}\!\inparen{T_m, A} := \overline{T}_m(A) - \alpha(T_m)\signedKtwo - \beta(T_m) \signedPtwo .
\end{equation}
The proof of Proposition \ref{prop:Matching_centeredSimpleTrees_and_Fone_mainResult} is immediate from the following three lemmas.
\begin{lemma}\label{lemma:matchingEdgeDensities_signedKtwo_Fone_cancels} With $F_1(A)$ defined in \eqref{eq:logprefactor_simplification} and $A \sim \cQ$,
    \begin{align}
    \Proj_{\signedKtwo}\inparen{ \overline{\simpleTrees}   } = -F_1(A) + O_\PP\!\inparen{  \frac{1}{\sqrt{nq}} }.
\end{align}
\end{lemma}
The main result of this section is the following.
\begin{lemma}\label{lemma:ProjP2simpleTrees_is_fluctuationPart} Let $A \sim \cQ$. Then
    \begin{align}
        \Proj_{\signedPtwo}\inparen{ \overline{\simpleTrees}   } = - \frac{1}{\sqrt{2n} q} \inparen{\frac{2\EE\abs{M}}{n}}^2 \frac{\signedPtwo(A)}{\sqrt{\Var \signedPtwo(A)}} + O_\PP\!\inparen{  \frac{1}{nq} } . %+ O_\PP\!\inparen{  \frac{(\log n)^2}{n} } .
    \end{align}
\end{lemma}
\begin{lemma}\label{lemma:ProjPerpKtwoPtwo_is_small}
    Let $A \sim \cQ$. Then
    \begin{align*}
        \Proj^{\perp}_{\signedKtwo, \signedPtwo}\inparen{ \overline{\simpleTrees}   } = O_\PP\!\inparen{\frac{1}{\sqrt{nq}}}.
    \end{align*}
\end{lemma}

\begin{proof}[Proof of Lemma \ref{lemma:matchingEdgeDensities_signedKtwo_Fone_cancels}] Similar to \eqref{eq:nonMatching_coeff_signedKtwo}, we compute 
    \begin{align*}
        \alpha(T_m) = \left. \frac{(n)_{m+1}}{\aut(T_m)} m q^m \middle/   \binom{n}{2}q  \right. .
    \end{align*}
    From the Taylor expansion \eqref{eq:identities_pq_(q/p)^r} giving $(q/p)^{m} = 1 + O(m c_n /np)$, we have 
    \begin{align*}
        &\Proj_{\signedKtwo}\inparen{ \overline{\simpleTrees}   } \\
        &= \inparen{  \sum_{m =1}^{2 \log n} \sum_{T_m} m! \phi(H(T_m)) m \lambda^m \frac{(n)_{m+1}}{\aut(T_m)}   }\frac{ \signedKtwo(A) }{\binom{n}{2}q} + O\!\inparen{  \frac{c_n}{np} \sum_{m =1}^{2 \log n} \sum_{T_m} m! \phi(H(T_m)) m^2 \lambda^m \frac{(n)_{m+1}}{\aut(T_m)}  }\frac{ \signedKtwo(A) }{\binom{n}{2}q} \\
        &= \underbrace{\frac{\EE\abs{M}}{\binom{n}{2}q}\signedKtwo(A)}_{ = - F_1(A)} + \underbrace{O\!\inparen{ 1 }\frac{ \signedKtwo(A) }{\binom{n}{2}q}}_{= O_\PP\!\inparen{ \frac{1}{n\sqrt{q}}  }} + O\!\inparen{ \frac{1}{n^2q^{3/2}} } \underbrace{ \sum_{m =1}^{2 \log n} \sum_{T_m} m! \phi(H(T_m)) m^2 \lambda^m \frac{(n)_{m+1}}{\aut(T_m)}  }_{= O(n)} \frac{ \signedKtwo(A) }{\sqrt{\Var \signedKtwo(A)}},
    \end{align*}
    where in the last line, we have used Proposition \ref{prop:GibbsMatching_mean_var_from_CE} to express the first term using $\EE|M|$, and used \eqref{eq:MD_Var_CE_formal} together with the proof of Proposition \ref{prop:Var|M|_Cn_bound} to bound the third term by $O(n)$, yielding that this term is $O_\PP\!\inparen{\frac{1}{\sqrt{nq}}}$. This completes the proof. 
\end{proof}

\begin{proof}[Proof of Lemma \ref{lemma:ProjPerpKtwoPtwo_is_small}]
    Note that $\Proj^{\perp}_{\signedKtwo, \signedPtwo}\inparen{ \overline{\simpleTrees}   }$ has mean zero. Therefore, it suffices to show that 
    % $\Var \Proj^{\perp}_{\signedKtwo, \signedPtwo}\inparen{ \overline{\simpleTrees}   }$ is vanishing. In fact, we will show
    \begin{equation*}
        \Var \Proj^{\perp}_{\signedKtwo, \signedPtwo}\inparen{ \overline{\simpleTrees}   } = O \left(\frac{1}{n q}\right) . 
    \end{equation*}
    By the triangle inequality,
    \begin{align}\label{eq:VarProjK2P2_Perp_simpleTrees_triangleInequality}
        \sqrt{\Var \Proj^{\perp}_{\signedKtwo, \signedPtwo}\inparen{ \overline{\simpleTrees}   } } \leq \sum_{m =1}^{2 \log n} \sum_{T_m} m! |\phi(H(T_m))| \frac{\lambda^m}{p^m} \sqrt{ \Var r^\perp_{\signedKtwo, \signedPtwo}\!\inparen{T_m, A} }.
    \end{align}
    Recall the definition of $\gamma(\cdot)$ in \eqref{eq:gamma_Tm_definition}. Compute 
    \begin{align*}
        \EE\insquare{  \overline{T}_m \signedKtwo(A)  } &= \frac{(n)_{m+1}m q^m (1-q)}{\aut(T_m)} \qquad\text{and}\qquad \EE\insquare{  \overline{T}_m \signedPtwo(A)  } = \frac{(n)_{m+1} \gamma(T_m) q^m (1-q)^2}{\aut(T_m)},
    \end{align*}
    as well as 
    \begin{align*}
        \Var \signedKtwo(A) = \binom{n}{2}q (1-q) \qquad\text{and}\qquad \Var \signedPtwo(A) = \binom{n}{3}\cdot 3 \cdot q^2 (1-q)^2. 
    \end{align*}
    From \eqref{eq:centeredSimpleTrees_def_rPperp_K2P2} we have 
    \begin{align}\label{eq:Var_rPerp_K2P2_cancelLeadSubLead}
        \Var r^\perp_{\signedKtwo, \signedPtwo}\!\inparen{T_m, A} &= \Var \overline{T}_m(A) - \frac{\EE\insquare{  \overline{T}_m \signedKtwo(A)  }^2 }{\Var \signedKtwo(A)}  - \frac{\EE\insquare{  \overline{T}_m \signedPtwo(A)  }^2 }{\Var \signedPtwo(A)}.
    \end{align}
    Using the variance estimate in \eqref{eq:Var_Tm_LeadAndSublead}, we find that the leading order term in $\Var \overline{T}_m(A)$ combines with the second term on the RHS of \eqref{eq:Var_rPerp_K2P2_cancelLeadSubLead} as 
    \begin{align*}
        \frac{2 m^2 (1-q)  q^{2m - 1} (n)_{m+1} (n)_{m-1} } {\aut(T_m)^2} - \frac{\inparen{\frac{(n)_{m+1}m q^m (1-q)}{\aut(T_m)} }^2}{ \Var \signedKtwo(A) } \leq C \frac{m^3 n^{2m-1}q^{2m-1}}{\aut(T_m)^2},
    \end{align*}
    where we used the inequality \eqref{eq:nonMatching_Var_rTMA_LeadCancellation_nIdentity}. The subleading order term in $\Var \overline{T}_m(A)$ in \eqref{eq:Var_Tm_LeadAndSublead} combines with the third term on the RHS of \eqref{eq:Var_rPerp_K2P2_cancelLeadSubLead} as 
    \begin{align*}
        &\frac{2 \gamma(T_m)^2 (1-q^2)  q^{2m - 2} (n)_{m+1} (n)_{m-2} } {\aut(T_m)^2} - \frac{\inparen{\frac{(n)_{m+1} \gamma(T_m) q^m (1-q)^2}{\aut(T_m)}}^2}{\binom{n}{3}\cdot 3 \cdot q^2 (1-q)^2} \\
        &\leq \frac{2q^{2m-2} \gamma(T_m)^2}{\aut(T_m)^2} (n)_{m+1} \insquare{  (n)_{m-2} - \frac{(n)_{m+1}}{n(n-1)(n-2)}   } + \frac{4(n)_{m+1} \gamma(T_m)^2 q^{2m-1}}{n(n-1)(n-2) \aut(T_m)^2} \\
        &\leq \frac{C  \gamma(T_m)^2 \cdot m \cdot n^{2m-2} q^{2m-2}}{\aut(T_m)^2} + \frac{C \gamma(T_m)^2 n^{2m-1} q^{2m-1}}{\aut(T_m)^2} \\
        &\leq \frac{C \gamma(T_m)^2 n^{2m-1} q^{2m-1}}{\aut(T_m)^2},
    \end{align*}
    where in the first inequality we bounded $1-q^2 \leq 1$ and $-(1 - q)^2 \leq -1 - 2q $, and in the second inequality we have used the inequality 
    \begin{align*}
        (n)_{m+1} \insquare{  (n)_{m-2} - \frac{(n)_{m+1}}{n(n-1)(n-2)}   } \leq C m n^{m-3},
    \end{align*}
    and also the approximation \eqref{eq:fallingfac_to_power_approx}. Thus, the dominant order of the RHS \eqref{eq:Var_rPerp_K2P2_cancelLeadSubLead} is contributed by the combined subleading term and the remainder term in \eqref{eq:Var_Tm_LeadAndSublead} to give 
    \begin{align*}
        \Var r^\perp_{\signedKtwo, \signedPtwo}\!\inparen{T_m, A} \leq C\frac{m^6 n^{2m-1} q^{2m-1}}{ \aut (T_m)^2},
    \end{align*}
    where we have used the bounds $\gamma(T_m) \leq m^2$ and $n q^2 = \Theta(1)$. The result is proved by plugging in this upper bound into \eqref{eq:VarProjK2P2_Perp_simpleTrees_triangleInequality} and following similar steps as in \eqref{eq:nonMatching_VarProjK_2_Perp_simpleTrees_finalBound}. Here there is a factor of $m^3$ instead of $m^2$ as in \eqref{eq:nonMatching_VarProjK_2_Perp_simpleTrees_finalBound}. Nevertheless, following the steps as in \eqref{eq:logZ_Kn_truncate_clogn_fixed_m_bound_sum_over_m}, we will obtain a derivative of a geometric series $\frac{n}{2} \sum_{m \geq 1} m (e\lambda \Delta)^m$ which is similarly bounded by $Cn$.
\end{proof}
% \begin{proposition}
%     \begin{align}
%         \Proj_{\signedPtwo}\inparen{ \overline{\simpleTrees}   } \overset{d}{\longrightarrow} \cN\!\inparen{ 0, \frac{c^2}{2\theta^2}  }.
%     \end{align}
% \end{proposition}
\begin{proof}[Proof of Lemma \ref{lemma:ProjP2simpleTrees_is_fluctuationPart}]
Write $\Proj_{\signedPtwo}\inparen{ \overline{\simpleTrees}   } = \mathsf{coeff}(\signedPtwo) \cdot \frac{\signedPtwo(A)}{\sqrt{\Var \signedPtwo(A)}} $, where 
\begin{align*}
    \mathsf{coeff}(\signedPtwo) &=  \sum_{m =2}^{2 \log n} \sum_{T_m} m! \phi(H(T_m)) \frac{\lambda^m}{p^m} \frac{\EE\insquare{\overline{T}_m(A)\cdot  \signedPtwo(A)}}{\sqrt{ \Var \signedPtwo(A)  }}.
\end{align*} 
It suffices to show that 
\begin{align}
    \label{eq:coeff_P2_goal}
    \mathsf{coeff}(\signedPtwo) = - \frac{1-q}{\sqrt{2n} q} \inparen{\frac{2\EE\abs{M}}{n}}^2 + O\!\inparen{  \frac{1}{nq} }. %+ O\!\inparen{  \frac{(\log n)^2}{n} }.
\end{align}

With $\gamma(\cdot)$ defined in \eqref{eq:gamma_Tm_definition}, compute
\begin{align*}
    \Cov\insquare{\overline{T}_m(A), \signedPtwo(A)} &= \frac{(n)_{m+1}}{\aut (T_m)} \gamma(T_m) q^m (1-q)^2, \quad\text{and}\quad
    \Var \signedPtwo(A) = 3 \binom{n}{3} q^2 (1-q)^2.
\end{align*}
Using \eqref{eq:fallingfac_to_power_approx}, we find that 
\begin{align}
    \label{eq:fallingfac_to_power_approx_coeff_P2}
    \frac{(n)_{m+1}}{\sqrt{(n)_3}} = n^{m - 1/2} \insquare{ 1 + O\!\inparen{  \frac{(\log n)^2}{n} }  }.
\end{align}
Applying \eqref{eq:fallingfac_to_power_approx_coeff_P2} followed by \eqref{eq:identities_pq_(q/p)^r} (zero-th order Taylor expansion) in the first and second lines respectively, we have
\begin{align*}
    \mathsf{coeff}(\signedPtwo) &= \insquare{1 + O\!\inparen{  \frac{(\log n)^2}{n} }}  \frac{\sqrt{2}(1-q)}{\sqrt{n}q}  \sum_{m =2}^{2 \log n} \sum_{T_m} m! \phi(H(T_m)) \gamma(T_m)\lambda^m \frac{q^m}{p^m} \frac{n^m}{\aut(T_m)} \\
    &= \insquare{1 + O\!\inparen{  \frac{(\log n)^2}{n} }}  \inparen{ W + O\!\inparen{ \frac{c_n(1-p)}{np} }W'  } \\
    &= W + O\!\inparen{  \frac{(\log n)^2}{n} } W + O\!\inparen{ \frac{1}{np} }W',
\end{align*}
where 
\begin{align*}
    W &:= \frac{\sqrt{2}(1-q)}{\sqrt{n}q}  \sum_{m =2}^{2 \log n} (n\lambda)^m \sum_{T_m} m! \phi(H(T_m)) \frac{\gamma(T_m)}{\aut(T_m)} , \\
    W' &:= \frac{\sqrt{2}(1-q)}{\sqrt{n}q}  \sum_{m =2}^{2 \log n} (n\lambda)^m \sum_{T_m} m! \phi(H(T_m)) m \frac{\gamma(T_m)}{\aut(T_m)} .
\end{align*}

On the other hand, using \eqref{eq:fallingfac_to_power_approx} and expanding the first term on the RHS of \eqref{eq:coeff_P2_goal} similarly as in \eqref{eq:nonMatching_EEoneRepTree_RHS_expand_in_X}, we obtain
\begin{align}
    - \frac{1-q}{\sqrt{2n} q} \inparen{\frac{2\EE\abs{M}}{n}}^2 &= X_{\leq 2 \log n} + X_{> 2 \log n } + O\!\inparen{  \frac{(\log n)^2}{n} }X + O\!\inparen{\frac{1}{n}},
\end{align}
where, with $\widetilde{\phi}(H) = m! \phi(H)$ denoting the unnormalized Ursell function for $H$ on $m$ vertices,
\begin{align*}
    X &:= - \frac{2\sqrt{2}(1-q)}{n^{5/2}q} \inparen{  \sum_{m=1}^{2 \log n} \sum_{T_m} m! m \phi(H(T_m)) \lambda^m \frac{n^{m+1} }{\aut(T_m)}   }^2 ,\\
    X_{\leq 2 \log n} &:= - \frac{2\sqrt{2}(1-q)}{\sqrt{n}q} \sum_{m = 2}^{2 \log n} (\lambda n)^m \sum_{\ell = 1}^{m-1} \sum_{(T_\ell, T_{m-\ell})} \frac{\ell (m-\ell) \widetilde{\phi}(H(T_\ell)) \widetilde{\phi}(H(T_{m-\ell})) }{\aut(T_\ell) \aut(T_{m - \ell})} ,\\
    X_{> 2 \log n} &:= - \frac{2\sqrt{2}(1-q)}{\sqrt{n}q} \sum_{m = 2\log n + 1}^{4 \log n} (\lambda n)^m \sum_{\ell = m - 2 \log n}^{2 \log n} \sum_{(T_\ell, T_{m-\ell})} \frac{\ell (m-\ell) \widetilde{\phi}(H(T_\ell)) \widetilde{\phi}(H(T_{m-\ell})) }{\aut(T_\ell) \aut(T_{m - \ell})}.
\end{align*}
We claim the following:
\begin{enumerate}[(i)]
    \item $O\!\inparen{\frac{(\log n)^2}{n}} W = O\!\inparen{\frac{(\log n)^2}{n}}$, and  $O\!\inparen{\frac{1}{np}} W' = O\!\inparen{\frac{1}{np}}$.
    \item $O\!\inparen{\frac{(\log n)^2}{n}} X = O\!\inparen{\frac{(\log n)^2}{n}}$.
    \item $X_{> 2 \log n} = O\!\inparen{ \frac{\log n}{n^2}  }$.
    \item $W = X_{\leq 2 \log n}$. 
\end{enumerate}
Combining the above will yield \eqref{eq:coeff_P2_goal}. It only remains to prove the claims. The proofs of Claims (ii) and (iii) are very similar to those in the proof of Lemma \ref{lemma:E_oneRepTrees_mainResult} and will not be repeated.

\smallskip
\noindent
\underline{\textbf{Proof of Claim (i):}} We have the bound $\gamma(T_m) \leq m^2$. Using \eqref{eq:fallingfac_to_power_approx} and then rewriting in terms of polymers, we have 
\begin{align*}
    \abs{W'} &\leq \frac{C}{n} \sum_{m = 2}^{2 \log n} \sum_{T_m} m! m^3 \abs{\phi(H(T_m))} \lambda^m \frac{(n)_{m+1}}{\aut(T_m)} \frac{n^{m+1}}{(n)_{m+1}} \\
    &= \frac{C}{n}\exp\inparen{ O\!\inparen{\frac{(\log n)^2}{n}}  } \sum_{m = 2}^{2 \log n} \sum_{ \substack{ e_1,\dots,e_m \\ \text{simple tree}} } m^3 \abs{\phi(H(e_1,\dots,e_m))} \lambda^m.
\end{align*}
Arguing as in \eqref{eq:logZ_Kn_truncate_clogn_fixed_m_bound_sum_over_m} using the Penrose tree-graph bound, we obtain
\begin{align*}
    &\sum_{m = 2}^{2 \log n} \sum_{ \substack{ e_1,\dots,e_m \\ \text{simple tree}} } m^3 \abs{\phi(H(e_1,\dots,e_m))} \lambda^m \leq \frac{n}{2} \sum_{m = 2}^{2 \log n} m (e\lambda \Delta)^m 
    \leq \frac{n}{2} \sum_{m \geq 2} m (e\lambda\Delta)^m \\
    &= \frac{n}{2} (e\lambda\Delta) \frac{\ud}{\ud \rho} \inparen{ \sum_{m \geq 2} \rho^{m+1} } \Bigg\rvert_{\rho = e\lambda\Delta} 
    = \frac{n}{2} (e\lambda\Delta) \frac{\ud}{\ud \rho} \inparen{\frac{\rho^3}{1-\rho} } \Bigg\rvert_{\rho = e\lambda\Delta} 
    \leq C n.
\end{align*}
This establishes that $W' = O(1)$. A similar argument will show that $W = O(1)$.

\smallskip
\noindent
\underline{\textbf{Proof of Claim (iv):}} Comparing the expressions for $X_{\leq 2 \log n}$ and $W$, we see that it is equivalent to show, for every $2 \leq m \leq 2\log n$, that \eqref{eq:MatchingEdgeDensity_fluctuation_fix_m} holds. Therefore, the proof is complete. 
\end{proof}

\subsection{Mean part}
\begin{comment}
\begin{itemize}
    \item Main result
    \item Structure, explanation, intuition
    \item `random parts' of oneRep and twoRep small
    \item Accounting of what's left
    \item Expanding E[simp], E[oneRep], E[twoRep]
    \begin{enumerate}
        \item E[simp] + E[oneRep] combine to 2nd logprefactor + $O(1)$ part (called combined)
        \item E[tripleEdge] $\sim$ -3rd log prefactor 
        \item combined writing as colored labeled combAdjDD, combSepDD 
        \item combSepDD + E[sepDD] $\sim 0$
        \item combAdjDD + E[adjDD] $\sim $ mean part.
    \end{enumerate}
\end{itemize}
\end{comment}

In this section, we establish Proposition \ref{prop:Matching_meanPart_mainResult}. Let us first show that $\oneRepTrees$ and $\twoRepTrees$ in \eqref{eq:matchingEdgeDensities_decompose_loglikelihood} concentrate around their respective expectations, so that they are essentially deterministic. 
\begin{lemma}\label{lemma:matchingEdgeDensity_oneRep_twoRep_are_deterministic}
    Let $A \sim \cQ$. Then the following holds
    \begin{align*}
        \oneRepTrees &= \EE\insquare{\oneRepTrees} + O_\PP\!\inparen{ \frac{1}{n q^{3/2}} } \\
        \twoRepTrees &= \EE\insquare{\twoRepTrees} + O_\PP\!\inparen{ \frac{1}{n^2 q^{5/2}} }
    \end{align*}
\end{lemma}
\begin{proof}[Proof of Lemma \ref{lemma:matchingEdgeDensity_oneRep_twoRep_are_deterministic}]
    The first statement follows by straightforward modifications of the proof of Lemma \ref{lemma:Var_oneRepTrees_small}. Here $\Var \overline{T}_m(A)$ is bounded in terms of $q$ instead of $p$ and this leads to an additional $(q/p)^m$ term. However, this does not lead to any additional difficulty as we can just expand $(q/p)^m \leq 1 + C m/(np)$. At this scale the lower order term can be absorbed into the dominant term. This leads to the same variance bound as in Lemma \ref{lemma:Var_oneRepTrees_small}. 

    The second statement follows by similar straightforward modifications.
\end{proof}
As a consequence of Lemma \ref{lemma:matchingEdgeDensity_oneRep_twoRep_are_deterministic}, we will not deal with any randomness in the remainder of this section. We first show that $\EE\insquare{\simpleTrees}$ is the sum of an $O(1/q)$ term and an $O(1)$ term that can be written as a sum over trees with one repeated edge. 
In this section, we write $\sim$ to mean equality to leading orders, hiding an at most $O\!\inparen{  \frac{(\log n)^2}{np}  }$ additive term. 
\begin{claim}
    \label{claim:EsimpleTrees_expandedInTmrep}
    We have
    \begin{align}
    \EE\insquare{ \simpleTrees  } &\sim \frac{c_n^2}{2} \frac{1-q}{q}  - \frac{c_n}{2nq^2} \sum_{m=1}^{2 \log n - 1} \sum_{\TmRep} \frac{(m+1)!}{2} \phi(H(\TmRep)) \frac{(n)_{m+1}}{\aut(\TmRep)} \lambda^{m+1} (m+3).
    \label{eq:EsimpleTrees_expandedInTmrep}
\end{align}
\end{claim}

\begin{proof}[Proof of Claim \ref{claim:EsimpleTrees_expandedInTmrep}]
Using \eqref{eq:identities_pq_(q/p)^r},
\begin{align*}
    &\EE\insquare{ \simpleTrees  } = \sum_{m =1}^{2 \log n} \sum_{T_m} m! \phi(H(T_m)) T_m(K_n) \lambda^m \insquare{  \frac{c_n}{n}\frac{1-p}{p}m + \frac{c_n^2}{2n^2}\inparen{\frac{1-p}{p}}^2 m (m-1) + O\!\inparen{ \frac{c_n^3}{n^2 p}m^3  } } \\
    &= \frac{c_n}{n}\frac{1-p}{p} \EE\abs{M} + \frac{c_n^2}{2n^2}\frac{(1-p)^2}{p^2} \sum_{m =1}^{2 \log n} \sum_{T_m} m! \phi(H(T_m)) T_m(K_n) \lambda^m    m (m-1)   + O\!\inparen{ \frac{1}{n p}  },
\end{align*} 
where the bound on the remainder term follows by a straightforward modification of the proof of Lemma \ref{prop:Var|M|_Cn_bound}. Furthermore, 
\begin{align*}
    \frac{c_n}{n} \frac{1-p}{p} \EE \abs{M} = \frac{c_n^2}{2} \frac{1-q}{q} + \frac{c_n^3}{2n} \frac{1-p}{pq} + O\!\inparen{\frac{1}{np}}.
\end{align*}
Thus 
\begin{align}
    \EE\insquare{ \simpleTrees  } &= \frac{c_n^2}{2} \frac{1-q}{q}  + \frac{c_n^3}{2n} \frac{1-p}{pq} \nonumber \\
    &\quad + \frac{c_n^2}{2n^2}\frac{(1-p)^2}{pq} \sum_{m =1}^{2 \log n} \sum_{T_m} m! \phi(H(T_m)) T_m(K_n) \lambda^m    m (m-1) + O\!\inparen{ \frac{1}{n p}  }.
    \label{eq:EsimpleTrees_expandedStillInTm} 
\end{align}

We now write all the $O(1)$ terms as a sum over repeated edge trees. We note that 
    \begin{align*}
        \frac{c_n^3}{2n} \frac{1-p}{pq} \sim \frac{c_n^2}{n^2q^2} \EE\abs{M} \sim \frac{c_n^2}{n^2q^2} \sum_{m =1}^{2 \log n} \sum_{T_m} m! \phi(H(T_m)) T_m(K_n) m \lambda^m.
    \end{align*}
    Hence the $O(1)$ terms in \eqref{eq:EsimpleTrees_expandedStillInTm} combine as
\begin{align}
    &\frac{c_n^3}{2n} \frac{1-p}{pq} + \frac{c_n^2}{2n^2}\frac{(1-p)^2}{pq} \sum_{m =1}^{2 \log n} \sum_{T_m} m! \phi(H(T_m)) T_m(K_n) \lambda^m    m (m-1) \nonumber \\
    &\quad\sim \frac{c_n^2}{2n^2 q^2} \sum_{m =1}^{2 \log n} \sum_{T_m } m! \phi(H(T_m)) T_m(K_n) \lambda^m (m+1)m \nonumber \\
    &\quad\sim \frac{c_n}{2 n q^2} \cdot \underbrace{ \frac{2\EE\abs{M}}{n^2} \sum_{m =1}^{2 \log n} \sum_{T_m} m! \phi(H(T_m)) T_m(K_n) \lambda^m    m (m+1) }_{ =: U}.
    \label{eq:EsimpleTrees_expanded_O1_parts}
\end{align}
Therefore, to prove \eqref{eq:EsimpleTrees_expandedInTmrep}, it suffices to show 
\begin{equation}
    \label{eq:EsimpleTrees_expandedInTmrep_proofGoal}
    U \sim - \sum_{m=1}^{2 \log n - 1} \sum_{\TmRep} \frac{(m+1)!}{2} \phi(H(\TmRep)) \frac{(n)_{m+1}}{\aut(\TmRep)}  \lambda^{m+1} (m+3).
\end{equation}

Note that Proposition \ref{prop:GibbsMatching_mean_var_from_CE} gives  
\begin{equation*}
    \frac{2\EE\abs{M}}{n^2} = \frac{2S(\lambda)}{n^2} + O\inparen{\frac{1}{n^2}}, \quad\text{where} \quad S(\lambda) := \sum_{m =1}^{2 \log n} \sum_{T_m} m! \phi(H(T_m)) T_m(K_n) \lambda^m    m.
\end{equation*}
We have the identity,
\begin{align}
    \frac{1}{\lambda} \frac{\ud}{\ud \lambda} \insquare{ \inparen{\frac{\lambda S(\lambda)}{n} }^2} &= \frac{2 S(\lambda)}{n^2} \frac{\ud}{\ud \lambda} \insquare{\lambda S(\lambda) } 
    = \frac{2 S(\lambda)}{n^2} \sum_{m =1}^{2 \log n} \sum_{T_m} m! \phi(H(T_m)) T_m(K_n) \lambda^m    m (m+1) \nonumber\\
    &= U + O\!\inparen{\frac{1}{n^2}}\sum_{m =1}^{2 \log n} \sum_{T_m} m! \phi(H(T_m)) T_m(K_n) \lambda^m    m (m+1) .
    \label{eq:EsimpleTrees_expandedInTmrep_partA}
\end{align}
On the other hand, by the approximation \eqref{eq:fallingfac_to_power_approx}, we have
\begin{align}
    \frac{S(\lambda)^2}{n^2} = X(\lambda) + O\!\inparen{ \frac{(\log n)^2}{n^3} } X(\lambda)
    \label{eq:EsimpleTrees_expandedInTmrep_partB} 
\end{align}
where 
\begin{align*}
    X(\lambda) := \frac{1}{n^2} \inparen{  \sum_{m=1}^{2 \log n} \sum_{T_m} m! m \phi(H(T_m)) \lambda^m \frac{n^{m+1} }{\aut(T_m)}   }^2.
\end{align*}
Importantly, we note that the factor $O\!\inparen{ \frac{(\log n)^2}{n^3} }$ is independent of $\lambda$. Further decompose 
\begin{align}
    X(\lambda) = X_{\leq 2 \log n}(\lambda) + X_{> 2 \log n}(\lambda)
    \label{eq:EsimpleTrees_expandedInTmrep_partC}
\end{align}
where 
\begin{align*}
    X_{\leq 2 \log n}(\lambda) &:=  \sum_{m = 1}^{2 \log n - 1} (n\lambda)^{m+1}  \sum_{\ell = 1}^{m} \sum_{(T_\ell, T_{m+1 -\ell})} \frac{  \ell \widetilde{\phi}(H(T_\ell))}{\aut (T_\ell)}  \frac{(m+1-\ell) \widetilde{\phi}(H(T_{m+1-\ell}))}{\aut(T_{m+1-\ell})}\\
    X_{> 2 \log n}(\lambda) &:= \sum_{m = 2 \log n}^{4 \log n - 1} (n\lambda)^{m+1} \sum_{\ell = m+1 - 2\log n}^{2 \log n} \sum_{(T_\ell, T_{m+1 -\ell})} \frac{  \ell \widetilde{\phi}(H(T_\ell))}{\aut (T_\ell)}  \frac{ (m+1-\ell) \widetilde{\phi}(H(T_{m+1-\ell}))}{\aut(T_{m+1-\ell})}.
\end{align*} 
Collecting equations \eqref{eq:EsimpleTrees_expandedInTmrep_partA}, \eqref{eq:EsimpleTrees_expandedInTmrep_partB}, and \eqref{eq:EsimpleTrees_expandedInTmrep_partC}, we have 
\begin{align*}
    U = \text{I} + \text{II} + \text{III} + \text{IV},
\end{align*}
where 
\begin{alignat*}{2}
    & \text{I} = \frac{1}{\lambda} \frac{\ud}{\ud \lambda} \insquare{\lambda^2 X_{\leq 2 \log n}(\lambda)}, 
    \quad&&\quad \text{II} = \frac{1}{\lambda} \frac{\ud}{\ud \lambda} \insquare{\lambda^2 X_{> 2 \log n}(\lambda)}, \\
    & \text{III} = O\!\inparen{ \frac{(\log n)^2}{n^3} } \frac{1}{\lambda} \frac{\ud}{\ud \lambda} \insquare{\lambda^2   X(\lambda)  },
    \quad&&\quad
    \text{IV} = O\!\inparen{\frac{1}{n^2}}\sum_{m =1}^{2 \log n} \sum_{T_m} m! \phi(H(T_m)) T_m(K_n) \lambda^m    m (m+1).
\end{alignat*}
% \begin{align*}
%     \text{I} &:= \frac{1}{\lambda} \frac{\ud}{\ud \lambda} \insquare{\lambda^2 X_{\leq 2 \log n}(\lambda)} \\
%     \text{II} &:= \frac{1}{\lambda} \frac{\ud}{\ud \lambda} \insquare{\lambda^2 X_{> 2 \log n}(\lambda)} \\
%     \text{III} &:= O\!\inparen{ \frac{(\log n)^2}{n^3} } \frac{1}{\lambda} \frac{\ud}{\ud \lambda} \insquare{\lambda^2   X(\lambda)  } \\
%     \text{IV} &:= O\!\inparen{\frac{1}{n^2}}\sum_{m =1}^{2 \log n} \sum_{T_m} m! \phi(H(T_m)) T_m(K_n) \lambda^m    m (m+1)
% \end{align*}
We state the following claims:
\begin{align*}
    &\text{(i)}  \quad \text{I} \sim \text{RHS of } \eqref{eq:EsimpleTrees_expandedInTmrep_proofGoal}, \qquad\qquad \text{(ii)} \quad \text{II} = O\!\inparen{\frac{(\log n)^2}{n^2}} ,\\
    &\text{(iii)} \quad \text{III} = O\!\inparen{\frac{(\log n)^3}{n }} ,\qquad\quad \text{(iv)}\quad \text{IV} = O\!\inparen{\frac{\log n}{n}}. 
\end{align*}
These claims will establish \eqref{eq:EsimpleTrees_expandedInTmrep_proofGoal}. It remains only to prove them.

\smallskip
\noindent
\underline{\textbf{Proof of Claim (i):}} From the proof of Lemma \ref{lemma:E_oneRepTrees_mainResult} (in particular Claim (iv) there) we have 
\begin{align*}
    X_{\leq 2 \log n}(\lambda) = -  \sum_{m = 1}^{2 \log n - 1} (n \lambda)^{m+1} \sum_{\TmRep} \frac{(m+1)!}{2 \aut (\TmRep)} \phi\inparen{H(\TmRep)}.
\end{align*}
The result is immediate by taking the derivative with respect to $\lambda$ in $\text{I}$ and then using the approximation \eqref{eq:fallingfac_to_power_approx}.

\smallskip
\noindent
\underline{\textbf{Proofs of Claims (ii) and (iii):}} These follow from straightforward modifications of the proof of Claims (iii) and (ii) respectively in the proof of Lemma \ref{lemma:E_oneRepTrees_mainResult}. The derivative with respect to $\lambda$ introduces an additional factor of $(m + \text{constant})$ which can be bounded in magnitude by $O(\log n)$.

\smallskip
\noindent
\underline{\textbf{Proof of Claims (iv):}} This follows by bounding the sum as $O(n)$ by similar arguments as in Claim (ii) in the proof of Lemma \ref{lemma:E_oneRepTrees_mainResult}.
\end{proof}

In light of Lemma \ref{lemma:matchingEdgeDensity_oneRep_twoRep_are_deterministic} we only need consider the mean part of $\oneRepTrees$:
% \begin{align*}
%     \oneRepTrees &= \sum_{m =1}^{2 \log n - 1} \sum_{\TmRep \in \mathscr{T}_m^{\text{rep}}} \frac{(m+1)!}{2} \phi(H(\TmRep))  \psi(H(\TmRep)) \insquare{  \frac{\lambda^{m+1}}{p^{m+1}} T_m(A) - \lambda^{m+1} T_m(K_n)  }.
% \end{align*} 
\begin{align*}
    \EE\insquare{\oneRepTrees} = \sum_{m =1}^{2 \log n - 1} \sum_{\TmRep} \frac{(m+1)!}{2} \phi(H(\TmRep)) \frac{(n)_{m+1}}{\aut(\TmRep)}  \lambda^{m+1}  \frac{1}{q} \insquare{ \inparen{ \frac{q}{p} }^{m+1} - q  }.
\end{align*}
We similarly decompose $\EE\insquare{\oneRepTrees}$ into the sum of an $O(1/q)$ and an $O(1)$ term.
\begin{claim}
    \label{claim:EoneRepTrees_expanded_final}
    We have
    \begin{align}
    \EE\insquare{\oneRepTrees} \nonumber 
    &= -\frac{1-q}{q} \frac{c_n^2}{4}   + O\!\inparen{\frac{(\log n)^2}{nq}} \nonumber \\
    &\qquad + \frac{c_n}{n} \frac{1-p}{pq} \sum_{m =1}^{2 \log n-1} \sum_{\TmRep} \frac{(m+1)!}{2} \phi(H(\TmRep))  \frac{(n)_{m+1}}{\aut(\TmRep)} \lambda^{m+1} (m+1) . 
    \label{eq:EoneRepTrees_expanded_final}
\end{align}
\end{claim}

\begin{proof}[Proof of Claim \ref{claim:EoneRepTrees_expanded_final}]
    From \eqref{eq:EoneRepTrees_is_meanSquare}, we have
\begin{align}\label{eq:TmRep_meanMDSquare_Relation}
    \sum_{m =1}^{2 \log n - 1} \sum_{\TmRep} \frac{(m+1)!}{2} \phi(H(\TmRep))  \frac{(n)_{m+1}}{\aut(\TmRep)} \lambda^{m+1} &= - \inparen{  \frac{\EE\abs{M}}{n} }^2 + O\!\inparen{\frac{(\log n)^2}{np}}.
\end{align}
The result follows by expanding LHS of \eqref{eq:EoneRepTrees_expanded_final} using \eqref{eq:identities_pq_(q/p)^r} (to first order). 
\end{proof}
Combining the $F_2$ term in \eqref{eq:logprefactor_simplification} with the expanded $\EE\insquare{ \simpleTrees  }$ in \eqref{eq:EsimpleTrees_expandedInTmrep} and the expanded $\EE\insquare{ \oneRepTrees  }$ in \eqref{eq:EoneRepTrees_expanded_final}, we see that the higher order $O(1/q)$ terms cancel; we are left with an $O(1)$ term which is called $\combined$. We record this as a lemma.
\begin{lemma}\label{lemma:combined_definition_as_Ftwo_EsimpleTrees_EonerepTrees}
We have
    \begin{align}
    &F_2 + \EE\insquare{ \simpleTrees  } + \EE\insquare{ \oneRepTrees  } \sim \combined,
    \label{eq:combine_secondLogPrefactor_Esimp_EoneRep}
    \end{align}
    where 
    \begin{align*}
        \combined \sim \frac{c_n}{2nq^2} \sum_{m =1}^{2 \log n-1} \sum_{\TmRep} \frac{(m+1)!}{2} \phi(H(\TmRep))  \frac{(n)_{m+1}}{\aut(\TmRep)} \lambda^{m+1} (m-1).
    \end{align*}
\end{lemma}

% was thinking about changing $\TmRep$ into $T_m$. Turns out the other way around might be better.
\begin{comment}
Note that it is possible to change the sum over $\TmRep$ into a sum over $T_m$: again from \eqref{eq:EoneRepTrees_is_meanSquare} we have by differentiating both sides (assuming the derivative of the correction is also small.)
\begin{align*}
    &\sum_{m =1}^{2 \log n} \sum_{\TmRep \in \mathscr{T}_m^{\text{rep}}} \frac{(m+1)!}{2} \phi(H(\TmRep))  \psi(H(\TmRep)) \lambda^{m+1} T_m(K_n) (m+1) \\
    &\quad= \lambda \frac{\ud}{\ud \lambda} \inparen{ - \inparen{  \frac{\EE\abs{M}}{n} }^2  }   + O\!\inparen{\frac{(\log n)^2}{np}} \\
    &\quad= - \frac{\lambda}{n^2} \cdot 2\EE\abs{M} \frac{\ud}{\ud \lambda} \EE\abs{M} + O\!\inparen{\frac{(\log n)^2}{np}}\\
    &\quad= - \frac{c_n}{n} \sum_{m =1}^{2 \log n} \sum_{T_m \in \mathscr{T}_m} m! \phi(H(T_m)) T_m(K_n) \lambda^m m^2 + O\!\inparen{\frac{(\log n)^2}{np}}
\end{align*}
Plugging this back, we find 
\begin{align}
    &\textcolor{forestgreen}{ - \frac{c_n^2}{4} \frac{1-q}{q} } + \EE\insquare{ \simpleTrees  } + \EE\insquare{ \oneRepTrees  } \nonumber\\
    &= - \frac{c_n^2}{2n^2 q^2}\sum_{m =1}^{2 \log n} \sum_{T_m \in \mathscr{T}_m} m! \phi(H(T_m)) T_m(K_n) \lambda^m    m (m-1)  + O\!\inparen{\frac{(\log n)^2}{np}} 
    \label{eq:logprefactorO1overq_plus_Esimp_plus_EoneRep}
\end{align}
\end{comment}
In light of Lemma \ref{lemma:matchingEdgeDensity_oneRep_twoRep_are_deterministic} we only need consider the mean part of $\twoRepTrees$ in \eqref{eq:matchingEdgeDensities_decompose_loglikelihood}, which we further decompose as follows:
\begin{equation}\label{eq:matchingEdgeDensities_EtwoRepTree_decomposition}
    \EE\insquare{\twoRepTrees} = \EE\insquare{\tripleEdge} + \EE\insquare{\adjDD} + \EE\insquare{\sepDD},
\end{equation}
where 
\begin{align*}
    \EE\insquare{\tripleEdge} &= \frac{1}{q^2} \sum_{m=1}^{2\log n - 2} \sum_{\TmTripleEdge} \frac{(m+2)!}{3!} \phi(H(\TmTripleEdge)) \lambda^{m+2} \insquare{ \frac{q^{m+2}}{p^{m+2}}  - q^2 } \frac{(n)_{m+1}}{\aut(\TmTripleEdge)} ,\\
    \EE\insquare{\adjDD} &= \frac{1}{q^2} \sum_{m=2}^{2\log n - 2} \sum_{\TmAdjDD} \frac{(m+2)!}{2!2!} \phi(H(\TmAdjDD)) \lambda^{m+2} \insquare{ \frac{q^{m+2}}{p^{m+2}}  - q^2 } \frac{(n)_{m+1}}{\aut(\TmAdjDD)} ,\\
    \EE\insquare{\sepDD} &= \frac{1}{q^2} \sum_{m=3}^{2\log n - 2} \sum_{\TmSepDD} \frac{(m+2)!}{2!2!} \phi(H(\TmSepDD)) \lambda^{m+2} \insquare{ \frac{q^{m+2}}{p^{m+2}}  - q^2 } \frac{(n)_{m+1}}{\aut(\TmSepDD)},
\end{align*}
where the $T_m^{\#}$'s are each unlabeled trees with $(m + 2)$ edges and $(m+1)$ vertices with the superscript $\#$ representing trees with:
\begin{center}
\begin{tabular}{lll}
    $\equiv$ & \text{exactly one edge repeated three times,} \\
    $==$ & \text{exactly two twice repeated edges that are incident,} \\
    $=\cdots =$ & \text{exactly two twice repeated edges that are not incident.}
\end{tabular}
\end{center}

We clarify the purpose of the $O(1)$ term $F_3$ in \eqref{eq:logprefactor_simplification}: to cancel the triple edge tree terms.
\begin{lemma}
    \label{lemma:tripleEdge_cancels_logprefactorO1}
    We have
    \begin{equation}
     \EE\insquare{\tripleEdge} \sim - F_3 .
    %  + O\!\inparen{  \frac{(\log n)^3}{np}  }
     \label{eq:tripleEdge_cancels_logprefactorO1}
    \end{equation}
\end{lemma}

\paragraph{Double-double edge terms.}

It remains to deal with the remaining $\combined$ \eqref{eq:combine_secondLogPrefactor_Esimp_EoneRep} and $\EE\insquare{\adjDD}$ and $\EE\insquare{\sepDD}$ terms. While not immediately apparent, $\combined$ can be interpreted as a sum over double-double repeated edge colored trees. To see this, split $\combined$ as follows. Define, with notation to be explained subsequently,
\begin{equation}\label{eq:combAdjDD_definition}
    \combAdjDD = \frac{1}{2nq^2} \sum_{m=2}^{2 \log n - 2} (n\lambda)^{m+2} \sum_{ \substack{   \inparen{ \widetilde{T}_{\text{red}}^{\text{rep}}(v_*), \widetilde{T}_{\text{blue}}(v_{**})    }  \\ \in \,  \adjDD }} \frac{1}{2(m+1)!} \widetilde{\phi}(H(\widetilde{T}_{\text{red}}^{\text{rep}}) ) \widetilde{\phi}(H(\widetilde{T}_{\text{blue}} ) ) 
\end{equation}
and 
\begin{equation}\label{eq:combSepDD_definition}
    \combSepDD = \frac{1}{2nq^2} \sum_{m=3}^{2 \log n - 2} (n\lambda)^{m+2} \sum_{ \substack{   \inparen{ \widetilde{T}_{\text{red}}^{\text{rep}}(v_*), \widetilde{T}_{\text{blue}}(v_{**})    }  \\ \in \,  \sepDD }} \frac{1}{2(m+1)!} \widetilde{\phi}(H(\widetilde{T}_{\text{red}}^{\text{rep}}) ) \widetilde{\phi}(H(\widetilde{T}_{\text{blue}} ) ).
\end{equation}
Here, for each $m$, the sum is over
\begin{equation}\label{eq:combined_genericPair}
    \inparen{ \widetilde{T}_{\text{red}}^{\text{rep}}(v_*), \widetilde{T}_{\text{blue}}(v_{**}) }
\end{equation}
satisfying:
\begin{itemize}
    \item $\widetilde{T}_{\text{red}}^{\text{rep}}(v_*)$ is a red colored vertex-labeled tree with exactly one repeated edge and one distinguished non-repeated edge $v_*$.
    \item $\widetilde{T}_{\text{blue}}(v_{**})$ is a blue colored vertex-labeled simple tree with one distinguished edge $v_{**}$.
    \item The label set of the two vertices incident to $v_*$ must coincide with that for $v_{**}$.
    \item Joining the trees by superimposing the distinguished edges $v_*$ and $v_{**}$ (matching their vertex labels) gives a labeled tree of size $m+2$ with $m+1$ vertices with two (twice) repeated edges. The vertices are labeled in $[m+1]$.

    \item
    The sets $\adjDD$ and $\sepDD$ collect the pairs of trees such that their joined trees have, respectively, adjacent double-double edges and separated double-double edges.
\end{itemize}
An example of a joined tree represented by the tuple \eqref{eq:combined_genericPair} in $\adjDD$ and $\sepDD$ is given on the left in Figures \ref{fig:equalAverageED_CombAdjDD} and \ref{fig:equalAverageED_sepDD} respectively (ignoring the other annotations of $v'$, $w_*$, $w_{**}$ for the moment). 
\begin{claim}\label{claim:comb_split_adjDD_sepDD}With $\combAdjDD$ and $\combSepDD$ defined in \eqref{eq:combAdjDD_definition} and \eqref{eq:combSepDD_definition} respectively, 
    \begin{equation*}
    \combined \sim \combAdjDD + \combSepDD.
\end{equation*}
\end{claim}
We remark that it is possible to extract out the relevant Ursell combinatorial identity involved in the double-double edge case as in Section \ref{sec:Ursell_CombinatorialIdentities}. However, this can only be done cleanly without splitting $\combined$ as above, and would obscure the different roles that $\combAdjDD$ and $\combSepDD$ play. The origin of the deterministic part of RHS of \eqref{eq:MatchingEdgeDensity_LLdistributionUnderNull} will instead be clearer from the below lemmas. The first lemma shows that we can forget about the separated double-double edge terms.
\begin{lemma}\label{lemma:sepDD_combSepDD_cancel}
    With $\EE\insquare{\sepDD}$ and $\combSepDD$ defined in \eqref{eq:matchingEdgeDensities_EtwoRepTree_decomposition} and \eqref{eq:combSepDD_definition} respectively, we have
    \begin{align*}
        \EE\insquare{\sepDD} + \combSepDD \sim 0.
    \end{align*}
\end{lemma}

\begin{lemma}
    \label{lemma:adjDD_combAdjDD_give_meanPart} With $\EE\insquare{\adjDD}$ and $\combAdjDD$ defined in \eqref{eq:matchingEdgeDensities_EtwoRepTree_decomposition} and \eqref{eq:combAdjDD_definition} respectively, we have
    \begin{align*}
        \EE\insquare{\adjDD} + \combAdjDD \sim -\frac{c_n^4}{4nq^2} \qquad \text{(deterministic part of RHS \eqref{eq:MatchingEdgeDensity_LLdistributionUnderNull})}.
    \end{align*}
\end{lemma}

\begin{proof}[Proof of Proposition \ref{prop:Matching_meanPart_mainResult}]
    The result follows from Lemmas \ref{lemma:matchingEdgeDensity_oneRep_twoRep_are_deterministic} and \ref{lemma:combined_definition_as_Ftwo_EsimpleTrees_EonerepTrees}--\ref{lemma:sepDD_combSepDD_cancel}.
    % , \ref{lemma:tripleEdge_cancels_logprefactorO1}, \ref{lemma:adjDD_combAdjDD_give_meanPart}, and \ref{lemma:sepDD_combSepDD_cancel}.
\end{proof}

Before embarking on the proofs of the above lemmas, we give a visual depiction in Figure \ref{fig:adjDDCartoon} for what cancellations to expect. (The $w_*$ and $w_{**}$ notation will be defined in the proofs.)
%%%% The adjDD cartoon
\begin{figure}[htp]\centering
\scalebox{.95}{
\begin{tikzpicture}
\hspace{-0.5cm}
\begin{scope}[scale=0.75, every node/.style={font=\small}, xshift=0cm, yshift=0cm] % First picture at default position
    \node (4) at (1,3) {};
    \node (6) at (2.5,2) {};
    \node (8) at (0,4) {};
    \node (1) at (0,2.2) {};
    \node (7) at (3.5,1) {};
    \node (2) at (4.7,1) {};
    \node (9) at (3.8,3) {};
    \node (5) at (5,2.2) {};
    \node (3) at (5,4) {};
    
    \draw[ultra thick, red] ($(4.center) + (0,1mm)$) -- (8);
    \draw[ultra thick, blue] (1) -- ($(4.center) + (0,-1mm)$);
    \draw[ultra thick, blue] (6) -- (7);
    \draw[ultra thick, blue] (2) -- ($(7.center) + (0,1mm)$);
    \draw[ultra thick, red] (3) -- ($(9.center) + (0,1mm)$);
    \draw[ultra thick, red] (5) -- ($(9.center) + (0,-1mm)$);

    %\draw[ultra thick, red] (6) -- node[label=below:$v'$, 
    \draw[ultra thick, transform canvas={xshift=0.15em, yshift=0.2em}, red] (4) -- node[label={[label distance=-0.3em]90:$v_{*}$}] {} (6);
    \draw[ultra thick, transform canvas={xshift=-0em, yshift=-0.2em}, blue] (4) -- node[label=below:$v_{**}$, xshift=-0.15cm, yshift=0.1cm] {} (6);
    
    \draw[ultra thick, transform canvas={xshift=-0.15em, yshift=0.21em}, red] (6) -- node[label={[label distance=-0.3em]90:$w_{*}$}] {} (9);
    \draw[ultra thick, transform canvas={xshift=0.15em, yshift=-0.21em}, red] (6) -- node[label=below:$w_{**}$,xshift=0.15cm, yshift=0.1cm] {} (9);
\end{scope}

\begin{scope}[scale=0.75, every node/.style={font=\small}, xshift=6cm, yshift=0cm] % Second picture 
    \node (4) at (1,3) {};
    \node (6) at (2.5,2) {};
    \node (8) at (0,4) {};
    \node (1) at (0,2.2) {};
    \node (7) at (3.5,1) {};
    \node (2) at (4.7,1) {};
    \node (9) at (3.8,3) {};
    \node (5) at (5,2.2) {};
    \node (3) at (5,4) {};
    
    \draw[ultra thick, red] ($(4.center) + (0,1mm)$) -- (8);
    \draw[ultra thick, blue] (1) -- ($(4.center) + (0,-1mm)$);
    \draw[ultra thick, blue] (6) -- (7);
    \draw[ultra thick, blue] (2) -- ($(7.center) + (0,1mm)$);
    \draw[ultra thick, blue] (3) -- ($(9.center) + (0,1mm)$);
    \draw[ultra thick, blue] (5) -- ($(9.center) + (0,-1mm)$);

    %\draw[ultra thick, red] (6) -- node[label=below:$v'$, 
    \draw[ultra thick, transform canvas={xshift=0.15em, yshift=0.2em}, red] (4) -- node[label={[label distance=-0.3em]90:$v_{*}$}] {} (6);
    \draw[ultra thick, transform canvas={xshift=-0em, yshift=-0.2em}, blue] (4) -- node[label=below:$v_{**}$, xshift=-0.15cm, yshift=0.1cm] {} (6);
    
    \draw[ultra thick, transform canvas={xshift=-0.15em, yshift=0.21em}, blue] (6) -- node[label={[label distance=-0.3em]90:$w_{*}$}] {} (9);
    \draw[ultra thick, transform canvas={xshift=0.15em, yshift=-0.21em}, blue] (6) -- node[label=below:$w_{**}$,xshift=0.15cm, yshift=0.1cm] {} (9);
\end{scope}

\begin{scope}[scale=0.75, every node/.style={font=\small}, xshift=12cm, yshift=0cm] % third picture 
    \node (4) at (1,3) {};
    \node (6) at (2.5,2) {};
    \node (8) at (0,4) {};
    \node (1) at (0,2.2) {};
    \node (7) at (3.5,1) {};
    \node (2) at (4.7,1) {};
    \node (9) at (3.8,3) {};
    \node (5) at (5,2.2) {};
    \node (3) at (5,4) {};
    
    \draw[ultra thick, red] ($(4.center) + (0,1mm)$) -- (8);
    \draw[ultra thick, blue] (1) -- ($(4.center) + (0,-1mm)$);
    \draw[ultra thick, blue] (6) -- (7);
    \draw[ultra thick, blue] (2) -- ($(7.center) + (0,1mm)$);
    \draw[ultra thick, red] (3) -- ($(9.center) + (0,1mm)$);
    \draw[ultra thick, blue] (5) -- ($(9.center) + (0,-1mm)$);

    %\draw[ultra thick, red] (6) -- node[label=below:$v'$, 
    \draw[ultra thick, transform canvas={xshift=0.15em, yshift=0.2em}, red] (4) -- node[label={[label distance=-0.3em]90:$v_{*}$}] {} (6);
    \draw[ultra thick, transform canvas={xshift=-0em, yshift=-0.2em}, blue] (4) -- node[label=below:$v_{**}$, xshift=-0.15cm, yshift=0.1cm] {} (6);
    
    \draw[ultra thick, transform canvas={xshift=-0.15em, yshift=0.21em}, red] (6) -- node[label={[label distance=-0.3em]90:$w_{*}$}] {} (9);
    \draw[ultra thick, transform canvas={xshift=0.15em, yshift=-0.21em}, blue] (6) -- node[label=below:$w_{**}$,xshift=0.15cm, yshift=0.1cm] {} (9);
\end{scope}

\begin{scope}[scale=0.75, every node/.style={font=\small}, xshift=18cm, yshift=0cm] % Fourth picture 
    \node (4) at (1,3) {};
    \node (6) at (2.5,2) {};
    \node (8) at (0,4) {};
    \node (1) at (0,2.2) {};
    \node (7) at (3.5,1) {};
    \node (2) at (4.7,1) {};
    \node (9) at (3.8,3) {};
    \node (5) at (5,2.2) {};
    \node (3) at (5,4) {};
    
    \draw[ultra thick, red] ($(4.center) + (0,1mm)$) -- (8);
    \draw[ultra thick, blue] (1) -- ($(4.center) + (0,-1mm)$);
    \draw[ultra thick, blue] (6) -- (7);
    \draw[ultra thick, blue] (2) -- ($(7.center) + (0,1mm)$);
    \draw[ultra thick, red] (3) -- ($(9.center) + (0,1mm)$);
    \draw[ultra thick, blue] (5) -- ($(9.center) + (0,-1mm)$);

    %\draw[ultra thick, red] (6) -- node[label=below:$v'$, 
    \draw[ultra thick, transform canvas={xshift=0.15em, yshift=0.2em}, red] (4) -- node[label={[label distance=-0.3em]90:$v_{*}$}] {} (6);
    \draw[ultra thick, transform canvas={xshift=-0em, yshift=-0.2em}, blue] (4) -- node[label=below:$v_{**}$, xshift=-0.15cm, yshift=0.1cm] {} (6);
    
    \draw[ultra thick, transform canvas={xshift=-0.15em, yshift=0.21em}, blue] (6) -- node[label={[label distance=-0.3em]90:$w_{*}$}] {} (9);
    \draw[ultra thick, transform canvas={xshift=0.15em, yshift=-0.21em}, red] (6) -- node[label=below:$w_{**}$,xshift=0.15cm, yshift=0.1cm] {} (9);
\end{scope}

%%%%%% Uncolored
\begin{scope}[scale=0.75, every node/.style={font=\small}, xshift=8cm, yshift=4cm] 
%\draw[step=1, lightgray] (0,0) grid (8,8);

\node (4) at (1,3) {};
\node (6) at (2.5,2) {};
\node (8) at (0,4) {};
\node (1) at (0,2.2) {};
\node (7) at (3.5,1) {};
\node (2) at (4.7,1) {};
\node (9) at (3.8,3) {};
\node (5) at (5,2.2) {};
\node (3) at (5,4) {};

\draw[ultra thick, black] ($(4.center) + (0,1mm)$) -- (8);
\draw[ultra thick, black] (1) -- ($(4.center) + (0,-1mm)$);
\draw[ultra thick, black] (6) -- (7);
\draw[ultra thick, black] (2) -- ($(7.center) + (0,1mm)$);
\draw[ultra thick, black] (3) -- ($(9.center) + (0,1mm)$);
\draw[ultra thick, black] (5) -- ($(9.center) + (0,-1mm)$);

%\draw[ultra thick, black] (6) -- node[label=below:$v'$, 
\draw[ultra thick, transform canvas={xshift=0.15em, yshift=0.2em}, black] (4) -- node {} (6);
\draw[ultra thick, transform canvas={xshift=-0em, yshift=-0.2em}, black] (4) -- node {} (6);

\draw[ultra thick, transform canvas={xshift=-0.15em, yshift=0.21em}, black] (6) -- node {} (9);
\draw[ultra thick, transform canvas={xshift=0.15em, yshift=-0.21em}, black] (6) -- node {} (9);
\end{scope}

%%%%%% Annotations
\node (adjDD) at (4.7,5) {$\EE\insquare{\mathsf{adjDD}}$};

\draw[->] (5,4.2) to [bend right=16] (2,3);
\draw[->] (7.2,4) to [bend right=20] (6.2,3);
\draw[->] (10,4) to [bend left=20] (11,3);
\draw[->] (12.5,4.2) to [bend left=16] (15,3);

% Define two points to underbrace
\coordinate (combAdjLeft) at (0,0);
\coordinate (combAdjRight) at (8,0);

% Draw the underbrace
\draw [decorate,decoration={brace,mirror,amplitude=6pt}, thick] (combAdjLeft) -- (combAdjRight)
    node [midway,below=1em] {Cancelled by \texttt{combAdjDD}};
% Optional: Add some text above the underbrace for context
%\node at (0,0.5) {Text 1};
%\node at (3,0.5) {Text 2};

% Define two points to underbrace
\coordinate (DetermPartLeft) at (9,0);
\coordinate (DetermPartRight) at (17,0);

% Draw the underbrace
\draw [decorate,decoration={brace,mirror,amplitude=6pt}, forestgreen, ultra thick] (DetermPartLeft) -- (DetermPartRight)
    node [midway,below=1em] {\textcolor{forestgreen}{\textbf{Deterministic part of RHS of \eqref{eq:MatchingEdgeDensity_LLdistributionUnderNull}!}}};
\end{tikzpicture}
}
\caption{A cartoon of the bi-colorings in \eqref{eq:adjDD_final_cancellations}. The above depicts one representative summand of each of the four sums. Here, $v_*$ and $v_{**}$ are always fixed to be red and blue respectively. The surviving contribution to the RHS of \eqref{eq:MatchingEdgeDensity_LLdistributionUnderNull} consists of those terms with a ``repeated wedge'' formed by superimposing two simple trees as indicated in the bottom right of the figure.}
\label{fig:adjDDCartoon}
\end{figure}
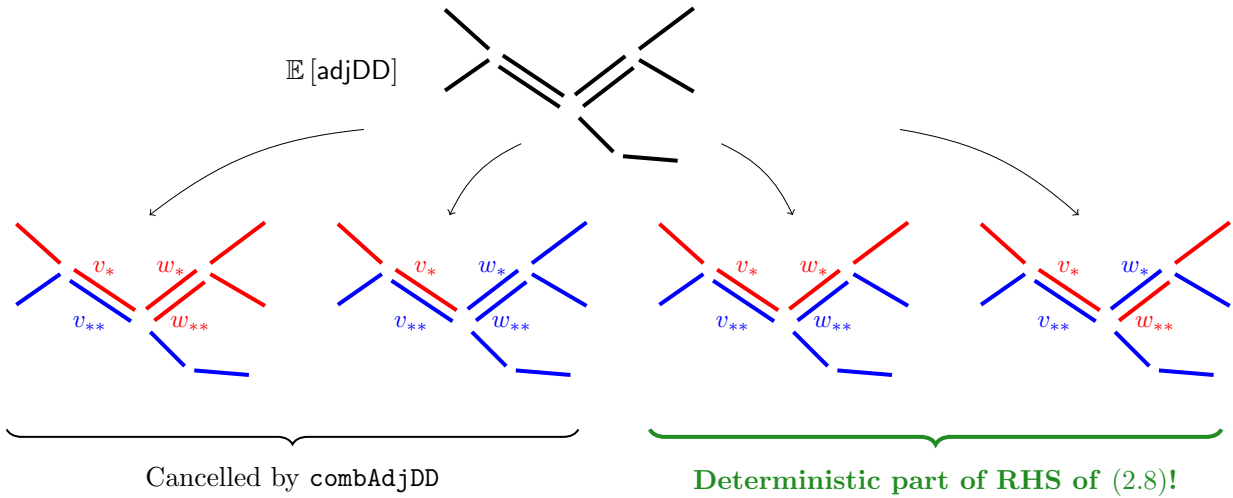

% We can now formally write the proof of Proposition \ref{prop:Matching_meanPart_mainResult}.

\begin{proof}[Proof of Lemma \ref{lemma:tripleEdge_cancels_logprefactorO1}]
    Using \eqref{eq:fallingfac_to_power_approx} and arguing similar to the proof of Lemma \ref{lemma:ProjP2simpleTrees_is_fluctuationPart}, we have 
    \begin{align}
        \EE\insquare{\tripleEdge} \sim \frac{1}{nq^2} \sum_{m = 1}^{2\log n -2} (n\lambda)^{m+2}  \sum_{\TmTripleEdge} \frac{(m+2)!}{3!} \frac{\phi(H(\TmTripleEdge)) }{\aut(\TmTripleEdge)} .
        \label{eq:tripleEdge_cancels_logprefactorO1_LHS_simp}
    \end{align}
    On the other hand, with \eqref{eq:TmRep_meanMDSquare_Relation} to substitute $(\EE\abs{M})^2$ and Proposition~\ref{prop:GibbsMatching_mean_var_from_CE}, \eqref{eq:EoneRepTrees_is_meanSquare} to substitute $\EE\abs{M}$, and similar approximations as in the proof of Lemma \ref{lemma:ProjP2simpleTrees_is_fluctuationPart}, we obtain 
    \begin{align}
        &-F_3 \sim \frac{1}{6n q^2} \inparen{ \frac{2 \EE\abs{M}}{n}  }^3  \sim \frac{4}{3n^4q^2} \inparen{\EE\abs{M}}^2 \inparen{\EE\abs{M}} \nonumber \\
        &\, \sim \frac{4}{3n^4 q^2} \inparen{ -n^2  \sum_{m=1}^{2 \log n -1} (n\lambda)^{m+1}  \sum_{\TmRep} \frac{(m+1)!}{2} \frac{\phi(H(\TmRep))}{\aut(\TmRep)}  } \inparen{ n \sum_{m=1}^{2 \log n} (n\lambda)^{m}  \sum_{T_m} m! m  \frac{\phi(H(T_m))}{\aut(T_m)}  } \nonumber \\
        &\, \sim -\frac{2}{3nq^2} \sum_{m = 1}^{2 \log n - 2} (n\lambda)^{m+2} \sum_{\ell = 1}^{m} \sum_{\inparen{T_\ell^{\text{rep}}, T_{m+1-\ell}} }  (m+1-\ell) \frac{\widetilde{\phi}(H(T_\ell^{\text{rep}})) \widetilde{\phi}(H(T_{m+1-\ell})) }{ \aut(T_\ell^{\text{rep}}) \aut(T_{m+1-\ell}) }.
        \label{eq:tripleEdge_cancels_logprefactorO1_RHS_simp}
    \end{align}
    Comparing \eqref{eq:tripleEdge_cancels_logprefactorO1_LHS_simp} and \eqref{eq:tripleEdge_cancels_logprefactorO1_RHS_simp}, we see that to prove \eqref{eq:tripleEdge_cancels_logprefactorO1}, it suffices to show, for every $1 \leq m \leq 2 \log n - 2$, that \eqref{eq:tripleEdge_cancels_logprefactorO1_fix_m} holds. 
    Therefore, the proof is complete.
\end{proof}

\begin{proof}[Proof of Claim \ref{claim:comb_split_adjDD_sepDD}]
    From Proposition \ref{prop:GibbsMatching_mean_var_from_CE} we have 
    \begin{align*}
        c_n \sim \frac{2}{n} \sum_{m=1}^{2 \log n} \sum_{T_m} m! \phi(H(T_m)) \frac{(n)_{m+1}}{\aut(T_m)} \lambda^m m + O\!\inparen{\frac{1}{n}}.
    \end{align*}
    Plugging this into $\combined$ \eqref{eq:combine_secondLogPrefactor_Esimp_EoneRep}, we have, hiding the lower order terms and expanding the sum:
    \begin{align*}
        &\combined \\
        &\sim \frac{1}{n^2 q^2} \inparen{ \sum_{m=1}^{2 \log n} \sum_{T_m} m! \phi(H(T_m)) \frac{(n)_{m+1}}{\aut(T_m)} \lambda^m m } \inparen{ \sum_{\TmRep} \frac{(m+1)!}{2} \phi(H(\TmRep))  \frac{(n)_{m+1}}{\aut(\TmRep)} \lambda^{m+1} (m-1)   } \\
        &= X^{\mathsf{comb}}_{\leq 2 \log n} + X^{\mathsf{comb}}_{> 2 \log n},
    \end{align*}
    where 
    \begin{align*}
        X^{\mathsf{comb}}_{\leq 2 \log n} &:= \frac{1}{2nq^2} \sum_{m=2}^{2 \log n - 2} (n\lambda)^{m+2} \sum_{\ell = 2}^{2 \log n - 1} \sum_{(T_\ell^{\text{rep}}, T_{m+1-\ell})} \frac{(\ell - 1)\widetilde{\phi}(H(T_\ell^{\text{rep}}))}{\aut(T_\ell^{\text{rep}})}  \frac{(m + 1 -\ell)\widetilde{\phi}(H(T_{m+1 -\ell}))}{\aut(T_{m+1-\ell})}, \\
        X^{\mathsf{comb}}_{> 2 \log n} &:= \frac{1}{2nq^2} \sum_{m=2 \log n - 1}^{4 \log n - 2} (n\lambda)^{m+2} \sum_{\ell = 2 \vee m+2 - 2 \log n}^{m \wedge 2 \log n} \sum_{(T_\ell^{\text{rep}}, T_{m+1-\ell})} \frac{(\ell - 1)\widetilde{\phi}(H(T_\ell^{\text{rep}}))}{\aut(T_\ell^{\text{rep}})}  \frac{(m + 1 -\ell)\widetilde{\phi}(H(T_{m+1 -\ell}))}{\aut(T_{m+1-\ell})}.
    \end{align*}
    We claim that 
    \begin{equation*}
        X^{\mathsf{comb}}_{> 2 \log n} = O\!\inparen{  \frac{C(\log n) \exp O\!\inparen{\frac{(\log n)^2}{n}}}{n^3 q^2}  } =  O\!\inparen{ \frac{\log n}{n^2}  }.
    \end{equation*}
    This follows by straightforward modifications of the arguments of Claim (iii) in the proof of Lemma \ref{lemma:E_oneRepTrees_mainResult}. On the other hand, rewrite the innermost sum in $X^{\mathsf{comb}}_{\leq 2 \log n}$ as a sum over colored, labeled pairs $\inparen{ \widetilde{T}_{\text{red}}^{\text{rep}}(v_*), \widetilde{T}_{\text{blue}}(v_{**}) }$ as in \eqref{eq:combined_genericPair}. The number of such pairs that can be generated from a single unlabeled, uncolored pair $(T_\ell^{\text{rep}}, T_{m+1-\ell})$ is 
    \begin{align*}
        \binom{m+1}{\ell + 1} \frac{(\ell + 1)!}{\aut(T_\ell^{\text{rep}})} (\ell - 1) (m+1-\ell) \cdot 2 \cdot \frac{(m-\ell)!}{\aut(T_{m+1-\ell})}
= \frac{2 (\ell - 1) (m+1-\ell) \cdot (m + 1)!}{\aut(T_\ell^{\text{rep}}) \aut(T_{m+1-\ell})} .
    \end{align*}
    Scaling $X^{\mathsf{comb}}_{\leq 2 \log n}$ appropriately by this combinatorial factor, we obtain 
    \begin{align*}
        \combined \sim \frac{1}{2nq^2} \sum_{m=2}^{2 \log n - 2} (n\lambda)^{m+2} \sum_{    \inparen{ \widetilde{T}_{\text{red}}^{\text{rep}}(v_*), \widetilde{T}_{\text{blue}}(v_{**})    }  } \frac{1}{2(m+1)!} \widetilde{\phi}(H(\widetilde{T}_{\text{red}}^{\text{rep}}) ) \widetilde{\phi}(H(\widetilde{T}_{\text{blue}} ) ).
    \end{align*}
    It remains to organize the sum into two terms: one collecting the pairs $\inparen{ \widetilde{T}_{\text{red}}^{\text{rep}}(v_*), \widetilde{T}_{\text{blue}}(v_{**}) }$ whose join gives a tree with adjacent twice repeated edges, and the other with separated twice repeated edges.
\end{proof}

\begin{proof}[Proof of Lemma \ref{lemma:sepDD_combSepDD_cancel}] 
    We will show that 
    \begin{align}\label{eq:sepDD_combSepDD_cancel_goal}
        \EE\insquare{\sepDD} \sim - \combSepDD.
    \end{align}
    Similar to \eqref{eq:adjDD_combAdjDD_give_meanPart_LHSRewrite}, we can rewrite the LHS of \eqref{eq:sepDD_combSepDD_cancel_goal} as 
    \begin{align}\label{eq:sepDD_combSepDD_cancel_LHSRewrite}
        \EE\insquare{\sepDD} \sim \frac{1}{nq^2} \sum_{m=3}^{2 \log n - 2} (n\lambda)^{m+2} \sum_{\widetilde{\TmSepDD}} \frac{m+2}{4} \phi(H(\widetilde{\TmSepDD})),
    \end{align}
    where the sum ranges over vertex-labeled (labels in $[m+1]$) trees $\widetilde{\TmSepDD}$ which have exactly two separated twice repeated edges. 
    
    Comparing \eqref{eq:sepDD_combSepDD_cancel_LHSRewrite} and \eqref{eq:combSepDD_definition}, we see that it suffices to show the following. For fixed $3 \leq m \leq 2 \log n - 2$, for fixed $\widetilde{\TmSepDD}$, let the two sets of repeated edges be $(v_*, v_{**})$ and $(w_{*}, w_{**})$. Then showing \eqref{eq:sepDD_combSepDD_cancel_goal} is equivalent to showing 
    \begin{align}
        \widetilde{\phi}(H(\widetilde{\TmSepDD})) &= - \sum_{ \substack{   \inparen{ \widetilde{T}_{\text{red}}^{\text{rep}}(v_*), \widetilde{T}_{\text{blue}}(v_{**})    }  \\ \cong \widetilde{\TmSepDD} } } \widetilde{\phi}(H(\widetilde{T}_{\text{red}}^{\text{rep}}) ) \widetilde{\phi}(H(\widetilde{T}_{\text{blue}} ) ) \nonumber \\
        &\qquad\qquad   - \sum_{ \substack{   \inparen{ \widetilde{T}_{\text{red}}^{\text{rep}}(w_*), \widetilde{T}_{\text{blue}}(w_{**})    }  \\ \cong   \widetilde{\TmSepDD} } } \widetilde{\phi}(H(\widetilde{T}_{\text{red}}^{\text{rep}}) ) \widetilde{\phi}(H(\widetilde{T}_{\text{blue}} ) ) ,
        \label{eq:sepDD_combSepDD_cancel_goal_fixm_fixTmSepDD}
    \end{align}
    where the sum constraint means that an uncolored version of the join of $\inparen{ \widetilde{T}_{\text{red}}^{\text{rep}}(\cdot), \widetilde{T}_{\text{blue}}(\cdot)    }$ by superimposing on their distinguished edges is isomorphic to $\widetilde{\TmSepDD}$. We claim that 
    \begin{equation}\label{eq:sepDD_combSepDD_cancel_vStar}\leqnomode\tag{\text{Case: }$v_* \text{ red}, v_{**} \text{ blue}$}
        \begin{split}
        \frac{1}{2} \widetilde{\phi}(H(\widetilde{\TmSepDD})) &= - \sum_{ \substack{   \inparen{ \widetilde{T}_{\text{red}}^{\text{rep}}(v_*), \widetilde{T}_{\text{blue}}(v_{**})    }  \\ \cong \widetilde{\TmSepDD} } } \widetilde{\phi}(H(\widetilde{T}_{\text{red}}^{\text{rep}}) ) \widetilde{\phi}(H(\widetilde{T}_{\text{blue}} ) ) ,
        \end{split}
    \end{equation}
    and 
    \begin{equation}\label{eq:sepDD_combSepDD_cancel_wStar}\leqnomode\tag{\text{Case: }$w_* \text{ red}, w_{**} \text{ blue}$}
        \begin{split}
        \frac{1}{2} \widetilde{\phi}(H(\widetilde{\TmSepDD})) &= - \sum_{ \substack{   \inparen{ \widetilde{T}_{\text{red}}^{\text{rep}}(w_*), \widetilde{T}_{\text{blue}}(w_{**})    }  \\ \cong \widetilde{\TmSepDD} } } \widetilde{\phi}(H(\widetilde{T}_{\text{red}}^{\text{rep}}) ) \widetilde{\phi}(H(\widetilde{T}_{\text{blue}} ) ). 
        \end{split}
    \end{equation}
    These will prove \eqref{eq:sepDD_combSepDD_cancel_goal_fixm_fixTmSepDD} which will finish the proof.

\smallskip
\noindent
    \textbf{Proof of Equation \eqref{eq:sepDD_combSepDD_cancel_vStar}} 
    Let $v'$ be the unique edge incident to $v_{*}$ (and $v_{**}$) in $\widetilde{\TmSepDD}$ which connects 
    $v_*$ to $w_*$ 
    % the vertices incident to $w_{*}$ and $w_{**}$ to $v_{*}$ and $v_{**}$ 
    (e.g., Figure \ref{fig:equalAverageED_sepDD} (Left)). In what follows, fix $H = H(\widetilde{\TmSepDD})$ the incompatibility graph of $\widetilde{\TmSepDD}$. Define the following subset of bi-colorings of $V(H)$:
    \begin{align*}
        \cC(H; v'\text{~red}) := \inbraces{  (V_{r}, V_{b}) : \, 
        \begin{aligned}
            & V_{r} \cup V_{b} = V(H) \text{ disjoint}, \, V_{r} \ni v_*, v', \, V_{b} \ni v_{**},\\
            & H[V_{r}] \text{ and } H[V_{b}] \text{ are each connected subgraphs}
        \end{aligned}
        }.
    \end{align*}
    Define $\cC(H; v'\text{~blue})$ analogously. There is a bijection between the sets 
    \begin{align*}
        \cC(H; v'\text{~red}) \qquad & \text{and}  \qquad \inbraces{\inparen{ \widetilde{T}_{\text{red}}^{\text{rep}}(v_*), \widetilde{T}_{\text{blue}}(v_{**})    }  \cong \widetilde{\TmSepDD} }.
    \end{align*}
    An example of such a bi-coloring in $\cC(H; v'\text{~red})$ is given in Figure \ref{fig:equalAverageED_sepDD} (Right). With $\cC(H; v_*, v_{**})$ defined in \eqref{eq:def-bi-coloring}, Equation \eqref{eq:sepDD_combSepDD_cancel_vStar} reduces to 
    \begin{align}
        \sum_{S \subseteq H(\text{conn.,spann.})} (-1)^{\abs{S}} &=  2 \sum_{(V_r, V_b) \in \cC(H; v'\text{~red})} \sum_{\substack{ S_r \subseteq H[V_r](\text{conn., spann.}) \\ S_b \subseteq H[V_b](\text{conn., spann.})  } } (-1)^{\abs{S_r} + \abs{S_b} + 1} \nonumber \\
        &= \sum_{(V_r, V_b) \in \cC(H;v'\text{~red})} (\cdots)  + \sum_{(V_r, V_b) \in \cC(H;v'\text{~blue})} (\cdots)\nonumber \\
        &= \sum_{(V_r, V_b) \in \cC(H; v_*, v_{**})} \sum_{\substack{ S_r \subseteq H[V_r](\text{conn., spann.}) \\ S_b \subseteq H[V_b](\text{conn., spann.})  } } (-1)^{\abs{S_r} + \abs{S_b} + 1} ,
    \end{align}
where for the second equality we have used symmetry. Equation \eqref{eq:sepDD_combSepDD_cancel_vStar} is then true by Lemma \ref{lemma:Ursell_trees_adjVertices_identity}.

\smallskip
\noindent
    \textbf{Proof of Equation \eqref{eq:sepDD_combSepDD_cancel_wStar}.} The argument is entirely analogous; we only have to switch the roles of $v_{*}$ and $v_{**}$ with those of $w_{*}$ and $w_{**}$. 
\end{proof}

%%%%%%%% COMBINING SEPDD FIGURES 
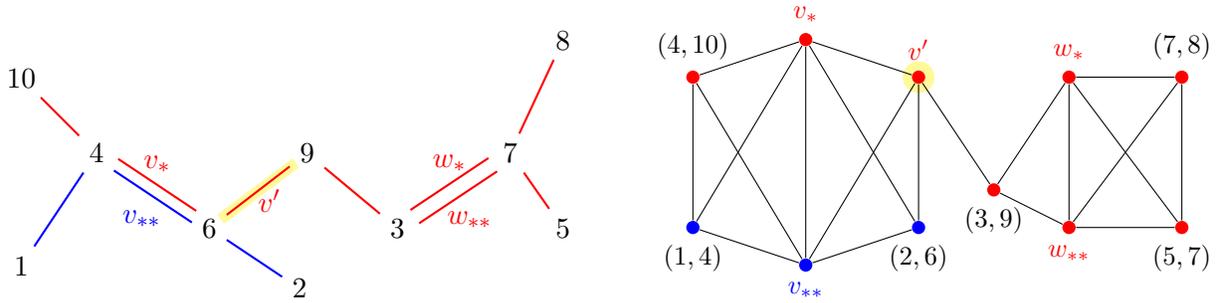
\begin{figure}[htp]\centering
%%%%%%%% (TRepRed(v*),Tblue(v**) %%%%%%%%
\begin{tikzpicture}
\tikzset{%
glow/.style={%
preaction={#1, draw, line join=round, line width=0.5pt, opacity=0.04,
preaction={#1, draw, line join=round, line width=1.0pt, opacity=0.04,
preaction={#1, draw, line join=round, line width=1.5pt, opacity=0.04,
preaction={#1, draw, line join=round, line width=2.0pt, opacity=0.04,
preaction={#1, draw, line join=round, line width=2.5pt, opacity=0.04,
preaction={#1, draw, line join=round, line width=3.0pt, opacity=0.04,
preaction={#1, draw, line join=round, line width=3.5pt, opacity=0.04,
preaction={#1, draw, line join=round, line width=4.0pt, opacity=0.04,
preaction={#1, draw, line join=round, line width=4.5pt, opacity=0.04,
preaction={#1, draw, line join=round, line width=5.0pt, opacity=0.04,
preaction={#1, draw, line join=round, line width=5.5pt, opacity=0.04,
preaction={#1, draw, line join=round, line width=6.0pt, opacity=0.04,
}}}}}}}}}}}}}}
%\draw[step=1, lightgray] (0,0) grid (8,8);

\draw[yellow, opacity=0.4, line width=0.15cm] (2.65,2.12) -- (3.65,2.87);

\node (4) at (1,3) {$4$};
\node (6) at (2.5,2) {$6$};
\node (10) at (0,4) {$10$};
\node (1) at (0,1.5) {$1$};
\node (2) at (3.7,1.2) {$2$};
\node (9) at (3.8,3) {$9$};
\node (3) at (5,2) {$3$};
\node (7) at (6.5,3) {$7$};
\node (8) at (7.2,4.5) {$8$};
\node (5) at (7.2,2) {$5$};

\draw[thick, red] (4) -- (10);
%\draw[thick, red, glow=green] (6) -- node[label=below:$v'$, xshift=0.15cm, yshift=0.3cm] {} (9);
\draw[thick, red] (6) -- node[label=below:$v'$, xshift=0.15cm, yshift=0.3cm] {} (9);
\draw[thick, red] (3) -- (9);
\draw[thick, red] (5) -- (7);
\draw[thick, red] (7) -- (8);
\draw[thick, blue] (1) -- (4);
\draw[thick, blue] (2) -- (6);

\draw[thick, transform canvas={xshift=0.15em, yshift=0.2em}, red] (4) -- node[label={[label distance=-0.3em]90:$v_{*}$}] {} (6);
\draw[thick, transform canvas={xshift=-0em, yshift=-0.2em}, blue] (4) -- node[label=below:$v_{**}$, xshift=-0.15cm, yshift=0.1cm] {} (6);
\draw[thick, transform canvas={xshift=-0.15em, yshift=0.23em}, red] (3) -- node[label={[label distance=-0.3em]90:$w_{*}$}] {} (7);
\draw[thick, transform canvas={xshift=0.15em, yshift=-0.23em}, red] (3) -- node[label=below:$w_{**}$,xshift=0.15cm, yshift=0.1cm] {} (7);
\end{tikzpicture}
\qquad % <----------------- SPACE BETWEEN PICTURES
%%%%%%%% H(TRepRed(v*),Tblue(v**)) with v3 %%%%%%%%
\begin{tikzpicture}[every node/.style={font=\small}, redCirc/.style={circle,fill=red, minimum size=5pt, inner sep=0pt}, blueCirc/.style={circle,fill=blue, minimum size=5pt, inner sep=0pt}]
\tikzset{%
glow/.style={%
preaction={#1, draw, line join=round, line width=0.5pt, opacity=0.04,
preaction={#1, draw, line join=round, line width=1.0pt, opacity=0.04,
preaction={#1, draw, line join=round, line width=1.5pt, opacity=0.04,
preaction={#1, draw, line join=round, line width=2.0pt, opacity=0.04,
preaction={#1, draw, line join=round, line width=2.5pt, opacity=0.04,
preaction={#1, draw, line join=round, line width=3.0pt, opacity=0.04,
preaction={#1, draw, line join=round, line width=3.5pt, opacity=0.04,
preaction={#1, draw, line join=round, line width=4.0pt, opacity=0.04,
preaction={#1, draw, line join=round, line width=4.5pt, opacity=0.04,
preaction={#1, draw, line join=round, line width=5.0pt, opacity=0.04,
preaction={#1, draw, line join=round, line width=5.5pt, opacity=0.04,
preaction={#1, draw, line join=round, line width=6.0pt, opacity=0.04,
}}}}}}}}}}}}}}
%\draw[step=1, lightgray] (0,0) grid (8,8);

\node[circle, fill=yellow, opacity=0.4, minimum size=12pt,inner sep=0pt] at (3.5,3) {};

\node[redCirc,label=above:{\textcolor{red}{$v_*$}}] (vStar) at (2,3.5) {};
\node[blueCirc,label=below:{\textcolor{blue}{$v_{**}$}}] (vStarStar) at (2,0.5) {};

\node[blueCirc, label=below:{$(1,4)$}] (1-4) at (0.5,1) {};
\node[blueCirc, label=below:{$(2,6)$}] (2-6) at (3.5,1) {};
\node[redCirc, label=above:{$(4,10)$}] (4-10) at (0.5,3) {};
%\node[redCirc, label=above:{\textcolor{red}{$v'$}}, glow=green] (vPrime) at (3.5,3) {};
\node[redCirc, label=above:{\textcolor{red}{$v'$}}] (vPrime) at (3.5,3) {};
\node[redCirc, label=below:{$(3,9)$}] (3-9) at (4.5,1.5) {};

\node[redCirc, label=above:{\textcolor{red}{$w_*$}}] (wStar) at (5.5,3) {};
\node[redCirc, label=below:{\textcolor{red}{$w_{**}$}}] (wStarStar) at (5.5,1) {};
\node[redCirc, label=above:{$(7,8)$}] (7-8) at (7,3) {};
\node[redCirc, label=below:{$(5,7)$}] (5-7) at (7,1) {};

\draw (vPrime) -- (vStar);
\draw (vPrime) -- (vStarStar);
\draw (vStar) -- (vStarStar);

\draw (1-4) -- (4-10);
\draw (1-4) -- (vStar);
\draw (1-4) -- (vStarStar);
\draw (4-10) -- (vStar);
\draw (4-10) -- (vStarStar);
\draw (2-6) -- (vStar);
\draw (2-6) -- (vStarStar);
\draw (2-6) -- (vPrime);
\draw (3-9) -- (vPrime);
\draw (3-9) -- (wStar);
\draw (3-9) -- (wStarStar);
\draw (wStar) -- (wStarStar);
\draw (7-8) -- (wStar);
\draw (7-8) -- (wStarStar);
\draw (5-7) -- (wStar);
\draw (5-7) -- (wStarStar);
\draw (5-7) -- (7-8);
\end{tikzpicture}
\caption{(Left) A joined tree represented by the tuple \eqref{eq:combined_genericPair} that is in the set $\sepDD$. The unique edge $v'$ that is adjacent to both $v_*$ and $v_{**}$ that connects between the separated repeated edges is highlighted in yellow. (Right) The corresponding incompatibility graph $H$.}
\label{fig:equalAverageED_sepDD}
\end{figure}

\begin{proof}[Proof of Lemma \ref{lemma:adjDD_combAdjDD_give_meanPart}] 
    We will show that 
    \begin{align}
        \EE\insquare{\adjDD} \sim -\combAdjDD -\frac{c_n^4}{4nq^2}.
        \label{eq:adjDD_combAdjDD_give_meanPart_Goal}
    \end{align}
    By Taylor expanding $(q/p)^{m+2}$ in the LHS of \eqref{eq:adjDD_combAdjDD_give_meanPart_Goal} as in \eqref{eq:identities_pq_(q/p)^r}, approximating $(n)_{m+1} \simeq n^{m+1}$ as in \eqref{eq:fallingfac_to_power_approx}, and rewriting the sum over vertex-labeled  trees $\widetilde{\TmAdjDD}$ (with labels in $[m+1]$) as in \eqref{eq:tripleEdge_cancels_logprefactorO1_fix_m}, we have 
    \begin{equation}\label{eq:adjDD_combAdjDD_give_meanPart_LHSRewrite}
        \EE\insquare{\adjDD} \sim \frac{1}{nq^2} \sum_{m=2}^{2 \log n - 2} (n\lambda)^{m+2} \sum_{\widetilde{\TmAdjDD}} \frac{m+2}{4} \phi(H(\widetilde{\TmAdjDD})).
    \end{equation}
    On the other hand, by the proof of Lemma \ref{lemma:ProjP2simpleTrees_is_fluctuationPart}, we have the identity
    \begin{align*}
        - \frac{2}{n^2} \inparen{\EE \abs{M}}^2 \sim \sum_{m=2}^{2 \log n} (n\lambda)^m \sum_{T_m} m! \phi(H(T_m)) \frac{\gamma(T_m)}{\aut(T_m)}.
    \end{align*}
    Therefore, by expanding the square, 
    % the RHS of \eqref{eq:adjDD_combAdjDD_give_meanPart_Goal} satisfies
    we obtain 
    \begin{align}
        - \frac{c_n^4}{ 4 nq^2} \sim - \frac{1}{n q^2} \inparen{ \sum_{m=2}^{2 \log n} (n\lambda)^m \sum_{T_m} m! \phi(H(T_m)) \frac{\gamma(T_m)}{\aut(T_m)}  }^2 := X^{\textcypr{\Cpi}}_{\leq 2 \log n} + X^{\textcypr{\Cpi}}_{> 2 \log n},
        \label{eq:adjDD_combAdjDD_give_meanPart_Goal_cn4ExpandSquare}
    \end{align}
    where 
    \begin{align*}
        X^{\textcypr{\Cpi}}_{\leq 2 \log n} &:= - \frac{1}{nq^2} \sum_{m=2}^{2 \log n - 2} (n\lambda)^{m+2} \sum_{\ell = 2}^{m} \sum_{(T_\ell, T_{m+2-\ell})} \frac{\widetilde{\phi}(H(T_{\ell})) \gamma(T_{\ell}) }{\aut(T_{\ell})}  \frac{\widetilde{\phi}(H(T_{m + 2 - \ell})) \gamma(T_{m + 2 - \ell}) }{\aut(T_{m + 2 - \ell})} , \\
        X^{\textcypr{\Cpi}}_{> 2 \log n} &:= - \frac{1}{nq^2} \sum_{m=2 \log n - 1}^{4 \log n - 2} (n\lambda)^{m+2} \sum_{\ell = 2 \vee m + 2 - 2 \log n}^{m \wedge 2 \log n} \sum_{(T_\ell, T_{m+2-\ell})} \frac{\widetilde{\phi}(H(T_{\ell})) \gamma(T_{\ell}) }{\aut(T_{\ell})}  \frac{\widetilde{\phi}(H(T_{m + 2 - \ell})) \gamma(T_{m + 2 - \ell}) }{\aut(T_{m + 2 - \ell})}.
    \end{align*}
    We claim that
    \begin{equation}
        X^{\textcypr{\Cpi}}_{> 2 \log n} = O\!\inparen{  \frac{(\log n)^2 \exp O\!\inparen{\frac{(\log n)^2}{n}}}{n^3 q^2}  } =  O\!\inparen{ \frac{(\log n)^2}{n^2}  }. \label{eq:adjDD_combAdjDD_give_meanPart_Goal_XBigger2logn_small}
    \end{equation}
    This follows by a straightforward modification of the arguments of Claim (iii) in the proof of Lemma \ref{lemma:E_oneRepTrees_mainResult}. Here we additionally use the bound $\gamma(T_m) \leq \binom{m}{2}$.

    For each $m$, rewrite the sum over pairs of unlabeled trees in $X^{\textcypr{\Cpi}}_{\leq 2 \log n}$ as a sum over pairs of trees generically denoted by
    \begin{equation}\label{eq:equalAverageED_determPart_genericTuple}
        \inparen{  \TmLabRed(u_r^{(1)}, u_r^{(2)} ), \, \TmLabBlue(u_b^{(1)}, u_b^{(2)} ) },
    \end{equation}
    satisfying the following: 
    \begin{itemize}
        \item $\TmLabRed(u_r^{(1)}, u_r^{(2)} )$ and $\TmLabBlue(u_b^{(1)}, u_b^{(2)} )$ are vertex-labeled, colored trees with two distinguished vertices indicated in parentheses. 
        \item The (unique) paths between the distinguished vertices $(u_r^{(1)}, u_r^{(2)})$ and $(u_b^{(1)}, u_b^{(2)})$ each form a $P_2$ in $\TmLabRed(u_r^{(1)}, u_r^{(2)} )$ and $\TmLabBlue(u_b^{(1)}, u_b^{(2)} )$ respectively. These are referred to as $P_2$ decorations.
        \item Say $u_r^{(1)} \text{---} u_r^{\text{join}} \text{---} u_r^{(2)}$ and $u_b^{(1)} \text{---} u_b^{\text{join}} \text{---} u_b^{(2)}$ are the $P_2$ decorations. Then the labels of $u_r^{\text{join}}$ and $u_b^{\text{join}}$ must coincide. The label sets of $(u_r^{(1)}, u_r^{(2)})$ and $(u_b^{(1)}, u_b^{(2)})$ must also coincide.
        \item Joining the two trees by superimposing on $P_2$ decorations (matching the vertex labels) gives a vertex-labeled tree with two adjacent twice repeated edges, with vertex labels in $[m+1]$. 
    \end{itemize}
    Figure \ref{fig:equalAverageED_DeterministicPartAdjDD} (left) gives an example of such a tuple \eqref{eq:equalAverageED_determPart_genericTuple}. The number of such pairs \eqref{eq:equalAverageED_determPart_genericTuple} that can be generated from a single $(T_\ell, T_{m+2-\ell})$ pair is 
    \begin{align*}
        \binom{m+1}{\ell + 1} \frac{(\ell+1)!}{\aut(T_\ell)} \gamma(T_\ell) \cdot  \gamma(T_{m+2-\ell}) \cdot 2 \cdot \frac{(m-\ell)!}{\aut(T_{m+2-\ell})}
        = \frac{2 \cdot (m+1)! \gamma(T_\ell) \gamma(T_{m+2-\ell})}{\aut(T_\ell) \aut(T_{m+2-\ell})} .
    \end{align*}
    The factor 2 arises because there are two ways to align the labels of $(u_r^{(1)}, u_r^{(2)} )$ and $(u_b^{(1)}, u_b^{(2)} )$. Therefore, scaling $X^{\textcypr{\Cpi}}_{\leq 2 \log n}$ by this combinatorial factor, and using \eqref{eq:adjDD_combAdjDD_give_meanPart_Goal_XBigger2logn_small}, we have from \eqref{eq:adjDD_combAdjDD_give_meanPart_Goal_cn4ExpandSquare} that  
    \begin{align}\label{eq:adjDD_combAdjDD_give_meanPart_Goal_cn4ExpandSquare_simplified}
        - \frac{c_n^4}{ 4 nq^2} \sim - \frac{1}{nq^2} \sum_{m=2}^{2 \log n - 2} (n\lambda)^{m+2}  \sum_{\inparen{  \TmLabRed(u_r^{(1)}, u_r^{(2)} ), \, \TmLabBlue(u_b^{(1)}, u_b^{(2)} ) }} \frac{1}{2(m+1)!} \widetilde{\phi}(H(\TmLabRed)) \widetilde{\phi}(H(\TmLabBlue)),
    \end{align}
    where we have suppressed mention of the distinguished vertices whenever clear from context.

    From \eqref{eq:adjDD_combAdjDD_give_meanPart_LHSRewrite}, \eqref{eq:combAdjDD_definition}, and \eqref{eq:adjDD_combAdjDD_give_meanPart_Goal_cn4ExpandSquare_simplified}, we see that to prove \eqref{eq:adjDD_combAdjDD_give_meanPart_Goal}, it suffices to show the following. Fix an $2 \leq m \leq 2 \log n - 2$ and fix a $\widetilde{\TmAdjDD}$. Denote the two sets of repeated edges in $\widetilde{\TmAdjDD}$ by $(v_{*}, v_{**})$ and $(w_{*}, w_{**})$. Then showing \eqref{eq:adjDD_combAdjDD_give_meanPart_Goal} is equivalent to showing
    % \begin{align}
    %     \sum_{\widetilde{\TmAdjDD}} \widetilde{\phi}(H(\widetilde{\TmAdjDD})) &= - 2 \sum_{\inparen{  \TmLabRed(u_r^{(1)}, u_r^{(2)} ), \, \TmLabBlue(u_b^{(1)}, u_b^{(2)} ) }} \widetilde{\phi}(H(\TmLabRed)) \widetilde{\phi}(H(\TmLabBlue)) \nonumber \\
    %     &\qquad\qquad\qquad\qquad\qquad   - \sum_{ \substack{   \inparen{ \widetilde{T}_{\text{red}}^{\text{rep}}(v_*), \widetilde{T}_{\text{blue}}(v_{**})    }  \\ \in \,  \adjDD }} \widetilde{\phi}(H(\widetilde{T}_{\text{red}}^{\text{rep}}) ) \widetilde{\phi}(H(\widetilde{T}_{\text{blue}} ) ) \label{eq:adjDD_combAdjDD_give_meanPart_Goal_fix_m}
    % \end{align}
    % In what follows, fix a $\widetilde{\TmAdjDD}$. Call the two sets of repeated edges $(v_{*}, v_{**})$ and $(w_{*}, w_{**})$. Then showing \eqref{eq:adjDD_combAdjDD_give_meanPart_Goal_fix_m} further reduces to showing 
    \begin{align}
        \widetilde{\phi}(H(\widetilde{\TmAdjDD})) &= - 2 \sum_{\substack{ \inparen{  \TmLabRed(u_r^{(1)}, u_r^{(2)} ), \, \TmLabBlue(u_b^{(1)}, u_b^{(2)} ) } \\ \cong \widetilde{\TmAdjDD} }  } \widetilde{\phi}(H(\TmLabRed)) \widetilde{\phi}(H(\TmLabBlue)) \nonumber \\
        &\qquad\qquad\qquad\qquad\qquad   - \sum_{ \substack{   \inparen{ \widetilde{T}_{\text{red}}^{\text{rep}}(v_*), \widetilde{T}_{\text{blue}}(v_{**})    }  \\ \cong \widetilde{\TmAdjDD} } } \widetilde{\phi}(H(\widetilde{T}_{\text{red}}^{\text{rep}}) ) \widetilde{\phi}(H(\widetilde{T}_{\text{blue}} ) ) \nonumber \\
        &\qquad\qquad\qquad\qquad\qquad   - \sum_{ \substack{   \inparen{ \widetilde{T}_{\text{red}}^{\text{rep}}(w_*), \widetilde{T}_{\text{blue}}(w_{**})    }  \\ \cong   \widetilde{\TmAdjDD} } } \widetilde{\phi}(H(\widetilde{T}_{\text{red}}^{\text{rep}}) ) \widetilde{\phi}(H(\widetilde{T}_{\text{blue}} ) ) .
        \label{eq:adjDD_combAdjDD_give_meanPart_Goal_fix_m_fix_TmAdjDD}
    \end{align}
    To clarify, the sum constraint in the first term on the RHS of \eqref{eq:adjDD_combAdjDD_give_meanPart_Goal_fix_m_fix_TmAdjDD} means that an uncolored version of the join of $\inparen{  \TmLabRed(u_r^{(1)}, u_r^{(2)} ), \, \TmLabBlue(u_b^{(1)}, u_b^{(2)} ) }$ formed by superimposing their $P_2$ decorations is isomorphic to $\widetilde{\TmAdjDD}$. The sum constraint in the second term on the RHS of \eqref{eq:adjDD_combAdjDD_give_meanPart_Goal_fix_m_fix_TmAdjDD} means that an uncolored version of the join of $\inparen{ \widetilde{T}_{\text{red}}^{\text{rep}}(v_*), \widetilde{T}_{\text{blue}}(v_{**})    }$ formed by superimposing the distinguished edges $v_{*}$ and $v_{**}$ is isomorphic to $\widetilde{\TmAdjDD}$. The last term is analogously defined. We will show 
    \begin{equation}\label{eq:adjDD_combAdjDD_give_meanPart_vStar}\leqnomode\tag{\text{Case: }$v_* \text{ red}, v_{**} \text{ blue}$}
        \begin{split}
        \frac{1}{2} \widetilde{\phi}(H(\widetilde{\TmAdjDD})) &= - \sum_{\substack{ \inparen{  \TmLabRed(u_r^{(1)}, u_r^{(2)} ), \, \TmLabBlue(u_b^{(1)}, u_b^{(2)} ) } \\ \cong \widetilde{\TmAdjDD} }  } \widetilde{\phi}(H(\TmLabRed)) \widetilde{\phi}(H(\TmLabBlue))  \\
        &\qquad    - \sum_{ \substack{   \inparen{ \widetilde{T}_{\text{red}}^{\text{rep}}(v_*), \widetilde{T}_{\text{blue}}(v_{**})    }  \\ \cong \widetilde{\TmAdjDD} } } \widetilde{\phi}(H(\widetilde{T}_{\text{red}}^{\text{rep}}) ) \widetilde{\phi}(H(\widetilde{T}_{\text{blue}} ) ) ,
        \end{split}
    \end{equation}
    %\reqnomode
    and 
    \begin{equation}\label{eq:adjDD_combAdjDD_give_meanPart_wStar}\leqnomode\tag{\text{Case: }$w_* \text{ red}, w_{**} \text{ blue}$}
        \begin{split}
        \frac{1}{2} \widetilde{\phi}(H(\widetilde{\TmAdjDD})) &= - \sum_{\substack{ \inparen{  \TmLabRed(u_r^{(1)}, u_r^{(2)} ), \, \TmLabBlue(u_b^{(1)}, u_b^{(2)} ) } \\ \cong \widetilde{\TmAdjDD} }  } \widetilde{\phi}(H(\TmLabRed)) \widetilde{\phi}(H(\TmLabBlue))  \\
        &\qquad    - \sum_{ \substack{   \inparen{ \widetilde{T}_{\text{red}}^{\text{rep}}(w_*), \widetilde{T}_{\text{blue}}(w_{**})    }  \\ \cong \widetilde{\TmAdjDD} } } \widetilde{\phi}(H(\widetilde{T}_{\text{red}}^{\text{rep}}) ) \widetilde{\phi}(H(\widetilde{T}_{\text{blue}} ) ). 
        \end{split}
    \end{equation}
    These will prove \eqref{eq:adjDD_combAdjDD_give_meanPart_Goal_fix_m_fix_TmAdjDD}, which will finish the proof.\\
    \textbf{Proof of equation \eqref{eq:adjDD_combAdjDD_give_meanPart_vStar}}. In what follows, fix $H = H(\widetilde{\TmAdjDD})$ the incompatibility graph of $\widetilde{\TmAdjDD}$.   Define the following subset of bi-colorings of $V(H)$:
    \begin{align*}
        \cC(H; w_*\text{~red}, w_{**}\text{~blue}) := \inbraces{  (V_{r}, V_{b}) : \, 
        \begin{aligned}
            & V_{r} \cup V_{b} = V(H) \text{ disjoint}, \, V_{r} \ni v_*, w_*, \, V_{b} \ni v_{**}, w_{**},\\
            & H[V_{r}] \text{ and } H[V_{b}] \text{ are each connected subgraphs}
        \end{aligned}
        }.
    \end{align*}
    Define $\cC(H; w_*\text{~blue}, w_{**} \text{~red})$, $\cC(H; w_* \text{~red}, w_{**} \text{~red})$, and $\cC(H; w_* \text{ ~blue}, w_{**} \text{~blue})$ analogously. Without loss of generality, suppose the $P_2$ decorations correspond to $v_*, v_{**}, w_*, w_{**}$ in the following way
    \begin{align}\label{eq:equalAverageED_determPart_identifyvStarvStarStar}
        \rlap{$\underbrace{\phantom{ u_r^{(1)} \text{---} u_r^{\text{join}}  } }_{= v_*}$} u_r^{(1)} \text{---} \overbrace{u_r^{\text{join}} \text{---} u_r^{(2)} }^{= w_*} \qquad \text{ and } \qquad \rlap{$\underbrace{\phantom{ u_b^{(1)} \text{---} u_b^{\text{join}}  } }_{= v_{**}}$} u_b^{(1)} \text{---} \overbrace{u_b^{\text{join}} \text{---} u_b^{(2)} }^{= w_{**}}.
    \end{align}
    There is a bijection between the sets 
    \begin{align*}
        \cC(H; w_*\text{~red}, w_{**}\text{~blue}) \qquad & \text{and}  \qquad \inbraces{\inparen{  \TmLabRed(u_r^{(1)}, u_r^{(2)} ), \, \TmLabBlue(u_b^{(1)}, u_b^{(2)} ) } \cong \widetilde{\TmAdjDD} }, \\
        \cC(H; w_* \text{~red}, w_{**} \text{~red}) \qquad & \text{and}  \qquad \inbraces{\inparen{ \widetilde{T}_{\text{red}}^{\text{rep}}(v_*), \widetilde{T}_{\text{blue}}(v_{**})    }  \cong \widetilde{\TmAdjDD} }.
    \end{align*}
    We refer to Figures \ref{fig:equalAverageED_CombAdjDD} and \ref{fig:equalAverageED_DeterministicPartAdjDD} for examples of such bijections. With $\cC(H; v_*, v_{**})$ defined in \eqref{eq:def-bi-coloring}, Equation \eqref{eq:adjDD_combAdjDD_give_meanPart_vStar} is equivalent to 
    \begin{align}
        \sum_{S \subseteq H(\text{conn.,spann.})} (-1)^{\abs{S}} &= 2 \sum_{(V_r, V_b) \in \cC(H; w_* \text{~red}, w_{**} \text{~blue})} \sum_{\substack{ S_r \subseteq H[V_r](\text{conn., spann.}) \\ S_b \subseteq H[V_b](\text{conn., spann.})  } } (-1)^{\abs{S_r} + \abs{S_b} + 1} \nonumber \\
        &\quad  + 2 \sum_{(V_r, V_b) \in \cC(H; w_* \text{~red}, w_{**} \text{~red})} \sum_{\substack{ S_r \subseteq H[V_r](\text{conn., spann.}) \\ S_b \subseteq H[V_b](\text{conn., spann.})  } } (-1)^{\abs{S_r} + \abs{S_b} + 1} \nonumber \\
        &= \sum_{(V_r, V_b) \in \cC(H; w_* \text{~red}, w_{**} \text{~blue})} (\cdots) + \sum_{(V_r, V_b) \in \cC(H; w_* \text{~blue}, w_{**} \text{~red})} (\cdots) \nonumber \\
        &\qquad \sum_{(V_r, V_b) \in \cC(H; w_* \text{~red}, w_{**} \text{~red})} (\cdots) + \sum_{(V_r, V_b) \in \cC(H; w_* \text{~blue}, w_{**} \text{~blue})} (\cdots) \label{eq:adjDD_final_cancellations}\\
        &= \sum_{(V_r, V_b) \in \cC(H; v_*, v_{**})} \sum_{\substack{ S_r \subseteq H[V_r](\text{conn., spann.}) \\ S_b \subseteq H[V_b](\text{conn., spann.})  } } (-1)^{\abs{S_r} + \abs{S_b} + 1} ,
        \nonumber
    \end{align}
where for the second equality we have used symmetry, recalling that $w_*$ and $w_{**}$ are the same repeated edge. Figure \ref{fig:adjDDCartoon} gives a cartoon of the different bi-colorings in \eqref{eq:adjDD_final_cancellations}. Then Equation \eqref{eq:adjDD_combAdjDD_give_meanPart_vStar} is true by Lemma \ref{lemma:Ursell_trees_adjVertices_identity}.
    
\smallskip
\noindent
   \textbf{Proof of equation \eqref{eq:adjDD_combAdjDD_give_meanPart_wStar}}. This argument is entirely analogous; we only have to swap the roles of $v_{*}$ and $v_{**}$ with those of $w_{*}$ and $w_{**}$. 
\end{proof}

%%%%%%%% COMBINING CombAdjDD FIGURES 
\begin{figure}[htp]\centering
%%%%%%%% (TRepRed(v*),Tblue(v**) adjDD %%%%%%%%
\begin{tikzpicture}
%\draw[step=1, lightgray] (0,0) grid (8,8);

\draw[yellow, opacity=0.4, line width=0.15cm] (2.65,2.25) -- (3.55,2.95);
\draw[yellow, opacity=0.4, line width=0.15cm] (2.75,2.05) -- (3.7,2.8);

\node (4) at (1,3) {$4$};
\node (6) at (2.5,2) {$6$};
\node (8) at (0,4) {$8$};
\node (1) at (0,2.2) {$1$};
\node (7) at (3.5,1) {$7$};
\node (2) at (4.7,1) {$2$};
\node (9) at (3.8,3) {$9$};
\node (5) at (5,2.2) {$5$};
\node (3) at (5,4) {$3$};

\draw[thick, red] (4) -- (8);
\draw[thick, blue] (1) -- (4);
\draw[thick, blue] (6) -- (7);
\draw[thick, blue] (2) -- (7);
\draw[thick, red] (3) -- (9);
\draw[thick, red] (5) -- (9);

%\draw[thick, red] (6) -- node[label=below:$v'$, 
\draw[thick, transform canvas={xshift=0.15em, yshift=0.2em}, red] (4) -- node[label={[label distance=-0.3em]90:$v_{*}$}] {} (6);
\draw[thick, transform canvas={xshift=-0em, yshift=-0.2em}, blue] (4) -- node[label=below:$v_{**}$, xshift=-0.15cm, yshift=0.1cm] {} (6);

\draw[thick, transform canvas={xshift=-0.15em, yshift=0.23em}, red] (6) -- node[label={[label distance=-0.3em]90:$w_{*}$}] {} (9);
\draw[thick, transform canvas={xshift=0.15em, yshift=-0.23em}, red] (6) -- node[label=below:$w_{**}$,xshift=0.15cm, yshift=0.1cm] {} (9);

\end{tikzpicture}
\qquad\qquad % <----------------- SPACE BETWEEN PICTURES
%%%%%%%% H(TRepRed(v*),Tblue(v**)) adjDD %%%%%%%%
\begin{tikzpicture}[every node/.style={font=\small}, redCirc/.style={circle,fill=red, minimum size=5pt, inner sep=0pt}, blueCirc/.style={circle,fill=blue, minimum size=5pt, inner sep=0pt}]
\node[circle, fill=yellow, opacity=0.4, minimum size=12pt,inner sep=0pt] at (3.5,3.5) {};
\node[circle, fill=yellow, opacity=0.4, minimum size=12pt,inner sep=0pt] at (3.5,0.5) {};
%\draw[step=1, lightgray] (0,0) grid (8,8);

\node[redCirc,label=above:{\textcolor{red}{$v_*$}}] (vStar) at (2,3.5) {};
\node[blueCirc,label=below:{\textcolor{blue}{$v_{**}$}}] (vStarStar) at (2,0.5) {};
\node[blueCirc, label=left:{$(1,4)$}] (1-4) at (1,1) {};
\node[redCirc, label=below:{\textcolor{red}{$w_{**}$}}] (wStarStar) at (3.5,0.5) {};
\node[redCirc, label=left:{$(4,8)$}] (4-8) at (1,3) {};
\node[redCirc, label=above:{\textcolor{red}{$w_{*}$}}] (wStar) at (3.5,3.5) {};
\node[blueCirc, label=below:{$(6,7)$}] (6-7) at (4.7,1.2) {};

\node[redCirc, label=right:{$(3,9)$}] (3-9) at (4.7,3) {};
\node[redCirc, label=right:{$(5,9)$}] (5-9) at (4.7,2) {};
\node[blueCirc, label=right:{$(2,7)$}] (2-7) at (5.7,1) {};

\draw (wStar) -- (vStar);
\draw (wStar) -- (vStarStar);
\draw (wStarStar) -- (vStar);
\draw (wStarStar) -- (vStarStar);
\draw (vStar) -- (vStarStar);
\draw (wStar) -- (wStarStar);

\draw (1-4) -- (4-8);
\draw (1-4) -- (vStar);
\draw (1-4) -- (vStarStar);
\draw (4-8) -- (vStar);
\draw (4-8) -- (vStarStar);
\draw (6-7) -- (vStar);
\draw (6-7) -- (vStarStar);
\draw (6-7) -- (wStar);
\draw (6-7) -- (wStarStar);
\draw (3-9) -- (wStar);
\draw (3-9) -- (wStarStar);
\draw (5-9) -- (wStar);
\draw (5-9) -- (wStarStar);
\draw (3-9) -- (5-9);
\draw (6-7) -- (2-7);

\end{tikzpicture}
\caption{(Left) A joined tree represented by the tuple \eqref{eq:combined_genericPair} that is in the set $\adjDD$. (Right) The corresponding incompatibility graph $H$.}
\label{fig:equalAverageED_CombAdjDD}
\end{figure}
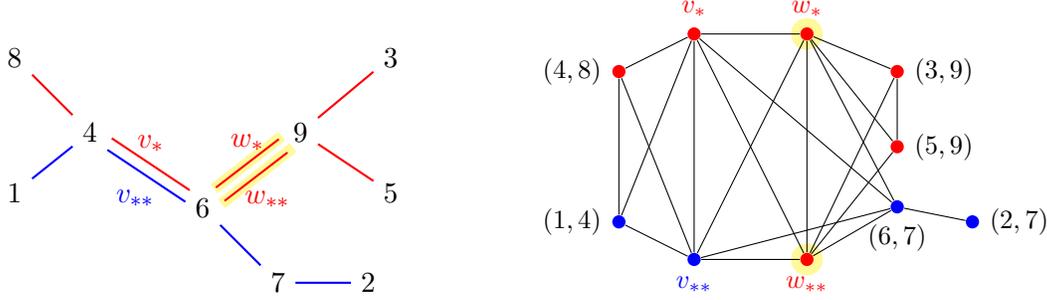

%%%%%%%% COMBINING Deterministic Part AdjDD FIGURES 
\begin{figure}[htp]\centering
%%%%%%%% (Tred(u_r^(1), u_r^(2),Tblue(u_b^(1), u_b^(2)) %%%%%%%%
\begin{tikzpicture}
%\draw[step=1, lightgray] (0,0) grid (8,8);

\draw[yellow, opacity=0.4, line width=0.15cm] (2.65,2.25) -- (3.55,2.95);
\draw[yellow, opacity=0.4, line width=0.15cm] (2.75,2.05) -- (3.7,2.8);

\node (4) at (1,3) {$4$};
\node (6) at (2.5,2) {$6$};
\node (8) at (0,4) {$8$};
\node (1) at (0,2.2) {$1$};
\node (7) at (3.5,1) {$7$};
\node (2) at (4.7,1) {$2$};
\node (9) at (3.8,3) {$9$};
\node (5) at (5,2.2) {$5$};
\node (3) at (5,4) {$3$};

\node [above=of 4, xshift=0.1cm, yshift=-1.1cm] {$\textcolor{red}{u_r^{(1)}}$};
\node [below=of 4, xshift=-0.1cm, yshift=1.1cm] {$\textcolor{blue}{u_b^{(1)}}$};
\node [above=of 9, xshift=0.1cm, yshift=-1.1cm] {$\textcolor{red}{u_r^{(2)}}$};
\node [below=of 9, xshift=0.1cm, yshift=1.1cm] {$\textcolor{blue}{u_b^{(2)}}$};

\draw[thick, red] (4) -- (8);
\draw[thick, blue] (1) -- (4);
\draw[thick, blue] (6) -- (7);
\draw[thick, blue] (2) -- (7);
\draw[thick, red] (3) -- (9);
\draw[thick, red] (5) -- (9);

%\draw[thick, red] (6) -- node[label=below:$v'$, 
\draw[thick, transform canvas={xshift=0.15em, yshift=0.2em}, red] (4) -- node[label={[label distance=-0.3em]90:$v_{*}$}] {} (6);
\draw[thick, transform canvas={xshift=-0em, yshift=-0.2em}, blue] (4) -- node[label=below:$v_{**}$, xshift=-0.15cm, yshift=0.1cm] {} (6);

\draw[thick, transform canvas={xshift=-0.15em, yshift=0.23em}, red] (6) -- node[label={[label distance=-0.3em]90:$w_{*}$}] {} (9);
\draw[thick, transform canvas={xshift=0.15em, yshift=-0.23em}, blue] (6) -- node[label=below:$w_{**}$,xshift=0.15cm, yshift=0.1cm] {} (9);

\end{tikzpicture}
\qquad\qquad % <----------------- SPACE BETWEEN PICTURES
%%%%%%%% H(Tred(u_r^(1), u_r^(2),Tblue(u_b^(1), u_b^(2)) adjDD %%%%%%%%
\begin{tikzpicture}[every node/.style={font=\small}, redCirc/.style={circle,fill=red, minimum size=5pt, inner sep=0pt}, blueCirc/.style={circle,fill=blue, minimum size=5pt, inner sep=0pt}]
\node[circle, fill=yellow, opacity=0.4, minimum size=12pt,inner sep=0pt] at (3.5,3.5) {};
\node[circle, fill=yellow, opacity=0.4, minimum size=12pt,inner sep=0pt] at (3.5,0.5) {};

%\draw[step=1, lightgray] (0,0) grid (8,8);

\node[redCirc,label=above:{\textcolor{red}{$v_*$}}] (vStar) at (2,3.5) {};
\node[blueCirc,label=below:{\textcolor{blue}{$v_{**}$}}] (vStarStar) at (2,0.5) {};
\node[blueCirc, label=left:{$(1,4)$}] (1-4) at (1,1) {};
\node[blueCirc, label=below:{\textcolor{blue}{$w_{**}$}}] (wStarStar) at (3.5,0.5) {};
\node[redCirc, label=left:{$(4,8)$}] (4-8) at (1,3) {};
\node[redCirc, label=above:{\textcolor{red}{$w_{*}$}}] (wStar) at (3.5,3.5) {};
\node[blueCirc, label=below:{$(6,7)$}] (6-7) at (4.7,1.2) {};

\node[redCirc, label=right:{$(3,9)$}] (3-9) at (4.7,3) {};
\node[redCirc, label=right:{$(5,9)$}] (5-9) at (4.7,2) {};
\node[blueCirc, label=right:{$(2,7)$}] (2-7) at (5.7,1) {};

\draw (wStar) -- (vStar);
\draw (wStar) -- (vStarStar);
\draw (wStarStar) -- (vStar);
\draw (wStarStar) -- (vStarStar);
\draw (vStar) -- (vStarStar);
\draw (wStar) -- (wStarStar);

\draw (1-4) -- (4-8);
\draw (1-4) -- (vStar);
\draw (1-4) -- (vStarStar);
\draw (4-8) -- (vStar);
\draw (4-8) -- (vStarStar);
\draw (6-7) -- (vStar);
\draw (6-7) -- (vStarStar);
\draw (6-7) -- (wStar);
\draw (6-7) -- (wStarStar);
\draw (3-9) -- (wStar);
\draw (3-9) -- (wStarStar);
\draw (5-9) -- (wStar);
\draw (5-9) -- (wStarStar);
\draw (3-9) -- (5-9);
\draw (6-7) -- (2-7);

\end{tikzpicture}
\caption{(Left) A joined tree represented by the tuple \eqref{eq:equalAverageED_determPart_genericTuple}. The repeated edges are identified as in \eqref{eq:equalAverageED_determPart_identifyvStarvStarStar}. (Right) The incompatibility graph $H$.}
\label{fig:equalAverageED_DeterministicPartAdjDD}
\end{figure}
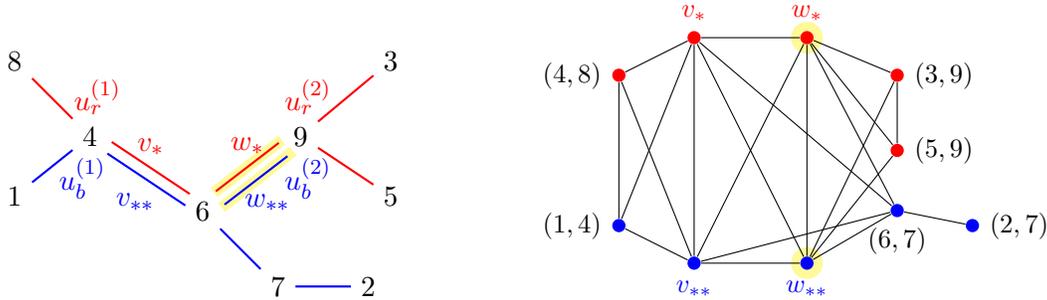

\subsection{\texorpdfstring{Dropping cycles and $\geq 3$ repeated edge subgraphs}{Dropping cycles and >= 3 repeated edge subgraphs}}
In this section we establish Proposition \ref{prop:Matching_remLessThanTwoLogn_mainResult}. This will follow from a series of lemmas: Lemmas \ref{lemma:moreThanThreeRep_small}, \ref{lemma:matchingEdgeDensity_simpleCyclic_small}, and \ref{lemma:matchingEdgeDensity_oneRep_andtwoRep_Cyclic_small} that bound the sub-sums of $\mathsf{rem}_{\leq 2 \log n}$ defined by 
\begin{align}
    \mathsf{rem}_{\leq 2 \log n} = \mathsf{simpleCyclic} + \mathsf{oneRepCyclic} + \mathsf{twoRepCyclic} + \mathsf{moreThanThreeRep},
\end{align}
where $\mathsf{simpleCyclic}$ and $\mathsf{oneRepCyclic}$ are defined exactly as in \eqref{eq:nonMatching_remainderLessThanTwoLogn_decomposition}, and 
\begin{align*}
    \mathsf{twoRepCyclic} &:= \sum_{m = 4}^{2 \log n} \sum_{ \substack{e_1,\dots,e_m \\ \text{only two rep.~edge,} \\ \text{contains cycle}  } } \phi(H(e_1,\dots,e_m)) \lambda^m \insquare{ \prod_{j=1}^{m} \frac{A_{e_j}}{p} - 1  }, \\
    \mathsf{moreThanThreeRep} &:= \sum_{m = 4}^{2 \log n} \sum_{\substack{ e_1,\dots,e_m \\ \geq 3 \text{ repeated edges}  }} \phi(H(e_1,\dots,e_m)) \lambda^m \insquare{ \prod_{j=1}^{m} \frac{A_{e_j}}{p} - 1  } .
\end{align*}

\begin{lemma}\label{lemma:moreThanThreeRep_small}
We have 
    $\mathsf{moreThanThreeRep} = O_\PP\!\inparen{\frac{1}{n^2q^3}}$.
\end{lemma}

\begin{proof}[Proof of Lemma \ref{lemma:moreThanThreeRep_small}]
    Decompose $\mathsf{moreThanThreeRep}$ into the sum of $\mathsf{moreThanThreeRepRandom}$ and $\mathsf{moreThanThreeRepDeterministic}$ where 
    \begin{align*}
        \mathsf{moreThanThreeRepRandom} := \sum_{m \geq 4} \sum_{\substack{ e_1,\dots,e_m \\ \geq 3 \text{ repeated edges}  }} \phi(H(e_1,\dots,e_m)) \prod_{j=1}^{m} \frac{\lambda}{p} A_{e_j} , \\
        \mathsf{moreThanThreeRepDeterministic} := \sum_{m \geq 4} \sum_{\substack{ e_1,\dots,e_m \\ \geq 3 \text{ repeated edges}  }} \phi(H(e_1,\dots,e_m)) \lambda^m.
    \end{align*}
    We bound 
    \begin{align}
        &\abs{ \mathsf{moreThanThreeRepRandom} } \nonumber \\
        & \leq  \sum_{m \geq 4} \sum_{\substack{ e_1,\dots,e_m \\ e_{i_1} = e_{i_2} = e_{i_3} = e_{i_4}  }} \abs{\phi(H(e_1,\dots,e_m))} \prod_{j=1}^{m} \frac{\lambda}{p} A_{e_j} + \sum_{m \geq 4} \sum_{\substack{ e_1,\dots,e_m \\ e_{i_1} = e_{i_2} = e_{i_3}, \\ e_{j_1} = e_{j_2}  }} (\cdots) + \sum_{m \geq 4} \sum_{\substack{ e_1,\dots,e_m \\ e_{i_1} = e_{i_2}, \\ e_{j_1} = e_{j_2}, e_{k_1} = e_{k_2} }} (\cdots),
        \label{eq:moreThanThreeRepRandom_decomposition}
    \end{align}
    where the sum constraint in the first term on RHS means there is a same edge appearing four times. In the second term there is an edge repeated three times and another edge repeated two times. Similarly for the last term. The summand $(\cdots)$ is the same for all terms.

    Apply the Penrose tree-graph bound similarly as in \eqref{eq:nonMatching_moreThanTwoRepRandom_ApplyPenroseTreeBound}, modifying the argument to have $\abs{V'} = 4$ instead of $\abs{V'} = 3$. We obtain 
    \begin{align*}
        &\sum_{m \geq 4} \sum_{\substack{ e_1,\dots,e_m \\ e_{i_1} = e_{i_2} = e_{i_3} = e_{i_4}  }} \abs{\phi(H(e_1,\dots,e_m))} \prod_{j=1}^{m} \frac{\lambda}{p} A_{e_j} \leq \sum_{m \geq 4} \frac{1}{m!} \frac{\lambda^m}{p^m} \sum_{t \in \cT_{m-1}^{\text{lab}}} K_2(A) \binom{m}{4} \inparen{2\Delta(A) - 1}^{m - 4} .%\label{eq:MatchingEdgeDensity_moreThanThreeRepRandom_PenroseTreeBound}
    \end{align*}
    The slight difference now from the argument in Lemma \ref{lemma:nonMatching_moreThanTwoRep} is that now $A \sim G(n,q)$ instead of $A \sim G(n,p)$. This does not present much additional difficulty since by hypothesis $1.01p \geq q \geq \frac{9 \log n}{n}$. Thus, similarly with probability at least $1 - \inparen{\frac{1}{n}}$, we have
    \begin{align*}
        \sum_{m \geq 4} \sum_{\substack{ e_1,\dots,e_m \\ e_{i_1} = e_{i_2} = e_{i_3} = e_{i_4}  }} \abs{\phi(H(e_1,\dots,e_m))} \prod_{j=1}^{m} \frac{\lambda}{p} A_{e_j} &\leq \frac{C}{n^2 q^3} \sum_{m=4}^{2 \log n} \inparen{\frac{q}{p}}^m m^2 (4.04e n \lambda)^m \\
        &\leq \frac{C}{n^2 q^3} \insquare{  \sum_{m=4}^{2 \log n}  m^2 (4.04 e n  \lambda)^m  + \frac{C}{np} \sum_{m=4}^{2 \log n}  m^3 (4.04 e n \lambda)^m }
    \end{align*}
    where the second inequality follows from \eqref{eq:identities_pq_(q/p)^r} which shows that for $4 \leq m \leq 2 \log n$,
    \begin{align}
        \frac{q^m}{p^m} \leq 1 +  C m \frac{c_n}{n} \frac{1-p}{p}
        \label{eq:qmpm_expandToZerothOrder}
    \end{align}
    for some universal constant $C$. By hypothesis, $\abs{4.04e n \lambda} < \frac{1}{e}$ so that the above (derivatives of) geometric series converges.

    The other terms in \eqref{eq:moreThanThreeRepRandom_decomposition} can be bounded by straightforward modifications of the proof in Lemma \ref{lemma:nonMatching_moreThanTwoRep}, with similar modifications for the $(q/p)^m$ factor as above.

    The bound for $\mathsf{moreThanThreeRepDeterministic}$ follows almost identically. We only have to replace every instance of the random $\Delta(A)$ and $\abs{A}$ with the deterministic $\Delta(K_n) = n-1$ and $\abs{K_n} = \binom{n}{2}$ respectively.  
\end{proof}

\begin{lemma}\label{lemma:matchingEdgeDensity_simpleCyclic_small}
We have 
    $\mathsf{simpleCyclic} = O_\PP\!\inparen{\frac{1}{nq}}$.
\end{lemma}

\begin{proof}[Proof of Lemma \ref{lemma:matchingEdgeDensity_simpleCyclic_small}]
Recall that $G_{N,m}$ denotes a generic unlabeled connected simple graph with $N$ vertices and $m$ edges. We also write $G_{N,m}(A)$ for the number of copies of $G_{N,m}$ appearing in the graph $A$. We have 
\begin{align*}
    \EE\insquare{\mathsf{simpleCyclic}} = \sum_{N = 3}^{2 \log n} \sum_{m = N}^{\binom{N}{2} \wedge 2 \log n} \sum_{G_{N,m}} m! \phi\inparen{ H(G_{N,m})  } \lambda^m G_{N,m}(K_n) \insquare{ \frac{q^m}{p^m} - 1  }.
\end{align*}
Applying \eqref{eq:qmpm_expandToZerothOrder}, we have 
\begin{align*}
    \abs{ \EE\insquare{\mathsf{simpleCyclic}} } \leq \frac{C}{np} \underbrace{ \sum_{N = 3}^{2 \log n} \sum_{m = N}^{\binom{N}{2} \wedge 2 \log n} \sum_{G_{N,m}} m! \abs{ \phi\inparen{ H(G_{N,m}) } } \lambda^m G_{N,m}(K_n) m}_{\leq C},
\end{align*}
where the sum is bounded by Proposition \ref{prop:MDMean_cycles_O(1)}. Thus $\EE\insquare{\mathsf{simpleCyclic}} = O\inparen{  \frac{1}{np}  } = O\inparen{  \frac{1}{nq}  }$.

On the other hand, by straightforward modifications of the proof of Lemma \ref{lemma:matchingEdgeDensity_simpleCyclic_small}, in particular using $q$ instead of $p$ in Claim \ref{claim:Var_G_N_m}, we will obtain
\begin{align*}
    \sqrt{\Var \mathsf{simpleCyclic}} \leq \frac{C}{n} \sqrt{\frac{1-q}{q}}  \underbrace{\sum_{N = 3}^{2 \log n}  \sum_{m = N}^{ \binom{N}{2} \wedge 2 \log n } \sum_{\substack{e_1,\dots,e_m \\ \text{$N$ vertices, $e_i$'s distinct}}} m^2 \abs{ \phi(H(e_1,\dots,e_m)) } \lambda^m}_{\leq C}.
\end{align*}
The sum is bounded by a slight modification of the proof of Proposition \ref{prop:MDMean_cycles_O(1)}, where we note that in \eqref{eq:MDMean_cycles_O(1)_afterCayley}, there was a ``spare'' factor of $1/(m+r)$ which will handle the additional factor of $m$ in the last line above. This establishes that $\Var\insquare{\mathsf{simpleCyclic}} \leq C/(n^2 q)$. Together with the bound on $\EE\insquare{\mathsf{simpleCyclic}}$, the proof is complete.
\end{proof}

\begin{lemma}\label{lemma:matchingEdgeDensity_oneRep_andtwoRep_Cyclic_small} 
We have $\mathsf{oneRepCyclic} = O_\PP\!\inparen{\frac{1}{nq}}$ and $\mathsf{twoRepCyclic} = O_\PP\!\inparen{\frac{1}{n^2q^2}}$.
\end{lemma}

\begin{proof}[Proof of Lemma \ref{lemma:matchingEdgeDensity_oneRep_andtwoRep_Cyclic_small}]
The proof is largely similar to that of Lemma \ref{lemma:nonMatching_oneRepCyclic_small}. We first decompose $\mathsf{oneRepCyclic}$ into the difference of $\mathsf{oneRepCyclicRandom}$ and $\mathsf{oneRepCyclicDeterministic}$.

By similar arguments that lead to \eqref{eq:nonMatching_oneRepCyclicRandom_afterPenroseCayley}, except here $A \sim G(n,q)$ where $q \neq p$, we obtain that with probability at least $1 - \frac{1}{n}$, 
\begin{align}\label{eq:matchingEdgeDensity_oneRepCyclicRandom_bound}
    \abs{\mathsf{oneRepCyclicRandom}} &\leq \frac{1}{6\Delta q} \sum_{m=1}^{2 \log n - 3} \sum_{r=3}^{2 \log n - m}  (4.04 e \lambda n)^{m+r} \underbrace{ \frac{q^{m+r}}{p^{m+r}} }_{\leq C} \binom{m+r}{r} \leq \frac{C}{nq}
\end{align}
where we have used \eqref{eq:qmpm_expandToZerothOrder} to bound $(q/p)^{m+r}$, and where the final inequality follows by similar arguments as in \eqref{eq:nonMatching_oneRepCyclicRandom_afterPenroseCayley_final}. This implies $\abs{\mathsf{oneRepCyclicRandom}} = O_\PP\!\inparen{\frac{1}{nq}}$.

An almost identical argument will show that $\mathsf{oneRepCyclicDeterministic} = O\!\inparen{\frac{1}{n}}$. We only have to replace random quantities with their deterministic counterparts as outlined in the proof of Lemma \ref{lemma:nonMatching_oneRepCyclic_small}.

Only small modifications of the above argument are needed for $\mathsf{twoRepCyclic}$. We similarly decompose into ``random'' and ``deterministic'' parts. In the former we bound 
\begin{align*}
    \abs{ \mathsf{twoRepCyclicRandom} } &\leq \sum_{m = 1}^{2 \log n - 3} \sum_{r = 3}^{2 \log n - m} \sum_{\substack{e_1,\dots,e_{m+r} \\ G \supseteq C_r \\  e_{i_1} = e_{i_2} = e_{i_3}  }  } \abs{ \phi(H(e_1,\dots,e_{m+r})) } \frac{\lambda^{m+r}}{p^{m+r}} \prod_{j=1}^{m+r} A_{e_j} \\
    &\qquad + \sum_{m = 1}^{2 \log n - 3} \sum_{r = 3}^{2 \log n - m} \sum_{\substack{e_1,\dots,e_{m+r} \\ G \supseteq C_r \\  e_{i_1} = e_{i_2}, \, e_{j_1} = e_{j_2}  }  } (\cdots),
\end{align*}
where the summand is the same for both terms. The first term on the RHS is bounded by $C/(n^2 q^2)$ by a small modification of the arguments that led to \eqref{eq:matchingEdgeDensity_oneRepCyclicRandom_bound}. Here, in Step 3 in the proof of Lemma \ref{lemma:nonMatching_oneRepCyclic_small}, we choose two edges in $t$ to correspond to the links between the (in total 3) repeated polymers. There are at most $\binom{m+r - 1}{2}$ ways to do this. Consequently, this introduces an additional factor of at most $m + r - 1$, but this can be handled by the ``spare'' $1/(m+r)$ factor in \eqref{eq:nonMatching_oneRepCyclicRandom_afterPenroseCayley_beforeFinal}. Additionally, we gain a factor of $1/(nq)$ because of the additional repeated edge. Altogether this leads to the claimed bound.

By analogous arguments, the second term on the RHS of above display is also bounded by $C/(n^2q^2)$. The deterministic part of $\mathsf{twoRepCyclic}$ can be shown to be $O\!\inparen{1/n^2}$. 
\end{proof}

\section{Proofs for planted perfect matching}
\label{sec:lambda_infty}

\begin{proof}[Proof of Theorem~\ref{thm:perfect-matching-p=q}]
The proofs of most statements in Theorems~\ref{thm:threshold-edge-count} and~\ref{thm:plantedMatchings_mean_is_minus_half_variance} also work for the case $\lambda = \infty$ and $c=1$. It suffices to show the asymptotic normality of the log-likelihood ratio under $\cQ$. 

The likelihood ratio satisfies %\notecm{Is the exponent of $(1-p)$ and $(1-q)$ supposed to be $1-A_{ij}$?} Corrected!
\begin{align}
    \frac{\ud \cP_\infty}{\ud \cQ}(A) = \EE_{M} \prod_{\inbraces{i,j} \in M} \frac{1}{q^{A_{ij}}} \prod_{\inbraces{i,j} \notin M} \frac{p^{A_{ij}} (1-p)^{1 - A_{ij}} }{ q^{A_{ij}} (1-q)^{1 - A_{ij}} } \boldsymbol{1}\!\inbraces{M \subset A}.
    \label{eq:Pinfty_planted_null_likelihoodRatio}
\end{align}
Setting $p = q$ in \eqref{eq:Pinfty_planted_null_likelihoodRatio}, we obtain 
\begin{equation*}\label{eq:likelihoodratio_perfectMatching_nonMatchingEdges}
    \frac{\ud \cP_\infty}{\ud \cQ}(A) = \frac{ \PP\insquare{M \subset A} }{q^{n/2}} = \frac{M(A)}{M(K_n) \cdot q^{n/2}} ,
\end{equation*}
where $M(G)$ denotes the number of perfect matchings in graph $G$. 
% (in line with notation \eqref{eq:ordinaryCenteredSignedSubgraphCount_definition}). 
% Note that $M(K_n) = (n-1)!!$.
We apply the result of \cite{janson1994numbers} Theorem~4 Equation (1.27) which states that (in our notation): for $A \sim \cQ$,
    \begin{equation*}
        \log \frac{M(A)}{\EE_{A 
        \sim \cQ} M(A) } \overset{d}{\longrightarrow} \cN\inparen{ - \frac{1-p}{4p}, \frac{1-p}{2p}  }.
    \end{equation*} 
The asymptotic normality of $\log \frac{\ud \cP_\infty}{\ud \cQ}(A)$ under $\cQ$ thus follows by combining the above two results. 
% \eqref{eq:likelihoodratio_perfectMatching_nonMatchingEdges}. 
\begin{comment}
The corresponding statement for $A \sim \cP_\infty$ is deduced from Le Cam's third lemma by considering the limiting joint distribution of $\inparen{\log \frac{\ud \cP_\infty}{\ud \cQ}, \log \frac{\ud \cP_\infty}{\ud \cQ}}$ under $\cQ$ as in \cite{van2000asymptotic} Example 6.7.
For the second statement, we use that 
\begin{align*}
    \TV(\cP_\infty,\cQ) = 1 - \PP_{A \sim \cP_\infty}\!\insquare{ \frac{\ud \cP_\infty}{\ud \cQ} (A) \leq 1  } - \PP_{A \sim \cQ}\!\insquare{ \frac{\ud \cP_\infty}{\ud \cQ} (A) \geq 1  }.
\end{align*}
The result follows since $\PP_{A \sim \cP_\infty}\!\insquare{ \log \frac{\ud \cP_\infty}{\ud \cQ} (A) \leq 0  }$ and $\PP_{A \sim \cQ}\!\insquare{ \log \frac{\ud \cP_\infty}{\ud \cQ} (A) \geq 0  }$ both converge to the same limit $\Phi\!\inparen{ - \sqrt{\frac{1-p}{8p}}  } = \frac{1}{2}\mathsf{erfc}\inparen{\sqrt{\frac{1-p}{16p}}}$. \notetim{Depending on the ordering the paper let's only state the Le Cam and TV part formally once when it first appears.}
\end{comment}
\end{proof}

\begin{proof}[Proof of Theorem~\ref{thm:perfect-matching-p-not-equal-q}]
% [Proof of Theorem \ref{thm:Pinfty_matchingEdgeDensities_mainResult}]
The proofs of most statements in Theorems~\ref{thm:matchingEdgeDensities_compEfficientStat_mainResult} and~\ref{thm:matchingEdgeDensities_mean_is_minus_half_variance} also work for the case $\lambda = \infty$ and $c=1$. It suffices to show the asymptotic normality of the log-likelihood ratio under $\cQ$. 

From \eqref{eq:Pinfty_planted_null_likelihoodRatio}, we can manipulate the likelihood ratio to be
\begin{align}\label{eq:likelihoodratio_MA_form}
    \frac{\ud \cP_\infty}{\ud \cQ}(A) = \left.  M(A) \inparen{ \frac{p(1-q)}{q(1-p)}  }^{K_2(A) - \frac{n}{2}} \middle/ q^{\frac{n}{2}} M(K_n) \inparen{ \frac{1 - p}{1 - q}  }^{\binom{n}{2} - \frac{n}{2}}    \right. .
\end{align}
Let $A \sim \cQ$. On the other hand, Equation (4.29) from \cite{janson1994numbers} states (note that `$c$' there translates into $\theta/2$ for us), 
 \begin{equation*}
    \log \frac{M(A)(1-a)^{K_2(A) - \frac{n}{2}}}{\EE_{A \sim \cQ}\!\insquare{M(A)(1-a)^{K_2(A) - \frac{n}{2}} }} \overset{d}{\longrightarrow} \cN\!\inparen{ -\frac{\tau^2}{4\theta^2}, \frac{\tau^2}{2\theta^2}  },
\end{equation*}
where $a := \frac{n/2}{\binom{n}{2}q}$, and $\tau$ is defined as the limit $ n^2( \kappa(P_2; M) - \kappa(K_2; M)  ) \rightarrow \tau$, where for any fixed (labeled) subgraph $G$, $\kappa(G; M)$ is the ratio of the number of perfect matchings containing $G$ to the number of perfect matchings in $K_n$. One computes that $\kappa(K_2; M) = (n/2)/\binom{n}{2} = 1/(n-1)$ and $\kappa(P_2; M) = 0$, yielding that $\tau = -1$. Furthermore, observe that
\begin{equation*}
    1 - a = \frac{p(1-q)}{q(1-p)}.
\end{equation*}
Thus,
\begin{align*}
    &\EE\!\insquare{M(A)(1-a)^{K_2(A) - \frac{n}{2}} } = \EE\!\insquare{ \inparen{ \sum_{M} \prod_{\inbraces{i,j}\in M} A_{ij}  } \inparen{\frac{p(1-q)}{q(1-p)}}^{K_2(A) - \frac{n}{2}} } \\
    &\quad= \EE\!\insquare{ \inparen{ \sum_{M} \prod_{\inbraces{i,j} \in \alpha} A_{ij} \prod_{\inbraces{i,j} \notin \alpha} \inparen{\frac{p(1-q)}{q(1-p)}}^{A_{ij}}  }  } \\
    &\quad= \sum_{M} \EE\!\insquare{ \prod_{\inbraces{i,j} \in \alpha} A_{ij} \prod_{\inbraces{i,j} \notin \alpha} \inparen{\frac{p(1-q)}{q(1-p)}}^{A_{ij}}   } = M(K_n) q^{n/2} \inparen{ \frac{1-q}{1-p}  }^{\binom{n}{2} - \frac{n}{2}}.
\end{align*}
Combining the above results finishes the proof.
% This finishes the proof of statement (i) for $A \sim \cQ$. The corresponding statement for $A \sim \cP_\infty$ is deduced from Le Cam's third lemma by considering the limiting joint distribution of $\inparen{\log \frac{\ud \cP_\infty}{\ud \cQ}, \log \frac{\ud \cP}{\ud \cQ}}$ under $\cQ$ (see \cite{van2000asymptotic} Example 6.7). Finally, we use that 
% \begin{align*}
%     \TV(\cP_\infty,\cQ) = 1 - \PP_{A \sim \cP_\infty}\!\insquare{ \frac{\ud \cP_\infty}{\ud \cQ} (A) \leq 1  } - \PP_{A \sim \cQ}\!\insquare{ \frac{\ud \cP}{\ud \cQ} (A) \geq 1  }.
% \end{align*}
% The second statement follows since $\PP_{A \sim \cP_\infty}\!\insquare{ \log \frac{\ud \cP_\infty}{\ud \cQ} (A) \leq 0  }$ and $\PP_{A \sim \cQ}\!\insquare{ \log \frac{\ud \cP_\infty}{\ud \cQ} (A) \geq 0  }$ both converge to the same limit $\Phi\!\inparen{ - \frac{1}{2\sqrt{2}\theta}}$. \notetim{Depending on the ordering the paper let's only state the Le Cam and TV part formally once when it first appears.}
\end{proof}

%%%%%% Leave this out in final product %%%%%%
\begin{comment}
\notetim{\eqref{eq:planted-PM-matchEdges-JansonApp-likelihood-ratio-limit_and_TV-limit} comes about from 
\begin{align*}
    \log \frac{\ud \cP}{\ud \cQ}(A) \simeq_{\cQ} \cN\!\inparen{  -\frac{\tau^2}{4\theta^2}, \frac{\tau^2}{2\theta^2}  },
\end{align*}
where $n^2(\gamma(P_2) - \gamma(K_2)^2) \longrightarrow \tau$, where $\gamma(\beta) = \cN_{M}(\beta) / \cN_M $. Here $\gamma(K_2) = \mu/\binom{n}{2} = \frac{1}{n-1}$ ($\mu = n/2$ size of the planted component) and $\gamma(P_2) = 0$, and $\cN_M = (n-1)!!$. Hence $\tau = -1$. In planted monomer-dimer case, replacing all $\mu$ by $\EE\abs{M}$, this should translate into
\begin{align}\label{eq:Janson_Gnm_loglikelihood_original_form}
    \log \frac{\ud \cP}{\ud \cQ}(A) \simeq_{\cQ} \cN\!\inparen{  -\frac{1}{4\theta^2} \inparen{\frac{2 \EE \abs{M}}{n} }^4, \frac{1}{2\theta^2}  \inparen{\frac{2 \EE \abs{M}}{n} }^4 }
\end{align}
}
\end{comment}

\end{document}